\documentclass[nosumlimits,twoside]{amsart}
\usepackage{amsfonts, amsmath, amssymb}
\usepackage[bookmarksnumbered,plainpages,hypertex]{hyperref}
\usepackage{graphicx}
\usepackage{float}
\usepackage{srcltx}
\usepackage{version}
\usepackage[all]{xy}
\usepackage[usenames]{color}
\usepackage{cancel}
\usepackage[normalem]{ulem}

\newtheorem{theorem}{\sc Theorem}[section]
\newtheorem{proposition}[theorem]{\sc Proposition}

\newtheorem{lemma}[theorem]{\sc Lemma}
\newtheorem{corollary}[theorem]{\sc Corollary}

\theoremstyle{definition}
\newtheorem{definition}[theorem]{\sc Definition}

\newtheorem{example}[theorem]{\sc Example}

\theoremstyle{remark}
\newtheorem{remark}[theorem]{\sc Remark}

\newtheorem{claim}[theorem]{}

\def\yd{_{H}^{H}\mathcal{YD}}

\newcommand{\ydg}{^{\Bbbk\Gamma}_{\Bbbk\Gamma}\mathcal{YD}}

\def\cB{\mathcal{B}}
\def\cD{\mathcal{D}}

\newcommand\N{\mathbb{N}}

\newcommand\ad{\operatorname{ad}}
\newcommand\Gr{\operatorname{Gr}}
\def\bq{\mathfrak{q}}
\def\bx{\mathbf{x}}

\allowdisplaybreaks
\IfFileExists{tcilatex.tex}{\input{tcilatex}}{}
\newenvironment{invisible}{{\noindent\sc \underline{\color{Blue}Invisible (To be hiden)}:\quad}\color{Red}}{\medskip}
\excludeversion{invisible}

\begin{document}
\title{Cohomology and Coquasi-bialgebras in the category of Yetter-Drinfeld
Modules}
\author{Iv\'{a}n Angiono}
\address{%
\parbox[b]{\linewidth}{FaMAF-CIEM (CONICET), Universidad Nacional
de C\'ordoba, Medina A\-llen\-de s/n, Ciudad Universitaria, 5000 C\' ordoba,
Argentina.}}
\email{angiono@famaf.unc.edu.ar}
\urladdr{\url{www.mate.uncor.edu/~angiono}}
\author{Alessandro Ardizzoni}
\address{%
\parbox[b]{\linewidth}{University of Turin, Department of Mathematics ``G. Peano'', via
Carlo Alberto 10, I-10123 Torino, Italy}}
\email{alessandro.ardizzoni@unito.it}
\urladdr{\url{sites.google.com/site/aleardizzonihome}}
\author{Claudia Menini}
\address{%
\parbox[b]{\linewidth}{University of Ferrara, Department of Mathematics and Computer Science, Via Machiavelli
35, Ferrara, I-44121, Italy}}
\email{men@unife.it}
\urladdr{\url{sites.google.com/a/unife.it/claudia-menini}}
\subjclass[2010]{Primary 16W30; Secondary 16E40}
\thanks{This paper was started during the visit, supported by INdAM, of the
first author to the University of Ferrara in November 2012. It was written
while both the second and the third authors were members of GNSAGA. The
first author's work was partially supported by CONICET, FONCyT-ANPCyT, Secyt
(UNC). The second author was partially supported by the research grant
``Progetti di Eccellenza 2011/2012'' from the ``Fondazione Cassa di
Risparmio di Padova e Rovigo''.}

\begin{abstract}
We prove that a finite-dimensional Hopf algebra with the dual Chevalley Property over a field of characteristic zero is quasi-isomorphic to a Radford-Majid bosonization whenever the third Hochschild cohomology group in the category of Yetter-Drinfeld modules of its diagram with coefficients in the base field vanishes. Moreover we show that this vanishing occurs in meaningful examples where the diagram is a Nichols algebra.
\end{abstract}
\keywords{Hopf Algebras, Coquasi-bialgebras, Bosonizations, Cocycle Deformations, Hochschild Cohomology.}
\maketitle
\tableofcontents

\section*{Introduction}

Let $A$ be a finite-dimensional Hopf algebra over a field $\Bbbk $ of characteristic zero such that the coradical $H$ of $A$ is a
sub-Hopf algebra (i.e. $A$ has the dual Chevalley Property). Denote by $\mathcal{D}\left( A\right)$ the diagram of $A$. The main aim of this paper (see Theorem \ref{teo:main}) is to prove that, if the third
Hochschild cohomology group in ${_{H}^{H}\mathcal{YD}}$ of the algebra $\mathcal{D}\left( A\right)$
with coefficients in $\Bbbk $ vanishes, in symbols $\mathrm{H}_{{\mathcal{YD}}}^{3}\left( \mathcal{D}\left( A\right) ,\Bbbk \right) =0$,
then $A$ is quasi-isomorphic to the Radford-Majid bosonization $E\#H$ of
some connected bialgebra $E$ in ${_{H}^{H}\mathcal{YD}}$ with $%
\mathrm{gr}\left( E\right) \cong \mathcal{D}\left( A\right) $ as bialgebras
in ${_{H}^{H}\mathcal{YD}}$.

The paper is organized as follows.
Let $H$ be a Hopf algebra over a field $\Bbbk$. In Section \ref{sec:1} we investigate the properties of coalgebras with multiplication and unit in the category ${_{H}^{H}\mathcal{YD}}$ (in particular of coquasi-bialgebras) and their associated graded coalgebra. The main result of this section, Theorem \ref{teo:grHopf}, establishes  that the associated graded coalgebra $\mathrm{gr}Q$ of a connected coquasi-bialgebra in ${_{H}^{H}%
\mathcal{YD}}$ is  a connected bialgebra in ${_{H}^{H}%
\mathcal{YD}}$.

In Section \ref{sec:2} we study the deformation of coquasi-bialgebras in ${_{H}^{H}\mathcal{YD}}$ by means of gauge transformations. In Proposition \ref{pro:deformSmash} we investigate its behaviour with respect to bosonization while in Proposition \ref{pro:grgaugeYD} with respect to the associated graded coalgebra.

In Section \ref{sec:3} we consider the associated graded coalgebra in case the Hopf algebra $H$ is semisimple and cosemisimple (e.g. $H$ is finite-dimensional cosemisimple over a field of characteristic zero). In particular, in Theorem \ref{teo:GelakiYD}, we prove that  a f.d. connected coquasi-bialgebra $Q$ in ${_{H}^{H}\mathcal{YD}}$ is gauge equivalent to a connected bialgebra in ${_{H}^{H}\mathcal{YD}}$ whenever $\mathrm{H}_{{\mathcal{YD}}}^{3}\left( \mathrm{gr}Q,\Bbbk \right) =0$. This result is inspired by \cite[Proposition 2.3]{EG}.

In Section \ref{sec:4}, we focus on the link between $\mathrm{H}_{{\mathcal{YD}}}^{n}\left( B,\Bbbk \right)$ and the invariants of $\mathrm{H} ^{n}\left( B,\Bbbk \right)$, where $B$ is a bialgebra in $\mathrm{H}_{{\mathcal{YD}}}^{n}\left( B,\Bbbk \right)$. In particular, in Proposition \ref{pro:D(H)} we show that $
\mathrm{H}_{{\mathcal{YD}}}^{n}\left( B,\Bbbk \right)$ is isomorphic to $\mathrm{H}^{n}\left( B,\Bbbk \right) ^{D(H)}$, which is a subspace of $\mathrm{H}^{n}\left( B,\Bbbk \right) ^{H}\cong\mathrm{H}^{n}\left( B\#H,\Bbbk \right),$ see Corollary \ref{coro:K}.

Section \ref{sec:5} is devoted to the proof of the main result of the paper, the aforementioned Theorem \ref{teo:main}.

In Section \ref{sec:6} we provide examples where $
\mathrm{H}_{{\mathcal{YD}}}^{n}\left( B,\Bbbk \right)=0$ in case $B$ is the Nichols algebra $\cB(V)$ of a Yetter-Drinfeld module $V$. In particular we show that that $\mathrm{H}_{{\mathcal{YD}}}^{3}\left( \cB(V),\Bbbk \right)$ can be zero
although $\mathrm{H}^3\left( \cB(V)\#H,\Bbbk \right)$ is non-trivial.

\section*{Preliminaries}

Given a category ${\mathcal{C}}$ and objects $M,N\in {\mathcal{C}}$, the
notation ${\mathcal{C}}\left( M,N\right) $ stands for the set of morphisms
in ${\mathcal{C}}$. This notation will be mainly applied to the case ${%
\mathcal{C}}$ is the category of vector space $\mathbf{Vec}_{\Bbbk }$ over a
field $\Bbbk $ or ${\mathcal{C}}$ is the category of Yetter-Drinfeld modules
${_{H}^{H}\mathcal{YD}}$ over a Hopf algebra $H$. The set of natural numbers including $0$ is denoted by $\N_0$ while $\N$ denotes the same set without $0$.

\section{Yetter-Drinfeld}\label{sec:1}

\begin{definition}
Let $C$ be a coalgebra. Denote by $C_{n}$ the $n$-th term of the coradical
filtration of $C$ and set $C_{-1}:=0.$ For every $x\in C,$ we set%
\begin{equation*}
\left\vert x\right\vert :=\min \left\{ i\in \N_0:x\in C_{i}\right\}
\qquad \text{and}\qquad \overline{x}:=x+C_{\left\vert x\right\vert -1}.
\end{equation*}%
Note that, for $x=0$, we have $\left\vert x\right\vert =0.$ One can define
the associated graded coalgebra%
\begin{equation*}
\mathrm{gr}C:=\oplus _{i\in \N_0}\frac{C_{i}}{C_{i-1}}
\end{equation*}%
with structure given, for every $x\in C$, by%
\begin{eqnarray}
\Delta _{\mathrm{gr}C}\left( \overline{x}\right) &=&\sum_{0\leq i\leq
\left\vert x\right\vert }\left( x_{1}+C_{i-1}\right) \otimes \left(
x_{2}+C_{\left\vert x\right\vert -i-1}\right) ,  \label{form:DeltaGr} \\
\varepsilon _{\mathrm{gr}C}\left( \overline{x}\right) &=&\delta _{\left\vert
x\right\vert ,0}\varepsilon _{C}\left( x\right) .  \label{form:EpsGr}
\end{eqnarray}
\end{definition}

\begin{claim}
\label{claim:basis}For every $i\in \N_0$, take a basis $\left\{
\overline{x^{i,j}}\mid j\in B_{i}\right\} $ of the $\Bbbk $-module $%
C_{i}/C_{i-1}\ $with $\overline{x^{i,j}}\neq \overline{x^{i,l}}$ for $j\neq
l $ and
\begin{equation*}
\left\vert x^{i,j}\right\vert =i.
\end{equation*}%
Then $\left\{ x^{i,j}\mid 0\leq i\leq n,j\in B_{i}\right\} $ is a basis of $%
C_{n}$ and $\left\{ x^{i,j}\mid i\in \N_0,j\in B_{i}\right\} $ is a
basis of $C$. Assume that $C$ has a distinguished grouplike element $%
1=1_{C}\neq 0$ and take $i>0.$ If $\varepsilon \left( x^{i,j}\right) \neq 0$
then we have that
\begin{equation*}
\overline{x^{i,j}-\varepsilon \left( x^{i,j}\right) 1}=\overline{x^{i,j}}
\end{equation*}%
so that we can take $x^{i,j}-\varepsilon \left( x^{i,j}\right) 1$ in place
of $x^{i,j}.$ In other words we can assume
\begin{equation}
\varepsilon \left( x^{i,j}\right) =0,\text{ for every }i>0,j\in B_{i}.
\label{form:XijNorm}
\end{equation}%
It is well-known there is a $\Bbbk $-linear isomorphism $\varphi
:C\rightarrow \mathrm{gr}C$ defined on the basis by $\varphi \left(
x^{i,j}\right) :=\overline{x^{i,j}}.$

We compute%
\begin{equation*}
\varepsilon _{\mathrm{gr}C}\varphi \left( x^{i,j}\right) =\varepsilon _{%
\mathrm{gr}C} \left( \overline{x^{i,j}}\right) \overset{(\ref{form:EpsGr})}{=%
}\delta _{i,0}\varepsilon \left( x^{0,j}\right) \overset{(\ref{form:XijNorm})%
}{=}\varepsilon \left( x^{i,j}\right) .
\end{equation*}%
Hence we obtain%
\begin{equation}
\varepsilon _{\mathrm{gr}C}\circ \varphi =\varepsilon .  \label{form:Gel1}
\end{equation}
\end{claim}

Let $H$ be a Hopf algebra. A \textbf{coalgebra with multiplication and unit}
in ${_{H}^{H}\mathcal{YD}}$ is a datum $\left( Q,m,u,\Delta ,\varepsilon
\right) $ where $\left( Q,\Delta ,\varepsilon \right) $ is a coalgebra in ${%
_{H}^{H}\mathcal{YD}}$, $m:Q\otimes Q\rightarrow Q$ is a coalgebra morphism
in ${_{H}^{H}\mathcal{YD}}$ called multiplication (which may fail to be
associative) and $u:\Bbbk \rightarrow Q$ is a coalgebra morphism in ${%
_{H}^{H}\mathcal{YD}}$ called unit. In this case we set $1_{Q}:=u\left(
1_{\Bbbk }\right) .$

Note that, for every $h\in H,k\in \Bbbk $, we have%
\begin{eqnarray}
h1_{Q} &=&hu\left( 1_{\Bbbk }\right) =u\left( h1_{\Bbbk }\right) =u\left(
\varepsilon _{H}\left( h\right) 1_{\Bbbk }\right) =\varepsilon _{H}\left(
h\right) u\left( 1_{\Bbbk }\right) =\varepsilon _{H}\left( h\right) 1_{Q},
\label{form:1Qlin} \\
\left( 1_{Q}\right) _{-1}\otimes \left( 1_{Q}\right) _{0} &=&\left( u\left(
1_{\Bbbk }\right) \right) _{-1}\otimes \left( u\left( 1_{\Bbbk }\right)
\right) _{0}=\left( 1_{\Bbbk }\right) _{-1}\otimes u\left( \left( 1_{\Bbbk
}\right) _{0}\right) =1_{H}\otimes u\left( 1_{\Bbbk }\right) =1_{H}\otimes
1_{Q}.  \label{form:1Qcolin}
\end{eqnarray}

\begin{proposition}
\label{pro:CMU}Let $H$ be a Hopf algebra and let $\left( Q,m,u,\Delta
,\varepsilon \right) $ be a coalgebra with multiplication and unit in ${%
_{H}^{H}\mathcal{YD}}$. If $Q_{0}$ is a subcoalgebra of $Q$ in ${_{H}^{H}%
\mathcal{YD}}$ such that $Q_{0}\cdot Q_{0}\subseteq Q_{0},$ then $Q_{n}$ is
a subcoalgebra of $Q$ in ${_{H}^{H}\mathcal{YD}}$ for every $n\in \N_0$%
. Moreover $Q_{a}\cdot Q_{b}\subseteq Q_{a+b}$ for every $a,b\in \N_0$
and the graded coalgebra $\mathrm{gr}Q,$ associated with the coradical
filtration of $Q,$ is a coalgebra with multiplication and unit in ${_{H}^{H}%
\mathcal{YD}}$ with respect to the usual coalgebra structure and with
multiplication and unit defined by
\begin{eqnarray}
m_{\mathrm{gr}Q}\left( \left( x+Q_{a-1}\right) \otimes \left(
y+Q_{b-1}\right) \right) &:&=xy+Q_{a+b-1},  \label{eq:mgr} \\
u_{\mathrm{gr}Q}\left( k\right) &:&=k1_{Q}+Q_{-1}  \notag
\end{eqnarray}
\end{proposition}

\begin{proof}
The coalgebra structure of $Q$ induces a coalgebra structure on $\mathrm{gr}%
Q $. Since $Q_{0}$ is a subcoalgebra of $Q$ in ${_{H}^{H}\mathcal{YD}}\ $%
and, for $n\geq 1$, one has $Q_{n}=Q_{n-1}\wedge _{Q}Q_{0},$ then
inductively one proves that $Q_{n}$ is a subcoalgebra of $Q$ in ${_{H}^{H}%
\mathcal{YD}}$. As a consequence one gets that $\mathrm{gr}Q$ is a coalgebra
in ${_{H}^{H}\mathcal{YD}}$ (this construction can be performed in the
setting of monoidal categories under suitable assumptions, see e.g. \cite[%
Theorem 2.10]{AM}). Let us prove that $\mathrm{gr}Q$ inherits also a
multiplication and unit. Let us check that $Q_{a}\cdot Q_{b}\subseteq
Q_{a+b} $ for every $a,b\in \N_0$. We proceed by induction on $n=a+b.$
If $n=0$ there is nothing to prove. Let $n\geq 1$ and assume that $%
Q_{i}\cdot Q_{j}\subseteq Q_{i+j}$ for every $i,j\in \N_0$ such that $%
0\leq i+j\leq n-1.$ Let $a,b\in \N_0$ be such that $n=a+b$. Since $%
\Delta \left( Q_{a}\right) \subseteq \sum_{i=0}^{a}Q_{i}\otimes Q_{a-i}$ and
$c_{Q,Q}\left( Q_{u}\otimes Q_{v}\right) \subseteq Q_{v}\otimes Q_{u},$
where $c_{Q,Q}$ denotes the braiding in ${_{H}^{H}\mathcal{YD}}$, using the
compatibility condition between $\Delta $ and $m,$ one easily gets that $%
\Delta \left( Q_{a}\cdot Q_{b}\right) \subseteq Q_{a+b-1}\otimes Q+Q\otimes
Q_{0}.$

\begin{invisible}
We compute%
\begin{eqnarray*}
\Delta \left( Q_{a}\cdot Q_{b}\right) &=&\Delta m\left( Q_{a}\otimes
Q_{b}\right) =\left( m\otimes m\right) \Delta _{Q\otimes Q}\left(
Q_{a}\otimes Q_{b}\right) \\
&=&\left( m\otimes m\right) \left( Q\otimes c_{Q,Q}\otimes Q\right) \left(
\Delta \otimes \Delta \right) \left( Q_{a}\otimes Q_{b}\right) \\
&\subseteq &\left( m\otimes m\right) \left( Q\otimes c_{Q,Q}\otimes Q\right)
\left( \left( \sum_{i=0}^{a}Q_{i}\otimes Q_{a-i}\right) \otimes \left(
\sum_{j=0}^{b}Q_{j}\otimes Q_{b-j}\right) \right) \\
&\subseteq &\sum_{i=0}^{a}\sum_{j=0}^{b}\left( m\otimes m\right) \left(
Q_{i}\otimes c_{Q,Q}\left( Q_{a-i}\otimes Q_{j}\right) \otimes Q_{b-j}\right)
\\
&\subseteq &\sum_{i=0}^{a}\sum_{j=0}^{b}\left( m\otimes m\right) \left(
Q_{i}\otimes Q_{j}\otimes Q_{a-i}\otimes Q_{b-j}\right) \\
&\subseteq &\sum_{i=0}^{a}\sum_{j=0}^{b}\left( Q_{i}\cdot Q_{j}\otimes
Q_{a-i}\cdot Q_{b-j}\right) \subseteq \sum_{\substack{ 0\leq i\leq a,  \\ %
0\leq j\leq b,  \\ i+j<a+b}}\left( Q_{i}\cdot Q_{j}\otimes Q_{a-i}\cdot
Q_{b-j}\right) +\left( Q_{a}\cdot Q_{b}\otimes Q_{0}\cdot Q_{0}\right) \\
&\subseteq &\sum_{\substack{ 0\leq i\leq a,  \\ 0\leq j\leq b,  \\ i+j<a+b}}%
Q_{i+j}\otimes Q+Q\otimes Q_{0}\subseteq Q_{a+b-1}\otimes Q+Q\otimes Q_{0}.
\end{eqnarray*}
\end{invisible}

Therefore $Q_{a}\cdot Q_{b}\subseteq Q_{a+b}.$ This property implies we have
a well-defined map in ${_{H}^{H}\mathcal{YD}}$%
\begin{equation*}
m_{\mathrm{gr}Q}^{a,b}:\frac{Q_{a}}{Q_{a-1}}\otimes \frac{Q_{b}}{Q_{b-1}}%
\rightarrow \frac{Q_{a+b}}{Q_{a+b-1}}
\end{equation*}%
defined, for $x\in Q_{a}$ and $y\in Q_{b},$ by (\ref{eq:mgr}). This can be
seen as the graded component of a morphism in ${_{H}^{H}\mathcal{YD}}$ that
we denote by $m_{\mathrm{gr}Q}:\mathrm{gr}Q\otimes \mathrm{gr}Q\rightarrow
\mathrm{gr}Q$. Let us check that $m_{\mathrm{gr}Q}$ is a coalgebra morphism
in ${_{H}^{H}\mathcal{YD}}$. Consider a basis of $Q$ with terms of the form $%
x^{i,j}$ as in \ref{claim:basis}. Hence we can write the comultiplication in
the form
\begin{equation*}
\Delta \left( x^{a,u}\right) =\sum_{s+t\leq a}\sum_{l,m}\eta
_{s,t,l,m}^{a,u}x^{s,l}\otimes x^{t,m}.
\end{equation*}%
Now, using (\ref{form:DeltaGr}), one gets that
\begin{equation}
\Delta _{\mathrm{gr}Q}\left( \overline{x^{a,u}}\right) =\sum_{0\leq i\leq
a}\sum_{l,m}\eta _{i,a-i,l,m}^{a,u}\overline{x^{i,l}}\otimes \overline{%
x^{a-i,m}}.  \label{form:deltagr}
\end{equation}

\begin{invisible}
Here is the proof:%
\begin{eqnarray*}
&&\Delta _{\mathrm{gr}Q}\left( \overline{x^{a,u}}\right) \overset{(\ref%
{form:DeltaGr})}{=}\sum_{0\leq i\leq a}\left( \left( x^{a,u}\right)
_{1}+Q_{i-1}\right) \otimes \left( \left( x^{a,u}\right)
_{2}+Q_{a-i-1}\right) \\
&=&\sum_{0\leq i\leq a}\sum_{s+t\leq a}\sum_{l,m}\eta _{s,t,l,m}^{a,u}\left(
x^{s,l}+Q_{i-1}\right) \otimes \left( x^{t,m}+Q_{a-i-1}\right) \\
&=&\sum_{0\leq i\leq a}\sum_{\substack{ s+t\leq a  \\ i\leq s,a-i\leq t}}%
\sum_{l,m}\eta _{s,t,l,m}^{a,u}\left( x^{s,l}+Q_{i-1}\right) \otimes \left(
x^{t,m}+Q_{a-i-1}\right) \\
&=&\sum_{0\leq i\leq a}\sum_{s=i,t=a-i}\sum_{l,m}\eta _{s,t,l,m}^{a,u}\left(
x^{s,l}+Q_{i-1}\right) \otimes \left( x^{t,m}+Q_{a-i-1}\right) \\
&=&\sum_{0\leq i\leq a}\sum_{l,m}\eta _{i,a-i,l,m}^{a,u}\left(
x^{i,l}+Q_{i-1}\right) \otimes \left( x^{a-i,m}+Q_{a-i-1}\right) \\
&=&\sum_{0\leq i\leq a}\sum_{l,m}\eta _{i,a-i,l,m}^{a,u}\overline{x^{i,l}}%
\otimes \overline{x^{a-i,m}}
\end{eqnarray*}
\end{invisible}

Using that $\Delta _{\mathrm{gr}Q\otimes \mathrm{gr}Q}=\left( \mathrm{gr}%
Q\otimes c_{\mathrm{gr}Q,\mathrm{gr}Q}\otimes \mathrm{gr}Q\right) \left(
\Delta _{\mathrm{gr}Q}\otimes \Delta _{\mathrm{gr}Q}\right) $ and (\ref%
{form:deltagr}), it is straightforward to check that $\left( m_{\mathrm{gr}%
Q}\otimes m_{\mathrm{gr}Q}\right) \Delta _{\mathrm{gr}Q\otimes \mathrm{gr}%
Q}\left( \overline{x^{a,u}}\otimes \overline{x^{b,v}}\right) =\Delta _{%
\mathrm{gr}Q}m_{\mathrm{gr}Q}\left( \overline{x^{a,u}}\otimes \overline{%
x^{b,v}}\right) .$

\begin{invisible}
We have%
\begin{eqnarray*}
&&\left( m_{\mathrm{gr}Q}\otimes m_{\mathrm{gr}Q}\right) \Delta _{\mathrm{gr}%
Q\otimes \mathrm{gr}Q}\left( \overline{x^{a,u}}\otimes \overline{x^{b,v}}%
\right) \\
&=&\left( m_{\mathrm{gr}Q}\otimes m_{\mathrm{gr}Q}\right) \left( \mathrm{gr}%
Q\otimes c_{\mathrm{gr}Q,\mathrm{gr}Q}\otimes \mathrm{gr}Q\right) \left(
\Delta _{\mathrm{gr}Q}\otimes \Delta _{\mathrm{gr}Q}\right) \left( \overline{%
x^{a,u}}\otimes \overline{x^{b,v}}\right) \\
&\overset{(\ref{form:deltagr})}{=}&\left[
\begin{array}{c}
\left( m_{\mathrm{gr}Q}\otimes m_{\mathrm{gr}Q}\right) \left( \mathrm{gr}%
Q\otimes c_{\mathrm{gr}Q,\mathrm{gr}Q}\otimes \mathrm{gr}Q\right) \\
\left( \sum_{0\leq i\leq a}\sum_{l,m}\eta _{i,a-i,l,m}^{a,u}\overline{x^{i,l}%
}\otimes \overline{x^{a-i,m}}\otimes \sum_{0\leq i^{\prime }\leq
b}\sum_{l^{\prime },m^{\prime }}\eta _{i^{\prime },b-i^{\prime },l^{\prime
},m^{\prime }}^{b,v}\overline{x^{i^{\prime },l^{\prime }}}\otimes \overline{%
x^{b-i^{\prime },m^{\prime }}}\right)%
\end{array}%
\right] \\
&=&\left[
\begin{array}{c}
\sum_{0\leq i\leq a}\sum_{l,m}\sum_{0\leq i^{\prime }\leq b}\sum_{l^{\prime
},m^{\prime }}\eta _{i^{\prime },b-i^{\prime },l^{\prime },m^{\prime
}}^{b,v}\eta _{i,a-i,l,m}^{a,u}\left( m_{\mathrm{gr}Q}\otimes m_{\mathrm{gr}%
Q}\right) \\
\left( \overline{x^{i,l}}\otimes \left( \overline{x^{a-i,m}}\right) _{-1}%
\overline{x^{i^{\prime },l^{\prime }}}\otimes \left( \overline{x^{a-i,m}}%
\right) _{0}\otimes \overline{x^{b-i^{\prime },m^{\prime }}}\right)%
\end{array}%
\right] \\
&=&\left[
\begin{array}{c}
\sum_{0\leq i\leq a}\sum_{l,m}\sum_{0\leq i^{\prime }\leq b}\sum_{l^{\prime
},m^{\prime }}\eta _{i^{\prime },b-i^{\prime },l^{\prime },m^{\prime
}}^{b,v}\eta _{i,a-i,l,m}^{a,u}\left( m_{\mathrm{gr}Q}\otimes m_{\mathrm{gr}%
Q}\right) \\
\left( \overline{x^{i,l}}\otimes \left( x^{a-i,m}+Q_{a-i-1}\right) _{-1}%
\overline{x^{i^{\prime },l^{\prime }}}\otimes \left(
x^{a-i,m}+Q_{a-i-1}\right) _{0}\otimes \overline{x^{b-i^{\prime },m^{\prime
}}}\right)%
\end{array}%
\right] \\
&=&\left[
\begin{array}{c}
\sum_{0\leq i\leq a}\sum_{l,m}\sum_{0\leq i^{\prime }\leq b}\sum_{l^{\prime
},m^{\prime }}\eta _{i^{\prime },b-i^{\prime },l^{\prime },m^{\prime
}}^{b,v}\eta _{i,a-i,l,m}^{a,u}\left( m_{\mathrm{gr}Q}\otimes m_{\mathrm{gr}%
Q}\right) \\
\left( \overline{x^{i,l}}\otimes \left( x^{a-i,m}\right) _{-1}\overline{%
x^{i^{\prime },l^{\prime }}}\otimes \left( \left( x^{a-i,m}\right)
_{0}+Q_{a-i-1}\right) \otimes \overline{x^{b-i^{\prime },m^{\prime }}}\right)%
\end{array}%
\right] \\
&=&\left[
\begin{array}{c}
\sum_{0\leq i\leq a}\sum_{l,m}\sum_{0\leq i^{\prime }\leq b}\sum_{l^{\prime
},m^{\prime }}\eta _{i^{\prime },b-i^{\prime },l^{\prime },m^{\prime
}}^{b,v}\eta _{i,a-i,l,m}^{a,u}\left( m_{\mathrm{gr}Q}\otimes m_{\mathrm{gr}%
Q}\right) \\
\left( \overline{x^{i,l}}\otimes \left( x^{a-i,m}\right) _{-1}\left(
x^{i^{\prime },l^{\prime }}+Q_{i^{\prime }-1}\right) \otimes \left( \left(
x^{a-i,m}\right) _{0}+Q_{a-i-1}\right) \otimes \overline{x^{b-i^{\prime
},m^{\prime }}}\right)%
\end{array}%
\right] \\
&=&\left[
\begin{array}{c}
\sum_{0\leq i\leq a}\sum_{l,m}\sum_{0\leq i^{\prime }\leq b}\sum_{l^{\prime
},m^{\prime }}\eta _{i^{\prime },b-i^{\prime },l^{\prime },m^{\prime
}}^{b,v}\eta _{i,a-i,l,m}^{a,u}\left( m_{\mathrm{gr}Q}\otimes m_{\mathrm{gr}%
Q}\right) \\
\left( \left( x^{i,l}+Q_{i-1}\right) \otimes \left( \left( x^{a-i,m}\right)
_{-1}x^{i^{\prime },l^{\prime }}+Q_{i^{\prime }-1}\right) \otimes \left(
\left( x^{a-i,m}\right) _{0}+Q_{a-i-1}\right) \otimes \left( x^{b-i^{\prime
},m^{\prime }}+Q_{b-i^{\prime }-1}\right) \right)%
\end{array}%
\right] \\
&=&\left[
\begin{array}{c}
\sum_{0\leq i\leq a}\sum_{0\leq i^{\prime }\leq b}\sum_{l,m}\sum_{l^{\prime
},m^{\prime }}\eta _{i^{\prime },b-i^{\prime },l^{\prime },m^{\prime
}}^{b,v}\eta _{i,a-i,l,m}^{a,u} \\
\left( \left( x^{i,l}\left( \left( x^{a-i,m}\right) _{-1}x^{i^{\prime
},l^{\prime }}\right) +Q_{i+i^{\prime }-1}\right) \otimes \left(
x^{a-i,m}\right) _{0}x^{b-i^{\prime },m^{\prime }}+Q_{a-i+b-i^{\prime
}-1}\right)%
\end{array}%
\right] \\
&=&\left[
\begin{array}{c}
\sum_{w=0}^{a+b}\sum_{\substack{ i+i^{\prime }=w  \\ i\leq a,i^{\prime }\leq
b }}\sum_{l,m}\sum_{l^{\prime },m^{\prime }}\eta _{i^{\prime },b-i^{\prime
},l^{\prime },m^{\prime }}^{b,v}\eta _{i,a-i,l,m}^{a,u} \\
\left( x^{i,l}\left( \left( x^{a-i,m}\right) _{-1}x^{i^{\prime },l^{\prime
}}\right) +Q_{w-1}\right) \otimes \left( \left( x^{a-i,m}\right)
_{0}x^{b-i^{\prime },m^{\prime }}+Q_{a+b-w-1}\right)%
\end{array}%
\right] \\
&=&\left[
\begin{array}{c}
\sum_{w=0}^{a+b}\sum_{\substack{ i+t=a,i^{\prime }+t^{\prime }=b,  \\ %
t+t^{\prime }=a+b-w,  \\ i+i^{\prime }=w}}\sum_{l,m}\sum_{l^{\prime
},m^{\prime }}\eta _{i^{\prime },t^{\prime },l^{\prime },m^{\prime
}}^{b,v}\eta _{i,t,l,m}^{a,u} \\
\left( x^{i,l}\left( \left( x^{t,m}\right) _{-1}x^{i^{\prime },l^{\prime
}}\right) +Q_{w-1}\right) \otimes \left( \left( x^{t,m}\right)
_{0}x^{t^{\prime },m^{\prime }}+Q_{a+b-w-1}\right)%
\end{array}%
\right] \\
&=&\left[
\begin{array}{c}
\sum_{w=0}^{a+b}\sum_{\substack{ s+t=a,s^{\prime }+t^{\prime }=b,  \\ %
t+t^{\prime }=a+b-w,  \\ s+s^{\prime }=w}}\sum_{l,m}\sum_{l^{\prime
},m^{\prime }}\eta _{s^{\prime },t^{\prime },l^{\prime },m^{\prime
}}^{b,v}\eta _{s,t,l,m}^{a,u} \\
\left( x^{s,l}\left( \left( x^{t,m}\right) _{-1}x^{s^{\prime },l^{\prime
}}\right) +Q_{w-1}\right) \otimes \left( \left( x^{t,m}\right)
_{0}x^{t^{\prime },m^{\prime }}+Q_{a+b-w-1}\right)%
\end{array}%
\right] \\
&=&\left[
\begin{array}{c}
\sum_{w=0}^{a+b}\sum_{l,m}\sum_{\substack{ s+t\leq a,s^{\prime }+t^{\prime
}\leq b,  \\ a+b-w\leq t+t^{\prime },  \\ w\leq s+s^{\prime }}}%
\sum_{l^{\prime },m^{\prime }}\eta _{s^{\prime },t^{\prime },l^{\prime
},m^{\prime }}^{b,v}\eta _{s,t,l,m}^{a,u} \\
\left( x^{s,l}\left( \left( x^{t,m}\right) _{-1}x^{s^{\prime },l^{\prime
}}\right) +Q_{w-1}\right) \otimes \left( \left( x^{t,m}\right)
_{0}x^{t^{\prime },m^{\prime }}+Q_{a+b-w-1}\right)%
\end{array}%
\right] \\
&=&\left[
\begin{array}{c}
\sum_{w=0}^{a+b}\sum_{s+t\leq a}\sum_{l,m}\eta
_{s,t,l,m}^{a,u}\sum_{s^{\prime }+t^{\prime }\leq b}\sum_{l^{\prime
},m^{\prime }}\eta _{s^{\prime },t^{\prime },l^{\prime },m^{\prime }}^{b,v}
\\
\left( x^{s,l}\left( \left( x^{t,m}\right) _{-1}x^{s^{\prime },l^{\prime
}}\right) +Q_{w-1}\right) \otimes \left( \left( x^{t,m}\right)
_{0}x^{t^{\prime },m^{\prime }}+Q_{a+b-w-1}\right)%
\end{array}%
\right] \\
&=&\sum_{w=0}^{a+b}\left( \left( x^{a,u}\right) _{1}\left( \left( \left(
x^{a,u}\right) _{2}\right) _{-1}\left( x^{b,v}\right) _{1}\right)
+Q_{w-1}\right) \otimes \left( \left( \left( x^{a,u}\right) _{2}\right)
_{0}\left( x^{b,v}\right) _{2}+Q_{a+b-w-1}\right) \\
&=&\sum_{w=0}^{a+b}\left( \left( x^{a,u}x^{b,v}\right) _{1}+Q_{w-1}\right)
\otimes \left( \left( x^{a,u}x^{b,v}\right) _{2}+Q_{a+b-w-1}\right) \\
&=&\Delta _{\mathrm{gr}Q}\left( x^{a,u}x^{b,v}+Q_{a+b-1}\right) \\
&=&\Delta _{\mathrm{gr}Q}m_{\mathrm{gr}Q}\left( \left(
x^{a,u}+Q_{a-1}\right) \otimes \left( x^{b,v}+Q_{b-1}\right) \right) \\
&=&\Delta _{\mathrm{gr}Q}m_{\mathrm{gr}Q}\left( \overline{x^{a,u}}\otimes
\overline{x^{b,v}}\right) .
\end{eqnarray*}
\end{invisible}

Moreover, since $\varepsilon _{\mathrm{gr}Q\otimes \mathrm{gr}Q}=\varepsilon
_{\mathrm{gr}Q}\otimes \varepsilon _{\mathrm{gr}Q},$ we get that $%
\varepsilon _{\mathrm{gr}Q}m_{\mathrm{gr}Q}\left( \overline{x^{a,u}}\otimes
\overline{x^{b,v}}\right) =\varepsilon _{\mathrm{gr}Q\otimes \mathrm{gr}%
Q}\left( \overline{x^{a,u}}\otimes \overline{x^{b,v}}\right) .$

\begin{invisible}
\begin{eqnarray*}
\varepsilon _{\mathrm{gr}Q}m_{\mathrm{gr}Q}\left( \overline{x^{a,u}}\otimes
\overline{x^{b,v}}\right) &=&\varepsilon _{\mathrm{gr}Q}m_{\mathrm{gr}%
Q}\left( \left( x^{a,u}+Q_{a-1}\right) \otimes \left( x^{b,v}+Q_{b-1}\right)
\right) \\
&=&\varepsilon _{\mathrm{gr}Q}\left( x^{a,u}x^{b,v}+Q_{a+b-1}\right) \\
&=&\delta _{a+b,0}\varepsilon \left( x^{a,u}x^{b,v}\right) \\
&=&\delta _{a+b,0}\varepsilon m\left( x^{a,u}\otimes x^{b,v}\right) \\
&=&\delta _{a+b,0}\varepsilon _{Q\otimes Q}\left( x^{a,u}\otimes
x^{b,v}\right) \\
&=&\delta _{a,0}\delta _{b,0}\varepsilon \left( x^{a,u}\right) \varepsilon
\left( x^{b,v}\right) \\
&=&\delta _{a,0}\varepsilon \left( x^{a,u}\right) \delta _{b,0}\varepsilon
\left( x^{b,v}\right) \\
&=&\varepsilon _{\mathrm{gr}Q}\left( \overline{x^{a,u}}\right) \varepsilon _{%
\mathrm{gr}Q}\left( \overline{x^{b,v}}\right) \\
&=&\varepsilon _{\mathrm{gr}Q\otimes \mathrm{gr}Q}\left( \overline{x^{a,u}}%
\otimes \overline{x^{b,v}}\right) .
\end{eqnarray*}
\end{invisible}

This proves that $m_{\mathrm{gr}Q}$ is a coalgebra morphism in ${_{H}^{H}%
\mathcal{YD}}$.

The fact that $u_{\mathrm{gr}Q}:\Bbbk \rightarrow \mathrm{gr}Q,$ defined by $%
u_{\mathrm{gr}Q}\left( k\right) :=k1_{Q}+Q_{-1}$ is a coalgebra morphism in $%
{_{H}^{H}\mathcal{YD}}$ easily follows by means of (\ref{form:1Qlin}) and (%
\ref{form:1Qcolin}).

\begin{invisible}
Let us check that $u_{\mathrm{gr}Q}:\Bbbk \rightarrow \mathrm{gr}Q,$ defined
by $u_{\mathrm{gr}Q}\left( k\right) :=k1_{Q}+Q_{-1}$ is a coalgebra morphism
in ${_{H}^{H}\mathcal{YD}}$ too. For $h\in H,k\in \Bbbk $, we have%
\begin{eqnarray*}
u_{\mathrm{gr}Q}\left( hk\right) &=&u_{\mathrm{gr}Q}\left( \varepsilon
_{H}\left( h\right) k\right) =\varepsilon _{H}\left( h\right) u_{\mathrm{gr}%
Q}\left( k\right) =\varepsilon _{H}\left( h\right) \left(
k1_{Q}+Q_{-1}\right) \\
&=&\varepsilon _{H}\left( h\right) k1_{Q}+Q_{-1}\overset{(\ref{form:1Qlin})}{%
=}kh1_{Q}+Q_{-1}=h\left( k1_{Q}+Q_{-1}\right) =hu_{\mathrm{gr}Q}\left(
k\right) , \\
\left[ u_{\mathrm{gr}Q}\left( k\right) \right] _{-1}\otimes \left[ u_{%
\mathrm{gr}Q}\left( k\right) \right] _{0} &=&\left[ k1_{Q}+Q_{-1}\right]
_{-1}\otimes \left[ k1_{Q}+Q_{-1}\right] _{0}=\left( k1_{Q}\right)
_{-1}\otimes \left[ \left( k1_{Q}\right) _{0}+Q_{-1}\right] \\
&=&\left( 1_{Q}\right) _{-1}\otimes \left[ k\left( 1_{Q}\right) _{0}+Q_{-1}%
\right] \overset{(\ref{form:1Qcolin})}{=}1_{H}\otimes \left[ k1_{Q}+Q_{-1}%
\right] =1_{H}\otimes u_{\mathrm{gr}Q}\left( k\right) .
\end{eqnarray*}%
Moreover%
\begin{eqnarray*}
\Delta _{\mathrm{gr}Q}u_{\mathrm{gr}Q}\left( k\right) &=&\Delta _{\mathrm{gr}%
Q}\left( k1_{Q}+Q_{-1}\right) \\
&=&\left( \left( k1_{Q}\right) _{1}+Q_{-1}\right) \otimes \left( \left(
k1_{Q}\right) _{2}+Q_{-1}\right) \\
&=&k\left( \left( 1_{Q}\right) _{1}+Q_{-1}\right) \otimes \left( \left(
1_{Q}\right) _{2}+Q_{-1}\right) \\
&=&k\left( 1_{Q}+Q_{-1}\right) \otimes \left( 1_{Q}+Q_{-1}\right) \\
&=&u_{\mathrm{gr}Q}\left( k\right) \otimes u_{\mathrm{gr}Q}\left( 1_{\Bbbk
}\right) \\
&=&\left( u_{\mathrm{gr}Q}\otimes u_{\mathrm{gr}Q}\right) \Delta _{\Bbbk
}\left( k\right) , \\
\varepsilon _{\mathrm{gr}Q}u_{\mathrm{gr}Q}\left( k\right) &=&\varepsilon _{%
\mathrm{gr}Q}\left( k1_{Q}+Q_{-1}\right) =\varepsilon \left( k1_{Q}\right)
=k\varepsilon \left( 1_{Q}\right) =k1_{\Bbbk }=k=\varepsilon _{\Bbbk }\left(
k\right) .
\end{eqnarray*}
\end{invisible}
\end{proof}

\begin{definition}[{\protect\cite[Definition 5.2]{ABM}}]
\label{def: dual quasi braided} Let $H$ be a Hopf algebra. Recall that a
\emph{coquasi-bialgebra} $(Q,m,u,\Delta ,\varepsilon ,\alpha )$ in the
pre-braided monoidal category $_{H}^{H}\mathcal{YD}$ is a coalgebra $\left(
Q,\Delta ,\varepsilon \right) $ in $_{H}^{H}\mathcal{YD}$ together with
coalgebra homomorphisms $m:Q\otimes Q\rightarrow Q$ and $u:\Bbbk \rightarrow
Q$ in $_{H}^{H}\mathcal{YD}$ and a convolution invertible element $\alpha
\in {_{H}^{H}\mathcal{YD}}\left( Q^{\otimes 3},\Bbbk \right) $ (\emph{%
braided reassociator}) such that%
\begin{eqnarray}
&&\alpha \left( Q\otimes Q\otimes m\right) \ast \alpha \left( m\otimes
Q\otimes Q\right) =\left( \varepsilon \otimes \alpha \right) \ast \alpha
\left( Q\otimes m\otimes Q\right) \ast \left( \alpha \otimes \varepsilon
\right) ,  \label{form: alpha 3-cocycle} \\
&&\alpha \left( Q\otimes u\otimes Q\right) =\alpha \left( u\otimes Q\otimes
Q\right) =\alpha \left( Q\otimes Q\otimes u\right) =\varepsilon _{Q\otimes
Q},  \label{form: alpha unital} \\
&&m\left( Q\otimes m\right) \ast \alpha =\alpha \ast m\left( m\otimes
Q\right) ,  \label{form: m quasi assoc} \\
&&m\left( u\otimes Q\right) =\mathrm{Id}_{Q}=m\left( Q\otimes u\right) .
\label{form: m unital}
\end{eqnarray}%
Here $\ast $ denotes the convolution product, where $Q^{\otimes 3}$ is the
tensor product of coalgebras in $_{H}^{H}\mathcal{YD}$ whence it depends on
the braiding of this category. Note that in (\ref{form: alpha unital}) any
of the three equalities such as $\alpha \left( u\otimes Q\otimes Q\right)
=\varepsilon _{Q\otimes Q}$ implies that $\alpha $ is unital.
\end{definition}

\begin{theorem}
\label{teo:grHopf}Let $H$ be a Hopf algebra and let $\left( Q,m,u,\Delta
,\varepsilon ,\omega \right) $ be a connected coquasi-bialgebra in ${_{H}^{H}%
\mathcal{YD}}$. Then $\mathrm{gr}Q$ is a connected bialgebra in ${_{H}^{H}%
\mathcal{YD}}$.
\end{theorem}

\begin{proof}
By Proposition \ref{pro:CMU}, we know that $\mathrm{gr}Q$ is a coalgebra
with multiplication and unit in ${_{H}^{H}\mathcal{YD}}$. We have to check
that the multiplication is associative and unitary.

Given two coalgebras $D,E$ in ${_{H}^{H}\mathcal{YD}}$ endowed with
coalgebras filtration $\left( D_{\left( n\right) }\right) _{n\in \N_0}$
and $\left( E_{\left( n\right) }\right) _{n\in \N_0}$ in ${_{H}^{H}%
\mathcal{YD}}$ such that $D_{\left( 0\right) }$ and $E_{\left( 0\right) }$
are one-dimensional, let us check that $C_{\left( n\right) }:=\sum_{0\leq
i\leq n}D_{\left( i\right) }\otimes E_{\left( n-i\right) }$ gives a
coalgebra filtration on $C:=D\otimes E$ in ${_{H}^{H}\mathcal{YD}}.$ First
note that the coalgebra structure of $C$ depends on the braiding. Thus, we
have
\begin{eqnarray*}
\Delta _{C}\left( C_{\left( n\right) }\right) &=&\left( D\otimes
c_{D,E}\otimes E\right) \left( \Delta _{D}\otimes \Delta _{E}\right) \left(
\sum_{i=0}^{n}D_{\left( i\right) }\otimes E_{\left( n-i\right) }\right) \\
&\subseteq &\left( D\otimes c_{D,E}\otimes E\right) \left(
\sum_{i=0}^{n}\sum_{a=0}^{i}\sum_{b=0}^{n-i}D_{\left( a\right) }\otimes
D_{\left( i-a\right) }\otimes E_{\left( b\right) }\otimes E_{\left(
n-i-b\right) }\right) \\
&\subseteq &\sum_{i=0}^{n}\sum_{a=0}^{i}\sum_{b=0}^{n-i}D_{\left( a\right)
}\otimes c_{D,E}\left( D_{\left( i-a\right) }\otimes E_{\left( b\right)
}\right) \otimes E_{\left( n-i-b\right) } \\
&\subseteq &\sum_{i=0}^{n}\sum_{a=0}^{i}\sum_{b=0}^{n-i}D_{\left( a\right)
}\otimes c_{D_{\left( i-a\right) },E_{\left( b\right) }}\left( D_{\left(
i-a\right) }\otimes E_{\left( b\right) }\right) \otimes E_{\left(
n-i-b\right) } \\
&\subseteq &\sum_{i=0}^{n}\sum_{a=0}^{i}\sum_{b=0}^{n-i}D_{\left( a\right)
}\otimes E_{\left( b\right) }\otimes D_{\left( i-a\right) }\otimes E_{\left(
n-i-b\right) } \\
&\subseteq &\sum_{i=0}^{n}\sum_{w=0}^{n}\sum_{\substack{ 0\leq a\leq i,  \\ %
0\leq b\leq n-i  \\ a+b=w}}D_{\left( a\right) }\otimes E_{\left( b\right)
}\otimes D_{\left( i-a\right) }\otimes E_{\left( n-i-b\right) } \\
&\subseteq &\sum_{w=0}^{n}C_{\left( w\right) }\otimes C_{\left( n-w\right) }.
\end{eqnarray*}%
Moreover, by \cite[Proposition 11.1.1]{Sw}, we have that the coradical of $%
C $ is contained in $D_{\left( 0\right) }\otimes E_{\left( 0\right) }\ $and
hence it is one-dimensional.

This argument can be used to produce a coalgebra filtration on $C:=Q\otimes
Q\otimes Q$ using as a filtration on $Q$ the coradical filtration. Let $n>0$
and let $w\in C_{\left( n\right) }=\sum_{i+j+k\leq n}Q_{i}\otimes
Q_{j}\otimes Q_{k}.$ By \cite[Lemma 3.69]{AMS}, we have that%
\begin{equation*}
\Delta _{C}\left( w\right) -w\otimes \left( 1_{Q}\right) ^{\otimes 3}-\left(
1_{Q}\right) ^{\otimes 3}\otimes w\in C_{\left( n-1\right) }\otimes
C_{\left( n-1\right) }.
\end{equation*}%
Thus we get%
\begin{equation*}
w_{1}\otimes w_{2}\otimes w_{3}-\Delta _{C}\left( w\right) \otimes \left(
1_{Q}\right) ^{\otimes 3}-\Delta _{C}\left( \left( 1_{Q}\right) ^{\otimes
3}\right) \otimes w\in \Delta _{C}\left( C_{\left( n-1\right) }\right)
\otimes C_{\left( n-1\right) }
\end{equation*}%
and hence, tensoring the first relation by $\left( 1_{Q}\right) ^{\otimes 3}$
on the right and adding it to the second one, we get
\begin{equation*}
w_{1}\otimes w_{2}\otimes w_{3}-w\otimes \left( 1_{Q}\right) ^{\otimes
3}\otimes \left( 1_{Q}\right) ^{\otimes 3}-\left( 1_{Q}\right) ^{\otimes
3}\otimes w\otimes \left( 1_{Q}\right) ^{\otimes 3}-\left( 1_{Q}\right)
^{\otimes 6}\otimes w\in C_{\left( n-1\right) }\otimes C_{\left( n-1\right)
}\otimes C_{\left( n-1\right) }.
\end{equation*}

For shortness, we set $\nu _{n}\left( z\right) :=m\left( Q\otimes m\right)
\left( z\right) +Q_{n-1}$ for every $z\in C.$ Thus, by applying to the last
displayed relation $C_{\left( n-1\right) }\otimes m\left( Q\otimes m\right)
\otimes C_{\left( n-1\right) }$ and factoring out the middle term by $%
Q_{n-1},$ we get
\begin{eqnarray*}
&&\left[
\begin{array}{c}
w_{1}\otimes \nu _{n}\left( w_{2}\right) \otimes w_{3}-w\otimes \nu
_{n}\left( \left( 1_{Q}\right) ^{\otimes 3}\right) \otimes \left(
1_{Q}\right) ^{\otimes 3}+ \\
-\left( 1_{Q}\right) ^{\otimes 3}\otimes \nu _{n}\left( w\right) \otimes
\left( 1_{Q}\right) ^{\otimes 3}-\left( 1_{Q}\right) ^{\otimes 3}\otimes \nu
_{n}\left( \left( 1_{Q}\right) ^{\otimes 3}\right) \otimes w%
\end{array}%
\right] \\
&\in &C_{\left( n-1\right) }\otimes \left( \frac{\nu _{n}\left( C_{\left(
n-1\right) }\right) }{Q_{n-1}}\right) \otimes C_{\left( n-1\right)
}\subseteq C_{\left( n-1\right) }\otimes \frac{Q_{n-1}}{Q_{n-1}}\otimes
C_{\left( n-1\right) }=0.
\end{eqnarray*}%
Thus we can express the first term with respect to the remaining ones as
follows
\begin{eqnarray*}
&&w_{1}\otimes \nu _{n}\left( w_{2}\right) \otimes w_{3} \\
&=&w\otimes \nu _{n}\left( \left( 1_{Q}\right) ^{\otimes 3}\right) \otimes
\left( 1_{Q}\right) ^{\otimes 3}+\left( 1_{Q}\right) ^{\otimes 3}\otimes \nu
_{n}\left( w\right) \otimes \left( 1_{Q}\right) ^{\otimes 3}+\left(
1_{Q}\right) ^{\otimes 3}\otimes \nu _{n}\left( \left( 1_{Q}\right)
^{\otimes 3}\right) \otimes w \\
&=&w\otimes \left( 1_{Q}+Q_{n-1}\right) \otimes \left( 1_{Q}\right)
^{\otimes 3}+\left( 1_{Q}\right) ^{\otimes 3}\otimes \nu _{n}\left( w\right)
\otimes \left( 1_{Q}\right) ^{\otimes 3}+\left( 1_{Q}\right) ^{\otimes
3}\otimes \left( 1_{Q}+Q_{n-1}\right) \otimes w \\
&\overset{n>0}{=}&\left( 1_{Q}\right) ^{\otimes 3}\otimes \nu _{n}\left(
w\right) \otimes \left( 1_{Q}\right) ^{\otimes 3}.
\end{eqnarray*}%
We have so proved that for $n>0$ and $w\in C_{\left( n\right) }$%
\begin{equation}
w_{1}\otimes \nu _{n}\left( w_{2}\right) \otimes w_{3}=\left( 1_{Q}\right)
^{\otimes 3}\otimes \nu _{n}\left( w\right) \otimes \left( 1_{Q}\right)
^{\otimes 3}.  \label{form:Dragos}
\end{equation}%
The same equation trivially holds also in the case $n=0$ as $C_{\left(
n\right) }$ is one-dimensional.

Let $x,y,z\in Q$. Then $x\otimes y\otimes z\in C_{\left( \left\vert
x\right\vert +\left\vert y\right\vert +\left\vert z\right\vert \right) }$ so
that
\begin{eqnarray*}
\left( \overline{x}\cdot \overline{y}\right) \cdot \overline{z} &=&\left(
\left( x+Q_{\left\vert x\right\vert -1}\right) \cdot \left( y+Q_{\left\vert
y\right\vert -1}\right) \right) \cdot \left( z+Q_{\left\vert z\right\vert
-1}\right) \\
&=&\left( \left( xy\right) +Q_{\left\vert x\right\vert +\left\vert
y\right\vert -1}\right) \cdot \left( z+Q_{\left\vert z\right\vert -1}\right)
\\
&=&\left( xy\right) z+Q_{\left\vert x\right\vert +\left\vert y\right\vert
+\left\vert z\right\vert -1} \\
&=&\omega ^{-1}\left( \left( x\otimes y\otimes z\right) _{1}\right) \nu
_{\left\vert x\right\vert +\left\vert y\right\vert +\left\vert z\right\vert
}\left( \left( x\otimes y\otimes z\right) _{2}\right) \omega \left( \left(
x\otimes y\otimes z\right) _{3}\right) \\
&\overset{(\ref{form:Dragos})}{=}&\omega ^{-1}\left( 1_{Q}\otimes
1_{Q}\otimes 1_{Q}\right) \nu _{\left\vert x\right\vert +\left\vert
y\right\vert +\left\vert z\right\vert }\left( x\otimes y\otimes z\right)
\omega \left( 1_{Q}\otimes 1_{Q}\otimes 1_{Q}\right) \\
&=&\nu _{\left\vert x\right\vert +\left\vert y\right\vert +\left\vert
z\right\vert }\left( x\otimes y\otimes z\right) \\
&=&x\left( yz\right) +Q_{\left\vert x\right\vert +\left\vert y\right\vert
+\left\vert z\right\vert -1}=\overline{x}\cdot \left( \overline{y}\cdot
\overline{z}\right) .
\end{eqnarray*}

Therefore the multiplication is associative. It is also unitary as
\begin{equation*}
\overline{x}\cdot \overline{1_{Q}}=\left( x+Q_{\left\vert x\right\vert
-1}\right) \cdot \left( 1_{Q}+Q_{-1}\right) =x\cdot 1_{Q}+Q_{\left\vert
x\right\vert -1}=x+Q_{\left\vert x\right\vert -1}=\overline{x}
\end{equation*}%
and similarly $\overline{1_{Q}}\cdot \overline{x}=\overline{x}$ for every $%
x\in Q.$
\end{proof}

\section{Gauge deformation}\label{sec:2}

\begin{definition}
Let $H$ be a Hopf algebra and let $\left( Q,m,u,\Delta ,\varepsilon ,\omega
\right) $ be a coquasi-bialgebra in ${_{H}^{H}\mathcal{YD}}$. A \textbf{%
gauge transformation} for $Q$ is a morphism $\gamma :Q\otimes Q\rightarrow
\Bbbk $ in ${_{H}^{H}\mathcal{YD}}$ which is convolution invertible in ${%
_{H}^{H}\mathcal{YD}}$ and which is also unitary on both entries.
\end{definition}

\begin{remark}
\label{rem:gamma-1gauge}For $\gamma $ as above, let us check that $\gamma
^{-1}$ is unitary whence a gauge transformation too.

First note that for all $x\in Q,$ by means of (\ref{form:1Qcolin}) and (\ref%
{form:1Qlin}), one gets%
\begin{eqnarray}
\left( 1_{Q}\otimes x\right) _{1}\otimes \left( 1_{Q}\otimes x\right) _{2}
&=&1_{Q}\otimes x_{1}\otimes 1_{Q}\otimes x_{2}  \label{form:delta1x} \\
\left( x\otimes 1_{Q}\right) _{1}\otimes \left( x\otimes 1_{Q}\right) _{2}
&=&x_{1}\otimes 1_{Q}\otimes x_{2}\otimes 1_{Q}  \label{form:deltax1}
\end{eqnarray}

\begin{invisible}
For all $x\in Q,$%
\begin{equation*}
\left( 1_{Q}\otimes x\right) _{1}\otimes \left( 1_{Q}\otimes x\right)
_{2}=1_{Q}\otimes \left( 1_{Q}\right) _{-1}x_{1}\otimes \left( 1_{Q}\right)
_{0}\otimes x_{2}\overset{(\ref{form:1Qcolin})}{=}1_{Q}\otimes x_{1}\otimes
1_{Q}\otimes x_{2}
\end{equation*}%
and%
\begin{eqnarray*}
\left( x\otimes 1_{Q}\right) _{1}\otimes \left( x\otimes 1_{Q}\right) _{2}
&=&x_{1}\otimes \left( x_{2}\right) _{-1}1_{Q}\otimes \left( x_{2}\right)
_{0}\otimes 1_{Q}\overset{(\ref{form:1Qlin})}{=}x_{1}\otimes \varepsilon
_{H}\left( \left( x_{2}\right) _{-1}\right) 1_{Q}\otimes \left( x_{2}\right)
_{0}\otimes 1_{Q} \\
&=&x_{1}\otimes 1_{Q}\otimes x_{2}\otimes 1_{Q}
\end{eqnarray*}
\end{invisible}

Thus%
\begin{equation*}
\gamma ^{-1}\left( 1_{Q}\otimes x\right) =\gamma ^{-1}\left( 1_{Q}\otimes
x_{1}\right) \varepsilon \left( x_{2}\right) =\gamma ^{-1}\left(
1_{Q}\otimes x_{1}\right) \gamma \left( 1_{Q}\otimes x_{2}\right) =\left(
\gamma ^{-1}\ast \gamma \right) \left( 1_{Q}\otimes x\right) =\varepsilon
\left( x\right)
\end{equation*}%
and similarly $\gamma ^{-1}\left( x\otimes 1_{Q}\right) =\varepsilon \left(
x\right) .$
\end{remark}

\begin{lemma}
\label{lem:InvYD}Let $H$ be a Hopf algebra and let $C$ be a coalgebra in ${%
_{H}^{H}\mathcal{YD}}$. Given a map $\gamma \in {_{H}^{H}\mathcal{YD}}\left(
C,\Bbbk \right) ,$ we have that $\gamma $ is convolution invertible in ${%
_{H}^{H}\mathcal{YD}}\left( C,\Bbbk \right) $ if and only if it is
convolution invertible in $\mathbf{Vec}_{\Bbbk }\left( C,\Bbbk \right) $.
Moreover the inverse is the same.
\end{lemma}

\begin{proof}
Assume there is a $\Bbbk $-linear map $\gamma ^{-1}:C\rightarrow \Bbbk $
which is a convolution inverse of $\gamma $ in $\mathbf{Vec}_{\Bbbk }\left(
C,\Bbbk \right) $. By \cite[Remark 2.4(ii)]{ABM-cocycleproj}, $\gamma ^{-1}$
is left $H$-linear. Let us check that $\gamma ^{-1}$ is left $H$-colinear:%
\begin{align*}
c_{-1}\otimes \gamma ^{-1}\left( c_{0}\right) & =\left( c_{1}\right)
_{-1}1_{H}\otimes \gamma ^{-1}\left( \left( c_{1}\right) _{0}\right) \gamma
\left( c_{2}\right) \gamma ^{-1}\left( c_{3}\right) \\
& =\left( c_{1}\right) _{-1}\left( c_{2}\right) _{-1}\otimes \gamma
^{-1}\left( \left( c_{1}\right) _{0}\right) \gamma \left( \left(
c_{2}\right) _{0}\right) \gamma ^{-1}\left( c_{3}\right) \\
& \overset{(\ast )}{=}\left( c_{1}\right) _{-1}\otimes \gamma ^{-1}\left(
\left( \left( c_{1}\right) _{0}\right) _{1}\right) \gamma \left( \left(
\left( c_{1}\right) _{0}\right) _{2}\right) \gamma ^{-1}\left( c_{2}\right)
\\
& =\left( c_{1}\right) _{-1}\otimes \left( \gamma ^{-1}\ast \gamma \right)
\left( \left( c_{1}\right) _{0}\right) \gamma ^{-1}\left( c_{2}\right) \\
& =\left( c_{1}\right) _{-1}\otimes \varepsilon _{C}\left( \left(
c_{1}\right) _{0}\right) \gamma ^{-1}\left( c_{2}\right) \\
& \overset{(\ast )}{=}1_{H}\otimes \varepsilon _{C}\left( c_{1}\right)
\gamma ^{-1}\left( c_{2}\right) =1_{H}\otimes \gamma ^{-1}\left( c\right)
\end{align*}%
where in (*) we used that the comultiplication or the counit of $C$ is left $%
H$-colinear. Thus $\gamma $ is convolution invertible in ${_{H}^{H}\mathcal{%
YD}}\left( C,\Bbbk \right) $. The other implication is obvious.
\end{proof}

\begin{proposition}
\label{pro:deformYD}Let $H$ be a Hopf algebra and let $\left( Q,m,u,\Delta
,\varepsilon ,\omega \right) $ be a coquasi-bialgebra in ${_{H}^{H}\mathcal{%
YD}}$. Let $\gamma :Q\otimes Q\rightarrow \Bbbk $ be a gauge transformation
in ${_{H}^{H}\mathcal{YD}}$. Then%
\begin{equation*}
Q^{\gamma }:=\left( Q,m^{\gamma },u,\Delta ,\varepsilon ,\omega ^{\gamma
}\right)
\end{equation*}%
is a coquasi-bialgebra in ${_{H}^{H}\mathcal{YD}}$, where%
\begin{eqnarray*}
m^{\gamma } &:=&\gamma \ast m\ast \gamma ^{-1} \\
\omega ^{\gamma } &:=&\left( \varepsilon \otimes \gamma \right) \ast \gamma
\left( Q\otimes m\right) \ast \omega \ast \gamma ^{-1}\left( m\otimes
Q\right) \ast \left( \gamma ^{-1}\otimes \varepsilon \right) .
\end{eqnarray*}
\end{proposition}

\begin{proof}
The proof is analogue to \cite[Proposition XV.3.2]{Kassel} in its dual
version. We include some details for the reader's sake. Note that $Q^{\gamma
}$ has the same underlying coalgebra of $Q$ which is a coalgebra in ${%
_{H}^{H}\mathcal{YD}}$. The unit is also the same and hence it is a
coalgebra map in ${_{H}^{H}\mathcal{YD}}$. Since $m^{\gamma }$ is the
convolution product of morphisms in ${_{H}^{H}\mathcal{YD}},$ it results
that $m^{\gamma }$ is in ${_{H}^{H}\mathcal{YD}}$ as well.

\begin{invisible}
Let us check that the multiplication is $m^{\gamma }$ is a morphism in ${%
_{H}^{H}\mathcal{YD}}.$ Let $C=Q\otimes Q.$ For $h\in H,c\in C$, we have%
\begin{eqnarray*}
\left( \gamma \ast m\ast \gamma ^{-1}\right) \left( hc\right) &=&\gamma
\left( \left( hc\right) _{1}\right) \cdot m\left( \left( hc\right)
_{2}\right) \cdot \gamma ^{-1}\left( \left( hc\right) _{3}\right) \\
&=&\gamma \left( h_{1}c_{1}\right) \cdot m\left( h_{2}c_{2}\right) \cdot
\gamma ^{-1}\left( h_{3}c_{3}\right) \\
&=&\left( \varepsilon _{H}\left( h_{1}\right) \gamma \left( c_{1}\right)
\right) \cdot \left( h_{2}m\left( c_{2}\right) \right) \cdot \left(
\varepsilon _{H}\left( h_{3}\right) \gamma ^{-1}\left( c_{3}\right) \right)
\\
&=&\gamma \left( c_{1}\right) \cdot \left( hm\left( c_{2}\right) \right)
\cdot \gamma ^{-1}\left( c_{3}\right) \\
&=&h\left( \gamma \left( c_{1}\right) \cdot m\left( c_{2}\right) \cdot
\gamma ^{-1}\left( c_{3}\right) \right) \\
&=&h\left( \gamma \ast m\ast \gamma ^{-1}\right) \left( c\right) , \\
c_{-1}\otimes \left( \gamma \ast m\ast \gamma ^{-1}\right) \left(
c_{0}\right) &=&c_{-1}\otimes \gamma \left( \left( c_{0}\right) _{1}\right)
\cdot m\left( \left( c_{0}\right) _{2}\right) \cdot \gamma ^{-1}\left(
\left( c_{0}\right) _{3}\right) \\
&=&\left( c_{1}\right) _{-1}\left( c_{2}\right) _{-1}\left( c_{3}\right)
_{-1}\otimes \gamma \left( \left( c_{1}\right) _{0}\right) \cdot m\left(
\left( c_{2}\right) _{0}\right) \cdot \gamma ^{-1}\left( \left( c_{3}\right)
_{0}\right) \\
&=&1_{H}\left( m\left( c_{2}\right) \right) _{-1}1_{H}\otimes \gamma \left(
c_{1}\right) \cdot \left( m\left( c_{2}\right) \right) _{0}\cdot \gamma
^{-1}\left( c_{3}\right) \\
&=&\left( m\left( c_{2}\right) \right) _{-1}\otimes \gamma \left(
c_{1}\right) \cdot \left( m\left( c_{2}\right) \right) _{0}\cdot \gamma
^{-1}\left( c_{3}\right) \\
&=&\left( \gamma \left( c_{1}\right) \cdot m\left( c_{2}\right) \cdot \gamma
^{-1}\left( c_{3}\right) \right) _{-1}\otimes \left( \gamma \left(
c_{1}\right) \cdot m\left( c_{2}\right) \cdot \gamma ^{-1}\left(
c_{3}\right) \right) _{0} \\
&=&\left( \left( \gamma \ast m\ast \gamma ^{-1}\right) \left( c\right)
\right) _{-1}\otimes \left( \left( \gamma \ast m\ast \gamma ^{-1}\right)
\left( c\right) \right) _{0}.
\end{eqnarray*}
\end{invisible}

Since $m$ is a coalgebra map in ${_{H}^{H}\mathcal{YD}}$ and $\gamma $ is
convolution invertible with convolution inverse $\gamma ^{-1},$ it follows
that $m^{\gamma }$ is a coalgebra map in ${_{H}^{H}\mathcal{YD}}$.

\begin{invisible}
Let us check that $m^{\gamma }$ is a coalgebra map in ${_{H}^{H}\mathcal{YD}}%
:$%
\begin{eqnarray*}
\Delta m^{\gamma }\left( c\right) &=&\gamma \left( c_{1}\right) \cdot \left(
m\left( c_{2}\right) \right) _{1}\otimes \left( m\left( c_{2}\right) \right)
_{2}\cdot \gamma ^{-1}\left( c_{3}\right) \\
&=&\gamma \left( c_{1}\right) \cdot m\left( \left( c_{2}\right) _{1}\right)
\otimes m\left( \left( c_{2}\right) _{2}\right) \cdot \gamma ^{-1}\left(
c_{3}\right) \\
&=&\gamma \left( c_{1}\right) \cdot m\left( c_{2}\right) \cdot \gamma
^{-1}\left( c_{3}\right) \otimes \gamma \left( c_{4}\right) \cdot m\left(
c_{5}\right) \cdot \gamma ^{-1}\left( c_{6}\right) \\
&=&m^{\gamma }\left( c_{1}\right) \otimes m^{\gamma }\left( c_{2}\right) , \\
\varepsilon m^{\gamma }\left( c\right) &=&\gamma \left( c_{1}\right) \cdot
\varepsilon m\left( c_{2}\right) \cdot \gamma ^{-1}\left( c_{3}\right)
=\gamma \left( c_{1}\right) \cdot \varepsilon _{C}\left( c_{2}\right) \cdot
\gamma ^{-1}\left( c_{3}\right) \\
&=&\gamma \left( c_{1}\right) \cdot \gamma ^{-1}\left( c_{2}\right) =\left(
\gamma \ast \gamma ^{-1}\right) \left( c\right) =\varepsilon _{C}\left(
c\right) .
\end{eqnarray*}
\end{invisible}

By means of (\ref{form:delta1x}) and (\ref{form:deltax1}), one gets that $%
m^{\gamma }\left( 1_{Q}\otimes x\right) =x=m^{\gamma }\left( x\otimes
1_{Q}\right) .$

\begin{invisible}
Let us check that $m^{\gamma }$ is unitary. For every $x\in Q,$ we have%
\begin{eqnarray*}
m^{\gamma }\left( 1_{Q}\otimes x\right) &=&\gamma \left( \left( 1_{Q}\otimes
x\right) _{1}\right) \cdot m\left( \left( 1_{Q}\otimes x\right) _{2}\right)
\cdot \gamma ^{-1}\left( \left( 1_{Q}\otimes x\right) _{3}\right) \\
&&\overset{(\ref{form:delta1x})}{=}\gamma \left( 1_{Q}\otimes x_{1}\right)
\cdot m\left( 1_{Q}\otimes x_{2}\right) \cdot \gamma ^{-1}\left(
1_{Q}\otimes x_{3}\right) \\
&=&\varepsilon \left( x_{1}\right) \cdot m\left( 1_{Q}\otimes x_{2}\right)
\cdot \varepsilon \left( x_{3}\right) =m\left( 1_{Q}\otimes x\right) =x.
\end{eqnarray*}%
Similarly, using (\ref{form:deltax1}), we get $m^{\gamma }\left( x\otimes
1_{Q}\right) =x.$
\end{invisible}

Let us consider now $\omega ^{\gamma }.$ Since it is the convolution product
of morphisms in ${_{H}^{H}\mathcal{YD}},$ it results that $\omega ^{\gamma }$
is in ${_{H}^{H}\mathcal{YD}}$ as well.

\begin{invisible}
Let us check it is a morphism in ${_{H}^{H}\mathcal{YD}}$. Set $D:=Q\otimes
Q\otimes Q$ and let $d\in D.$ We compute%
\begin{eqnarray*}
\omega ^{\gamma }\left( hd\right) &=&\left( \varepsilon \otimes \gamma
\right) \left( h_{1}d_{1}\right) \cdot \gamma \left( Q\otimes m\right)
\left( h_{2}d_{2}\right) \cdot \omega \left( h_{3}d_{3}\right) \cdot \gamma
^{-1}\left( m\otimes Q\right) \left( h_{4}d_{4}\right) \cdot \left( \gamma
^{-1}\otimes \varepsilon \right) \left( h_{5}d_{5}\right) \\
&=&\left[
\begin{array}{c}
\varepsilon _{H}\left( h_{1}\right) \left( \varepsilon \otimes \gamma
\right) \left( d_{1}\right) \cdot \varepsilon _{H}\left( h_{2}\right) \gamma
\left( Q\otimes m\right) \left( d_{2}\right) \cdot \varepsilon _{H}\left(
h_{3}\right) \omega \left( d_{3}\right) \\
\cdot \varepsilon _{H}\left( h_{4}\right) \gamma ^{-1}\left( m\otimes
Q\right) \left( d_{4}\right) \cdot \varepsilon _{H}\left( h_{5}\right)
\left( \gamma ^{-1}\otimes \varepsilon \right) \left( d_{5}\right)%
\end{array}%
\right] \\
&=&\varepsilon _{H}\left( h\right) \omega ^{\gamma }\left( d\right)
\end{eqnarray*}%
and%
\begin{eqnarray*}
&&d_{-1}\otimes \omega ^{\gamma }\left( d_{0}\right) \\
&=&d_{-1}\otimes \left[
\begin{array}{c}
\left( \varepsilon \otimes \gamma \right) \left( \left( d_{0}\right)
_{1}\right) \cdot \gamma \left( Q\otimes m\right) \left( \left( d_{0}\right)
_{2}\right) \cdot \omega \left( \left( d_{0}\right) _{3}\right) \\
\cdot \gamma ^{-1}\left( m\otimes Q\right) \left( \left( d_{0}\right)
_{4}\right) \cdot \left( \gamma ^{-1}\otimes \varepsilon \right) \left(
\left( d_{0}\right) _{5}\right)%
\end{array}%
\right] \\
&=&\left( d_{1}\right) _{-1}\left( d_{2}\right) _{-1}\left( d_{3}\right)
_{-1}\left( d_{4}\right) _{-1}\left( d_{5}\right) _{-1}\otimes \left[
\begin{array}{c}
\left( \varepsilon \otimes \gamma \right) \left( \left( d_{1}\right)
_{0}\right) \cdot \gamma \left( Q\otimes m\right) \left( \left( d_{2}\right)
_{0}\right) \cdot \omega \left( \left( d_{3}\right) _{0}\right) \\
\cdot \gamma ^{-1}\left( m\otimes Q\right) \left( \left( d_{4}\right)
_{0}\right) \cdot \left( \gamma ^{-1}\otimes \varepsilon \right) \left(
\left( d_{5}\right) _{0}\right)%
\end{array}%
\right] \\
&=&1_{H}\otimes \left[
\begin{array}{c}
\left( \varepsilon \otimes \gamma \right) \left( d_{1}\right) \cdot \gamma
\left( Q\otimes m\right) \left( d_{2}\right) \cdot \omega \left( d_{3}\right)
\\
\cdot \gamma ^{-1}\left( m\otimes Q\right) \left( d_{4}\right) \cdot \left(
\gamma ^{-1}\otimes \varepsilon \right) \left( d_{5}\right)%
\end{array}%
\right] =1_{H}\otimes \omega ^{\gamma }\left( d\right) .
\end{eqnarray*}
\end{invisible}

Let us check that $\omega ^{\gamma }$ is unitary. Consider the map $\alpha
_{2}:Q\otimes Q\rightarrow Q\otimes Q\otimes Q$ defined by $\alpha
_{2}\left( x\otimes y\right) =x\otimes 1_{Q}\otimes y.$ The equalities (\ref%
{form:deltax1}) and (\ref{form:1Qcolin}) yield%
\begin{eqnarray*}
\left( \alpha _{2}\left( x\otimes y\right) \right) _{1}\otimes \left( \alpha
_{2}\left( x\otimes y\right) \right) _{2} &=&\alpha _{2}\left( x_{1}\otimes
\left( x_{2}\right) _{-1}y_{1}\right) \otimes \alpha _{2}\left( \left(
x_{2}\right) _{0}\otimes y_{2}\right) \\
&=&\alpha _{2}\left( \left( x\otimes y\right) _{1}\right) \otimes \alpha
_{2}\left( \left( x\otimes y\right) _{2}\right)
\end{eqnarray*}%
so that $\alpha _{2}$ is comultiplicative.

\begin{invisible}
For $x,y\in Q$, we have%
\begin{eqnarray*}
\left( x\otimes 1_{Q}\otimes y\right) _{1}\otimes \left( x\otimes
1_{Q}\otimes y\right) _{2} &=&\left( x\otimes 1_{Q}\right) _{1}\otimes
\left( \left( x\otimes 1_{Q}\right) _{2}\right) _{-1}y_{1}\otimes \left(
\left( x\otimes 1_{Q}\right) _{2}\right) _{0}\otimes y_{2} \\
&&\overset{(\ref{form:deltax1})}{=}\left( x_{1}\otimes 1_{Q}\right) \otimes
\left( x_{2}\otimes 1_{Q}\right) _{-1}y_{1}\otimes \left( x_{2}\otimes
1_{Q}\right) _{0}\otimes y_{2} \\
&=&x_{1}\otimes 1_{Q}\otimes \left( x_{2}\right) _{-1}\left( 1_{Q}\right)
_{-1}y_{1}\otimes \left( x_{2}\right) _{0}\otimes \left( 1_{Q}\right)
_{0}\otimes y_{2} \\
&&\overset{(\ref{form:1Qcolin})}{=}x_{1}\otimes 1_{Q}\otimes \left(
x_{2}\right) _{-1}y_{1}\otimes \left( x_{2}\right) _{0}\otimes 1_{Q}\otimes
y_{2}.
\end{eqnarray*}
\end{invisible}

Thus%
\begin{equation*}
\omega ^{\gamma }\alpha _{2}:=\left( \varepsilon \otimes \gamma \right)
\alpha _{2}\ast \gamma \left( Q\otimes m\right) \alpha _{2}\ast \omega
\alpha _{2}\ast \gamma ^{-1}\left( m\otimes Q\right) \alpha _{2}\ast \left(
\gamma ^{-1}\otimes \varepsilon \right) \alpha _{2}
\end{equation*}%
and computing the factors of this convolution products one gets
\begin{gather*}
\left( \varepsilon \otimes \gamma \right) \alpha _{2}=\varepsilon \otimes
\varepsilon ,\quad \gamma \left( Q\otimes m\right) \alpha _{2}=\gamma ,\quad
\omega \alpha _{2}=\varepsilon \otimes \varepsilon , \\
\gamma ^{-1}\left( m\otimes Q\right) \alpha _{2}=\gamma ^{-1},\quad \left(
\gamma ^{-1}\otimes \varepsilon \right) \alpha _{2}=\varepsilon \otimes
\varepsilon
\end{gather*}%
and hence $\omega ^{\gamma }\alpha _{2}=\gamma \ast \gamma ^{-1}=\varepsilon
\otimes \varepsilon ,$ which means that $\omega ^{\gamma }\left( x\otimes
1_{Q}\otimes y\right) =\varepsilon \left( x\right) \varepsilon \left(
y\right) $ for every $x,y\in Q.$

\begin{invisible}
Thus, for every $x,y\in Q$, we have
\begin{eqnarray*}
\omega ^{\gamma }\left( x\otimes 1_{Q}\otimes y\right) &=&\omega ^{\gamma
}\alpha _{2}\left( x\otimes y\right) \\
&=&\left[
\begin{array}{c}
\left( \varepsilon \otimes \gamma \right) \left( \left( \alpha _{2}\left(
x\otimes y\right) \right) _{1}\right) \cdot \gamma \left( Q\otimes m\right)
\left( \left( \alpha _{2}\left( x\otimes y\right) \right) _{2}\right) \\
\cdot \omega \left( \left( \alpha _{2}\left( x\otimes y\right) \right)
_{3}\right) \cdot \gamma ^{-1}\left( m\otimes Q\right) \left( \left( \alpha
_{2}\left( x\otimes y\right) \right) _{4}\right) \cdot \left( \gamma
^{-1}\otimes \varepsilon \right) \left( \left( \alpha _{2}\left( x\otimes
y\right) \right) _{5}\right)%
\end{array}%
\right] \\
&=&\left[
\begin{array}{c}
\left( \varepsilon \otimes \gamma \right) \left( \alpha _{2}\left( \left(
x\otimes y\right) _{1}\right) \right) \cdot \gamma \left( Q\otimes m\right)
\left( \alpha _{2}\left( \left( x\otimes y\right) _{2}\right) \right) \\
\cdot \omega \left( \alpha _{2}\left( \left( x\otimes y\right) _{3}\right)
\right) \cdot \gamma ^{-1}\left( m\otimes Q\right) \left( \alpha _{2}\left(
\left( x\otimes y\right) _{4}\right) \right) \cdot \left( \gamma
^{-1}\otimes \varepsilon \right) \left( \alpha _{2}\left( \left( x\otimes
y\right) _{5}\right) \right)%
\end{array}%
\right] .
\end{eqnarray*}%
Note that%
\begin{eqnarray*}
\left( \varepsilon \otimes \gamma \right) \alpha _{2}\left( x^{\prime
}\otimes y^{\prime }\right) &=&\left( \varepsilon \otimes \gamma \right)
\left( x^{\prime }\otimes 1_{Q}\otimes y^{\prime }\right) =\varepsilon
\left( x^{\prime }\right) \gamma \left( 1_{Q}\otimes y^{\prime }\right)
=\varepsilon \left( x^{\prime }\right) \varepsilon \left( y^{\prime }\right)
, \\
\gamma \left( Q\otimes m\right) \alpha _{2}\left( x^{\prime }\otimes
y^{\prime }\right) &=&\gamma \left( Q\otimes m\right) \left( x^{\prime
}\otimes 1_{Q}\otimes y^{\prime }\right) =\gamma \left( x^{\prime }\otimes
y^{\prime }\right) , \\
\omega \alpha _{2}\left( x^{\prime }\otimes y^{\prime }\right) &=&\omega
\left( x^{\prime }\otimes 1_{Q}\otimes y^{\prime }\right) =\varepsilon
\left( x^{\prime }\right) \varepsilon \left( y^{\prime }\right) , \\
\gamma ^{-1}\left( m\otimes Q\right) \alpha _{2}\left( x^{\prime }\otimes
y^{\prime }\right) &=&\gamma ^{-1}\left( m\otimes Q\right) \left( x^{\prime
}\otimes 1_{Q}\otimes y^{\prime }\right) =\gamma ^{-1}\left( x^{\prime
}\otimes y^{\prime }\right) , \\
\left( \gamma ^{-1}\otimes \varepsilon \right) \alpha _{2}\left( x^{\prime
}\otimes y^{\prime }\right) &=&\left( \gamma ^{-1}\otimes \varepsilon
\right) \left( x^{\prime }\otimes 1_{Q}\otimes y^{\prime }\right) =\gamma
^{-1}\left( x^{\prime }\otimes 1_{Q}\right) \varepsilon \left( y^{\prime
}\right) =\varepsilon \left( x^{\prime }\right) \varepsilon \left( y^{\prime
}\right)
\end{eqnarray*}%
so that%
\begin{equation*}
\omega ^{\gamma }\left( x\otimes 1_{Q}\otimes y\right) =\gamma \left( \left(
x\otimes y\right) _{1}\right) \cdot \gamma ^{-1}\left( \left( x\otimes
y\right) _{2}\right) =\varepsilon \left( x\right) \varepsilon \left(
y\right) .
\end{equation*}
\end{invisible}

Similarly, considering $\alpha _{1}:Q\otimes Q\rightarrow Q\otimes Q\otimes
Q $ defined by $\alpha _{2}\left( x\otimes y\right) =1_{Q}\otimes x\otimes y$%
, one proves that $\omega ^{\gamma }\left( 1_{Q}\otimes x\otimes y\right)
=\varepsilon \left( x\right) \varepsilon \left( y\right) .$ A symmetric
argument shows that $\omega ^{\gamma }\left( x\otimes y\otimes 1_{Q}\right)
=\varepsilon \left( x\right) \varepsilon \left( y\right) .$

\begin{invisible}
We have%
\begin{eqnarray*}
\left( 1_{Q}\otimes x\otimes y\right) _{1}\otimes \left( 1_{Q}\otimes
x\otimes y\right) _{2} &=&\left( 1_{Q}\otimes x\right) _{1}\otimes \left(
\left( 1_{Q}\otimes x\right) _{2}\right) _{-1}y_{1}\otimes \left( \left(
1_{Q}\otimes x\right) _{2}\right) _{0}\otimes y_{2} \\
&&\overset{(\ref{form:delta1x})}{=}\left( 1_{Q}\otimes x_{1}\right) \otimes
\left( 1_{Q}\otimes x_{2}\right) _{-1}y_{1}\otimes \left( 1_{Q}\otimes
x_{2}\right) _{0}\otimes y_{2} \\
&=&\left( 1_{Q}\otimes x_{1}\right) \otimes \left( x_{2}\right)
_{-1}y_{1}\otimes 1_{Q}\otimes \left( x_{2}\right) _{0}\otimes y_{2}
\end{eqnarray*}%
and hence also the map $\alpha _{1}:Q\otimes Q\rightarrow Q\otimes Q\otimes
Q $ defined by $\alpha _{2}\left( x\otimes y\right) =1_{Q}\otimes x\otimes y$
is a coalgebra map. Hence one gets that $\omega ^{\gamma }\left(
1_{Q}\otimes x\otimes y\right) =\varepsilon \left( x\right) \varepsilon
\left( y\right) $ as a consequence of the following computations%
\begin{eqnarray*}
\left( \varepsilon \otimes \gamma \right) \alpha _{1}\left( x^{\prime
}\otimes y^{\prime }\right) &=&\left( \varepsilon \otimes \gamma \right)
\left( 1_{Q}\otimes x^{\prime }\otimes y^{\prime }\right) =\gamma \left(
x^{\prime }\otimes y^{\prime }\right) , \\
\gamma \left( Q\otimes m\right) \alpha _{1}\left( x^{\prime }\otimes
y^{\prime }\right) &=&\gamma \left( Q\otimes m\right) \left( 1_{Q}\otimes
x^{\prime }\otimes y^{\prime }\right) =\gamma \left( 1_{Q}\otimes x^{\prime
}y^{\prime }\right) =\varepsilon \left( x^{\prime }\right) \varepsilon
\left( y^{\prime }\right) , \\
\omega \alpha _{1}\left( x^{\prime }\otimes y^{\prime }\right) &=&\omega
\left( 1_{Q}\otimes x^{\prime }\otimes y^{\prime }\right) =\varepsilon
\left( x^{\prime }\right) \varepsilon \left( y^{\prime }\right) , \\
\gamma ^{-1}\left( m\otimes Q\right) \alpha _{1}\left( x^{\prime }\otimes
y^{\prime }\right) &=&\gamma ^{-1}\left( m\otimes Q\right) \left(
1_{Q}\otimes x^{\prime }\otimes y^{\prime }\right) =\gamma ^{-1}\left(
x^{\prime }\otimes y^{\prime }\right) , \\
\left( \gamma ^{-1}\otimes \varepsilon \right) \alpha _{1}\left( x^{\prime
}\otimes y^{\prime }\right) &=&\left( \gamma ^{-1}\otimes \varepsilon
\right) \left( 1_{Q}\otimes x^{\prime }\otimes y^{\prime }\right) =\gamma
^{-1}\left( 1_{Q}\otimes x^{\prime }\right) \varepsilon \left( y^{\prime
}\right) =\varepsilon \left( x^{\prime }\right) \varepsilon \left( y^{\prime
}\right) .
\end{eqnarray*}%
A symmetric argument shows that $\omega ^{\gamma }\left( x\otimes y\otimes
1_{Q}\right) =\varepsilon \left( x\right) \varepsilon \left( y\right) .$
\end{invisible}

Note that, by Lemma \ref{lem:InvYD}, $\omega ^{\gamma }$ is convolution
invertible in ${_{H}^{H}\mathcal{YD}}\left( D,\Bbbk \right) $ as it is
convolution invertible in $\mathbf{Vec}_{\Bbbk }\left( D,\Bbbk \right) $.

Let us check that the multiplication is quasi-associative. By \cite[Lemma
2.10 formula (2.7)]{ABM}, we have
\begin{eqnarray*}
m^{\gamma }\left( Q\otimes \gamma \ast m\ast \gamma ^{-1}\right) &=&\left(
\varepsilon \otimes \gamma \right) \ast m^{\gamma }\left( Q\otimes m\right)
\ast \left( \varepsilon \otimes \gamma ^{-1}\right) \text{,} \\
\left( \varepsilon \otimes \gamma ^{-1}\right) \ast \left( \varepsilon
\otimes \gamma \right) &=&\varepsilon \otimes \left( \gamma ^{-1}\ast \gamma
\right) =\varepsilon \otimes \varepsilon \otimes \varepsilon , \\
m^{\gamma }\left( m^{\gamma }\otimes Q\right) &=&m^{\gamma }\left( \gamma
\ast m\ast \gamma ^{-1}\otimes Q\right) =\left( \gamma \otimes \varepsilon
\right) \ast m^{\gamma }\left( m\ast \gamma ^{-1}\otimes Q\right) \\
&=&\left( \gamma \otimes \varepsilon \right) \ast m^{\gamma }\left( m\otimes
Q\right) \ast \left( \gamma ^{-1}\otimes \varepsilon \right) , \\
\left( \gamma ^{-1}\otimes \varepsilon \right) \ast \left( \gamma \otimes
\varepsilon \right) &=&\left( \left( \gamma ^{-1}\ast \gamma \right) \otimes
\varepsilon \right) =\varepsilon \otimes \varepsilon \otimes \varepsilon .
\end{eqnarray*}%
By using these equalities one obtains%
\begin{eqnarray*}
m^{\gamma }\left( Q\otimes m^{\gamma }\right) \ast \omega ^{\gamma }
&=&\left( \varepsilon \otimes \gamma \right) \ast \gamma \left( Q\otimes
m\right) \ast m\left( Q\otimes m\right) \ast \omega \ast \gamma ^{-1}\left(
m\otimes Q\right) \ast \left( \gamma ^{-1}\otimes \varepsilon \right) , \\
\omega ^{\gamma }\ast m^{\gamma }\left( m^{\gamma }\otimes Q\right)
&=&\left( \varepsilon \otimes \gamma \right) \ast \gamma \left( Q\otimes
m\right) \ast \omega \ast m\left( m\otimes Q\right) \ast \gamma ^{-1}\left(
m\otimes Q\right) \ast \left( \gamma ^{-1}\otimes \varepsilon \right)
\end{eqnarray*}%
so that $\omega ^{\gamma }\ast m^{\gamma }\left( m^{\gamma }\otimes Q\right)
=m^{\gamma }\left( Q\otimes m^{\gamma }\right) \ast \omega ^{\gamma }.$

\begin{invisible}
We have that%
\begin{eqnarray*}
&&m^{\gamma }\left( Q\otimes m^{\gamma }\right) \ast \omega ^{\gamma } \\
&=&m^{\gamma }\left( Q\otimes \gamma \ast m\ast \gamma ^{-1}\right) \ast
\omega ^{\gamma } \\
&=&\left( \varepsilon \otimes \gamma \right) \ast m^{\gamma }\left( Q\otimes
m\right) \ast \left( \varepsilon \otimes \gamma ^{-1}\right) \ast \omega
^{\gamma } \\
&=&\left( \varepsilon \otimes \gamma \right) \ast m^{\gamma }\left( Q\otimes
m\right) \ast \left( \varepsilon \otimes \gamma ^{-1}\right) \ast \left(
\varepsilon \otimes \gamma \right) \ast \gamma \left( Q\otimes m\right) \ast
\omega \ast \gamma ^{-1}\left( m\otimes Q\right) \ast \left( \gamma
^{-1}\otimes \varepsilon \right) \\
&=&\left( \varepsilon \otimes \gamma \right) \ast m^{\gamma }\left( Q\otimes
m\right) \ast \gamma \left( Q\otimes m\right) \ast \omega \ast \gamma
^{-1}\left( m\otimes Q\right) \ast \left( \gamma ^{-1}\otimes \varepsilon
\right) \\
&=&\left( \varepsilon \otimes \gamma \right) \ast \left( \gamma \ast m\ast
\gamma ^{-1}\right) \left( Q\otimes m\right) \ast \gamma \left( Q\otimes
m\right) \ast \omega \ast \gamma ^{-1}\left( m\otimes Q\right) \ast \left(
\gamma ^{-1}\otimes \varepsilon \right) \\
&=&\left( \varepsilon \otimes \gamma \right) \ast \gamma \left( Q\otimes
m\right) \ast m\left( Q\otimes m\right) \ast \gamma ^{-1}\left( Q\otimes
m\right) \ast \gamma \left( Q\otimes m\right) \ast \omega \ast \gamma
^{-1}\left( m\otimes Q\right) \ast \left( \gamma ^{-1}\otimes \varepsilon
\right) \\
&=&\left( \varepsilon \otimes \gamma \right) \ast \gamma \left( Q\otimes
m\right) \ast m\left( Q\otimes m\right) \ast \left( \gamma ^{-1}\ast \gamma
\right) \left( Q\otimes m\right) \ast \omega \ast \gamma ^{-1}\left(
m\otimes Q\right) \ast \left( \gamma ^{-1}\otimes \varepsilon \right) \\
&=&\left( \varepsilon \otimes \gamma \right) \ast \gamma \left( Q\otimes
m\right) \ast m\left( Q\otimes m\right) \ast \omega \ast \gamma ^{-1}\left(
m\otimes Q\right) \ast \left( \gamma ^{-1}\otimes \varepsilon \right)
\end{eqnarray*}%
and similarly we have%
\begin{eqnarray*}
&&\omega ^{\gamma }\ast m^{\gamma }\left( m^{\gamma }\otimes Q\right) \\
&=&\omega ^{\gamma }\ast \left( \gamma \otimes \varepsilon \right) \ast
m^{\gamma }\left( m\otimes Q\right) \ast \left( \gamma ^{-1}\otimes
\varepsilon \right) \\
&=&\left( \varepsilon \otimes \gamma \right) \ast \gamma \left( Q\otimes
m\right) \ast \omega \ast \gamma ^{-1}\left( m\otimes Q\right) \ast \left(
\gamma ^{-1}\otimes \varepsilon \right) \ast \left( \gamma \otimes
\varepsilon \right) \ast m^{\gamma }\left( m\otimes Q\right) \ast \left(
\gamma ^{-1}\otimes \varepsilon \right) \\
&=&\left( \varepsilon \otimes \gamma \right) \ast \gamma \left( Q\otimes
m\right) \ast \omega \ast \gamma ^{-1}\left( m\otimes Q\right) \ast
m^{\gamma }\left( m\otimes Q\right) \ast \left( \gamma ^{-1}\otimes
\varepsilon \right) \\
&=&\left( \varepsilon \otimes \gamma \right) \ast \gamma \left( Q\otimes
m\right) \ast \omega \ast \gamma ^{-1}\left( m\otimes Q\right) \ast \left(
\gamma \ast m\ast \gamma ^{-1}\right) \left( m\otimes Q\right) \ast \left(
\gamma ^{-1}\otimes \varepsilon \right) \\
&=&\left( \varepsilon \otimes \gamma \right) \ast \gamma \left( Q\otimes
m\right) \ast \omega \ast \gamma ^{-1}\left( m\otimes Q\right) \ast \gamma
\left( m\otimes Q\right) \ast m\left( m\otimes Q\right) \ast \gamma
^{-1}\left( m\otimes Q\right) \ast \left( \gamma ^{-1}\otimes \varepsilon
\right) \\
&=&\left( \varepsilon \otimes \gamma \right) \ast \gamma \left( Q\otimes
m\right) \ast \omega \ast \left( \gamma ^{-1}\ast \gamma \right) \left(
m\otimes Q\right) \ast m\left( m\otimes Q\right) \ast \gamma ^{-1}\left(
m\otimes Q\right) \ast \left( \gamma ^{-1}\otimes \varepsilon \right) \\
&=&\left( \varepsilon \otimes \gamma \right) \ast \gamma \left( Q\otimes
m\right) \ast \omega \ast m\left( m\otimes Q\right) \ast \gamma ^{-1}\left(
m\otimes Q\right) \ast \left( \gamma ^{-1}\otimes \varepsilon \right) \\
&=&\left( \varepsilon \otimes \gamma \right) \ast \gamma \left( Q\otimes
m\right) \ast m\left( Q\otimes m\right) \ast \omega \ast \gamma ^{-1}\left(
m\otimes Q\right) \ast \left( \gamma ^{-1}\otimes \varepsilon \right)
\end{eqnarray*}%
so that $\omega ^{\gamma }\ast m^{\gamma }\left( m^{\gamma }\otimes Q\right)
=m^{\gamma }\left( Q\otimes m^{\gamma }\right) \ast \omega ^{\gamma }.$
\end{invisible}

It remains to check that $\omega ^{\gamma }$ is a reassociator. By \cite[%
Lemma 2.10 formula (2.7)]{ABM}, we have%
\begin{eqnarray*}
\omega ^{\gamma }\left( Q\otimes Q\otimes \gamma \ast m\ast \gamma
^{-1}\right) &=&\left( \varepsilon \otimes \varepsilon \otimes \gamma
\right) \ast \omega ^{\gamma }\left( Q\otimes Q\otimes m\right) \ast \left(
\varepsilon \otimes \varepsilon \otimes \gamma ^{-1}\right) , \\
\omega ^{\gamma }\left( \gamma \ast m\ast \gamma ^{-1}\otimes Q\otimes
Q\right) &=&\left( \gamma \otimes \varepsilon \otimes \varepsilon \right)
\ast \omega ^{\gamma }\left( m\otimes Q\otimes Q\right) \ast \left( \gamma
^{-1}\otimes \varepsilon \otimes \varepsilon \right) , \\
\left( \gamma \otimes \varepsilon \otimes \varepsilon \right) \ast \left(
\varepsilon \otimes \varepsilon \otimes \gamma \right) &=&\gamma \otimes
\gamma =\left( \varepsilon \otimes \varepsilon \otimes \gamma \right) \ast
\left( \gamma \otimes \varepsilon \otimes \varepsilon \right) .
\end{eqnarray*}%
By using these equalities one obtains%
\begin{eqnarray*}
&&\omega ^{\gamma }\left( Q\otimes Q\otimes m^{\gamma }\right) \ast \omega
^{\gamma }\left( m^{\gamma }\otimes Q\otimes Q\right) \\
&&=\left[
\begin{array}{c}
\left( \varepsilon \otimes \varepsilon \otimes \gamma \right) \ast \left(
\varepsilon \otimes \gamma \left( Q\otimes m\right) \right) \ast \gamma
\left( Q\otimes m\left( Q\otimes m\right) \right) \\
\ast \omega \left( Q\otimes Q\otimes m\right) \ast \omega \left( m\otimes
Q\otimes Q\right) \\
\ast \gamma ^{-1}\left( m\left( m\otimes Q\right) \otimes Q\right) \ast
\left( \gamma ^{-1}\left( m\otimes Q\right) \otimes \varepsilon \right) \ast
\left( \gamma ^{-1}\otimes \varepsilon \otimes \varepsilon \right)%
\end{array}%
\right]
\end{eqnarray*}%
and%
\begin{eqnarray*}
&&\left( \varepsilon \otimes \omega ^{\gamma }\right) \ast \omega ^{\gamma
}\left( Q\otimes m^{\gamma }\otimes Q\right) \ast \left( \omega ^{\gamma
}\otimes \varepsilon \right) \\
&&=\left[
\begin{array}{c}
\left( \varepsilon \otimes \varepsilon \otimes \gamma \right) \ast \left(
\varepsilon \otimes \gamma \left( Q\otimes m\right) \right) \ast \gamma
\left( Q\otimes m\left( Q\otimes m\right) \right) \\
\ast \left( \varepsilon \otimes \omega \right) \ast \omega \left( Q\otimes
m\otimes Q\right) \ast \left( \omega \otimes \varepsilon \right) \\
\ast \gamma ^{-1}\left( m\left( m\otimes Q\right) \otimes Q\right) \ast
\left( \gamma ^{-1}\left( m\otimes Q\right) \otimes \varepsilon \right) \ast
\left( \gamma ^{-1}\otimes \varepsilon \otimes \varepsilon \right)%
\end{array}%
\right] .
\end{eqnarray*}

\begin{invisible}
We compute%
\begin{eqnarray*}
&&\omega ^{\gamma }\left( Q\otimes Q\otimes m^{\gamma }\right) \ast \omega
^{\gamma }\left( m^{\gamma }\otimes Q\otimes Q\right) \\
&=&\left[
\begin{array}{c}
\left( \varepsilon \otimes \varepsilon \otimes \gamma \right) \ast \omega
^{\gamma }\left( Q\otimes Q\otimes m\right) \ast \left( \varepsilon \otimes
\varepsilon \otimes \gamma ^{-1}\right) \\
\ast \left( \gamma \otimes \varepsilon \otimes \varepsilon \right) \ast
\omega ^{\gamma }\left( m\otimes Q\otimes Q\right) \ast \left( \gamma
^{-1}\otimes \varepsilon \otimes \varepsilon \right)%
\end{array}%
\right] \\
&=&\left[
\begin{array}{c}
\left( \varepsilon \otimes \varepsilon \otimes \gamma \right) \ast \left(
\varepsilon \otimes \gamma \right) \left( Q\otimes Q\otimes m\right) \ast
\gamma \left( Q\otimes m\right) \left( Q\otimes Q\otimes m\right) \ast
\omega \left( Q\otimes Q\otimes m\right) \\
\ast \gamma ^{-1}\left( m\otimes Q\right) \left( Q\otimes Q\otimes m\right)
\ast \left( \gamma ^{-1}\otimes \varepsilon \right) \left( Q\otimes Q\otimes
m\right) \ast \left( \varepsilon \otimes \varepsilon \otimes \gamma
^{-1}\right) \\
\ast \left( \gamma \otimes \varepsilon \otimes \varepsilon \right) \ast
\left( \varepsilon \otimes \gamma \right) \left( m\otimes Q\otimes Q\right)
\ast \gamma \left( Q\otimes m\right) \left( m\otimes Q\otimes Q\right) \\
\ast \omega \left( m\otimes Q\otimes Q\right) \ast \gamma ^{-1}\left(
m\otimes Q\right) \left( m\otimes Q\otimes Q\right) \ast \left( \gamma
^{-1}\otimes \varepsilon \right) \left( m\otimes Q\otimes Q\right) \ast
\left( \gamma ^{-1}\otimes \varepsilon \otimes \varepsilon \right)%
\end{array}%
\right] \\
&=&\left[
\begin{array}{c}
\left( \varepsilon \otimes \varepsilon \otimes \gamma \right) \ast \left(
\varepsilon \otimes \gamma \left( Q\otimes m\right) \right) \ast \gamma
\left( Q\otimes m\left( Q\otimes m\right) \right) \ast \omega \left(
Q\otimes Q\otimes m\right) \\
\ast \gamma ^{-1}\left( m\otimes m\right) \ast \left( \gamma ^{-1}\otimes
\varepsilon \otimes \varepsilon \right) \ast \left( \varepsilon \otimes
\varepsilon \otimes \gamma ^{-1}\right) \\
\ast \left( \gamma \otimes \varepsilon \otimes \varepsilon \right) \ast
\left( \varepsilon \otimes \varepsilon \otimes \gamma \right) \ast \gamma
\left( m\otimes m\right) \\
\ast \omega \left( m\otimes Q\otimes Q\right) \ast \gamma ^{-1}\left(
m\left( m\otimes Q\right) \otimes Q\right) \ast \left( \gamma ^{-1}\left(
m\otimes Q\right) \otimes \varepsilon \right) \ast \left( \gamma
^{-1}\otimes \varepsilon \otimes \varepsilon \right)%
\end{array}%
\right] \\
&=&\left[
\begin{array}{c}
\left( \varepsilon \otimes \varepsilon \otimes \gamma \right) \ast \left(
\varepsilon \otimes \gamma \left( Q\otimes m\right) \right) \ast \gamma
\left( Q\otimes m\left( Q\otimes m\right) \right) \ast \omega \left(
Q\otimes Q\otimes m\right) \\
\ast \gamma ^{-1}\left( m\otimes m\right) \ast \left( \gamma ^{-1}\otimes
\varepsilon \otimes \varepsilon \right) \ast \left( \varepsilon \otimes
\varepsilon \otimes \gamma ^{-1}\right) \\
\ast \left( \varepsilon \otimes \varepsilon \otimes \gamma \right) \ast
\left( \gamma \otimes \varepsilon \otimes \varepsilon \right) \ast \gamma
\left( m\otimes m\right) \\
\ast \omega \left( m\otimes Q\otimes Q\right) \ast \gamma ^{-1}\left(
m\left( m\otimes Q\right) \otimes Q\right) \ast \left( \gamma ^{-1}\left(
m\otimes Q\right) \otimes \varepsilon \right) \ast \left( \gamma
^{-1}\otimes \varepsilon \otimes \varepsilon \right)%
\end{array}%
\right] \\
&=&\left[
\begin{array}{c}
\left( \varepsilon \otimes \varepsilon \otimes \gamma \right) \ast \left(
\varepsilon \otimes \gamma \left( Q\otimes m\right) \right) \ast \gamma
\left( Q\otimes m\left( Q\otimes m\right) \right) \ast \omega \left(
Q\otimes Q\otimes m\right) \\
\ast \gamma ^{-1}\left( m\otimes m\right) \ast \gamma \left( m\otimes
m\right) \\
\ast \omega \left( m\otimes Q\otimes Q\right) \ast \gamma ^{-1}\left(
m\left( m\otimes Q\right) \otimes Q\right) \ast \left( \gamma ^{-1}\left(
m\otimes Q\right) \otimes \varepsilon \right) \ast \left( \gamma
^{-1}\otimes \varepsilon \otimes \varepsilon \right)%
\end{array}%
\right] \\
&=&\left[
\begin{array}{c}
\left( \varepsilon \otimes \varepsilon \otimes \gamma \right) \ast \left(
\varepsilon \otimes \gamma \left( Q\otimes m\right) \right) \ast \gamma
\left( Q\otimes m\left( Q\otimes m\right) \right) \\
\ast \omega \left( Q\otimes Q\otimes m\right) \ast \omega \left( m\otimes
Q\otimes Q\right) \\
\ast \gamma ^{-1}\left( m\left( m\otimes Q\right) \otimes Q\right) \ast
\left( \gamma ^{-1}\left( m\otimes Q\right) \otimes \varepsilon \right) \ast
\left( \gamma ^{-1}\otimes \varepsilon \otimes \varepsilon \right)%
\end{array}%
\right] .
\end{eqnarray*}%
Moreover%
\begin{eqnarray*}
&&\left( \varepsilon \otimes \omega ^{\gamma }\right) \ast \omega ^{\gamma
}\left( Q\otimes m^{\gamma }\otimes Q\right) \ast \left( \omega ^{\gamma
}\otimes \varepsilon \right) \\
&=&\left( \varepsilon \otimes \omega ^{\gamma }\right) \ast \omega ^{\gamma
}\left( Q\otimes \left( \gamma \ast m\ast \gamma ^{-1}\right) \otimes
Q\right) \ast \left( \omega ^{\gamma }\otimes \varepsilon \right) \\
&=&\left( \varepsilon \otimes \omega ^{\gamma }\right) \ast \omega ^{\gamma
}\left( Q\otimes \left( \left( \gamma \otimes \varepsilon \right) \ast
\left( m\otimes Q\right) \ast \left( \gamma ^{-1}\otimes \varepsilon \right)
\right) \right) \ast \left( \omega ^{\gamma }\otimes \varepsilon \right) \\
&=&\left( \varepsilon \otimes \omega ^{\gamma }\right) \ast \left(
\varepsilon \otimes \gamma \otimes \varepsilon \right) \ast \omega ^{\gamma
}\left( Q\otimes m\otimes Q\right) \ast \left( \varepsilon \otimes \gamma
^{-1}\otimes \varepsilon \right) \ast \left( \omega ^{\gamma }\otimes
\varepsilon \right) \\
&=&\left[
\begin{array}{c}
\left( \varepsilon \otimes \varepsilon \otimes \gamma \right) \ast \left(
\varepsilon \otimes \gamma \left( Q\otimes m\right) \right) \ast \left(
\varepsilon \otimes \omega \right) \ast \left( \varepsilon \otimes \gamma
^{-1}\left( m\otimes Q\right) \right) \ast \left( \varepsilon \otimes \gamma
^{-1}\otimes \varepsilon \right) \ast \\
\left( \varepsilon \otimes \gamma \otimes \varepsilon \right) \ast \left(
\varepsilon \otimes \gamma \right) \left( Q\otimes m\otimes Q\right) \ast \\
\gamma \left( Q\otimes m\right) \left( Q\otimes m\otimes Q\right) \ast
\omega \left( Q\otimes m\otimes Q\right) \ast \gamma ^{-1}\left( m\otimes
Q\right) \left( Q\otimes m\otimes Q\right) \\
\ast \left( \gamma ^{-1}\otimes \varepsilon \right) \left( Q\otimes m\otimes
Q\right) \ast \left( \varepsilon \otimes \gamma ^{-1}\otimes \varepsilon
\right) \\
\ast \left( \varepsilon \otimes \gamma \otimes \varepsilon \right) \ast
\left( \gamma \left( Q\otimes m\right) \otimes \varepsilon \right) \ast
\left( \omega \otimes \varepsilon \right) \ast \left( \gamma ^{-1}\left(
m\otimes Q\right) \otimes \varepsilon \right) \ast \left( \gamma
^{-1}\otimes \varepsilon \otimes \varepsilon \right)%
\end{array}%
\right] \\
&=&\left[
\begin{array}{c}
\left( \varepsilon \otimes \varepsilon \otimes \gamma \right) \ast \left(
\varepsilon \otimes \gamma \left( Q\otimes m\right) \right) \ast \left(
\varepsilon \otimes \omega \right) \ast \left( \varepsilon \otimes \gamma
^{-1}\left( m\otimes Q\right) \right) \ast \\
\left( \varepsilon \otimes \gamma \left( m\otimes Q\right) \right) \ast \\
\gamma \left( Q\otimes m\left( m\otimes Q\right) \right) \ast \omega \left(
Q\otimes m\otimes Q\right) \ast \gamma ^{-1}\left( m\left( Q\otimes m\right)
\otimes Q\right) \\
\ast \left( \gamma ^{-1}\left( Q\otimes m\right) \otimes \varepsilon \right)
\ast \\
\ast \left( \gamma \left( Q\otimes m\right) \otimes \varepsilon \right) \ast
\left( \omega \otimes \varepsilon \right) \ast \left( \gamma ^{-1}\left(
m\otimes Q\right) \otimes \varepsilon \right) \ast \left( \gamma
^{-1}\otimes \varepsilon \otimes \varepsilon \right)%
\end{array}%
\right] \\
&=&\left[
\begin{array}{c}
\left( \varepsilon \otimes \varepsilon \otimes \gamma \right) \ast \left(
\varepsilon \otimes \gamma \left( Q\otimes m\right) \right) \ast \left(
\varepsilon \otimes \omega \right) \ast \\
\gamma \left( Q\otimes m\left( m\otimes Q\right) \right) \ast \omega \left(
Q\otimes m\otimes Q\right) \ast \gamma ^{-1}\left( m\left( Q\otimes m\right)
\otimes Q\right) \\
\ast \left( \omega \otimes \varepsilon \right) \ast \left( \gamma
^{-1}\left( m\otimes Q\right) \otimes \varepsilon \right) \ast \left( \gamma
^{-1}\otimes \varepsilon \otimes \varepsilon \right)%
\end{array}%
\right] \\
&=&\left[
\begin{array}{c}
\left( \varepsilon \otimes \varepsilon \otimes \gamma \right) \ast \left(
\varepsilon \otimes \gamma \left( Q\otimes m\right) \right) \ast \\
\gamma \left( Q\otimes \omega \ast m\left( m\otimes Q\right) \right) \ast
\omega \left( Q\otimes m\otimes Q\right) \ast \gamma ^{-1}\left( m\left(
Q\otimes m\right) \ast \omega \otimes Q\right) \\
\ast \left( \gamma ^{-1}\left( m\otimes Q\right) \otimes \varepsilon \right)
\ast \left( \gamma ^{-1}\otimes \varepsilon \otimes \varepsilon \right)%
\end{array}%
\right] \\
&=&\left[
\begin{array}{c}
\left( \varepsilon \otimes \varepsilon \otimes \gamma \right) \ast \left(
\varepsilon \otimes \gamma \left( Q\otimes m\right) \right) \ast \\
\gamma \left( Q\otimes m\left( Q\otimes m\right) \ast \omega \right) \ast
\omega \left( Q\otimes m\otimes Q\right) \ast \gamma ^{-1}\left( \omega \ast
m\left( m\otimes Q\right) \otimes Q\right) \\
\ast \left( \gamma ^{-1}\left( m\otimes Q\right) \otimes \varepsilon \right)
\ast \left( \gamma ^{-1}\otimes \varepsilon \otimes \varepsilon \right)%
\end{array}%
\right] \\
&=&\left[
\begin{array}{c}
\left( \varepsilon \otimes \varepsilon \otimes \gamma \right) \ast \left(
\varepsilon \otimes \gamma \left( Q\otimes m\right) \right) \ast \gamma
\left( Q\otimes m\left( Q\otimes m\right) \right) \\
\left( \varepsilon \otimes \omega \right) \ast \omega \left( Q\otimes
m\otimes Q\right) \ast \left( \omega \otimes \varepsilon \right) \ast \\
\gamma ^{-1}\left( m\left( m\otimes Q\right) \otimes Q\right) \ast \left(
\gamma ^{-1}\left( m\otimes Q\right) \otimes \varepsilon \right) \ast \left(
\gamma ^{-1}\otimes \varepsilon \otimes \varepsilon \right)%
\end{array}%
\right] \\
&=&\left[
\begin{array}{c}
\left( \varepsilon \otimes \varepsilon \otimes \gamma \right) \ast \left(
\varepsilon \otimes \gamma \left( Q\otimes m\right) \right) \ast \gamma
\left( Q\otimes m\left( Q\otimes m\right) \right) \\
\omega \left( Q\otimes Q\otimes m\right) \ast \omega \left( m\otimes
Q\otimes Q\right) \ast \\
\gamma ^{-1}\left( m\left( m\otimes Q\right) \otimes Q\right) \ast \left(
\gamma ^{-1}\left( m\otimes Q\right) \otimes \varepsilon \right) \ast \left(
\gamma ^{-1}\otimes \varepsilon \otimes \varepsilon \right)%
\end{array}%
\right] .
\end{eqnarray*}
\end{invisible}

Therefore%
\begin{equation*}
\omega ^{\gamma }\left( Q\otimes Q\otimes m^{\gamma }\right) \ast \omega
^{\gamma }\left( m^{\gamma }\otimes Q\otimes Q\right) =\left( \varepsilon
\otimes \omega ^{\gamma }\right) \ast \omega ^{\gamma }\left( Q\otimes
m^{\gamma }\otimes Q\right) \ast \left( \omega ^{\gamma }\otimes \varepsilon
\right) .
\end{equation*}
\end{proof}

In analogy to the case of Hopf algebras, one can define the bosonization
$E\#H$ of a coquasi-bialgebra in ${_{H}^{H}\mathcal{YD}}$ by a Hopf algebra $%
H,$ see \cite[Definition 5.4]{ABM} for further details on the structure. The
following result was originally stated for $E$ a Hopf algebra. Yorck Sommerh\"{a}user suggested the present more general form which investigates the
behaviour of the bosonization under a suitable gauge transformation.

\begin{proposition}
\label{pro:deformSmash}Let $H$ be a Hopf algebra and let $\left(
E,m,u,\Delta ,\varepsilon ,\omega \right) $ be a coquasi-bialgebra in ${%
_{H}^{H}\mathcal{YD}}$. Let $\gamma :E\otimes E\rightarrow \Bbbk $ be a
gauge transformation in ${_{H}^{H}\mathcal{YD}}$. Set%
\begin{equation*}
\Gamma :\left( E\#H\right) \otimes \left( E\#H\right) \rightarrow \Bbbk
:\left( x\#h\right) \otimes \left( x^{\prime }\#h^{\prime }\right) \mapsto
\gamma \left( x\otimes hx^{\prime }\right) \varepsilon _{H}\left( h^{\prime
}\right) .
\end{equation*}%
Then $\Gamma $ is a gauge transformation and $\left( E\#H\right) ^{\Gamma
}=E^{\gamma }\#H$ as ordinary coquasi-bialgebras.
\end{proposition}

\begin{proof}
By \cite[Lemma 2.15 and what follows]{ABM}, we have that $\Gamma $ is
convolution invertible $H$-bilinear and $H$-balanced. Moreover $\Gamma
^{-1}\left( \left( x\#h\right) \otimes \left( x^{\prime }\#h^{\prime
}\right) \right) =\gamma ^{-1}\left( x\otimes hx^{\prime }\right)
\varepsilon _{H}\left( h^{\prime }\right) .$ If $\alpha :\left( E\#H\right)
\otimes \left( E\#H\right) \rightarrow E\#H$ is $H$-bilinear and $H$%
-balanced, it is easy to check that $\Gamma \ast \alpha \ast \Gamma ^{-1}$
is $H$-bilinear and $H$-balanced too.

\begin{invisible}
We check this for our sake. Note that $E\#H$ is an $H$-bimodule coalgebra
with respect to  $\left( r\#h\right) l=r\#hl$ and $l\left( s\#h\right)
=l_{1}s\#l_{2}h$. Let, in general, $A$ be an $H$-bimodule coalgebra and let $%
\alpha :A\otimes A\rightarrow A$ be an $H$-bilinear and $H$-balanced map.
Let $\Gamma :A\otimes A\rightarrow \Bbbk $ be $H$-bilinear and $H$-balanced
map. Then%
\begin{eqnarray*}
\left( \Gamma \ast \alpha \right) \left( hxh^{\prime }\otimes x^{\prime
}h^{\prime \prime }\right)  &=&\Gamma \left( \left( hxh^{\prime }\right)
_{1}\otimes \left( x^{\prime }h^{\prime \prime }\right) _{1}\right) \alpha
\left( \left( hxh^{\prime }\right) _{2}\otimes \left( x^{\prime }h^{\prime
\prime }\right) _{2}\right)  \\
&=&\Gamma \left( h_{1}x_{1}h_{1}^{\prime }\otimes x_{1}^{\prime
}h_{1}^{\prime \prime }\right) \alpha \left( h_{2}x_{2}h_{2}^{\prime
}\otimes x_{2}^{\prime }h_{2}^{\prime \prime }\right)  \\
&=&\varepsilon _{H}\left( h_{1}\right) \Gamma \left( x_{1}\otimes
h_{1}^{\prime }x_{1}^{\prime }\right) \varepsilon _{H}\left( h_{1}^{\prime
\prime }\right) h_{2}\alpha \left( x_{2}\otimes h_{2}^{\prime }x_{2}^{\prime
}\right) h_{2}^{\prime \prime } \\
&=&\Gamma \left( x_{1}\otimes h_{1}^{\prime }x_{1}^{\prime }\right) h\alpha
\left( x_{2}\otimes h_{2}^{\prime }x_{2}^{\prime }\right) h^{\prime \prime }
\\
&=&h\Gamma \left( x_{1}\otimes \left( h^{\prime }x^{\prime }\right)
_{1}\right) \alpha \left( x_{2}\otimes \left( h^{\prime }x^{\prime }\right)
_{2}\right) h^{\prime \prime } \\
&=&h\left( \Gamma \ast \alpha \right) \left( x\otimes h^{\prime }x^{\prime
}\right) h^{\prime \prime }
\end{eqnarray*}%
and%
\begin{eqnarray*}
\left( \alpha \ast \Gamma \right) \left( hxh^{\prime }\otimes x^{\prime
}h^{\prime \prime }\right)  &=&\alpha \left( \left( hxh^{\prime }\right)
_{1}\otimes \left( x^{\prime }h^{\prime \prime }\right) _{1}\right) \Gamma
\left( \left( hxh^{\prime }\right) _{2}\otimes \left( x^{\prime }h^{\prime
\prime }\right) _{2}\right)  \\
&=&\alpha \left( h_{1}x_{1}h_{1}^{\prime }\otimes x_{1}^{\prime
}h_{1}^{\prime \prime }\right) \Gamma \left( h_{2}x_{2}h_{2}^{\prime
}\otimes x_{2}^{\prime }h_{2}^{\prime \prime }\right)  \\
&=&h_{1}\alpha \left( x_{1}\otimes h_{1}^{\prime }x_{1}^{\prime }\right)
h_{1}^{\prime \prime }\varepsilon _{H}\left( h_{2}\right) \Gamma \left(
x_{2}\otimes h_{2}^{\prime }x_{2}^{\prime }\right) \varepsilon _{H}\left(
h_{2}^{\prime \prime }\right)  \\
&=&h\alpha \left( x_{1}\otimes h_{1}^{\prime }x_{1}^{\prime }\right)
h^{\prime \prime }\Gamma \left( x_{2}\otimes h_{2}^{\prime }x_{2}^{\prime
}\right)  \\
&=&h\alpha \left( x_{1}\otimes \left( h^{\prime }x^{\prime }\right)
_{1}\right) \Gamma \left( x_{2}\otimes \left( h^{\prime }x^{\prime }\right)
_{2}\right) h^{\prime \prime } \\
&=&h\left( \alpha \ast \Gamma \right) \left( x\otimes h^{\prime }x^{\prime
}\right) h^{\prime \prime }.
\end{eqnarray*}
\end{invisible}

In particular, since
\begin{equation*}
m_{E\#H}\left( \left( x\#h\right) \otimes \left( x^{\prime }\#h^{\prime
}\right) \right) =m\left( x\otimes h_{1}x^{\prime }\right) \otimes
h_{2}h^{\prime }
\end{equation*}%
we have that $m_{E\#H}$ is $H$-bilinear and $H$-balanced where $E\#H$
carries the left $H$-diagonal action and the right regular action over $H$.

\begin{invisible}
We have%
\begin{eqnarray*}
m_{E\#H}\left( l\left( x\#h\right) \otimes \left( x^{\prime }\#h^{\prime
}\right) \right)  &=&m_{E\#H}\left( \left( l_{1}x\#l_{2}h\right) \otimes
\left( x^{\prime }\#h^{\prime }\right) \right)  \\
&=&m_{E}\left( l_{1}x\otimes \left( l_{2}h_{1}\right) x^{\prime }\right)
\otimes \left( l_{3}h_{2}\right) h^{\prime } \\
&=&m_{E}\left( l_{1}x\otimes l_{2}\left( h_{1}x^{\prime }\right) \right)
\otimes l_{3}\left( h_{2}h^{\prime }\right)  \\
&=&l_{1}m_{E}\left( x\otimes h_{1}x^{\prime }\right) \otimes l_{2}\left(
h_{2}h^{\prime }\right)  \\
&=&l\left( m_{E}\left( x\otimes h_{1}x^{\prime }\right) \otimes
h_{2}h^{\prime }\right) =lm_{E\#H}\left( \left( x\#h\right) \otimes \left(
x^{\prime }\#h^{\prime }\right) \right) ,
\end{eqnarray*}%
\begin{eqnarray*}
m_{E\#H}\left( \left( x\#h\right) \otimes \left( x^{\prime }\#h^{\prime
}\right) l\right)  &=&m_{E\#H}\left( \left( x\#h\right) \otimes \left(
x^{\prime }\#h^{\prime }l\right) \right) =m_{E}\left( x\otimes
h_{1}x^{\prime }\right) \otimes h_{2}\left( h^{\prime }l\right)  \\
&=&m_{E}\left( x\otimes h_{1}x^{\prime }\right) \otimes \left(
h_{2}h^{\prime }\right) l=\left( m_{E}\left( x\otimes h_{1}x^{\prime
}\right) \otimes \left( h_{2}h^{\prime }\right) \right) l \\
&=&m_{E\#H}\left( \left( x\#h\right) \otimes \left( x^{\prime }\#h^{\prime
}\right) \right) l,
\end{eqnarray*}%
and%
\begin{eqnarray*}
m_{E\#H}\left( \left( x\#h\right) l\otimes \left( x^{\prime }\#h^{\prime
}\right) \right)  &=&m_{E\#H}\left( \left( x\#hl\right) \otimes \left(
x^{\prime }\#h^{\prime }\right) \right)  \\
&=&m_{E}\left( x\otimes \left( h_{1}l_{1}\right) x^{\prime }\right) \otimes
\left( h_{2}l_{2}\right) h^{\prime }=m_{E}\left( x\otimes h_{1}\left(
l_{1}x^{\prime }\right) \right) \otimes h_{2}\left( l_{2}h^{\prime }\right)
\\
&=&m_{E\#H}\left( \left( x\#h\right) \otimes \left( l_{1}x^{\prime
}\#l_{2}h^{\prime }\right) \right) =m_{E\#H}\left( \left( x\#h\right)
\otimes l\left( x^{\prime }\#h^{\prime }\right) \right) .
\end{eqnarray*}
\end{invisible}

Thus $m_{\left( E\#H\right) ^{\Gamma }}=\Gamma \ast m_{E\#H}\ast \Gamma ^{-1}
$ is $H$-bilinear and $H$-balanced. Moreover, since $E^{\gamma }$ is also a
coquasi-bialgebra in ${_{H}^{H}\mathcal{YD}}$ we have that $m_{E^{\gamma
}\#H}:\left( E\#H\right) \otimes \left( E\#H\right) \rightarrow E\#H$ is $H$%
-bilinear and $H$-balanced too.

\begin{invisible}
\begin{eqnarray*}
m_{E^{\gamma }\#H}\left( l\left( x\#h\right) s\otimes \left( x^{\prime
}\#h^{\prime }\right) t\right) &=&m_{E^{\gamma }\#H}\left( \left(
l_{1}x\#l_{2}hs\right) \otimes \left( x^{\prime }\#h^{\prime }t\right)
\right) \\
&=&\left( \gamma \ast m_{E}\ast \gamma ^{-1}\right) \left( l_{1}x\otimes
l_{2}h_{1}s_{1}x^{\prime }\right) \#l_{3}h_{2}s_{2}h^{\prime }t \\
&=&l_{1}\left( \gamma \ast m_{E}\ast \gamma ^{-1}\right) \left( x\otimes
h_{1}s_{1}x^{\prime }\right) \#l_{2}h_{2}s_{2}h^{\prime }t \\
&=&l\left[ \left( \gamma \ast m_{E}\ast \gamma ^{-1}\right) \left( x\otimes
h_{1}s_{1}x^{\prime }\right) \#h_{2}s_{2}h^{\prime }\right] t \\
&=&lm_{E^{\gamma }\#H}\left( \left( x\#h\right) \otimes \left(
s_{1}x^{\prime }\#s_{2}h^{\prime }\right) \right) t \\
&=&lm_{E^{\gamma }\#H}\left( \left( x\#h\right) \otimes s\left( x^{\prime
}\#h^{\prime }\right) \right) t
\end{eqnarray*}
\end{invisible}

Therefore, in order to check that $m_{\left( E\#H\right) ^{\Gamma
}}=m_{E^{\gamma }\#H},$ it suffices to prove that they coincide on elements
of the form $\left( x\#1_{H}\right) \otimes \left( x^{\prime }\#1_{H}\right)
.$

\begin{invisible}
If $\alpha :\left( E\#H\right) \otimes \left( E\#H\right) \rightarrow E\#H$
is $H$-bilinear and $H$-balanced, then
\begin{eqnarray*}
\alpha \left( \left( x\#h\right) \otimes \left( x^{\prime }\#h^{\prime
}\right) \right) &=&\alpha \left( \left( x\#1_{H}\right) h\otimes \left(
x^{\prime }\#h^{\prime }\right) \right) \\
&=&\alpha \left( \left( x\#1_{H}\right) \otimes h\left( x^{\prime
}\#h^{\prime }\right) \right) \\
&=&\alpha \left( \left( x\#1_{H}\right) \otimes \left( h_{1}x^{\prime
}\#h_{2}h^{\prime }\right) \right) \\
&=&\alpha \left( \left( x\#1_{H}\right) \otimes \left( h_{1}x^{\prime
}\#1_{H}\right) \right) h_{2}h^{\prime }.
\end{eqnarray*}
\end{invisible}

Let us consider the multiplication%
\begin{eqnarray*}
&&m_{\left( E\#H\right) ^{\Gamma }}\left( \left( x\#1_{H}\right) \otimes
\left( x^{\prime }\#1_{H}\right) \right) \\
&=&\left( \Gamma \ast m_{E\#H}\ast \Gamma ^{-1}\right) \left( \left(
x\#1_{H}\right) \otimes \left( x^{\prime }\#1_{H}\right) \right) \\
&=&\Gamma \left( \left( x\#1_{H}\right) _{1}\otimes \left( x^{\prime
}\#1_{H}\right) _{1}\right) \cdot m_{E\#H}\left( \left( x\#1_{H}\right)
_{2}\otimes \left( x^{\prime }\#1_{H}\right) _{2}\right) \cdot \Gamma
^{-1}\left( \left( x\#1_{H}\right) _{3}\otimes \left( x^{\prime
}\#1_{H}\right) _{3}\right) .
\end{eqnarray*}%
Now, from
\begin{equation*}
\Delta _{E\#H}\left( x\#h\right) =\sum \left( x^{\left( 1\right) }\#{%
x^{\left( 2\right) }}_{\left\langle -1\right\rangle }h_{1}\right) \otimes
\left( {x^{\left( 2\right) }}_{\left\langle 0\right\rangle }\#h_{2}\right)
\end{equation*}%
we get
\begin{eqnarray*}
&&\left( x\#1_{H}\right) _{1}\otimes \left( x\#1_{H}\right) _{2}\otimes
\left( x\#1_{H}\right) _{3} \\
&=&\sum \left( x^{\left( 1\right) }\#{x^{\left( 2\right) }}_{\left\langle
-1\right\rangle }{x^{\left( 3\right) }}_{\left\langle -2\right\rangle
}\right) \otimes \left( {x^{\left( 2\right) }}_{\left\langle 0\right\rangle
}\#{x^{\left( 3\right) }}_{\left\langle -1\right\rangle }\right) \otimes
\left( {x^{\left( 3\right) }}_{\left\langle 0\right\rangle }\#1_{H}\right)
\end{eqnarray*}

\begin{invisible}
\begin{eqnarray*}
&&\left( x\#1_{H}\right) _{1}\otimes \left( x\#1_{H}\right) _{2}\otimes
\left( x\#1_{H}\right) _{3} \\
&=&\sum \Delta _{E\#H}\left( x^{\left( 1\right) }\#\left( x^{\left( 2\right)
}\right) _{\left\langle -1\right\rangle }\right) \otimes \left( \left(
x^{\left( 2\right) }\right) _{\left\langle 0\right\rangle }\#1_{H}\right) \\
&=&\sum \left( \left( x^{\left( 1\right) }\right) ^{\left( 1\right)
}\#\left( \left( x^{\left( 1\right) }\right) ^{\left( 2\right) }\right)
_{\left\langle -1\right\rangle }\left( \left( x^{\left( 2\right) }\right)
_{\left\langle -1\right\rangle }\right) _{1}\right) \otimes \left( \left(
\left( x^{\left( 1\right) }\right) ^{\left( 2\right) }\right) _{\left\langle
0\right\rangle }\#\left( \left( x^{\left( 2\right) }\right) _{\left\langle
-1\right\rangle }\right) _{2}\right) \otimes \left( \left( x^{\left(
2\right) }\right) _{\left\langle 0\right\rangle }\#1_{H}\right) \\
&=&\sum \left( x^{\left( 1\right) }\#\left( x^{\left( 2\right) }\right)
_{\left\langle -1\right\rangle }\left( x^{\left( 3\right) }\right)
_{\left\langle -2\right\rangle }\right) \otimes \left( \left( x^{\left(
2\right) }\right) _{\left\langle 0\right\rangle }\#\left( x^{\left( 3\right)
}\right) _{\left\langle -1\right\rangle }\right) \otimes \left( \left(
x^{\left( 3\right) }\right) _{\left\langle 0\right\rangle }\#1_{H}\right)
\end{eqnarray*}
\end{invisible}

so that%
\begin{eqnarray*}
&&m_{\left( E\#H\right) ^{\Gamma }}\left( \left( x\#1_{H}\right) \otimes
\left( x^{\prime }\#1_{H}\right) \right)  \\
&=&\Gamma \left( \left( x\#1_{H}\right) _{1}\otimes \left( x^{\prime
}\#1_{H}\right) _{1}\right) \cdot m_{E\#H}\left( \left( x\#1_{H}\right)
_{2}\otimes \left( x^{\prime }\#1_{H}\right) _{2}\right) \cdot \Gamma
^{-1}\left( \left( x\#1_{H}\right) _{3}\otimes \left( x^{\prime
}\#1_{H}\right) _{3}\right)  \\
&=&\left[
\begin{array}{c}
\sum \Gamma \left( x^{\left( 1\right) }\#{x^{\left( 2\right) }}%
_{\left\langle -1\right\rangle }{x^{\left( 3\right) }}_{\left\langle
-2\right\rangle }\otimes x^{\prime \left( 1\right) }\#x^{\prime \left(
2\right) }{}_{\left\langle -1\right\rangle }x^{\prime \left( 3\right)
}{}_{\left\langle -2\right\rangle }\right)  \\
\cdot m_{E\#H}\left( {x^{\left( 2\right) }}_{\left\langle 0\right\rangle }\#{%
x^{\left( 3\right) }}_{\left\langle -1\right\rangle }\otimes x^{\prime
\left( 2\right) }{}_{\left\langle 0\right\rangle }\#x^{\prime \left(
3\right) }{}_{\left\langle -1\right\rangle }\right)  \\
\cdot \Gamma ^{-1}\left( {x^{\left( 3\right) }}_{\left\langle 0\right\rangle
}\#1_{H}\otimes x^{\prime \left( 3\right) }{}_{\left\langle 0\right\rangle
}\#1_{H}\right)
\end{array}%
\right]  \\
&=&\left[
\begin{array}{c}
\sum \gamma \left( x^{\left( 1\right) }\otimes {x^{\left( 2\right) }}%
_{\left\langle -1\right\rangle }{x^{\left( 3\right) }}_{\left\langle
-2\right\rangle }x^{\prime \left( 1\right) }\right)  \\
\cdot m_{E\#H}\left( {x^{\left( 2\right) }}_{\left\langle 0\right\rangle }\#{%
x^{\left( 3\right) }}_{\left\langle -1\right\rangle }\otimes x^{\prime
\left( 2\right) }\#x^{\prime \left( 3\right) }{}_{\left\langle
-1\right\rangle }\right)  \\
\cdot \gamma ^{-1}\left( {x^{\left( 3\right) }}_{\left\langle 0\right\rangle
}\otimes x^{\prime \left( 3\right) }{}_{\left\langle 0\right\rangle }\right)
\end{array}%
\right]  \\
&=&\left[
\begin{array}{c}
\sum \gamma \left( x^{\left( 1\right) }\otimes {x^{\left( 2\right) }}%
_{\left\langle -1\right\rangle }{x^{\left( 3\right) }}_{\left\langle
-2\right\rangle }x^{\prime \left( 1\right) }\right)  \\
\cdot m\left( {x^{\left( 2\right) }}_{\left\langle 0\right\rangle }\otimes {%
x^{\left( 3\right) }}_{\left\langle -2\right\rangle }x^{\prime \left(
2\right) }\right) \otimes x^{\left( 3\right) }{}_{\left\langle
-1\right\rangle }x^{\prime \left( 3\right) }{}_{\left\langle -1\right\rangle
} \\
\cdot \gamma ^{-1}\left( {x^{\left( 3\right) }}_{\left\langle 0\right\rangle
}\otimes x^{\prime \left( 3\right) }{}_{\left\langle 0\right\rangle }\right)
\end{array}%
\right]  \\
&=&\left[
\begin{array}{c}
\sum \gamma \left( x^{\left( 1\right) }\otimes {x^{\left( 2\right) }}%
_{\left\langle -1\right\rangle }x^{\left( 3\right) }{}_{\left\langle
-2\right\rangle }x^{\prime \left( 1\right) }\right)  \\
\cdot m\left( {x^{\left( 2\right) }}_{\left\langle 0\right\rangle }\otimes {%
x^{\left( 3\right) }}_{\left\langle -1\right\rangle }x^{\prime \left(
2\right) }\right) \otimes \left( {x^{\left( 3\right) }}_{\left\langle
0\right\rangle }\otimes x^{\prime \left( 3\right) }\right) _{\left\langle
-1\right\rangle } \\
\cdot \gamma ^{-1}\left( {x^{\left( 3\right) }}_{\left\langle 0\right\rangle
}\otimes x^{\prime \left( 3\right) }{}_{\left\langle 0\right\rangle }\right)
\end{array}%
\right]  \\
&\overset{\gamma ^{-1}\text{ colin.}}{=}&\left[
\begin{array}{c}
\sum \gamma \left( x^{\left( 1\right) }\otimes {x^{\left( 2\right) }}%
_{\left\langle -1\right\rangle }{x^{\left( 3\right) }}_{\left\langle
-2\right\rangle }x^{\prime \left( 1\right) }\right) \cdot m\left( {x^{\left(
2\right) }}_{\left\langle 0\right\rangle }\otimes {x^{\left( 3\right) }}%
_{\left\langle -1\right\rangle }x^{\prime \left( 2\right) }\right) \otimes
1_{H} \\
\cdot \gamma ^{-1}\left( {x^{\left( 3\right) }}_{\left\langle 0\right\rangle
}\otimes x^{\prime \left( 3\right) }\right)
\end{array}%
\right]  \\
&=&\left[
\begin{array}{c}
\sum \gamma \left( x^{\left( 1\right) }\otimes {x^{\left( 2\right) }}%
_{\left\langle -1\right\rangle }{x^{\left( 3\right) }}_{\left\langle
-2\right\rangle }x^{\prime \left( 1\right) }\right) m\left( {x^{\left(
2\right) }}_{\left\langle 0\right\rangle }\otimes {x^{\left( 3\right) }}%
_{\left\langle -1\right\rangle }x^{\prime \left( 2\right) }\right)  \\
\gamma ^{-1}\left( {x^{\left( 3\right) }}_{\left\langle 0\right\rangle
}\otimes x^{\prime \left( 3\right) }\right)
\end{array}%
\right] \otimes 1_{H}.
\end{eqnarray*}%
Now we have%
\begin{equation*}
\sum \left( x\otimes y\right) ^{\left( 1\right) }\otimes \left( x\otimes
y\right) ^{\left( 2\right) }=\sum x^{\left( 1\right) }\otimes x^{\left(
2\right) }{}_{\left\langle -1\right\rangle }y^{\left( 1\right) }\otimes
x^{\left( 2\right) }{}_{\left\langle 0\right\rangle }\otimes y^{\left(
2\right) }
\end{equation*}%
so that

\begin{invisible}
\begin{eqnarray*}
&&\sum \left( x\otimes y\right) ^{\left( 1\right) }\otimes \left( x\otimes
y\right) ^{\left( 2\right) }\otimes \left( x\otimes y\right) ^{\left(
3\right) } \\
&=&\sum \left( x^{\left( 1\right) }\otimes \left( x^{\left( 2\right)
}\right) _{\left\langle -1\right\rangle }y^{\left( 1\right) }\right)
^{\left( 1\right) }\otimes \left( x^{\left( 1\right) }\otimes \left(
x^{\left( 2\right) }\right) _{\left\langle -1\right\rangle }y^{\left(
1\right) }\right) ^{\left( 2\right) }\otimes \left( x^{\left( 2\right)
}\right) _{\left\langle 0\right\rangle }\otimes y^{\left( 2\right) } \\
&=&\sum \left( x^{\left( 1\right) \left( 1\right) }\otimes \left( x^{\left(
1\right) \left( 2\right) }\right) _{\left\langle -1\right\rangle }\left(
\left( x^{\left( 2\right) }\right) _{\left\langle -1\right\rangle }y^{\left(
1\right) }\right) ^{\left( 1\right) }\otimes \left( x^{\left( 1\right)
\left( 2\right) }\right) _{\left\langle 0\right\rangle }\otimes \left(
\left( x^{\left( 2\right) }\right) _{\left\langle -1\right\rangle }y^{\left(
1\right) }\right) ^{\left( 2\right) }\right) \otimes \left( x^{\left(
2\right) }\right) _{\left\langle 0\right\rangle }\otimes y^{\left( 2\right) }
\\
&=&\sum \left( x^{\left( 1\right) }\otimes \left( x^{\left( 2\right)
}\right) _{\left\langle -1\right\rangle }\left( \left( x^{\left( 3\right)
}\right) _{\left\langle -1\right\rangle }y^{\left( 1\right) }\right)
^{\left( 1\right) }\otimes \left( x^{\left( 2\right) }\right) _{\left\langle
0\right\rangle }\otimes \left( \left( x^{\left( 3\right) }\right)
_{\left\langle -1\right\rangle }y^{\left( 1\right) }\right) ^{\left(
2\right) }\right) \otimes \left( x^{\left( 3\right) }\right) _{\left\langle
0\right\rangle }\otimes y^{\left( 2\right) } \\
&=&\sum \left( x^{\left( 1\right) }\otimes \left( x^{\left( 2\right)
}\right) _{\left\langle -1\right\rangle }\left( x^{\left( 3\right) }\right)
_{\left\langle -2\right\rangle }\left( y^{\left( 1\right) }\right) ^{\left(
1\right) }\otimes \left( x^{\left( 2\right) }\right) _{\left\langle
0\right\rangle }\otimes \left( x^{\left( 3\right) }\right) _{\left\langle
-1\right\rangle }\left( y^{\left( 1\right) }\right) ^{\left( 2\right)
}\right) \otimes \left( x^{\left( 3\right) }\right) _{\left\langle
0\right\rangle }\otimes y^{\left( 2\right) } \\
&=&\sum \left( x^{\left( 1\right) }\otimes \left( x^{\left( 2\right)
}\right) _{\left\langle -1\right\rangle }\left( x^{\left( 3\right) }\right)
_{\left\langle -2\right\rangle }y^{\left( 1\right) }\right) \otimes \left(
\left( x^{\left( 2\right) }\right) _{\left\langle 0\right\rangle }\otimes
\left( x^{\left( 3\right) }\right) _{\left\langle -1\right\rangle }y^{\left(
2\right) }\right) \otimes \left( \left( x^{\left( 3\right) }\right)
_{\left\langle 0\right\rangle }\otimes y^{\left( 3\right) }\right)
\end{eqnarray*}%
i.e.
\end{invisible}

\begin{eqnarray*}
&&\sum \left( x\otimes y\right) ^{\left( 1\right) }\otimes \left( x\otimes
y\right) ^{\left( 2\right) }\otimes \left( x\otimes y\right) ^{\left(
3\right) } \\
&=&\sum \left( x^{\left( 1\right) }\otimes x^{\left( 2\right)
}{}_{\left\langle -1\right\rangle }x^{\left( 3\right) }{}_{\left\langle
-2\right\rangle }y^{\left( 1\right) }\right) \otimes \left( x^{\left(
2\right) }{}_{\left\langle 0\right\rangle }\otimes x^{\left( 3\right)
}{}_{\left\langle -1\right\rangle }y^{\left( 2\right) }\right) \otimes
\left( x^{\left( 3\right) }{}_{\left\langle 0\right\rangle }\otimes
y^{\left( 3\right) }\right) .
\end{eqnarray*}%
Using this equality we can proceed in our computation:%
\begin{eqnarray*}
&&m_{\left( E\#H\right) ^{\Gamma }}\left( \left( x\#1_{H}\right) \otimes
\left( x^{\prime }\#1_{H}\right) \right)  \\
&=&\left[
\begin{array}{c}
\sum \gamma \left( x^{\left( 1\right) }\otimes x^{\left( 2\right)
}{}_{\left\langle -1\right\rangle }x^{\left( 3\right) }{}_{\left\langle
-2\right\rangle }x^{\prime \left( 1\right) }\right)  \\
m\left( x^{\left( 2\right) }{}_{\left\langle 0\right\rangle }\otimes
x^{\left( 3\right) }{}_{\left\langle -1\right\rangle }x^{\prime \left(
2\right) }\right) \gamma ^{-1}\left( {x^{\left( 3\right) }}_{\left\langle
0\right\rangle }\otimes x^{\prime \left( 3\right) }\right)
\end{array}%
\right] \otimes 1_{H} \\
&=&\left[ \sum \gamma \left( \left( x\otimes x^{\prime }\right) ^{\left(
1\right) }\right) \cdot m\left( \left( x\otimes x^{\prime }\right) ^{\left(
2\right) }\right) \cdot \gamma ^{-1}\left( \left( x\otimes x^{\prime
}\right) ^{\left( 3\right) }\right) \right] \#1_{H} \\
&=&\left( \gamma \ast m\ast \gamma ^{-1}\right) \left( x\otimes x^{\prime
}\right) \#1_{H} \\
&=&m_{E^{\gamma }}\left( x\otimes x^{\prime }\right) \#1_{H} \\
&=&m_{E^{\gamma }\#H}\left( \left( x\#1_{H}\right) \otimes \left( x^{\prime
}\#1_{H}\right) \right) .
\end{eqnarray*}%
Finally $u_{\left( E\#H\right) ^{\Gamma
}}=u_{E\#H}=1_{E}\#1_{H}=1_{E^{\gamma }}\#1_{H}=u_{E^{\gamma }\#H}.$

As a coalgebra $\left( E\#H\right) ^{\Gamma }$ coincides with $E\#H$ and
hence with $E^{\gamma }\#H$.

Finally let us check that $\omega _{E^{\gamma }\#H}$ and $\omega _{\left(
E\#H\right) ^{\Gamma }}$ coincide. To this aim, let us use the maps $\mho
_{H,-}^{\ast }$ of \cite[Lemma 2.15]{ABM}. First note that $\omega
_{E^{\gamma }\#H}=\mho _{H,E^{\gamma }}^{3}\left( \omega _{E^{\gamma
}}\right) $ by \cite[Proposition 5.3]{ABM}. Now%
\begin{eqnarray*}
\omega _{\left( E\#H\right) ^{\Gamma }} &=&\left( \varepsilon _{E\#H}\otimes
\Gamma \right) \ast \Gamma \left( E\#H\otimes m_{E\#H}\right) \ast \omega
_{E\#H}\ast \Gamma ^{-1}\left( m_{E\#H}\otimes E\#H\right) \ast \left(
\Gamma ^{-1}\otimes \varepsilon _{E\#H}\right)  \\
&=&\left( \mho _{H,E}^{1}\left( \varepsilon \right) \otimes \mho
_{H,E}^{2}\left( \gamma \right) \right) \ast \mho _{H,E}^{2}\left( \gamma
\right) \left( E\#H\otimes m_{E\#H}\right) \ast \mho _{H,E}^{3}\left( \omega
\right)  \\
&&\ast \mho _{H,E}^{2}\left( \gamma ^{-1}\right) \left( m_{E\#H}\otimes
E\#H\right) \ast \left( \mho _{H,E}^{2}\left( \gamma ^{-1}\right) \otimes
\mho _{H,E}^{1}\left( \varepsilon \right) \right)
\end{eqnarray*}%
One easily checks that%
\begin{eqnarray*}
\mho _{H,E}^{1}\left( \varepsilon \right) \otimes \mho _{H,E}^{2}\left(
\gamma \right)  &=&\mho _{H,E^{\gamma }}^{3}\left( \varepsilon \otimes
\gamma \right) , \\
\mho _{H,E}^{2}\left( \gamma \right) \left( E\#H\otimes m_{E\#H}\right)
&=&\mho _{H,E^{\gamma }}^{3}\left( \gamma \left( E\otimes m\right) \right) ,
\\
\mho _{H,E}^{2}\left( \gamma ^{-1}\right) \left( m_{E\#H}\otimes E\#H\right)
&=&\mho _{H,E^{\gamma }}^{3}\left( \gamma ^{-1}\left( m\otimes E\right)
\right) , \\
\mho _{H,E}^{2}\left( \gamma ^{-1}\right) \otimes \mho _{H,E}^{1}\left(
\varepsilon _{E}\right)  &=&\mho _{H,E^{\gamma }}^{3}\left( \gamma
^{-1}\otimes \varepsilon \right) .
\end{eqnarray*}%
Thus we obtain%
\begin{eqnarray*}
\omega _{\left( E\#H\right) ^{\Gamma }} &=&\mho _{H,E^{\gamma }}^{3}\left(
\varepsilon \otimes \gamma \right) \ast \mho _{H,E^{\gamma }}^{3}\left(
\gamma \left( E\otimes m\right) \right) \ast \mho _{H,E}^{3}\left( \omega
\right) \ast \mho _{H,E^{\gamma }}^{3}\left( \gamma ^{-1}\left( m\otimes
E\right) \right) \ast \mho _{H,E^{\gamma }}^{3}\left( \gamma ^{-1}\otimes
\varepsilon \right)  \\
&=&\mho _{H,E^{\gamma }}^{3}\left[ \left( \varepsilon \otimes \gamma \right)
\ast \gamma \left( E\otimes m\right) \ast \omega \ast \gamma ^{-1}\left(
m\otimes E\right) \ast \left( \gamma ^{-1}\otimes \varepsilon \right) \right]
\\
&=&\mho _{H,E^{\gamma }}^{3}\left( \omega _{E^{\gamma }}\right) =\omega
_{E^{\gamma }\#H}.
\end{eqnarray*}

\begin{invisible}
We have%
\begin{eqnarray*}
\left[ \mho _{H,E}^{1}\left( \varepsilon \right) \otimes \mho
_{H,E}^{2}\left( \gamma \right) \right] \left( r\#h\otimes s\#l\otimes
t\#k\right)  &=&\varepsilon _{E\#H}\left( r\#h\right) \mho _{H,E}^{2}\left(
\gamma \right) \left( s\#l\otimes t\#k\right)  \\
&=&\varepsilon \left( r\right) \varepsilon _{H}\left( h\right) \gamma \left(
s\otimes lt\right) \varepsilon _{H}\left( k\right)  \\
&&\overset{\gamma \text{ lin.}}{=}\varepsilon \left( r\right) \gamma \left(
h_{1}s\otimes h_{2}lt\right) \varepsilon _{H}\left( k\right)  \\
&=&\left( \varepsilon \otimes \gamma \right) \left( r\otimes h_{1}s\otimes
h_{2}lt\right) \varepsilon _{H}\left( k\right)  \\
&=&\mho _{H,E^{\gamma }}^{3}\left( \varepsilon \otimes \gamma \right) \left(
r\#h\otimes s\#l\otimes t\#k\right) ,
\end{eqnarray*}%
\begin{eqnarray*}
&&\mho _{H,E}^{2}\left( \gamma \right) \left( E\#H\otimes m_{E\#H}\right)
\left( r\#h\otimes s\#l\otimes t\#k\right)  \\
&=&\mho _{H,E}^{2}\left( \gamma \right) \left( r\#h\otimes m\left( s\otimes
l_{1}t\right) \#l_{2}k\right) =\gamma \left( r\otimes hm\left( s\otimes
l_{1}t\right) \right) \varepsilon _{H}\left( l_{2}k\right)  \\
&&\overset{m\text{ lin.}}{=}\gamma \left( r\otimes m\left( h_{1}s\otimes
h_{2}lt\right) \right) \varepsilon _{H}\left( k\right)  \\
&=&\gamma \left( E\otimes m\right) \left( r\otimes h_{1}s\otimes
h_{2}lt\right) \varepsilon _{H}\left( k\right)  \\
&=&\mho _{H,E^{\gamma }}^{3}\left( \gamma \left( E\otimes m\right) \right) ,
\end{eqnarray*}%
\begin{eqnarray*}
&&\mho _{H,E}^{2}\left( \gamma ^{-1}\right) \left( m_{E\#H}\otimes
E\#H\right) \left( r\#h\otimes s\#l\otimes t\#k\right)  \\
&=&\mho _{H,E}^{2}\left( \gamma ^{-1}\right) \left( m\left( r\otimes
h_{1}s\right) \#h_{2}l\otimes t\#k\right) =\gamma ^{-1}\left( m\left(
r\otimes h_{1}s\right) \otimes h_{2}lt\right) \varepsilon _{H}\left(
k\right)  \\
&=&\gamma ^{-1}\left( m\otimes E\right) \left( r\otimes h_{1}s\otimes
h_{2}lt\right) \varepsilon _{H}\left( k\right) =\mho _{H,E^{\gamma
}}^{3}\left( \gamma ^{-1}\left( m\otimes E\right) \right) \left( r\#h\otimes
s\#l\otimes t\#k\right) ,
\end{eqnarray*}%
\begin{eqnarray*}
&&\left( \mho _{H,E}^{2}\left( \gamma ^{-1}\right) \otimes \mho
_{H,E}^{1}\left( \varepsilon _{E}\right) \right) \left( r\#h\otimes
s\#l\otimes t\#k\right)  \\
&=&\gamma ^{-1}\left( r\otimes hs\right) \varepsilon _{H}\left( l\right)
\varepsilon \left( t\right) \varepsilon _{H}\left( k\right)  \\
&=&\gamma ^{-1}\left( r\otimes h_{1}s\right) \varepsilon \left(
h_{2}lt\right) \varepsilon _{H}\left( k\right)  \\
&=&\left( \gamma ^{-1}\otimes \varepsilon \right) \left( r\otimes
h_{1}s\otimes h_{2}lt\right) \varepsilon _{H}\left( k\right)  \\
&=&\mho _{H,E^{\gamma }}^{3}\left( \gamma ^{-1}\otimes \varepsilon \right)
\left( r\#h\otimes s\#l\otimes t\#k\right) .
\end{eqnarray*}
\end{invisible}
\end{proof}

\begin{proposition}
\label{pro:grgaugeYD} Let $H$ be a Hopf algebra and let $\left( Q,m,u,\Delta
,\varepsilon ,\omega \right) $ be a connected coquasi-bialgebra in ${_{H}^{H}%
\mathcal{YD}}$. Let $\gamma :Q\otimes Q\rightarrow \Bbbk $ be a gauge
transformation in ${_{H}^{H}\mathcal{YD}}$. Then $\mathrm{gr}\left(
Q^{\gamma }\right) $ and $\mathrm{gr}\left( Q\right) $ coincide as
bialgebras in ${_{H}^{H}\mathcal{YD}}$.
\end{proposition}

\begin{proof}
By Proposition \ref{pro:deformYD}, $Q^{\gamma }$ is a coquasi-bialgebra in ${%
_{H}^{H}\mathcal{YD}}$. It is obviously connected as it coincides with $Q$
as a coalgebra. By Theorem \ref{teo:grHopf}, both $\mathrm{gr}Q$ and $%
\mathrm{gr}\left( Q^{\gamma }\right) $ are connected bialgebras in ${_{H}^{H}%
\mathcal{YD}}$. Let us check they coincide.

Note that, by Remark \ref{rem:gamma-1gauge}, we have that $\gamma ^{-1}$ is
a gauge transformation, hence it is trivial on $\Bbbk 1_{Q}\otimes 1_{Q}.$
Let $C:=Q\otimes Q$. Let $n>0$ and let $w\in C_{\left( n\right)
}=\sum_{i+j\leq n}Q_{i}\otimes Q_{j}.$ By \cite[Lemma 3.69]{AMS}, we have
that $\Delta _{C}\left( w\right) -w\otimes \left( 1_{Q}\right) ^{\otimes
2}-\left( 1_{Q}\right) ^{\otimes 2}\otimes w\in C_{\left( n-1\right)
}\otimes C_{\left( n-1\right) }$. Thus we get%
\begin{equation*}
w_{1}\otimes w_{2}\otimes w_{3}-\Delta _{C}\left( w\right) \otimes \left(
1_{Q}\right) ^{\otimes 2}-\Delta _{C}\left( \left( 1_{Q}\right) ^{\otimes
2}\right) \otimes w\in \Delta _{C}\left( C_{\left( n-1\right) }\right)
\otimes C_{\left( n-1\right) }
\end{equation*}%
and hence%
\begin{equation*}
w_{1}\otimes w_{2}\otimes w_{3}-w\otimes \left( 1_{Q}\right) ^{\otimes
2}\otimes \left( 1_{Q}\right) ^{\otimes 2}-\left( 1_{Q}\right) ^{\otimes
2}\otimes w\otimes \left( 1_{Q}\right) ^{\otimes 2}-\left( 1_{Q}\right)
^{\otimes 4}\otimes w\in C_{\left( n-1\right) }\otimes C_{\left( n-1\right)
}\otimes C_{\left( n-1\right) }.
\end{equation*}%
Since $m\left( C_{\left( n-1\right) }\right) \subseteq Q_{n-1}$ we get%
\begin{equation*}
w_{1}\otimes m\left( w_{2}\right) \otimes w_{3}-w\otimes 1_{Q}\otimes \left(
1_{Q}\right) ^{\otimes 2}-\left( 1_{Q}\right) ^{\otimes 2}\otimes m\left(
w\right) \otimes \left( 1_{Q}\right) ^{\otimes 2}-\left( 1_{Q}\right)
^{\otimes 3}\otimes w\in C_{\left( n-1\right) }\otimes Q_{n-1}\otimes
C_{\left( n-1\right) }
\end{equation*}

and hence%
\begin{equation}
w_{1}\otimes \left( m\left( w_{2}\right) +Q_{n-1}\right) \otimes
w_{3}=\left( 1_{Q}\right) ^{\otimes 2}\otimes \left( m\left( w\right)
+Q_{n-1}\right) \otimes \left( 1_{Q}\right) ^{\otimes 2}.  \label{form:gr1}
\end{equation}

Let $x,y\in Q$. We compute%
\begin{eqnarray*}
\overline{x}\cdot _{\gamma }\overline{y} &=&\left( x+Q_{\left\vert
x\right\vert -1}\right) \cdot _{\gamma }\left( y+Q_{\left\vert y\right\vert
-1}\right) \\
&=&\left( x\cdot _{\gamma }y\right) +Q_{\left\vert x\right\vert +\left\vert
y\right\vert -1} \\
&=&\gamma \left( \left( x\otimes y\right) _{1}\right) m\left( \left(
x\otimes y\right) _{2}\right) \gamma ^{-1}\left( \left( x\otimes y\right)
_{3}\right) +Q_{\left\vert x\right\vert +\left\vert y\right\vert -1} \\
&=&\gamma \left( \left( x\otimes y\right) _{1}\right) \left( m\left( \left(
x\otimes y\right) _{2}\right) +Q_{\left\vert x\right\vert +\left\vert
y\right\vert -1}\right) \gamma ^{-1}\left( \left( x\otimes y\right)
_{3}\right) \\
&\overset{(\ref{form:gr1})}{=}&\gamma \left( \left( 1_{Q}\right) ^{\otimes
2}\right) \left( m\left( x\otimes y\right) +Q_{\left\vert x\right\vert
+\left\vert y\right\vert -1}\right) \gamma ^{-1}\left( \left( 1_{Q}\right)
^{\otimes 2}\right) \\
&=&m\left( x\otimes y\right) +Q_{\left\vert x\right\vert +\left\vert
y\right\vert -1}=\left( x\cdot y\right) +Q_{\left\vert x\right\vert
+\left\vert y\right\vert -1}=\overline{x}\cdot \overline{y}.
\end{eqnarray*}%
Note that $Q^{\gamma }$ and $Q$ have the same unit so that $\mathrm{gr}Q$
and $\mathrm{gr}Q^{\gamma }$ have.
\end{proof}

\section{(Co)semisimple case}\label{sec:3}

Assume $H$ is a semisimple and cosemisimple Hopf algebra (e.g. $H$ is
finite-dimensional cosemisimple over a field of characteristic zero). Note
that $H$ is then separable (see e.g. \cite[Corollary 3.7]{St} or \cite[%
Theorem 2.34]{AMS}) whence finite-dimensional. Let $\left( Q,m,u,\Delta
,\varepsilon \right) $ be a f.d. coalgebra with multiplication and unit in ${%
_{H}^{H}\mathcal{YD}}$. Assume that the coradical $Q_{0}$ is a subcoalgebra
of $Q$ in ${_{H}^{H}\mathcal{YD}}$ such that $Q_{0}\cdot Q_{0}\subseteq
Q_{0}.$ Let $y^{n,i}$ with $1\leq i\leq \dim \left( Q_{n}/Q_{n-1}\right) $
be a basis for $Q_{n}/Q_{n-1}.$ Consider, for every $n>0,$ the exact
sequence in ${_{H}^{H}\mathcal{YD}}$ given by
\begin{equation*}
\xymatrixrowsep{25pt}\xymatrixcolsep{1cm} \xymatrix{0\ar[r]&Q_{n-1}%
\ar[r]^{s_n}&Q_n\ar[r]^{\pi_n}&\frac{Q_n}{Q_{n-1}}\ar[r]&0}
\end{equation*}
Now, since $H$ is semisimple and cosemisimple, by \cite[Proposition 7]{Ra}
the Drinfeld double $D(H)$ is semisimple. By a result essentially due to
Majid (see \cite[Proposition 10.6.16]{Mo}) and by \cite[Proposition 6]{RT},
we get that the category ${_{H}^{H}\mathcal{YD}}\cong {{_{D(H)}}}\mathfrak{M}
$ is a semisimple category. Therefore $\pi _{n}$ cosplits i.e. there is a
morphism $\sigma _{n}:\left( Q_{n}/Q_{n-1}\right) \rightarrow Q_{n}$ in ${%
_{H}^{H}\mathcal{YD}}$ such that $\pi _{n}\sigma _{n}=\mathrm{Id}.$ Let $%
u_{n}:\Bbbk \rightarrow Q_{n}$ be the corestriction of the unit $u:\Bbbk
\rightarrow Q$ and let $\varepsilon _{n}=\varepsilon _{\mid
Q_{n}}:Q_{n}\rightarrow \Bbbk $ be the counit of the subcoalgebra $Q_{n}.$
Set%
\begin{equation*}
\sigma _{n}^{\prime }:=\sigma _{n}-u_{n}\circ \varepsilon _{n}\circ \sigma
_{n}
\end{equation*}%
This is a morphism in ${_{H}^{H}\mathcal{YD}}$. Moreover
\begin{eqnarray*}
\pi _{n}\circ \sigma _{n}^{\prime } &=&\pi _{n}\circ \sigma _{n}-\pi
_{n}\circ u_{n}\circ \varepsilon _{n}\circ \sigma _{n}\overset{n>0}{=}%
\mathrm{Id}_{Q_{n}/Q_{n-1}}-0=\mathrm{Id}_{Q_{n}/Q_{n-1}}, \\
\varepsilon _{n}\circ \sigma _{n}^{\prime } &=&\varepsilon _{n}\circ \sigma
_{n}-\varepsilon _{n}\circ u_{n}\circ \varepsilon _{n}\circ \sigma
_{n}=\varepsilon _{n}\circ \sigma _{n}-\varepsilon _{n}\circ \sigma _{n}=0.
\end{eqnarray*}%
Therefore, without loss of generality we can assume that $\varepsilon
_{n}\circ \sigma _{n}=0.$ A standard argument on split short exact sequences
shows that there exists a morphism $p_{n}:Q_{n}\rightarrow Q_{n-1}$ in ${%
_{H}^{H}\mathcal{YD}}$ such that $s_{n}p_{n}+\sigma _{n}\pi _{n}=\mathrm{Id}%
_{Q_{n}}$, $p_{n}s_{n}=\mathrm{Id}_{Q_{n-1}}$ and $p_{n}\sigma _{n}=0$. We
set
\begin{equation*}
x^{n,i}:=\sigma _{n}\left( y^{n,i}\right) .
\end{equation*}%
Therefore%
\begin{equation*}
y^{n,i}=\pi _{n}\sigma _{n}\left( y^{n,i}\right) =\pi _{n}\left(
x^{n,i}\right) =x^{n,i}+Q_{n-1}=\overline{x^{n,i}}.
\end{equation*}%
These terms $x^{n,i}$ define a $\Bbbk $-basis for $Q.$ As $Q$ is
finite-dimensional, there exists $d\in \N_0$ such that $Q=Q_{d}$; we
fix $d$ minimal. For all $0\leq a\leq b,$ define the maps%
\begin{eqnarray*}
p_{a,b} &:&Q_{b}\rightarrow Q_{a},\qquad p_{a,b}:=p_{a+1}\circ p_{a+2}\circ
\cdots \circ p_{b-1}\circ p_{b}, \\
s_{b,a} &:&Q_{a}\rightarrow Q_{b},\qquad s_{b,a}:=s_{b}\circ s_{b-1}\circ
\cdots \circ s_{a+2}\circ s_{a+1}.
\end{eqnarray*}%
Clearly one has%
\begin{equation*}
p_{a,b}\circ s_{b,a}=\mathrm{Id}_{Q_{a}}\text{.}
\end{equation*}%
Thus, for $0\leq i,a\leq b$ we have%
\begin{equation}
p_{i,b}\circ s_{b,a}=\left\{
\begin{array}{cc}
p_{i,b}\circ s_{b,i}\circ s_{i,a} & i>a \\
p_{i,a}\circ p_{a,b}\circ s_{b,a} & i\leq a%
\end{array}%
\right. =\left\{
\begin{array}{cc}
s_{i,a} & i>a \\
p_{i,a} & i\leq a%
\end{array}%
\right.  \label{form:ps1}
\end{equation}

Thus we get an isomorphism $\varphi :Q\rightarrow \mathrm{gr}Q$ of objects
in ${_{H}^{H}\mathcal{YD}}$ given by%
\begin{eqnarray*}
\varphi \left( x\right) := &&p_{0,d}\left( x\right) +\pi _{1}p_{1,d}\left(
x\right) +\pi _{2}p_{2,d}\left( x\right) +\cdots +\pi _{d-2}p_{d-2,d}\left(
x\right) +\pi _{d-1}p_{d-1,d}\left( x\right) +\pi _{d}\left( x\right) \\
&=&\sum_{0\leq t\leq d}\pi _{t}p_{t,d}\left( x\right) ,\text{ for every }%
x\in Q,
\end{eqnarray*}%
where we set%
\begin{equation*}
\pi _{0}=\mathrm{Id}_{Q_{0}},\qquad p_{d,d}=\mathrm{Id}_{Q_{d}}.
\end{equation*}%
For $0\leq n\leq d,$ we have%
\begin{eqnarray*}
\varphi \left( x^{n,i}\right) &=&\varphi \left( s_{d,n}\left( x^{n,i}\right)
\right) =\varphi \left( s_{d,n}\sigma _{n}\left( y^{n,i}\right) \right)
=\sum_{0\leq t\leq d}\pi _{t}p_{t,d}s_{d,n}\left( \sigma _{n}\left(
y^{n,i}\right) \right) \\
&=&\sum_{n<t\leq d}\pi _{t}p_{t,d}s_{d,n}\left( \sigma _{n}\left(
y^{n,i}\right) \right) +\sum_{0\leq t\leq n}\pi _{t}p_{t,d}s_{d,n}\left(
\sigma _{n}\left( y^{n,i}\right) \right) \\
&\overset{(\ref{form:ps1})}{=}&\sum_{n<t\leq d}\pi _{t}s_{t,n}\left( \sigma
_{n}\left( y^{n,i}\right) \right) +\sum_{0\leq t<n}\pi _{t}p_{t,n}\left(
\sigma _{n}\left( y^{n,i}\right) \right) +\pi _{n}p_{n,d}s_{d,n}\left(
\sigma _{n}\left( y^{n,i}\right) \right) \\
&=&\sum_{n<t\leq d}\pi _{t}s_{t,t-1}s_{t-1,n}\left( \sigma _{n}\left(
y^{n,i}\right) \right) +\sum_{0\leq t<n}\pi _{t}p_{t,n-1}p_{n-1,n}\left(
\sigma _{n}\left( y^{n,i}\right) \right) + \\
&&+\pi _{n}p_{n,d}s_{d,n}\left( \sigma _{n}\left( y^{n,i}\right) \right) \\
&=&\sum_{n<t\leq d}\pi _{t}s_{t}s_{t-1,n}\sigma _{n}\left( y^{n,i}\right)
+\sum_{0\leq t<n}\pi _{t}p_{t,n-1}p_{n}\sigma _{n}\left( y^{n,i}\right) +\pi
_{n}\sigma _{n}\left( y^{n,i}\right) \\
&=&0+0+y^{n,i}=y^{n,i}.
\end{eqnarray*}%
Hence $\varphi \left( x^{n,i}\right) =y^{n,i}.$ Since $y^{n,i}$ with $1\leq
i\leq \dim \left( Q_{n}/Q_{n-1}\right) =:d_{n}$ form a basis for $%
Q_{n}/Q_{n-1}$ we have that%
\begin{equation*}
hy^{n,i}\in \frac{Q_{n}}{Q_{n-1}},\qquad \left( y^{n,i}\right) _{-1}\otimes
\left( y^{n,i}\right) _{0}\in H\otimes \frac{Q_{n}}{Q_{n-1}}.
\end{equation*}%
Therefore there are $\chi _{t,i}^{n}\in H^{\ast }$ and $h_{t,i}^{n}\in H$
such that
\begin{equation}
hy^{n,i}=\sum_{1\leq t\leq d_{n}}\chi _{t,i}^{n}\left( h\right)
y^{n,t},\qquad \left( y^{n,i}\right) _{-1}\otimes \left( y^{n,i}\right)
_{0}=\sum_{1\leq t\leq d_{n}}h_{i,t}^{n}\otimes y^{n,t}.  \label{eq:chi}
\end{equation}%
We have%
\begin{eqnarray*}
h\left( h^{\prime }y^{n,i}\right) &=&\sum_{1\leq s\leq d_{n}}\chi
_{s,i}^{n}\left( h^{\prime }\right) hy^{n,s}=\sum_{1\leq s\leq d_{n}}\chi
_{s,i}^{n}\left( h^{\prime }\right) \sum_{1\leq t\leq d_{n}}\chi
_{t,s}^{n}\left( h\right) y^{n,t} \\
&=&\sum_{1\leq s\leq d_{n}}\sum_{1\leq t\leq d_{n}}\chi _{t,s}^{n}\left(
h\right) \chi _{s,i}^{n}\left( h^{\prime }\right) y^{n,t}, \\
\left( hh^{\prime }\right) y^{n,i} &=&\sum_{1\leq t\leq d_{n}}\chi
_{t,i}^{n}\left( hh^{\prime }\right) y^{n,t}
\end{eqnarray*}%
and hence%
\begin{equation*}
\chi _{t,i}^{n}\left( hh^{\prime }\right) =\sum_{1\leq s\leq d_{n}}\chi
_{t,s}^{n}\left( h\right) \chi _{s,i}^{n}\left( h^{\prime }\right) .
\end{equation*}%
Moreover
\begin{equation*}
y^{n,i}=1_{H}y^{n,i}=\sum_{1\leq t\leq d_{n}}\chi _{t,i}^{n}\left(
1_{H}\right) y^{n,t}
\end{equation*}%
and hence
\begin{equation*}
\chi _{t,i}^{n}\left( 1_{H}\right) =\delta _{t,i}.
\end{equation*}%
We also have%
\begin{eqnarray*}
\left( y^{n,i}\right) _{-1}\otimes \left( \left( y^{n,i}\right) _{0}\right)
_{-1}\otimes \left( \left( y^{n,i}\right) _{0}\right) _{0} &=&\sum_{1\leq
s\leq d_{n}}h_{i,s}^{n}\otimes \left( y^{n,s}\right) _{-1}\otimes \left(
y^{n,s}\right) _{0} \\
&=&\sum_{1\leq s\leq d_{n}}h_{i,s}^{n}\otimes \sum_{1\leq t\leq
d_{n}}h_{s,t}^{n}\otimes y^{n,t} \\
&=&\sum_{1\leq s\leq d_{n}}\sum_{1\leq t\leq d_{n}}h_{i,s}^{n}\otimes
h_{s,t}^{n}\otimes y^{n,t}, \\
\left( \left( y^{n,i}\right) _{-1}\right) _{1}\otimes \left( \left(
y^{n,i}\right) _{-1}\right) _{2}\otimes \left( y^{n,i}\right) _{0}
&=&\sum_{1\leq t\leq d_{n}}\Delta _{H}\left( h_{t,i}^{n}\right) \otimes
y^{n,t}
\end{eqnarray*}%
so that%
\begin{equation*}
\Delta _{H}\left( h_{t,i}^{n}\right) =\sum_{1\leq s\leq
d_{n}}h_{i,s}^{n}\otimes h_{s,t}^{n}.
\end{equation*}%
Moreover%
\begin{equation*}
y^{n,i}=\varepsilon _{H}\left( \left( y^{n,i}\right) _{-1}\right) \left(
y^{n,i}\right) _{0}=\sum_{1\leq t\leq d_{n}}\varepsilon _{H}\left(
h_{t,i}^{n}\right) y^{n,t}
\end{equation*}%
and hence%
\begin{equation*}
\varepsilon _{H}\left( h_{t,i}^{n}\right) =\delta _{t,i}.
\end{equation*}%
Finally%
\begin{eqnarray*}
\left( h_{1}y^{n,i}\right) _{-1}h_{2}\otimes \left( h_{1}y^{n,i}\right) _{0}
&=&\sum_{1\leq s\leq d_{n}}\chi _{s,i}^{n}\left( h_{1}\right) \left(
y^{n,s}\right) _{-1}h_{2}\otimes \left( y^{n,s}\right) _{0} \\
&=&\sum_{1\leq s\leq d_{n}}\chi _{s,i}^{n}\left( h_{1}\right) \sum_{1\leq
t\leq d_{n}}h_{s,t}^{n}h_{2}\otimes y^{n,t} \\
&=&\sum_{1\leq s\leq d_{n}}\sum_{1\leq t\leq d_{n}}h_{s,t}^{n}\chi
_{s,i}^{n}\left( h_{1}\right) h_{2}\otimes y^{n,t}, \\
h_{1}\left( y^{n,i}\right) _{-1}\otimes h_{2}\left( y^{n,i}\right) _{0}
&=&\sum_{1\leq s\leq d_{n}}h_{1}h_{i,s}^{n}\otimes h_{2}y^{n,s}=\sum_{1\leq
s\leq d_{n}}h_{1}h_{i,s}^{n}\otimes \sum_{1\leq t\leq d_{n}}\chi
_{t,s}^{n}\left( h_{2}\right) y^{n,t} \\
&=&\sum_{1\leq s\leq d_{n}}\sum_{1\leq t\leq d_{n}}h_{1}\chi
_{t,s}^{n}\left( h_{2}\right) h_{i,s}^{n}\otimes y^{n,t}
\end{eqnarray*}%
Therefore, we get%
\begin{equation*}
\sum_{1\leq s\leq d_{n}}h_{s,t}^{n}\chi _{s,i}^{n}\left( h_{1}\right)
h_{2}=\sum_{1\leq s\leq d_{n}}h_{1}\chi _{t,s}^{n}\left( h_{2}\right)
h_{i,s}^{n}.
\end{equation*}%
We have%
\begin{eqnarray*}
hx^{n,i} &=&h\sigma _{n}\left( y^{n,i}\right) =\sigma _{n}\left(
hy^{n,i}\right) =\sigma _{n}\left( \sum_{1\leq t\leq d_{n}}\chi
_{t,i}^{n}\left( h\right) y^{n,t}\right) =\sum_{1\leq t\leq d_{n}}\chi
_{t,i}^{n}\left( h\right) x^{n,t}, \\
\left( x^{n,i}\right) _{-1}\otimes \left( x^{n,i}\right) _{0} &=&\left(
\sigma _{n}\left( y^{n,i}\right) \right) _{-1}\otimes \left( \sigma
_{n}\left( y^{n,i}\right) \right) _{0}=\left( y^{n,i}\right) _{-1}\otimes
\sigma _{n}\left( \left( y^{n,i}\right) _{0}\right) =\sum_{1\leq t\leq
d_{n}}h_{i,t}^{n}\otimes x^{n,t}, \\
\varepsilon _{Q}\left( x^{n,i}\right) &=&\varepsilon _{n}\left(
x^{n,i}\right) =\varepsilon _{n}\sigma _{n}\left( y^{n,i}\right) =0\text{
for }n>0.
\end{eqnarray*}%
If $Q$ is connected, then $d_{0}=1$ so we may assume $y^{0,0}:=1_{Q}+Q_{-1}$%
. Since $\pi _{0}=\mathrm{Id}_{Q_{0}}$ we get
\begin{equation*}
\sigma _{0}=\mathrm{Id}_{Q_{0}}\circ \sigma _{0}=\pi _{0}\circ \sigma _{0}=%
\mathrm{Id}_{Q_{0}}
\end{equation*}%
and hence
\begin{equation*}
x^{0,0}=\sigma _{0}\left( y^{0,0}\right) =\sigma _{0}\left(
1_{Q}+Q_{-1}\right) =1_{Q}.
\end{equation*}%
Since, by Proposition \ref{pro:CMU}, $Q_{a}\cdot Q_{a^{\prime }}\subseteq
Q_{a+a^{\prime }}$ for every $a,a^{\prime }\in \N_0$, we can write the
product of two elements of the basis in the form%
\begin{equation}
x^{a,l}x^{a^{\prime },l^{\prime }}=\sum_{u\leq a+a^{\prime }}\sum_{v}\mu
_{u,v}^{a,l,a^{\prime },l^{\prime }}x^{u,v}.  \label{form:GelYD1}
\end{equation}%
We compute%
\begin{eqnarray*}
\overline{x^{a,l}}\cdot \overline{x^{a^{\prime },l^{\prime }}} &=&\left(
x^{a,l}+Q_{a-1}\right) \left( x^{a^{\prime },l^{\prime }}+Q_{a^{\prime
}-1}\right) \\
&=&\left( x^{a,l}x^{a^{\prime },l^{\prime }}\right) +Q_{a+a^{\prime }-1} \\
&&\overset{(\ref{form:GelYD1})}{=}\left( \sum_{u\leq a+a^{\prime
}}\sum_{v}\mu _{u,v}^{a,l,a^{\prime },l^{\prime }}x^{u,v}\right)
+Q_{a+a^{\prime }-1} \\
&=&\left( \sum_{v}\mu _{a+a^{\prime },v}^{a,l,a^{\prime },l^{\prime
}}x^{a+a^{\prime },v}\right) +Q_{a+a^{\prime }-1} \\
&=&\sum_{v}\mu _{a+a^{\prime },v}^{a,l,a^{\prime },l^{\prime }}\left(
x^{a+a^{\prime },v}+Q_{a+a^{\prime }-1}\right) \\
&=&\sum_{v}\mu _{a+a^{\prime },v}^{a,l,a^{\prime },l^{\prime }}\overline{%
x^{a+a^{\prime },v}}.
\end{eqnarray*}%
which gives%
\begin{equation}
\overline{x^{a,l}}\cdot \overline{x^{a^{\prime },l^{\prime }}}=\sum_{v}\mu
_{a+a^{\prime },v}^{a,l,a^{\prime },l^{\prime }}\overline{x^{a+a^{\prime },v}%
}.  \label{form:GelYD2}
\end{equation}

\begin{remark}
\label{rem:HochYD}Let $H$ be a Hopf algebra and let $\left(
A,m_{A},u_{A}\right) $ be an algebra in ${_{H}^{H}\mathcal{YD}}$. Let $%
\varepsilon _{A}:A\rightarrow \Bbbk $ be an algebra map in ${_{H}^{H}%
\mathcal{YD}}$. The Hochschild cohomology in a monoidal category is known, see e.g. \cite{AMS-Hoch}.
%\textcolor{blue}{[Maybe one should further investigate the bibliography quoted in \url{http://zbmath.org/?q=an:1116.18004}. For the moment we will use the results in \cite{AMS-Hoch} although they are not the newest in the literature. We just quote what we need.]}
Consider $\Bbbk $ as an $A$-bimodule in ${_{H}^{H}\mathcal{YD}}$ through $\varepsilon _{A}$. Then,
following \cite[1.24]{AMS-Hoch}, we can consider an analogue of the standard
complex
\begin{equation*}
\xymatrixrowsep{25pt}\xymatrixcolsep{1cm} \xymatrix{\yd( \Bbbk ,\Bbbk)
\ar[r]^{\partial ^0}& \yd(A ,\Bbbk) \ar[r]^{\partial ^1}&\yd( A^{\otimes2}
,\Bbbk) \ar[r]^-{\partial ^2}&\yd( A^{\otimes3} ,\Bbbk) \ar[r]^-{\partial
^3}&\cdots}
\end{equation*}
Explicitly, given $f$ in the corresponding domain of $\partial ^{n},$ for $%
n=0,1,2,3$, we have
\begin{eqnarray*}
\partial ^{0}\left( f\right) &=&f\left( 1\right) \varepsilon
_{A}-\varepsilon _{A}f\left( 1\right) =0, \\
\partial ^{1}\left( f\right) &=&f\otimes \varepsilon _{A}-fm_{A}+\varepsilon
_{A}\otimes f, \\
\partial ^{2}\left( f\right) &=&f\otimes \varepsilon _{A}-f\left( A\otimes
m_{A}\right) +f\left( m_{A}\otimes A\right) -\varepsilon _{A}\otimes f, \\
\partial ^{3}\left( f\right) &=&f\otimes \varepsilon _{A}-f\left( A\otimes
A\otimes m_{A}\right) +f\left( A\otimes m_{A}\otimes A\right) -f\left(
m_{A}\otimes A\otimes A\right) +\varepsilon _{A}\otimes f.
\end{eqnarray*}%
For every $n\geq 1$ denote by
\begin{equation*}
\mathrm{Z}_{{\mathcal{YD}}}^{n}\left( A,\Bbbk \right) :=\mathrm{ker}\left(
\partial ^{n}\right) ,\qquad \mathrm{B}_{{\mathcal{YD}}}^{n}\left( A,\Bbbk
\right) :=\mathrm{Im}\left( \partial ^{n-1}\right) \qquad \text{and}\qquad
\mathrm{H}_{{\mathcal{YD}}}^{n}\left( A,\Bbbk \right) :=\frac{\mathrm{Z}_{{%
\mathcal{YD}}}^{n}\left( A,\Bbbk \right) }{\mathrm{B}_{{\mathcal{YD}}%
}^{n}\left( A,\Bbbk \right) }
\end{equation*}%
the abelian groups of $n$-cocycles, of $n$-coboundaries and the $n$-th
Hochschild cohomology group in ${_{H}^{H}\mathcal{YD}}$ of the algebra $A$
with coefficients in $\Bbbk $. We point out that the construction above
works for an arbitrary $A$-bimodule $M$ in ${_{H}^{H}\mathcal{YD}}$ instead
of $\Bbbk $.
\end{remark}

Next result is inspired by \cite[Proposition 2.3]{EG}. Two
coquasi-bialgebras $Q$ and $Q^{\prime }$ in ${_{H}^{H}\mathcal{YD}}$ will be
called \textbf{gauge equivalent} whenever there is some gauge transformation
$\gamma :Q\otimes Q\rightarrow \Bbbk $ in ${_{H}^{H}\mathcal{YD}}$ such that
$Q^{\gamma }\cong Q^{\prime }$ as coquasi-bialgebras in ${_{H}^{H}\mathcal{YD%
}},$ see Proposition \ref{pro:deformYD} for the structure of $Q^{\gamma }$.

\begin{theorem}
\label{teo:GelakiYD} Let $H$ be a semisimple and cosemisimple Hopf algebra
and let $\left( Q,m,u,\Delta ,\varepsilon ,\omega \right) $ be a f.d.
connected coquasi-bialgebra in ${_{H}^{H}\mathcal{YD}}$. If $\mathrm{H}_{{%
\mathcal{YD}}}^{3}\left( \mathrm{gr}Q,\Bbbk \right) =0$ then $Q$ is gauge
equivalent to a connected bialgebra in ${_{H}^{H}\mathcal{YD}}$.
\end{theorem}

\begin{proof}
For $t\in \N_0$, and $x,y,z$ in the basis of $Q$, we set%
\begin{equation*}
\omega _{t}\left( x\otimes y\otimes z\right) :=\delta _{\left\vert
x\right\vert +\left\vert y\right\vert +\left\vert z\right\vert ,t}\omega
\left( x\otimes y\otimes z\right) .
\end{equation*}%
Let us check it defines a morphism $\omega _{t}:Q\otimes Q\otimes
Q\rightarrow \Bbbk $ in ${_{H}^{H}\mathcal{YD}}$. It is left $H$-linear as,
by means of (\ref{eq:chi}), the definition of $\omega _{t}$ and the $H$%
-linearity of $\omega $, we can prove that $\omega _{t}\left( h\left(
x^{n,i}\otimes x^{n^{\prime },i^{\prime }}\otimes x^{n^{\prime \prime
},i^{\prime \prime }}\right) \right) =\varepsilon _{H}\left( h\right) \omega
_{t}\left( x^{n,i}\otimes x^{n^{\prime },i^{\prime }}\otimes x^{n^{\prime
\prime },i^{\prime \prime }}\right) .$

\begin{invisible}
\begin{eqnarray*}
&&\omega _{t}\left( h\left( x^{n,i}\otimes x^{n^{\prime },i^{\prime
}}\otimes x^{n^{\prime \prime },i^{\prime \prime }}\right) \right) \\
&=&\omega _{t}\left( h_{1}x^{n,i}\otimes h_{2}x^{n^{\prime },i^{\prime
}}\otimes h_{3}x^{n^{\prime \prime },i^{\prime \prime }}\right) \\
&&\overset{(\ref{eq:chi})}{=}\omega _{t}\left( \sum_{1\leq w\leq d_{n}}\chi
_{w,i}^{n}\left( h_{1}\right) x^{n,w}\otimes \sum_{1\leq w^{\prime }\leq
d_{n^{\prime }}}\chi _{w^{\prime },i^{\prime }}^{n^{\prime }}\left(
h_{2}\right) x^{n^{\prime },w^{\prime }}\otimes \sum_{1\leq w^{\prime \prime
}\leq d_{n^{\prime \prime }}}\chi _{w^{\prime \prime },i^{\prime \prime
}}^{n^{\prime \prime }}\left( h_{3}\right) x^{n^{\prime \prime },w^{\prime
\prime }}\right) \\
&=&\sum_{1\leq w\leq d_{n}}\sum_{1\leq w^{\prime }\leq d_{n^{\prime
}}}\sum_{1\leq w^{\prime \prime }\leq d_{n^{\prime \prime }}}\chi
_{w,i}^{n}\left( h_{1}\right) \chi _{w^{\prime },i^{\prime }}^{n^{\prime
}}\left( h_{2}\right) \chi _{w^{\prime \prime },i^{\prime \prime
}}^{n^{\prime \prime }}\left( h_{3}\right) \omega _{t}\left( x^{n,w}\otimes
x^{n^{\prime },w^{\prime }}\otimes x^{n^{\prime \prime },w^{\prime \prime
}}\right) \\
&=&\sum_{1\leq w\leq d_{n}}\sum_{1\leq w^{\prime }\leq d_{n^{\prime
}}}\sum_{1\leq w^{\prime \prime }\leq d_{n^{\prime \prime }}}\chi
_{w,i}^{n}\left( h_{1}\right) \chi _{w^{\prime },i^{\prime }}^{n^{\prime
}}\left( h_{2}\right) \chi _{w^{\prime \prime },i^{\prime \prime
}}^{n^{\prime \prime }}\left( h_{3}\right) \delta _{n+n^{\prime }+n^{\prime
\prime },t}\omega \left( x^{n,w}\otimes x^{n^{\prime },w^{\prime }}\otimes
x^{n^{\prime \prime },w^{\prime \prime }}\right) \\
&&\overset{(\ref{eq:chi})}{=}\delta _{n+n^{\prime }+n^{\prime \prime
},t}\omega \left( h_{1}x^{n,i}\otimes h_{2}x^{n^{\prime },i^{\prime
}}\otimes h_{3}x^{n^{\prime \prime },i^{\prime \prime }}\right) \\
&=&\delta _{n+n^{\prime }+n^{\prime \prime },t}\omega \left( h\left(
x^{n,i}\otimes x^{n^{\prime },i^{\prime }}\otimes x^{n^{\prime \prime
},i^{\prime \prime }}\right) \right) \\
&=&\delta _{n+n^{\prime }+n^{\prime \prime },t}\varepsilon _{H}\left(
h\right) \omega \left( x^{n,i}\otimes x^{n^{\prime },i^{\prime }}\otimes
x^{n^{\prime \prime },i^{\prime \prime }}\right) =\varepsilon _{H}\left(
h\right) \omega _{t}\left( x^{n,i}\otimes x^{n^{\prime },i^{\prime }}\otimes
x^{n^{\prime \prime },i^{\prime \prime }}\right) .
\end{eqnarray*}
\end{invisible}

Moreover it is left $H$-colinear as, by means of (\ref{eq:chi}), the
definition of $\omega _{t}$ and the $H$-colinearity of $\omega $, we can
prove that%
\begin{equation*}
\left( x^{n,i}\otimes x^{n^{\prime },i^{\prime }}\otimes x^{n^{\prime \prime
},i^{\prime \prime }}\right) _{\left\langle -1\right\rangle }\otimes \omega
_{t}\left( \left( x^{n,i}\otimes x^{n^{\prime },i^{\prime }}\otimes
x^{n^{\prime \prime },i^{\prime \prime }}\right) _{\left\langle
0\right\rangle }\right) =1_{H}\otimes \omega _{t}\left( x^{n,i}\otimes
x^{n^{\prime },i^{\prime }}\otimes x^{n^{\prime \prime },i^{\prime \prime
}}\right) .
\end{equation*}

\begin{invisible}
\begin{eqnarray*}
&&\left( x^{n,i}\otimes x^{n^{\prime },i^{\prime }}\otimes x^{n^{\prime
\prime },i^{\prime \prime }}\right) _{\left\langle -1\right\rangle }\otimes
\omega _{t}\left( \left( x^{n,i}\otimes x^{n^{\prime },i^{\prime }}\otimes
x^{n^{\prime \prime },i^{\prime \prime }}\right) _{\left\langle
0\right\rangle }\right) \\
&=&\left( x^{n,i}\right) _{-1}\left( x^{n^{\prime },i^{\prime }}\right)
_{-1}\left( x^{n^{\prime \prime },i^{\prime \prime }}\right) _{-1}\otimes
\omega _{t}\left( \left( x^{n,i}\right) _{0}\otimes \left( x^{n^{\prime
},i^{\prime }}\right) _{0}\otimes \left( x^{n^{\prime \prime },i^{\prime
\prime }}\right) _{0}\right) \\
&&\overset{(\ref{eq:chi})}{=}\sum_{1\leq w\leq d_{n}}h_{i,w}^{n}\sum_{1\leq
w^{\prime }\leq d_{n^{\prime }}}h_{i^{\prime },w^{\prime }}^{n^{\prime
}}\sum_{1\leq w^{\prime \prime }\leq d_{n^{\prime \prime }}}h_{i^{\prime
\prime },w^{\prime \prime }}^{n^{\prime \prime }}\otimes \omega _{t}\left(
x^{n,w}\otimes x^{n^{\prime },w^{\prime }}\otimes x^{n^{\prime \prime
},w^{\prime \prime }}\right) \\
&=&\sum_{1\leq w\leq d_{n}}h_{i,w}^{n}\sum_{1\leq w^{\prime }\leq
d_{n^{\prime }}}h_{i^{\prime },w^{\prime }}^{n^{\prime }}\sum_{1\leq
w^{\prime \prime }\leq d_{n^{\prime \prime }}}h_{i^{\prime \prime
},w^{\prime \prime }}^{n^{\prime \prime }}\otimes \delta _{n+n^{\prime
}+n^{\prime \prime },t}\omega \left( x^{n,w}\otimes x^{n^{\prime },w^{\prime
}}\otimes x^{n^{\prime \prime },w^{\prime \prime }}\right) \\
&&\overset{(\ref{eq:chi})}{=}\delta _{n+n^{\prime }+n^{\prime \prime
},t}\left( x^{n,i}\right) _{-1}\left( x^{n^{\prime },i^{\prime }}\right)
_{-1}\left( x^{n^{\prime \prime },i^{\prime \prime }}\right) _{-1}\otimes
\omega \left( \left( x^{n,i}\right) _{0}\otimes \left( x^{n^{\prime
},i^{\prime }}\right) _{0}\otimes \left( x^{n^{\prime \prime },i^{\prime
\prime }}\right) _{0}\right) \\
&=&\delta _{n+n^{\prime }+n^{\prime \prime },t}\left( x^{n,i}\otimes
x^{n^{\prime },i^{\prime }}\otimes x^{n^{\prime \prime },i^{\prime \prime
}}\right) _{-1}\otimes \omega \left( \left( x^{n,i}\otimes x^{n^{\prime
},i^{\prime }}\otimes x^{n^{\prime \prime },i^{\prime \prime }}\right)
_{0}\right) \\
&=&\delta _{n+n^{\prime }+n^{\prime \prime },t}\left( \omega \left(
x^{n,i}\otimes x^{n^{\prime },i^{\prime }}\otimes x^{n^{\prime \prime
},i^{\prime \prime }}\right) \right) _{-1}\otimes \left( \omega \left(
x^{n,i}\otimes x^{n^{\prime },i^{\prime }}\otimes x^{n^{\prime \prime
},i^{\prime \prime }}\right) \right) _{0} \\
&=&1_{H}\otimes \delta _{n+n^{\prime }+n^{\prime \prime },t}\omega \left(
x^{n,i}\otimes x^{n^{\prime },i^{\prime }}\otimes x^{n^{\prime \prime
},i^{\prime \prime }}\right) \\
&=&1_{H}\otimes \omega _{t}\left( x^{n,i}\otimes x^{n^{\prime },i^{\prime
}}\otimes x^{n^{\prime \prime },i^{\prime \prime }}\right) .
\end{eqnarray*}
\end{invisible}

Clearly, for $x,y,z\in Q$ in the basis, one has%
\begin{equation*}
\sum\limits_{t\in \N_0}\omega _{t}\left( x\otimes y\otimes z\right)
=\sum\limits_{t\in \N_0}\delta _{\left\vert x\right\vert +\left\vert
y\right\vert +\left\vert z\right\vert ,t}\omega \left( x\otimes y\otimes
z\right) =\omega \left( x\otimes y\otimes z\right)
\end{equation*}%
so that we can formally write
\begin{equation}
\omega =\sum\limits_{t\in \N_0}\omega _{t}.  \label{eq:omegapiece}
\end{equation}

Since $\varepsilon $ is trivial on elements in the basis of strictly
positive degree, one gets%
\begin{equation}
\omega _{0}=\varepsilon \otimes \varepsilon \otimes \varepsilon .
\label{form:Omega0YD}
\end{equation}

\begin{invisible}
In fact%
\begin{eqnarray*}
\omega _{0}\left( x\otimes y\otimes z\right) &=&\delta _{\left\vert
x\right\vert +\left\vert y\right\vert +\left\vert z\right\vert ,0}\omega
\left( x\otimes y\otimes z\right) \\
&=&\delta _{\left\vert x\right\vert ,0}\delta _{\left\vert y\right\vert
,0}\delta _{\left\vert z\right\vert ,0}\omega \left( x\otimes y\otimes
z\right) \\
&=&\delta _{\left\vert x\right\vert ,0}\delta _{\left\vert y\right\vert
,0}\delta _{\left\vert z\right\vert ,0}xyz\omega \left( 1_{\Bbbk }\otimes
1_{\Bbbk }\otimes 1_{\Bbbk }\right) \\
&=&\delta _{\left\vert x\right\vert ,0}\delta _{\left\vert y\right\vert
,0}\delta _{\left\vert z\right\vert ,0}xyz=\delta _{\left\vert x\right\vert
,0}\delta _{\left\vert y\right\vert ,0}\delta _{\left\vert z\right\vert
,0}\varepsilon \left( x\right) \varepsilon \left( y\right) \varepsilon
\left( z\right) \\
&\overset{(\ast )}{=}&\varepsilon \left( x\right) \varepsilon \left(
y\right) \varepsilon \left( z\right) =\left( \varepsilon \otimes \varepsilon
\otimes \varepsilon \right) \left( x\otimes y\otimes z\right)
\end{eqnarray*}%
where (*) we used that $\varepsilon $ is trivial on elements in the basis of
strictly positive degree.
\end{invisible}

If $\omega =\omega _{0}$ then $Q$ is a (connected) bialgebra in ${_{H}^{H}%
\mathcal{YD}}$ and the proof is finished. Thus we can assume $\omega \neq
\omega _{0}$ and set
\begin{eqnarray*}
s &:&=\min \left\{ i\in \N :\omega
_{i}\neq 0\right\} , \\
\overline{\omega }_{s} &:&=\omega _{s}\left( \varphi ^{-1}\otimes \varphi
^{-1}\otimes \varphi ^{-1}\right) , \\
\overline{Q} &:&=\mathrm{gr}Q.
\end{eqnarray*}%
Note that $\overline{\omega }_{s}$ is a morphism in ${_{H}^{H}\mathcal{YD}}$
as a composition of morphisms in ${_{H}^{H}\mathcal{YD}}$.

Let $n\in \N_0,$ let $C^{4}=Q\otimes Q\otimes Q\otimes Q$ and let $%
u\in C_{\left( n\right) }^{4}=\sum_{i+j+k+l\leq n}Q_{i}\otimes Q_{j}\otimes
Q_{k}\otimes Q_{l}$.

\begin{invisible}
Similarly one has $C^{3}=Q\otimes Q\otimes Q$ and the corresponding $%
C_{\left( n\right) }^{3}.$ Then
\begin{equation*}
u_{1}\otimes u_{2}\in \sum_{i+j\leq n}C_{\left( i\right) }^{4}\otimes
C_{\left( j\right) }^{4},\qquad u_{1}\otimes u_{2}\otimes u_{3}\in
\sum_{a+b+c\leq n}C_{\left( a\right) }^{4}\otimes C_{\left( b\right)
}^{4}\otimes C_{\left( c\right) }^{4}.
\end{equation*}%
Note that%
\begin{eqnarray*}
\omega _{|C_{\left( n\right) }^{3}} &=&\left( \sum\limits_{t\in \N_0%
}\omega _{t}\right) _{|C_{\left( n\right) }^{3}}=\left( \sum\limits_{0\leq
t\leq n}\omega _{t}\right) _{|C_{\left( n\right) }^{3}}, \\
\left( Q\otimes Q\otimes m\right) \left( C_{\left( n\right) }^{4}\right)
&\subseteq &C_{\left( n\right) }^{3},\left( m\otimes Q\otimes Q\right)
\left( C_{\left( n\right) }^{4}\right) \subseteq C_{\left( n\right)
}^{3},\left( Q\otimes m\otimes Q\right) \left( C_{\left( n\right)
}^{4}\right) \subseteq C_{\left( n\right) }^{3}, \\
\left( \varepsilon \otimes Q\otimes Q\otimes Q\right) \left( C_{\left(
n\right) }^{4}\right) &\subseteq &C_{\left( n\right) }^{3},\left( Q\otimes
Q\otimes Q\otimes \varepsilon \right) \left( C_{\left( n\right) }^{4}\right)
\subseteq C_{\left( n\right) }^{3}.
\end{eqnarray*}%
Hence%
\begin{eqnarray*}
&&\left( Q\otimes Q\otimes m\right) \left( u_{1}\right) \otimes \left(
m\otimes Q\otimes Q\right) \left( u_{2}\right) \\
&\in &\sum_{i+j\leq n}\left( Q\otimes Q\otimes m\right) \left( C_{\left(
i\right) }^{4}\right) \otimes \left( m\otimes Q\otimes Q\right) \left(
C_{\left( j\right) }^{4}\right) \\
&\subseteq &\sum_{i+j\leq n}C_{\left( i\right) }^{3}\otimes C_{\left(
j\right) }^{3}\subseteq \sum_{i+j\leq n}C_{\left( i+j\right) }^{6}\subseteq
C_{\left( n\right) }^{6}, \\
&&\left( \varepsilon \otimes Q^{\otimes 3}\right) \left( u_{1}\right)
\otimes \left( Q\otimes m\otimes Q\right) \left( u_{2}\right) \cdot \left(
Q^{\otimes 3}\otimes \varepsilon \right) \left( u_{3}\right) \\
&\in &\sum_{a+b+c\leq n}\left( \varepsilon \otimes Q^{\otimes 3}\right)
\left( C_{\left( a\right) }^{4}\right) \otimes \left( Q\otimes m\otimes
Q\right) \left( C_{\left( b\right) }^{4}\right) \cdot \left( Q^{\otimes
3}\otimes \varepsilon \right) \left( C_{\left( c\right) }^{4}\right) \\
&\subseteq &\sum_{a+b+c\leq n}C_{\left( a\right) }^{3}\otimes C_{\left(
b\right) }^{3}\otimes C_{\left( c\right) }^{3}\subseteq \sum_{a+b+c\leq
n}C_{\left( a+b+c\right) }^{9}\subseteq C_{\left( n\right) }^{9}
\end{eqnarray*}%
A direct computation on a generic element $x^{n_{1},i_{1}}\otimes
x^{n_{2},i_{2}}\otimes x^{n_{3},i_{3}}\otimes x^{n_{4},i_{4}}\otimes
x^{n_{5},i_{5}}\otimes x^{n_{6},i_{6}}$ of the basis of $C_{\left( n\right)
}^{6}$ shows that%
\begin{equation*}
\left( \omega \otimes \omega \right) _{\mid C_{\left( n\right) }^{6}}=\left(
\sum_{i+j\leq n}\omega _{i}\otimes \omega _{j}\right) _{\mid C_{\left(
n\right) }^{6}}.
\end{equation*}%
Similarly one proves%
\begin{equation*}
\left( \omega \otimes \omega \otimes \omega \right) _{\mid C_{\left(
n\right) }^{9}}=\left( \sum_{a+b+c\leq n}\omega _{a}\otimes \omega
_{b}\otimes \omega _{c}\right) _{\mid C_{\left( n\right) }^{9}}.
\end{equation*}

Since $\omega $ is a reassociator, we have%
\begin{equation*}
\omega \left( Q\otimes Q\otimes m\right) \ast \omega \left( m\otimes
Q\otimes Q\right) =\left( \varepsilon \otimes \omega \right) \ast \omega
\left( Q\otimes m\otimes Q\right) \ast \left( \omega \otimes \varepsilon
\right)
\end{equation*}%
and hence%
\begin{equation*}
\omega \left( Q\otimes Q\otimes m\right) \left( u_{1}\right) \cdot \omega
\left( m\otimes Q\otimes Q\right) \left( u_{2}\right) =\omega \left( \left(
\varepsilon \otimes Q^{\otimes 3}\right) \left( u_{1}\right) \right) \cdot
\omega \left( Q\otimes m\otimes Q\right) \left( u_{2}\right) \cdot \omega
\left( \left( Q^{\otimes 3}\otimes \varepsilon \right) \left( u_{3}\right)
\right)
\end{equation*}%
which, by the foregoing, rereads as
\begin{equation*}
\sum\limits_{0\leq i+j\leq n}\left[
\begin{array}{c}
\omega _{i}\left( \left( Q\otimes Q\otimes m\right) \left( u_{1}\right)
\right) \cdot \\
\omega _{j}\left( \left( m\otimes Q\otimes Q\right) \left( u_{2}\right)
\right)%
\end{array}%
\right] =\sum\limits_{0\leq a+b+c\leq n}\left[
\begin{array}{c}
\left( \varepsilon \otimes \omega _{a}\right) \left( u_{1}\right) \cdot \\
\omega _{b}\left( \left( Q\otimes m\otimes Q\right) \left( u_{2}\right)
\right) \cdot \\
\left( \omega _{c}\otimes \varepsilon \right) \left( u_{3}\right)%
\end{array}%
\right] .
\end{equation*}%
This leads to the following equality.
\end{invisible}

A direct computation rewriting the cocycle condition using (\ref%
{eq:omegapiece}) proves that, for every $n\in \N_0$, and $u\in
C_{\left( n\right) }^{4}$%
\begin{eqnarray}
&&\sum\limits_{0\leq i+j\leq n}\left[ \omega _{i}\left( Q\otimes Q\otimes
m\right) \ast \omega _{j}\left( m\otimes Q\otimes Q\right) \right] \left(
u\right)  \label{form:goldenYD} \\
&=&\sum\limits_{0\leq a+b+c\leq n}\left[ \left( \varepsilon \otimes \omega
_{a}\right) \ast \omega _{b}\left( Q\otimes m\otimes Q\right) \ast \left(
\omega _{c}\otimes \varepsilon \right) \right] \left( u\right) .  \notag
\end{eqnarray}

Next aim is to check that $\left[ \overline{\omega }_{s}\right] \in \mathrm{H%
}_{{\mathcal{YD}}}^{3}\left( \mathrm{gr}Q,\Bbbk \right) $ i.e. that%
\begin{equation*}
\overline{\omega }_{s}\left( m_{\overline{Q}}\otimes \overline{Q}\otimes
\overline{Q}\right) +\overline{\omega }_{s}\left( \overline{Q}\otimes
\overline{Q}\otimes m_{\overline{Q}}\right) =\left( \varepsilon _{\overline{Q%
}}\otimes \overline{\omega }_{s}\right) +\overline{\omega }_{s}\left(
\overline{Q}\otimes m_{\overline{Q}}\otimes \overline{Q}\right) +\left(
\overline{\omega }_{s}\otimes \varepsilon _{\overline{Q}}\right) .
\end{equation*}%
This is achieved by evaluating the two sides of the equality above on $%
\overline{u}:=\overline{x}\otimes \overline{y}\otimes \overline{z}\otimes
\overline{t}$ where $x,y,z,t$ are elements in the basis and using (\ref%
{form:GelYD2}). If $\overline{u}$ has homogeneous degree greater than $s,$
then both terms are zero. Otherwise, i.e. if $\overline{u}$ has homogeneous
degree at most $s$, one has $\overline{\omega }_{s}\left( m_{\overline{Q}%
}\otimes \overline{Q}\otimes \overline{Q}\right) \left( \overline{u}\right)
=\omega _{s}\left( m_{Q}\otimes Q\otimes Q\right) \left( u\right) $ and
similarly for the other pieces so that one has to check that
\begin{equation*}
\omega _{s}\left( m\otimes Q\otimes Q\right) \left( u\right) +\omega
_{s}\left( Q\otimes Q\otimes m\right) \left( u\right) =\left( \varepsilon
\otimes \omega _{s}\right) \left( u\right) +\omega _{s}\left( Q\otimes
m\otimes Q\right) \left( u\right) +\left( \omega _{s}\otimes \varepsilon
\right) \left( u\right) .
\end{equation*}

This equality follows by using (\ref{form:goldenYD}) and the definition of $%
s $.

\begin{invisible}
We compute this equality on $\overline{u}:=\overline{x}\otimes \overline{y}%
\otimes \overline{z}\otimes \overline{t}$ where $x,y,z,t$ are elements in
the basis. In this case take $x=x^{a,l}$ and $y=x^{a^{\prime },l^{\prime }}$%
. First we have%
\begin{eqnarray*}
&&\overline{\omega }_{s}\left( m_{\overline{Q}}\otimes \overline{Q}\otimes
\overline{Q}\right) \left( \overline{u}\right) \\
&=&\overline{\omega }_{s}\left( \overline{x}\cdot \overline{y}\otimes
\overline{z}\otimes \overline{t}\right) \\
&=&\omega _{s}\left( \varphi ^{-1}\left( \overline{x}\cdot \overline{y}%
\right) \otimes z\otimes t\right) \\
&\overset{(\ref{form:GelYD2})}{=}&\omega _{s}\left( \varphi ^{-1}\left(
\sum_{v}\mu _{a+a^{\prime },v}^{a,l,a^{\prime },l^{\prime }}\overline{%
x^{a+a^{\prime },v}}\right) \otimes z\otimes t\right) \\
&=&\omega _{s}\left( \sum_{v}\mu _{a+a^{\prime },v}^{a,l,a^{\prime
},l^{\prime }}x^{a+a^{\prime },v}\otimes z\otimes t\right) \\
&=&\sum_{v}\mu _{a+a^{\prime },v}^{a,l,a^{\prime },l^{\prime }}\omega
_{s}\left( x^{a+a^{\prime },v}\otimes z\otimes t\right) =\sum_{v}\mu
_{a+a^{\prime },v}^{a,l,a^{\prime },l^{\prime }}\delta _{a+a^{\prime
}+\left\vert z\right\vert +\left\vert t\right\vert ,s}\omega \left(
x^{a+a^{\prime },v}\otimes z\otimes t\right) \\
&=&\sum_{v}\mu _{a+a^{\prime },v}^{a,l,a^{\prime },l^{\prime }}\delta
_{\left\vert x\right\vert +\left\vert y\right\vert +\left\vert z\right\vert
+\left\vert t\right\vert ,s}\omega \left( x^{a+a^{\prime },v}\otimes
z\otimes t\right) .
\end{eqnarray*}%
so that
\begin{equation}
\overline{\omega }_{s}\left( m_{\overline{Q}}\otimes \overline{Q}\otimes
\overline{Q}\right) \left( \overline{u}\right) =\sum_{v}\mu _{a+a^{\prime
},v}^{a,l,a^{\prime },l^{\prime }}\delta _{\left\vert x\right\vert
+\left\vert y\right\vert +\left\vert z\right\vert +\left\vert t\right\vert
,s}\omega \left( x^{a+a^{\prime },v}\otimes z\otimes t\right) .
\label{form:ubar1}
\end{equation}

CASE 1) $\left\vert x\right\vert +\left\vert y\right\vert +\left\vert
z\right\vert +\left\vert t\right\vert >s.$ In this case by (\ref{form:ubar1}%
), we get%
\begin{equation*}
\overline{\omega }_{s}\left( m_{\overline{Q}}\otimes \overline{Q}\otimes
\overline{Q}\right) \left( \overline{u}\right) =\sum_{v}\mu _{a+a^{\prime
},v}^{a,l,a^{\prime },l^{\prime }}\delta _{\left\vert x\right\vert
+\left\vert y\right\vert +\left\vert z\right\vert +\left\vert t\right\vert
,s}\omega \left( x^{a+a^{\prime },v}\otimes z\otimes t\right) =0.
\end{equation*}%
Similarly $\overline{\omega }_{s}\left( \overline{Q}\otimes \overline{Q}%
\otimes m_{\overline{Q}}\right) \left( \overline{u}\right) =0=\overline{%
\omega }_{s}\left( \overline{Q}\otimes m_{\overline{Q}}\otimes \overline{Q}%
\right) \left( \overline{u}\right) .$ Moreover%
\begin{eqnarray*}
\left( \overline{\omega }_{s}\otimes \varepsilon _{\overline{A}}\right)
\left( \overline{u}\right) &=&\overline{\omega }_{s}\left( \overline{x}%
\otimes \overline{y}\otimes \overline{z}\right) \varepsilon _{\overline{A}%
}\left( \overline{t}\right) \\
&=&\omega _{s}\left( x\otimes y\otimes z\right) \delta _{\left\vert
t\right\vert ,0}\varepsilon \left( t\right) \\
&=&\delta _{\left\vert x\right\vert +\left\vert y\right\vert +\left\vert
z\right\vert ,s}\omega \left( x\otimes y\otimes z\right) \delta _{\left\vert
t\right\vert ,0}\varepsilon \left( t\right) =0.
\end{eqnarray*}%
Similarly $\left( \varepsilon _{\overline{A}}\otimes \overline{\omega }%
_{s}\right) \left( \overline{u}\right) =0.$

CASE 2) $\left\vert x\right\vert +\left\vert y\right\vert +\left\vert
z\right\vert +\left\vert t\right\vert \leq s.$ In this case we will use (\ref%
{form:goldenYD}) for $u=x\otimes y\otimes z\otimes t$ and hence%
\begin{eqnarray*}
&&\sum\limits_{0\leq i+j\leq s}\left[ \omega _{i}\left( Q\otimes Q\otimes
m\right) \ast \omega _{j}\left( m\otimes Q\otimes Q\right) \right] \left(
u\right) \\
&=&\sum\limits_{0\leq a+b+c\leq s}\left[ \left( \varepsilon \otimes \omega
_{a}\right) \ast \omega _{b}\left( Q\otimes m\otimes Q\right) \ast \left(
\omega _{c}\otimes \varepsilon \right) \right] \left( u\right) .
\end{eqnarray*}%
By minimality of $s$ we obtain%
\begin{eqnarray*}
&&\left[
\begin{array}{c}
\left[ \omega _{0}\left( Q\otimes Q\otimes m\right) \ast \omega _{s}\left(
m\otimes Q\otimes Q\right) \right] \left( u\right) + \\
\left[ \omega _{s}\left( Q\otimes Q\otimes m\right) \ast \omega _{0}\left(
m\otimes Q\otimes Q\right) \right] \left( u\right)%
\end{array}%
\right] \\
&=&\left[
\begin{array}{c}
\left[ \left( \varepsilon \otimes \omega _{0}\right) \ast \omega _{0}\left(
Q\otimes m\otimes Q\right) \ast \left( \omega _{s}\otimes \varepsilon
\right) \right] \left( u\right) + \\
\left[ \left( \varepsilon \otimes \omega _{0}\right) \ast \omega _{s}\left(
Q\otimes m\otimes Q\right) \ast \left( \omega _{0}\otimes \varepsilon
\right) \right] \left( u\right) + \\
\left[ \left( \varepsilon \otimes \omega _{s}\right) \ast \omega _{0}\left(
Q\otimes m\otimes Q\right) \ast \left( \omega _{0}\otimes \varepsilon
\right) \right] \left( u\right)%
\end{array}%
\right] .
\end{eqnarray*}

By (\ref{form:Omega0YD}) the equality above becomes%
\begin{eqnarray}
&&\omega _{s}\left( m\otimes Q\otimes Q\right) \left( u\right) +\omega
_{s}\left( Q\otimes Q\otimes m\right) \left( u\right)  \label{form:OmegaSYD}
\\
&=&\left( \omega _{s}\otimes \varepsilon \right) \left( u\right) +\omega
_{s}\left( Q\otimes m\otimes Q\right) \left( u\right) +\left( \varepsilon
\otimes \omega _{s}\right) \left( u\right) .  \notag
\end{eqnarray}%
We will rewrite this formula in terms of $\overline{\omega }_{s}$. For $%
x=x^{a,l}$ and $y=x^{a^{\prime },l^{\prime }}$, by (\ref{form:ubar1}) we have%
\begin{equation*}
\overline{\omega }_{s}\left( m_{\overline{Q}}\otimes \overline{Q}\otimes
\overline{Q}\right) \left( \overline{u}\right) =\sum_{v}\mu _{a+a^{\prime
},v}^{a,l,a^{\prime },l^{\prime }}\delta _{\left\vert x\right\vert
+\left\vert y\right\vert +\left\vert z\right\vert +\left\vert t\right\vert
,s}\omega \left( x^{a+a^{\prime },v}\otimes z\otimes t\right) .
\end{equation*}%
Moreover%
\begin{eqnarray*}
&&\omega _{s}\left( m\otimes Q\otimes Q\right) \left( u\right) \\
&=&\omega _{s}\left( m\otimes Q\otimes Q\right) \left( x\otimes y\otimes
z\otimes t\right) \\
&=&\omega _{s}\left( xy\otimes z\otimes t\right) \\
&=&\omega _{s}\left( x^{a,l}x^{a^{\prime },l^{\prime }}\otimes z\otimes
t\right) \\
&\overset{(\ref{form:GelYD1})}{=}&\omega _{s}\left( \sum_{w\leq a+a^{\prime
}}\sum_{v}\mu _{w,v}^{a,l,a^{\prime },l^{\prime }}x^{w,v}\otimes z\otimes
t\right) \\
&=&\sum_{w\leq a+a^{\prime }}\sum_{v}\mu _{w,v}^{a,l,a^{\prime },l^{\prime
}}\omega _{s}\left( x^{w,v}\otimes z\otimes t\right) \\
&=&\sum_{w\leq \left\vert x\right\vert +\left\vert y\right\vert }\sum_{v}\mu
_{w,v}^{a,l,a^{\prime },l^{\prime }}\delta _{w+\left\vert z\right\vert
+\left\vert t\right\vert ,s}\omega \left( x^{w,v}\otimes z\otimes t\right) \\
&&\left( \left\vert x\right\vert +\left\vert y\right\vert +\left\vert
z\right\vert +\left\vert t\right\vert \leq s\Rightarrow w+\left\vert
z\right\vert +\left\vert t\right\vert \leq \left\vert x\right\vert
+\left\vert y\right\vert +\left\vert z\right\vert +\left\vert t\right\vert
\leq s\right) \\
&=&\sum_{v}\mu _{a+a^{\prime },v}^{a,l,a^{\prime },l^{\prime }}\delta
_{\left\vert x\right\vert +\left\vert y\right\vert +\left\vert z\right\vert
+\left\vert t\right\vert ,s}\omega \left( x^{\left\vert x\right\vert
+\left\vert y\right\vert ,v}\otimes z\otimes t\right)
\end{eqnarray*}%
so that $\overline{\omega }_{s}\left( m_{\overline{Q}}\otimes \overline{Q}%
\otimes \overline{Q}\right) \left( \overline{u}\right) =\omega _{s}\left(
m\otimes Q\otimes Q\right) \left( u\right) .$ Similarly one proves
\begin{eqnarray*}
\overline{\omega }_{s}\left( \overline{Q}\otimes \overline{Q}\otimes m_{%
\overline{Q}}\right) \left( \overline{u}\right) &=&\omega _{s}\left(
Q\otimes Q\otimes m\right) \left( u\right) , \\
\overline{\omega }_{s}\left( \overline{Q}\otimes m_{\overline{Q}}\otimes
\overline{Q}\right) \left( \overline{u}\right) &=&\omega _{s}\left( Q\otimes
m\otimes Q\right) \left( u\right) .
\end{eqnarray*}%
Since $x,y,z,t$ are elements in the basis we have%
\begin{eqnarray*}
&&\left( \overline{\omega }_{s}\otimes \varepsilon _{\overline{Q}}\right)
\left( \overline{u}\right) \\
&=&\overline{\omega }_{s}\left( \overline{x}\otimes \overline{y}\otimes
\overline{z}\right) \varepsilon _{\overline{Q}}\left( \overline{t}\right) \\
&=&\omega _{s}\left( x\otimes y\otimes z\right) \delta _{\left\vert
t\right\vert ,0}\varepsilon \left( t\right) \text{ (note }t\text{ is of the
form }x^{i,j}\text{ and that }\varepsilon \left( x^{i,j}\right) =\delta
_{i,0}\varepsilon \left( x^{i,j}\right) \text{)} \\
&=&\omega _{s}\left( x\otimes y\otimes z\right) \varepsilon \left( t\right)
\\
&=&\left( \omega _{s}\otimes \varepsilon \right) \left( u\right)
\end{eqnarray*}%
and hence $\left( \overline{\omega }_{s}\otimes \varepsilon _{\overline{Q}%
}\right) \left( \overline{u}\right) =\left( \omega _{s}\otimes \varepsilon
\right) \left( u\right) .$ In a similar way one gets%
\begin{equation*}
\left( \varepsilon _{\overline{Q}}\otimes \overline{\omega }_{s}\right)
\left( \overline{u}\right) =\left( \varepsilon \otimes \omega _{s}\right)
\left( u\right) .
\end{equation*}%
Summing up we get that (\ref{form:OmegaSYD}) rewrites as
\begin{equation*}
\overline{\omega }_{s}\left( m_{\overline{Q}}\otimes \overline{Q}\otimes
\overline{Q}\right) \left( \overline{u}\right) +\overline{\omega }_{s}\left(
\overline{Q}\otimes \overline{Q}\otimes m_{\overline{Q}}\right) \left(
\overline{u}\right) =\left( \varepsilon _{\overline{Q}}\otimes \overline{%
\omega }_{s}\right) \left( \overline{u}\right) +\overline{\omega }_{s}\left(
\overline{Q}\otimes m_{\overline{Q}}\otimes \overline{Q}\right) \left(
\overline{u}\right) +\left( \overline{\omega }_{s}\otimes \varepsilon _{%
\overline{Q}}\right) \left( \overline{u}\right) .
\end{equation*}%
We have so proved that $\overline{\omega }_{s}\in \mathrm{Z}_{{\mathcal{YD}}%
}^{3}\left( \mathrm{gr}Q,\Bbbk \right) .$
\end{invisible}

By assumption $\mathrm{H}_{{\mathcal{YD}}}^{3}\left( \mathrm{gr}Q,\Bbbk
\right) =0$ so that there exists a morphism $\overline{v}:\overline{Q}%
\otimes \overline{Q}\rightarrow \Bbbk $ in ${_{H}^{H}\mathcal{YD}}$ such that%
\begin{equation*}
\overline{\omega }_{s}=\partial ^{2}\overline{v}=\overline{v}\otimes
\varepsilon _{\overline{Q}}-\overline{v}\left( \overline{Q}\otimes m_{%
\overline{Q}}\right) +\overline{v}\left( m_{\overline{Q}}\otimes \overline{Q}%
\right) -\varepsilon _{\overline{Q}}\otimes \overline{v}.
\end{equation*}%
Explicitly, on elements in the basis we get%
\begin{equation*}
\overline{\omega }_{s}\left( \overline{x}\otimes \overline{y}\otimes
\overline{z}\right) =\overline{v}\left( \overline{x}\otimes \overline{y}%
\right) \varepsilon _{\overline{Q}}\left( \overline{z}\right) -\overline{v}%
\left( \overline{x}\otimes \overline{y}\cdot \overline{z}\right) +\overline{v%
}\left( \overline{x}\cdot \overline{y}\otimes \overline{z}\right)
-\varepsilon _{\overline{Q}}\left( \overline{x}\right) \overline{v}\left(
\overline{y}\otimes \overline{z}\right) .
\end{equation*}%
Define $\overline{\zeta }:\overline{Q}\otimes \overline{Q}\rightarrow \Bbbk $
on the basis by setting
\begin{equation*}
\overline{\zeta }\left( \overline{x}\otimes \overline{y}\right) :=\delta
_{\left\vert x\right\vert +\left\vert y\right\vert ,s}\overline{v}\left(
\overline{x}\otimes \overline{y}\right) .
\end{equation*}%
As we have done for $\omega _{t},$ one can check that $\overline{\zeta }$ is
a morphism in ${_{H}^{H}\mathcal{YD}}$.

\begin{invisible}
It is a morphism in ${_{H}^{H}\mathcal{YD}}$ as%
\begin{eqnarray*}
\overline{\zeta }\left( h\left( y^{n,i}\otimes y^{n^{\prime },i^{\prime
}}\right) \right) &=&\overline{\zeta }\left( h_{1}y^{n,i}\otimes
h_{2}y^{n^{\prime },i^{\prime }}\right) \\
&=&\overline{\zeta }\left( \sum_{1\leq w\leq d_{n}}\chi _{w,i}^{n}\left(
h_{1}\right) y^{n,w}\otimes \sum_{1\leq w^{\prime }\leq d_{n^{\prime }}}\chi
_{w^{\prime },i^{\prime }}^{n^{\prime }}\left( h_{2}\right) y^{n^{\prime
},w^{\prime }}\right) \\
&=&\delta _{n+n^{\prime },s}\overline{v}\left( \sum_{1\leq w\leq d_{n}}\chi
_{w,i}^{n}\left( h_{1}\right) y^{n,w}\otimes \sum_{1\leq w^{\prime }\leq
d_{n^{\prime }}}\chi _{w^{\prime },i^{\prime }}^{n^{\prime }}\left(
h_{2}\right) y^{n^{\prime },w^{\prime }}\right) \\
&=&\delta _{n+n^{\prime },s}\overline{v}\left( h_{1}y^{n,i}\otimes
h_{2}y^{n^{\prime },i^{\prime }}\right) \\
&=&\delta _{n+n^{\prime },s}\overline{v}\left( h\left( y^{n,i}\otimes
y^{n^{\prime },i^{\prime }}\right) \right) =\varepsilon _{H}\left( h\right)
\delta _{n+n^{\prime },s}\overline{v}\left( y^{n,i}\otimes y^{n^{\prime
},i^{\prime }}\right) \\
&=&\varepsilon _{H}\left( h\right) \overline{\zeta }\left( y^{n,i}\otimes
y^{n^{\prime },i^{\prime }}\right)
\end{eqnarray*}%
and%
\begin{eqnarray*}
&&\left( y^{n,i}\otimes y^{n^{\prime },i^{\prime }}\right) _{-1}\otimes
\overline{\zeta }\left( \left( y^{n,i}\otimes y^{n^{\prime },i^{\prime
}}\right) _{0}\right) \\
&=&\left( y^{n,i}\right) _{-1}\left( y^{n^{\prime },i^{\prime }}\right)
_{-1}\otimes \overline{\zeta }\left( \left( y^{n,i}\right) _{0}\otimes
\left( y^{n^{\prime },i^{\prime }}\right) _{0}\right) \\
&=&\sum_{1\leq w\leq d_{n}}h_{i,w}^{n}\sum_{1\leq w^{\prime }\leq
d_{n^{\prime }}}h_{i^{\prime },w^{\prime }}^{n^{\prime }}\otimes \overline{%
\zeta }\left( y^{n,w}\otimes y^{n^{\prime },w^{\prime }}\right) \\
&=&\delta _{n+n^{\prime },s}\sum_{1\leq w\leq d_{n}}h_{i,w}^{n}\sum_{1\leq
w^{\prime }\leq d_{n^{\prime }}}h_{i^{\prime },w^{\prime }}^{n^{\prime
}}\otimes \overline{v}\left( y^{n,w}\otimes y^{n^{\prime },w^{\prime
}}\right) \\
&=&\delta _{n+n^{\prime },s}\left( y^{n,i}\otimes y^{n^{\prime },i^{\prime
}}\right) _{-1}\otimes \overline{v}\left( \left( y^{n,i}\otimes y^{n^{\prime
},i^{\prime }}\right) _{0}\right) \\
&=&\delta _{n+n^{\prime },s}\left( \overline{v}\left( y^{n,i}\otimes
y^{n^{\prime },i^{\prime }}\right) \right) _{-1}\otimes \left( \overline{v}%
\left( y^{n,i}\otimes y^{n^{\prime },i^{\prime }}\right) \right) _{0} \\
&=&\delta _{n+n^{\prime },s}1_{H}\otimes \overline{v}\left( y^{n,i}\otimes
y^{n^{\prime },i^{\prime }}\right) \\
&=&1_{H}\otimes \overline{\zeta }\left( y^{n,i}\otimes y^{n^{\prime
},i^{\prime }}\right) .
\end{eqnarray*}
\end{invisible}

Moreover on elements in the basis we get%
\begin{eqnarray*}
&&\left( \partial ^{2}\overline{\zeta }\right) \left( \overline{x}\otimes
\overline{y}\otimes \overline{z}\right) \\
&=&\left( \overline{\zeta }\otimes \varepsilon _{\overline{Q}}\right) \left(
\overline{x}\otimes \overline{y}\otimes \overline{z}\right) -\overline{\zeta
}\left( \overline{Q}\otimes m_{\overline{Q}}\right) \left( \overline{x}%
\otimes \overline{y}\otimes \overline{z}\right) +\overline{\zeta }\left( m_{%
\overline{Q}}\otimes \overline{Q}\right) \left( \overline{x}\otimes
\overline{y}\otimes \overline{z}\right) -\left( \varepsilon _{\overline{Q}%
}\otimes \overline{\zeta }\right) \left( \overline{x}\otimes \overline{y}%
\otimes \overline{z}\right) \\
&=&\overline{\zeta }\left( \overline{x}\otimes \overline{y}\right)
\varepsilon _{\overline{Q}}\left( \overline{z}\right) -\overline{\zeta }%
\left( \overline{x}\otimes \overline{y}\cdot \overline{z}\right) +\overline{%
\zeta }\left( \overline{x}\cdot \overline{y}\otimes \overline{z}\right)
-\varepsilon _{\overline{Q}}\left( \overline{x}\right) \overline{\zeta }%
\left( \overline{y}\otimes \overline{z}\right) .
\end{eqnarray*}

\begin{invisible}
Note that $\overline{y}$ and $\overline{z}$ are of the form $\overline{y}=%
\overline{x^{a,l}}$ and $\overline{z}=\overline{x^{a^{\prime },l^{\prime }}}$
and $)$hence one obtains
\begin{eqnarray*}
&&\overline{\zeta }\left( \overline{x}\otimes \overline{x^{a,l}}\cdot
\overline{x^{a^{\prime },l^{\prime }}}\right) \overset{(\ref{form:GelYD2})}{=%
}\overline{\zeta }\left( \overline{x}\otimes \sum_{v}\mu _{a+a^{\prime
},v}^{a,l,a^{\prime },l^{\prime }}\overline{x^{a+a^{\prime },v}}\right) \\
&=&\delta _{\left\vert x\right\vert +a+a^{\prime },s}\overline{v}\left(
\overline{x}\otimes \sum_{v}\mu _{a+a^{\prime },v}^{a,l,a^{\prime
},l^{\prime }}\overline{x^{a+a^{\prime },v}}\right) \\
&\overset{(\ref{form:GelYD2})}{=}&\delta _{\left\vert x\right\vert
+a+a^{\prime },s}\overline{v}\left( \overline{x}\otimes \overline{x^{a,l}}%
\cdot \overline{x^{a^{\prime },l^{\prime }}}\right)
\end{eqnarray*}
\end{invisible}

By using (\ref{form:GelYD2}), one gets
\begin{equation*}
\overline{\zeta }\left( \overline{x}\otimes \overline{y}\cdot \overline{z}%
\right) =\delta _{\left\vert x\right\vert +\left\vert y\right\vert
+\left\vert z\right\vert ,s}\overline{v}\left( \overline{x}\otimes \overline{%
y}\cdot \overline{z}\right) \qquad \text{and}\qquad \overline{\zeta }\left(
\overline{x}\cdot \overline{y}\otimes \overline{z}\right) =\delta
_{\left\vert x\right\vert +\left\vert y\right\vert +\left\vert z\right\vert
,s}\overline{v}\left( \overline{x}\cdot \overline{y}\otimes \overline{z}%
\right) .
\end{equation*}%
By means of these equalities one gets%
\begin{eqnarray*}
\left( \partial ^{2}\overline{\zeta }\right) \left( \overline{x}\otimes
\overline{y}\otimes \overline{z}\right) &=&\delta _{\left\vert x\right\vert
+\left\vert y\right\vert +\left\vert z\right\vert ,s}\left( \partial ^{2}%
\overline{v}\right) \left( \overline{x}\otimes \overline{y}\otimes \overline{%
z}\right) =\delta _{\left\vert x\right\vert +\left\vert y\right\vert
+\left\vert z\right\vert ,s}\overline{\omega }_{s}\left( \overline{x}\otimes
\overline{y}\otimes \overline{z}\right) \\
&=&\delta _{\left\vert x\right\vert +\left\vert y\right\vert +\left\vert
z\right\vert ,s}\omega _{s}\left( x\otimes y\otimes z\right) =\delta
_{\left\vert x\right\vert +\left\vert y\right\vert +\left\vert z\right\vert
,s}\delta _{\left\vert x\right\vert +\left\vert y\right\vert +\left\vert
z\right\vert ,s}\omega \left( x\otimes y\otimes z\right) \\
&=&\delta _{\left\vert x\right\vert +\left\vert y\right\vert +\left\vert
z\right\vert ,s}\omega \left( x\otimes y\otimes z\right) =\omega _{s}\left(
x\otimes y\otimes z\right) =\overline{\omega }_{s}\left( \overline{x}\otimes
\overline{y}\otimes \overline{z}\right) .
\end{eqnarray*}

\begin{invisible}
\begin{eqnarray*}
&&\left( \partial ^{2}\overline{\zeta }\right) \left( \overline{x}\otimes
\overline{y}\otimes \overline{z}\right) \\
&=&\overline{\zeta }\left( \overline{x}\otimes \overline{y}\right)
\varepsilon _{\overline{Q}}\left( \overline{z}\right) -\overline{\zeta }%
\left( \overline{x}\otimes \overline{y}\cdot \overline{z}\right) +\overline{%
\zeta }\left( \overline{x}\cdot \overline{y}\otimes \overline{z}\right)
-\varepsilon _{\overline{Q}}\left( \overline{x}\right) \overline{\zeta }%
\left( \overline{y}\otimes \overline{z}\right) \\
&=&\left[
\begin{array}{c}
\delta _{\left\vert x\right\vert +\left\vert y\right\vert ,s}\overline{v}%
\left( \overline{x}\otimes \overline{y}\right) \delta _{\left\vert
z\right\vert ,0}\varepsilon _{\overline{Q}}\left( \overline{z}\right)
-\delta _{\left\vert x\right\vert +\left\vert y\right\vert +\left\vert
z\right\vert ,s}\overline{v}\left( \overline{x}\otimes \overline{y}\cdot
\overline{z}\right) \\
+\delta _{\left\vert x\right\vert +\left\vert y\right\vert +\left\vert
z\right\vert ,s}\overline{v}\left( \overline{x}\cdot \overline{y}\otimes
\overline{z}\right) -\delta _{\left\vert x\right\vert ,0}\varepsilon _{%
\overline{Q}}\left( \overline{x}\right) \delta _{\left\vert y\right\vert
+\left\vert z\right\vert ,s}\overline{v}\left( \overline{y}\otimes \overline{%
z}\right)%
\end{array}%
\right] \\
&=&\left[
\begin{array}{c}
\delta _{\left\vert x\right\vert +\left\vert y\right\vert ,s}\delta
_{\left\vert z\right\vert ,0}\left( \overline{v}\otimes \varepsilon _{%
\overline{Q}}\right) \left( \overline{x}\otimes \overline{y}\otimes
\overline{z}\right) -\delta _{\left\vert x\right\vert +\left\vert
y\right\vert +\left\vert z\right\vert ,s}\overline{v}\left( \overline{x}%
\otimes \overline{y}\cdot \overline{z}\right) \\
+\delta _{\left\vert x\right\vert +\left\vert y\right\vert +\left\vert
z\right\vert ,s}\overline{v}\left( \overline{x}\cdot \overline{y}\otimes
\overline{z}\right) -\delta _{\left\vert x\right\vert ,0}\delta _{\left\vert
y\right\vert +\left\vert z\right\vert ,s}\left( \varepsilon _{\overline{Q}%
}\otimes \overline{v}\right) \left( \overline{x}\otimes \overline{y}\otimes
\overline{z}\right)%
\end{array}%
\right] \\
&=&\left[
\begin{array}{c}
\delta _{\left\vert x\right\vert +\left\vert y\right\vert +\left\vert
z\right\vert ,s}\left( \overline{v}\otimes \varepsilon _{\overline{Q}%
}\right) \left( \overline{x}\otimes \overline{y}\otimes \overline{z}\right)
-\delta _{\left\vert x\right\vert +\left\vert y\right\vert +\left\vert
z\right\vert ,s}\overline{v}\left( \overline{x}\otimes \overline{y}\cdot
\overline{z}\right) \\
+\delta _{\left\vert x\right\vert +\left\vert y\right\vert +\left\vert
z\right\vert ,s}\overline{v}\left( \overline{x}\cdot \overline{y}\otimes
\overline{z}\right) -\delta _{\left\vert x\right\vert +\left\vert
y\right\vert +\left\vert z\right\vert ,s}\left( \varepsilon _{\overline{Q}%
}\otimes \overline{v}\right) \left( \overline{x}\otimes \overline{y}\otimes
\overline{z}\right)%
\end{array}%
\right] \\
&=&\delta _{\left\vert x\right\vert +\left\vert y\right\vert +\left\vert
z\right\vert ,s}\left( \partial ^{2}\overline{v}\right) \left( \overline{x}%
\otimes \overline{y}\otimes \overline{z}\right) \\
&=&\delta _{\left\vert x\right\vert +\left\vert y\right\vert +\left\vert
z\right\vert ,s}\overline{\omega }_{s}\left( \overline{x}\otimes \overline{y}%
\otimes \overline{z}\right) \\
&=&\delta _{\left\vert x\right\vert +\left\vert y\right\vert +\left\vert
z\right\vert ,s}\omega _{s}\left( x\otimes y\otimes z\right) =\delta
_{\left\vert x\right\vert +\left\vert y\right\vert +\left\vert z\right\vert
,s}\delta _{\left\vert x\right\vert +\left\vert y\right\vert +\left\vert
z\right\vert ,s}\omega \left( x\otimes y\otimes z\right) =\delta
_{\left\vert x\right\vert +\left\vert y\right\vert +\left\vert z\right\vert
,s}\omega \left( x\otimes y\otimes z\right) \\
&=&\omega _{s}\left( x\otimes y\otimes z\right) =\overline{\omega }%
_{s}\left( \overline{x}\otimes \overline{y}\otimes \overline{z}\right) .
\end{eqnarray*}
\end{invisible}

Therefore $\partial ^{2}\overline{\zeta }=\overline{\omega }_{s}.$ This
means that we can assume that $\overline{v}\left( \overline{x}\otimes
\overline{y}\right) =0$ for $\left\vert x\right\vert +\left\vert
y\right\vert \neq s.$ Equivalently%
\begin{equation}
\overline{v}\left( \overline{x}\otimes \overline{y}\right) =\delta
_{\left\vert x\right\vert +\left\vert y\right\vert ,s}\overline{v}\left(
\overline{x}\otimes \overline{y}\right) \text{ for }x,y\text{ in the basis.}
\label{form:vbarYD}
\end{equation}

\begin{invisible}
\lbrack NOTE that we essentially proved that $\partial ^{2}\overline{v}$
graded implies we can take $\overline{v}$ graded. Is this equivalent to
prove that $\partial ^{2}$ is graded? Does it make sense to write it in
general for a graded algebra $A=\oplus _{i\in \N_0}A_{i}$ instead of $%
\mathrm{gr}Q$???? It this case probably we have to take on ${_{H}^{H}%
\mathcal{YD}}\left( A^{\otimes n},\Bbbk \right) $ the graduation ${_{H}^{H}%
\mathcal{YD}}\left( A^{\otimes n},\Bbbk \right) _{t}:={_{H}^{H}\mathcal{YD}}%
\left( \left( A^{\otimes n}\right) _{t},\Bbbk \right) $ where $\left(
A^{\otimes n}\right) _{t}=\sum\limits_{i_{1}+\cdots
ti_{n}=t}A_{i_{1}}\otimes \cdots \otimes A_{i_{n}}.$]
\end{invisible}

Set
\begin{equation*}
v:=\overline{v}\circ \left( \varphi \otimes \varphi \right) \qquad \text{and}%
\qquad \gamma :=\left( \varepsilon \otimes \varepsilon \right) +v.
\end{equation*}%
In particular, one gets%
\begin{equation}
v\left( x\otimes y\right) =\delta _{\left\vert x\right\vert +\left\vert
y\right\vert ,s}v\left( x\otimes y\right) \text{ for }x,y\text{ in the basis.%
}  \label{form:vYD}
\end{equation}%
Note also that both $v$ and $\gamma $ are morphisms in ${_{H}^{H}\mathcal{YD}%
}$ as they are obtained as composition or sum of morphisms in this category.
Let us check that $\gamma $ is a gauge transformation on $Q$ in ${_{H}^{H}%
\mathcal{YD}}$.

Recall that $x^{0,0}=1_{Q}$ is in the basis. For $x$ in the basis, we have $%
\gamma \left( x\otimes 1_{Q}\right) =\varepsilon \left( x\right) +v\left(
x\otimes 1_{Q}\right) .$ Note that%
\begin{eqnarray*}
0 &=&\delta _{\left\vert x\right\vert ,s}\varepsilon \left( x\right) =\delta
_{\left\vert x\right\vert +\left\vert 1_{Q}\right\vert +\left\vert
1_{Q}\right\vert ,s}\omega \left( x\otimes 1_{Q}\otimes 1_{Q}\right) \\
&=&\omega _{s}\left( x\otimes 1_{Q}\otimes 1_{Q}\right) =\overline{\omega }%
_{s}\left( \overline{x}\otimes \overline{1_{Q}}\otimes \overline{1_{Q}}%
\right) \\
&=&\overline{v}\left( \overline{x}\otimes \overline{1_{Q}}\right)
\varepsilon _{\overline{Q}}\left( \overline{1_{Q}}\right) -\overline{v}%
\left( \overline{x}\otimes \overline{1_{Q}}\cdot \overline{1_{Q}}\right) +%
\overline{v}\left( \overline{x}\cdot \overline{1_{Q}}\otimes \overline{1_{Q}}%
\right) -\varepsilon _{\overline{Q}}\left( \overline{x}\right) \overline{v}%
\left( \overline{1_{Q}}\otimes \overline{1_{Q}}\right) \\
&\overset{(\ref{form:vbarYD})}{=}&\overline{v}\left( \overline{x}\otimes
\overline{1_{Q}}\right) -\overline{v}\left( \overline{x}\otimes \overline{%
1_{Q}}\right) +\overline{v}\left( \overline{x}\otimes \overline{1_{Q}}%
\right) -\varepsilon _{\overline{Q}}\left( \overline{x}\right) \delta
_{\left\vert 1_{Q}\right\vert +\left\vert 1_{Q}\right\vert ,s}\overline{v}%
\left( \overline{1_{Q}}\otimes \overline{1_{Q}}\right) \\
&=&v\left( x\otimes 1_{Q}\right)
\end{eqnarray*}%
so that $v\left( x\otimes 1_{Q}\right) =0$ and hence $\gamma \left( x\otimes
1_{Q}\right) =\varepsilon \left( x\right) +v\left( x\otimes 1_{Q}\right)
=\varepsilon \left( x\right) .$ Similarly one proves $\gamma \left(
1_{Q}\otimes x\right) =\varepsilon \left( x\right) .$ Hence $\gamma $ is
unital. Note that the coalgebra $C=Q\otimes Q$ is connected as $Q$ is. Thus,
in order to prove that $\gamma :Q\otimes Q\rightarrow \Bbbk $ is convolution
invertible it suffices to check (see \cite[Lemma 5.2.10]{Mo}) that $\gamma
_{\mid \Bbbk 1_{Q}\otimes \Bbbk 1_{Q}}$ is convolution invertible. But for $%
k,k^{\prime }\in \Bbbk $ we have
\begin{equation*}
\gamma \left( k1_{Q}\otimes k^{\prime }1_{Q}\right) =kk^{\prime }\gamma
\left( 1_{Q}\otimes 1_{Q}\right) =kk^{\prime }\varepsilon \left(
1_{Q}\right) =kk^{\prime }=\left( \varepsilon \otimes \varepsilon \right)
\left( k1_{Q}\otimes k^{\prime }1_{Q}\right)
\end{equation*}%
Hence $\gamma _{\mid \Bbbk 1_{Q}\otimes \Bbbk 1_{Q}}=\left( \varepsilon
\otimes \varepsilon \right) _{\mid \Bbbk 1_{Q}\otimes \Bbbk 1_{Q}}$ which is
convolution invertible. Thus there is a $\Bbbk $-linear map $\gamma
^{-1}:Q\otimes Q\rightarrow \Bbbk $ and such that%
\begin{equation*}
\gamma \ast \gamma ^{-1}=\varepsilon \otimes \varepsilon =\gamma ^{-1}\ast
\gamma .
\end{equation*}%
Note that, by Lemma \ref{lem:InvYD}, $\gamma \in {_{H}^{H}\mathcal{YD}}$
implies $\gamma ^{-1}\in {_{H}^{H}\mathcal{YD}}$.

Therefore $\gamma $ is a gauge transformation in ${_{H}^{H}\mathcal{YD}}$.
By Proposition \ref{pro:deformYD}, $Q^{\gamma }$ is a coquasi-bialgebra in ${%
_{H}^{H}\mathcal{YD}}$. By Proposition \ref{pro:grgaugeYD}, we have that $\mathrm{%
gr}Q^{\gamma }$ and $\mathrm{gr}Q$ coincide as bialgebras in ${_{H}^{H}%
\mathcal{YD}}$. Hence $\mathrm{H}_{{\mathcal{YD}}}^{3}\left( \mathrm{gr}%
Q^{\gamma },\Bbbk \right) =\mathrm{H}_{{\mathcal{YD}}}^{3}\left( \mathrm{gr}%
Q,\Bbbk \right) =0.$ Therefore $Q^{\gamma }$ fulfills the same requirement
of $Q$ as in the statement. Let us check that $\left( \omega ^{\gamma
}\right) _{t}=0$ for $1\leq t\leq s$ (this will complete the proof by an
induction process as $Q$ is finite-dimensional).

Note that the definition of $\gamma $ and (\ref{form:vYD}) imply%
\begin{equation}
\gamma \left( x\otimes y\right) =\delta _{\left\vert x\right\vert
+\left\vert y\right\vert ,0}\gamma \left( x\otimes y\right) +\delta
_{\left\vert x\right\vert +\left\vert y\right\vert ,s}\gamma \left( x\otimes
y\right) \text{ for }x,y\text{ in the basis.}  \label{form:gamma1YD}
\end{equation}

Let $C^{2}=Q\otimes Q$ and let $C_{\left( n\right) }^{2}=\sum_{i+j\leq
n}Q_{i}\otimes Q_{j}$. For $u\in C_{\left( 2s-1\right) }^{2}$ we have%
\begin{equation*}
\left[ \gamma \ast \left( \left( \varepsilon \otimes \varepsilon \right)
-v\right) \right] \left( u\right) =\left( \varepsilon \otimes \varepsilon
\right) \left( u\right) -v\left( u\right) +v\left( u\right) -v\left(
u_{1}\right) v\left( u_{2}\right) \overset{(\ref{form:vYD})}{=}\left(
\varepsilon \otimes \varepsilon \right) \left( u\right) .
\end{equation*}

\begin{invisible}
Here is the computation%
\begin{eqnarray*}
&&\left[ \gamma \ast \left( \left( \varepsilon \otimes \varepsilon \right)
-v\right) \right] \left( u\right) \\
&=&\left[ \left( \left( \varepsilon \otimes \varepsilon \right) +v\right)
\ast \left( \left( \varepsilon \otimes \varepsilon \right) -v\right) \right]
\left( u\right) \\
&=&\left( \left( \varepsilon \otimes \varepsilon \right) +v\right) \left(
u_{1}\right) \cdot \left( \left( \varepsilon \otimes \varepsilon \right)
-v\right) \left( u_{2}\right) \\
&=&\left[
\begin{array}{c}
\left( \varepsilon \otimes \varepsilon \right) \left( u_{1}\right) \left(
\varepsilon \otimes \varepsilon \right) \left( u_{2}\right) + \\
\left( \varepsilon \otimes \varepsilon \right) \left( u_{1}\right) \left(
-v\right) \left( u_{2}\right) + \\
\left( v\right) \left( u_{1}\right) \left( \varepsilon \otimes \varepsilon
\right) \left( u_{2}\right) + \\
v\left( u_{1}\right) \left( -v\right) \left( u_{2}\right)%
\end{array}%
\right] \\
&=&\left( \varepsilon \otimes \varepsilon \right) \left( u\right) -v\left(
u\right) +v\left( u\right) -v\left( u_{1}\right) v\left( u_{2}\right) \\
&=&\left( \varepsilon \otimes \varepsilon \right) \left( u\right) -v\left(
u_{1}\right) v\left( u_{2}\right) \\
&\overset{(\ref{form:vYD})}{=}&\left( \varepsilon \otimes \varepsilon
\right) \left( u\right)
\end{eqnarray*}
\end{invisible}

Therefore $\left[ \gamma \ast \left( \left( \varepsilon \otimes \varepsilon
\right) -v\right) \right] _{|C_{\left( 2s-1\right) }^{2}}=\left( \varepsilon
\otimes \varepsilon \right) _{|C_{\left( 2s-1\right) }^{2}}.$ By uniqueness
of the convolution inverse, we deduce
\begin{equation}
\gamma ^{-1}\left( u\right) =\left( \varepsilon \otimes \varepsilon \right)
\left( u\right) -v\left( u\right) ,\text{ for }u\in C_{\left( 2s-1\right)
}^{2}.  \label{form:gamma-1}
\end{equation}

\begin{invisible}
Explicitly, for $u\in C_{\left( 2s-1\right) }^{2}$%
\begin{eqnarray*}
\gamma ^{-1}\left( u\right) &=&\gamma ^{-1}\left( u_{1}\right) \left(
\varepsilon \otimes \varepsilon \right) \left( u_{2}\right) \text{ (use that
}u_{2}\in C_{\left( 2s-1\right) }^{2}\text{)} \\
&=&\gamma ^{-1}\left( u_{1}\right) \left[ \gamma \ast \left( \left(
\varepsilon \otimes \varepsilon \right) -v\right) \right] \left( u_{2}\right)
\\
&=&\left[ \gamma ^{-1}\ast \gamma \ast \left( \left( \varepsilon \otimes
\varepsilon \right) -v\right) \right] \left( u\right) \\
&=&\left( \left( \varepsilon \otimes \varepsilon \right) -v\right) \left(
u\right)
\end{eqnarray*}
\end{invisible}

Let $x,y,z$ be in the basis. Set $\overline{u}:=\overline{x}\otimes
\overline{y}\otimes \overline{z}$ and $u:=x\otimes y\otimes z.$ We compute%
\begin{eqnarray*}
\left( \omega ^{\gamma }\right) _{s}\left( u\right) &=&\delta _{\left\vert
x\right\vert +\left\vert y\right\vert +\left\vert z\right\vert ,s}\omega
^{\gamma }\left( u\right) \\
&=&\delta _{\left\vert x\right\vert +\left\vert y\right\vert +\left\vert
z\right\vert ,s}\left[ \left( \varepsilon \otimes \gamma \right) \ast \gamma
\left( Q\otimes m\right) \ast \omega \ast \gamma ^{-1}\left( m\otimes
Q\right) \ast \left( \gamma ^{-1}\otimes \varepsilon \right) \right] \left(
u\right) \\
&=&\delta _{\left\vert x\right\vert +\left\vert y\right\vert +\left\vert
z\right\vert ,s}\left[ \left( \varepsilon \otimes \gamma \right) \ast \gamma
\left( Q\otimes m\right) \ast \left( \omega _{0}+\omega _{s}\right) \ast
\gamma ^{-1}\left( m\otimes Q\right) \ast \left( \gamma ^{-1}\otimes
\varepsilon \right) \right] \left( u\right) \\
&\overset{(\ref{form:Omega0YD})}{=}&\delta _{\left\vert x\right\vert
+\left\vert y\right\vert +\left\vert z\right\vert ,s}\left[
\begin{array}{c}
\left( \varepsilon \otimes \gamma \right) \ast \gamma \left( Q\otimes
m\right) \ast \gamma ^{-1}\left( m\otimes Q\right) \ast \left( \gamma
^{-1}\otimes \varepsilon \right) + \\
\left( \varepsilon \otimes \gamma \right) \ast \gamma \left( Q\otimes
m\right) \ast \omega _{s}\ast \gamma ^{-1}\left( m\otimes Q\right) \ast
\left( \gamma ^{-1}\otimes \varepsilon \right)%
\end{array}%
\right] \left( u\right) \\
&=&\left[
\begin{array}{c}
\delta _{\left\vert x\right\vert +\left\vert y\right\vert +\left\vert
z\right\vert ,s}\left( \varepsilon \otimes \gamma \right) \left(
u_{1}\right) \cdot \gamma \left( Q\otimes m\right) \left( u_{2}\right) \cdot
\gamma ^{-1}\left( m\otimes Q\right) \left( u_{3}\right) \cdot \left( \gamma
^{-1}\otimes \varepsilon \right) \left( u_{4}\right) + \\
\delta _{\left\vert x\right\vert +\left\vert y\right\vert +\left\vert
z\right\vert ,s}\left( \varepsilon \otimes \gamma \right) \left(
u_{1}\right) \cdot \gamma \left( Q\otimes m\right) \left( u_{2}\right) \cdot
\omega _{s}\left( u_{3}\right) \cdot \gamma ^{-1}\left( m\otimes Q\right)
\left( u_{4}\right) \cdot \left( \gamma ^{-1}\otimes \varepsilon \right)
\left( u_{5}\right)%
\end{array}%
\right] .
\end{eqnarray*}

\begin{invisible}
Here is the full computation:
\begin{eqnarray*}
&&\left( \omega ^{\gamma }\right) _{s}\left( u\right) =\delta _{\left\vert
x\right\vert +\left\vert y\right\vert +\left\vert z\right\vert ,s}\omega
^{\gamma }\left( u\right) \\
&=&\delta _{\left\vert x\right\vert +\left\vert y\right\vert +\left\vert
z\right\vert ,s}\left[ \left( \varepsilon \otimes \gamma \right) \ast \gamma
\left( Q\otimes m\right) \ast \omega \ast \gamma ^{-1}\left( m\otimes
Q\right) \ast \left( \gamma ^{-1}\otimes \varepsilon \right) \right] \left(
u\right) \\
&=&\delta _{\left\vert x\right\vert +\left\vert y\right\vert +\left\vert
z\right\vert ,s}\left[ \left( \varepsilon \otimes \gamma \right) \ast \gamma
\left( Q\otimes m\right) \ast \left( \sum\limits_{t\in \N_0}\omega
_{t}\right) \ast \gamma ^{-1}\left( m\otimes Q\right) \ast \left( \gamma
^{-1}\otimes \varepsilon \right) \right] \left( u\right) \\
&=&\delta _{\left\vert x\right\vert +\left\vert y\right\vert +\left\vert
z\right\vert ,s}\left[ \left( \varepsilon \otimes \gamma \right) \ast \gamma
\left( Q\otimes m\right) \ast \left( \omega _{0}+\omega _{s}\right) \ast
\gamma ^{-1}\left( m\otimes Q\right) \ast \left( \gamma ^{-1}\otimes
\varepsilon \right) \right] \left( u\right) \\
&=&\delta _{\left\vert x\right\vert +\left\vert y\right\vert +\left\vert
z\right\vert ,s}\left[
\begin{array}{c}
\left( \varepsilon \otimes \gamma \right) \ast \gamma \left( Q\otimes
m\right) \ast \omega _{0}\ast \gamma ^{-1}\left( m\otimes Q\right) \ast
\left( \gamma ^{-1}\otimes \varepsilon \right) + \\
\left( \varepsilon \otimes \gamma \right) \ast \gamma \left( Q\otimes
m\right) \ast \omega _{s}\ast \gamma ^{-1}\left( m\otimes Q\right) \ast
\left( \gamma ^{-1}\otimes \varepsilon \right)%
\end{array}%
\right] \left( u\right) \\
&\overset{(\ref{form:Omega0YD})}{=}&\delta _{\left\vert x\right\vert
+\left\vert y\right\vert +\left\vert z\right\vert ,s}\left[
\begin{array}{c}
\left( \varepsilon \otimes \gamma \right) \ast \gamma \left( Q\otimes
m\right) \ast \gamma ^{-1}\left( m\otimes Q\right) \ast \left( \gamma
^{-1}\otimes \varepsilon \right) + \\
\left( \varepsilon \otimes \gamma \right) \ast \gamma \left( Q\otimes
m\right) \ast \omega _{s}\ast \gamma ^{-1}\left( m\otimes Q\right) \ast
\left( \gamma ^{-1}\otimes \varepsilon \right)%
\end{array}%
\right] \left( u\right) \\
&=&\left[
\begin{array}{c}
\delta _{\left\vert x\right\vert +\left\vert y\right\vert +\left\vert
z\right\vert ,s}\left[ \left( \varepsilon \otimes \gamma \right) \ast \gamma
\left( Q\otimes m\right) \ast \gamma ^{-1}\left( m\otimes Q\right) \ast
\left( \gamma ^{-1}\otimes \varepsilon \right) \right] \left( u\right) + \\
\delta _{\left\vert x\right\vert +\left\vert y\right\vert +\left\vert
z\right\vert ,s}\left[ \left( \varepsilon \otimes \gamma \right) \ast \gamma
\left( Q\otimes m\right) \ast \omega _{s}\ast \gamma ^{-1}\left( m\otimes
Q\right) \ast \left( \gamma ^{-1}\otimes \varepsilon \right) \right] \left(
u\right)%
\end{array}%
\right] \\
&=&\left[
\begin{array}{c}
\delta _{\left\vert x\right\vert +\left\vert y\right\vert +\left\vert
z\right\vert ,s}\left( \varepsilon \otimes \gamma \right) \left(
u_{1}\right) \cdot \gamma \left( Q\otimes m\right) \left( u_{2}\right) \cdot
\gamma ^{-1}\left( m\otimes Q\right) \left( u_{3}\right) \cdot \left( \gamma
^{-1}\otimes \varepsilon \right) \left( u_{4}\right) + \\
\delta _{\left\vert x\right\vert +\left\vert y\right\vert +\left\vert
z\right\vert ,s}\left( \varepsilon \otimes \gamma \right) \left(
u_{1}\right) \cdot \gamma \left( Q\otimes m\right) \left( u_{2}\right) \cdot
\omega _{s}\left( u_{3}\right) \cdot \gamma ^{-1}\left( m\otimes Q\right)
\left( u_{4}\right) \cdot \left( \gamma ^{-1}\otimes \varepsilon \right)
\left( u_{5}\right)%
\end{array}%
\right] .
\end{eqnarray*}
\end{invisible}

Now, all terms appearing in the last two lines, excepted $\omega _{s}$,
vanish out of degrees $0$ and $s$ and coincide with $\varepsilon \otimes
\varepsilon \otimes \varepsilon $ on degree $0.$ On the other hand $\omega
_{s}$ vanishes out of $s$. Since $\gamma :=\left( \varepsilon \otimes
\varepsilon \right) +v$ and in view of (\ref{form:gamma-1}), the term $%
\delta _{\left\vert x\right\vert +\left\vert y\right\vert +\left\vert
z\right\vert ,s}$ forces the following simplification%
\begin{equation*}
\left( \omega ^{\gamma }\right) _{s}\left( u\right) =\left[
\begin{array}{c}
\delta _{\left\vert x\right\vert +\left\vert y\right\vert +\left\vert
z\right\vert ,s}\left[ \left( \varepsilon \otimes v\right) \left( u\right)
+v\left( Q\otimes m\right) \left( u\right) -v\left( m\otimes Q\right) \left(
u\right) -\left( v\otimes \varepsilon \right) \left( u\right) \right] + \\
+\delta _{\left\vert x\right\vert +\left\vert y\right\vert +\left\vert
z\right\vert ,s}\omega _{s}\left( u\right)%
\end{array}%
\right] .
\end{equation*}

\begin{invisible}
\begin{eqnarray*}
&&\left( \omega ^{\gamma }\right) _{s}\left( u\right) \\
&=&\left[
\begin{array}{c}
\delta _{\left\vert x\right\vert +\left\vert y\right\vert +\left\vert
z\right\vert ,s}\left( \varepsilon \otimes \gamma \right) \left(
u_{1}\right) \cdot \gamma \left( Q\otimes m\right) \left( u_{2}\right) \cdot
\gamma ^{-1}\left( m\otimes Q\right) \left( u_{3}\right) \cdot \left( \gamma
^{-1}\otimes \varepsilon \right) \left( u_{4}\right) + \\
\delta _{\left\vert x\right\vert +\left\vert y\right\vert +\left\vert
z\right\vert ,s}\left( \varepsilon \otimes \gamma \right) \left(
u_{1}\right) \cdot \gamma \left( Q\otimes m\right) \left( u_{2}\right) \cdot
\omega _{s}\left( u_{3}\right) \cdot \gamma ^{-1}\left( m\otimes Q\right)
\left( u_{4}\right) \cdot \left( \gamma ^{-1}\otimes \varepsilon \right)
\left( u_{5}\right)%
\end{array}%
\right] \\
&=&\left[
\begin{array}{c}
\delta _{\left\vert x\right\vert +\left\vert y\right\vert +\left\vert
z\right\vert ,s}\left( \varepsilon \otimes \left( \left( \varepsilon \otimes
\varepsilon \right) +v\right) \right) \left( u_{1}\right) \cdot \gamma
\left( Q\otimes m\right) \left( u_{2}\right) \cdot \gamma ^{-1}\left(
m\otimes Q\right) \left( u_{3}\right) \cdot \left( \gamma ^{-1}\otimes
\varepsilon \right) \left( u_{4}\right) + \\
\delta _{\left\vert x\right\vert +\left\vert y\right\vert +\left\vert
z\right\vert ,s}\left( \varepsilon \otimes \left( \left( \varepsilon \otimes
\varepsilon \right) +v\right) \right) \left( u_{1}\right) \cdot \gamma
\left( Q\otimes m\right) \left( u_{2}\right) \cdot \omega _{s}\left(
u_{3}\right) \cdot \gamma ^{-1}\left( m\otimes Q\right) \left( u_{4}\right)
\cdot \left( \gamma ^{-1}\otimes \varepsilon \right) \left( u_{5}\right)%
\end{array}%
\right] \\
&=&\left[
\begin{array}{c}
\delta _{\left\vert x\right\vert +\left\vert y\right\vert +\left\vert
z\right\vert ,s}\left( \varepsilon \otimes \varepsilon \otimes \varepsilon
\right) \left( u_{1}\right) \cdot \gamma \left( Q\otimes m\right) \left(
u_{2}\right) \cdot \gamma ^{-1}\left( m\otimes Q\right) \left( u_{3}\right)
\cdot \left( \gamma ^{-1}\otimes \varepsilon \right) \left( u_{4}\right) +
\\
\delta _{\left\vert x\right\vert +\left\vert y\right\vert +\left\vert
z\right\vert ,s}\left( \varepsilon \otimes v\right) \left( u_{1}\right)
\cdot \gamma \left( Q\otimes m\right) \left( u_{2}\right) \cdot \gamma
^{-1}\left( m\otimes Q\right) \left( u_{3}\right) \cdot \left( \gamma
^{-1}\otimes \varepsilon \right) \left( u_{4}\right) + \\
\delta _{\left\vert x\right\vert +\left\vert y\right\vert +\left\vert
z\right\vert ,s}\left( \varepsilon \otimes \varepsilon \otimes \varepsilon
\right) \left( u_{1}\right) \cdot \gamma \left( Q\otimes m\right) \left(
u_{2}\right) \cdot \omega _{s}\left( u_{3}\right) \cdot \gamma ^{-1}\left(
m\otimes Q\right) \left( u_{4}\right) \cdot \left( \gamma ^{-1}\otimes
\varepsilon \right) \left( u_{5}\right) + \\
\delta _{\left\vert x\right\vert +\left\vert y\right\vert +\left\vert
z\right\vert ,s}\left( \varepsilon \otimes v\right) \left( u_{1}\right)
\cdot \gamma \left( Q\otimes m\right) \left( u_{2}\right) \cdot \omega
_{s}\left( u_{3}\right) \cdot \gamma ^{-1}\left( m\otimes Q\right) \left(
u_{4}\right) \cdot \left( \gamma ^{-1}\otimes \varepsilon \right) \left(
u_{5}\right)%
\end{array}%
\right] \\
&\overset{\left\vert u_{1}\right\vert +\left\vert u_{3}\right\vert \leq s}{=}%
&\left[
\begin{array}{c}
\delta _{\left\vert x\right\vert +\left\vert y\right\vert +\left\vert
z\right\vert ,s}\left( \varepsilon \otimes \varepsilon \otimes \varepsilon
\right) \left( u_{1}\right) \cdot \gamma \left( Q\otimes m\right) \left(
u_{2}\right) \cdot \gamma ^{-1}\left( m\otimes Q\right) \left( u_{3}\right)
\cdot \left( \gamma ^{-1}\otimes \varepsilon \right) \left( u_{4}\right) +
\\
\delta _{\left\vert x\right\vert +\left\vert y\right\vert +\left\vert
z\right\vert ,s}\left( \varepsilon \otimes v\right) \left( u_{1}\right)
\cdot \gamma \left( Q\otimes m\right) \left( u_{2}\right) \cdot \gamma
^{-1}\left( m\otimes Q\right) \left( u_{3}\right) \cdot \left( \gamma
^{-1}\otimes \varepsilon \right) \left( u_{4}\right) + \\
\delta _{\left\vert x\right\vert +\left\vert y\right\vert +\left\vert
z\right\vert ,s}\left( \varepsilon \otimes \varepsilon \otimes \varepsilon
\right) \left( u_{1}\right) \cdot \gamma \left( Q\otimes m\right) \left(
u_{2}\right) \cdot \omega _{s}\left( u_{3}\right) \cdot \gamma ^{-1}\left(
m\otimes Q\right) \left( u_{4}\right) \cdot \left( \gamma ^{-1}\otimes
\varepsilon \right) \left( u_{5}\right) +0%
\end{array}%
\right] \\
&=&\left[
\begin{array}{c}
\delta _{\left\vert x\right\vert +\left\vert y\right\vert +\left\vert
z\right\vert ,s}\gamma \left( Q\otimes m\right) \left( u_{1}\right) \cdot
\gamma ^{-1}\left( m\otimes Q\right) \left( u_{2}\right) \cdot \left( \gamma
^{-1}\otimes \varepsilon \right) \left( u_{3}\right) + \\
\delta _{\left\vert x\right\vert +\left\vert y\right\vert +\left\vert
z\right\vert ,s}\left( \varepsilon \otimes v\right) \left( u_{1}\right)
\cdot \gamma \left( Q\otimes m\right) \left( u_{2}\right) \cdot \gamma
^{-1}\left( m\otimes Q\right) \left( u_{3}\right) \cdot \left( \gamma
^{-1}\otimes \varepsilon \right) \left( u_{4}\right) + \\
\delta _{\left\vert x\right\vert +\left\vert y\right\vert +\left\vert
z\right\vert ,s}\gamma \left( Q\otimes m\right) \left( u_{1}\right) \cdot
\omega _{s}\left( u_{2}\right) \cdot \gamma ^{-1}\left( m\otimes Q\right)
\left( u_{3}\right) \cdot \left( \gamma ^{-1}\otimes \varepsilon \right)
\left( u_{4}\right) +%
\end{array}%
\right] \\
&=&\left[
\begin{array}{c}
\delta _{\left\vert x\right\vert +\left\vert y\right\vert +\left\vert
z\right\vert ,s}\left( \left( \varepsilon \otimes \varepsilon \right)
+v\right) \left( Q\otimes m\right) \left( u_{1}\right) \cdot \gamma
^{-1}\left( m\otimes Q\right) \left( u_{2}\right) \cdot \left( \gamma
^{-1}\otimes \varepsilon \right) \left( u_{3}\right) + \\
\delta _{\left\vert x\right\vert +\left\vert y\right\vert +\left\vert
z\right\vert ,s}\left( \varepsilon \otimes v\right) \left( u_{1}\right)
\cdot \left( \left( \varepsilon \otimes \varepsilon \right) +v\right) \left(
Q\otimes m\right) \left( u_{2}\right) \cdot \gamma ^{-1}\left( m\otimes
Q\right) \left( u_{3}\right) \cdot \left( \gamma ^{-1}\otimes \varepsilon
\right) \left( u_{4}\right) + \\
\delta _{\left\vert x\right\vert +\left\vert y\right\vert +\left\vert
z\right\vert ,s}\left( \left( \varepsilon \otimes \varepsilon \right)
+v\right) \left( Q\otimes m\right) \left( u_{1}\right) \cdot \omega
_{s}\left( u_{2}\right) \cdot \gamma ^{-1}\left( m\otimes Q\right) \left(
u_{3}\right) \cdot \left( \gamma ^{-1}\otimes \varepsilon \right) \left(
u_{4}\right) +%
\end{array}%
\right] \\
&=&\left[
\begin{array}{c}
\delta _{\left\vert x\right\vert +\left\vert y\right\vert +\left\vert
z\right\vert ,s}\left( \varepsilon \otimes \varepsilon \right) \left(
Q\otimes m\right) \left( u_{1}\right) \cdot \gamma ^{-1}\left( m\otimes
Q\right) \left( u_{2}\right) \cdot \left( \gamma ^{-1}\otimes \varepsilon
\right) \left( u_{3}\right) + \\
\delta _{\left\vert x\right\vert +\left\vert y\right\vert +\left\vert
z\right\vert ,s}v\left( Q\otimes m\right) \left( u_{1}\right) \cdot \gamma
^{-1}\left( m\otimes Q\right) \left( u_{2}\right) \cdot \left( \gamma
^{-1}\otimes \varepsilon \right) \left( u_{3}\right) + \\
\delta _{\left\vert x\right\vert +\left\vert y\right\vert +\left\vert
z\right\vert ,s}\left( \varepsilon \otimes v\right) \left( u_{1}\right)
\cdot \left( \varepsilon \otimes \varepsilon \right) \left( Q\otimes
m\right) \left( u_{2}\right) \cdot \gamma ^{-1}\left( m\otimes Q\right)
\left( u_{3}\right) \cdot \left( \gamma ^{-1}\otimes \varepsilon \right)
\left( u_{4}\right) + \\
\delta _{\left\vert x\right\vert +\left\vert y\right\vert +\left\vert
z\right\vert ,s}\left( \varepsilon \otimes v\right) \left( u_{1}\right)
\cdot v\left( Q\otimes m\right) \left( u_{2}\right) \cdot \gamma ^{-1}\left(
m\otimes Q\right) \left( u_{3}\right) \cdot \left( \gamma ^{-1}\otimes
\varepsilon \right) \left( u_{4}\right) + \\
\delta _{\left\vert x\right\vert +\left\vert y\right\vert +\left\vert
z\right\vert ,s}\left( \varepsilon \otimes \varepsilon \right) \left(
Q\otimes m\right) \left( u_{1}\right) \cdot \omega _{s}\left( u_{2}\right)
\cdot \gamma ^{-1}\left( m\otimes Q\right) \left( u_{3}\right) \cdot \left(
\gamma ^{-1}\otimes \varepsilon \right) \left( u_{4}\right) + \\
\delta _{\left\vert x\right\vert +\left\vert y\right\vert +\left\vert
z\right\vert ,s}v\left( Q\otimes m\right) \left( u_{1}\right) \cdot \omega
_{s}\left( u_{2}\right) \cdot \gamma ^{-1}\left( m\otimes Q\right) \left(
u_{3}\right) \cdot \left( \gamma ^{-1}\otimes \varepsilon \right) \left(
u_{4}\right)%
\end{array}%
\right] \\
&\overset{\left\vert u_{1}\right\vert +\left\vert u_{2}\right\vert \leq s}{=}%
&\left[
\begin{array}{c}
\delta _{\left\vert x\right\vert +\left\vert y\right\vert +\left\vert
z\right\vert ,s}\gamma ^{-1}\left( m\otimes Q\right) \left( u_{1}\right)
\cdot \left( \gamma ^{-1}\otimes \varepsilon \right) \left( u_{2}\right) +
\\
\delta _{\left\vert x\right\vert +\left\vert y\right\vert +\left\vert
z\right\vert ,s}v\left( Q\otimes m\right) \left( u_{1}\right) \cdot \gamma
^{-1}\left( m\otimes Q\right) \left( u_{2}\right) \cdot \left( \gamma
^{-1}\otimes \varepsilon \right) \left( u_{3}\right) + \\
\delta _{\left\vert x\right\vert +\left\vert y\right\vert +\left\vert
z\right\vert ,s}\left( \varepsilon \otimes v\right) \left( u_{1}\right)
\cdot \gamma ^{-1}\left( m\otimes Q\right) \left( u_{2}\right) \cdot \left(
\gamma ^{-1}\otimes \varepsilon \right) \left( u_{3}\right) +0+ \\
\delta _{\left\vert x\right\vert +\left\vert y\right\vert +\left\vert
z\right\vert ,s}\omega _{s}\left( u_{1}\right) \cdot \gamma ^{-1}\left(
m\otimes Q\right) \left( u_{2}\right) \cdot \left( \gamma ^{-1}\otimes
\varepsilon \right) \left( u_{3}\right) +0%
\end{array}%
\right] \\
&\overset{(\ref{form:gamma-1})}{=}&\left[
\begin{array}{c}
\delta _{\left\vert x\right\vert +\left\vert y\right\vert +\left\vert
z\right\vert ,s}\left( \left( \varepsilon \otimes \varepsilon \right)
-v\right) \left( m\otimes Q\right) \left( u_{1}\right) \cdot \left( \gamma
^{-1}\otimes \varepsilon \right) \left( u_{2}\right) + \\
\delta _{\left\vert x\right\vert +\left\vert y\right\vert +\left\vert
z\right\vert ,s}v\left( Q\otimes m\right) \left( u_{1}\right) \cdot \left(
\left( \varepsilon \otimes \varepsilon \right) -v\right) \left( m\otimes
Q\right) \left( u_{2}\right) \cdot \left( \gamma ^{-1}\otimes \varepsilon
\right) \left( u_{3}\right) + \\
\delta _{\left\vert x\right\vert +\left\vert y\right\vert +\left\vert
z\right\vert ,s}\left( \varepsilon \otimes v\right) \left( u_{1}\right)
\cdot \left( \left( \varepsilon \otimes \varepsilon \right) -v\right) \left(
m\otimes Q\right) \left( u_{2}\right) \cdot \left( \gamma ^{-1}\otimes
\varepsilon \right) \left( u_{3}\right) + \\
\delta _{\left\vert x\right\vert +\left\vert y\right\vert +\left\vert
z\right\vert ,s}\omega _{s}\left( u_{1}\right) \cdot \left( \left(
\varepsilon \otimes \varepsilon \right) -v\right) \left( m\otimes Q\right)
\left( u_{2}\right) \cdot \left( \gamma ^{-1}\otimes \varepsilon \right)
\left( u_{3}\right)%
\end{array}%
\right] \\
&=&\left[
\begin{array}{c}
\delta _{\left\vert x\right\vert +\left\vert y\right\vert +\left\vert
z\right\vert ,s}\left( \varepsilon \otimes \varepsilon \right) \left(
m\otimes Q\right) \left( u_{1}\right) \cdot \left( \gamma ^{-1}\otimes
\varepsilon \right) \left( u_{2}\right) + \\
-\delta _{\left\vert x\right\vert +\left\vert y\right\vert +\left\vert
z\right\vert ,s}v\left( m\otimes Q\right) \left( u_{1}\right) \cdot \left(
\gamma ^{-1}\otimes \varepsilon \right) \left( u_{2}\right) + \\
\delta _{\left\vert x\right\vert +\left\vert y\right\vert +\left\vert
z\right\vert ,s}v\left( Q\otimes m\right) \left( u_{1}\right) \cdot \left(
\varepsilon \otimes \varepsilon \right) \left( m\otimes Q\right) \left(
u_{2}\right) \cdot \left( \gamma ^{-1}\otimes \varepsilon \right) \left(
u_{3}\right) + \\
-\delta _{\left\vert x\right\vert +\left\vert y\right\vert +\left\vert
z\right\vert ,s}v\left( Q\otimes m\right) \left( u_{1}\right) \cdot v\left(
m\otimes Q\right) \left( u_{2}\right) \cdot \left( \gamma ^{-1}\otimes
\varepsilon \right) \left( u_{3}\right) + \\
\delta _{\left\vert x\right\vert +\left\vert y\right\vert +\left\vert
z\right\vert ,s}\left( \varepsilon \otimes v\right) \left( u_{1}\right)
\cdot \left( \varepsilon \otimes \varepsilon \right) \left( m\otimes
Q\right) \left( u_{2}\right) \cdot \left( \gamma ^{-1}\otimes \varepsilon
\right) \left( u_{3}\right) + \\
-\delta _{\left\vert x\right\vert +\left\vert y\right\vert +\left\vert
z\right\vert ,s}\left( \varepsilon \otimes v\right) \left( u_{1}\right)
\cdot v\left( m\otimes Q\right) \left( u_{2}\right) \cdot \left( \gamma
^{-1}\otimes \varepsilon \right) \left( u_{3}\right) + \\
\delta _{\left\vert x\right\vert +\left\vert y\right\vert +\left\vert
z\right\vert ,s}\omega _{s}\left( u_{1}\right) \cdot \left( \varepsilon
\otimes \varepsilon \right) \left( m\otimes Q\right) \left( u_{2}\right)
\cdot \left( \gamma ^{-1}\otimes \varepsilon \right) \left( u_{3}\right) \\
-\delta _{\left\vert x\right\vert +\left\vert y\right\vert +\left\vert
z\right\vert ,s}\omega _{s}\left( u_{1}\right) \cdot v\left( m\otimes
Q\right) \left( u_{2}\right) \cdot \left( \gamma ^{-1}\otimes \varepsilon
\right) \left( u_{3}\right)%
\end{array}%
\right] \\
&\overset{\left\vert u_{1}\right\vert +\left\vert u_{2}\right\vert \leq s}{=}%
&\left[
\begin{array}{c}
\delta _{\left\vert x\right\vert +\left\vert y\right\vert +\left\vert
z\right\vert ,s}\left( \gamma ^{-1}\otimes \varepsilon \right) \left(
u\right) + \\
-\delta _{\left\vert x\right\vert +\left\vert y\right\vert +\left\vert
z\right\vert ,s}v\left( m\otimes Q\right) \left( u_{1}\right) \cdot \left(
\gamma ^{-1}\otimes \varepsilon \right) \left( u_{2}\right) + \\
\delta _{\left\vert x\right\vert +\left\vert y\right\vert +\left\vert
z\right\vert ,s}v\left( Q\otimes m\right) \left( u_{1}\right) \cdot \left(
\gamma ^{-1}\otimes \varepsilon \right) \left( u_{2}\right) -0+ \\
\delta _{\left\vert x\right\vert +\left\vert y\right\vert +\left\vert
z\right\vert ,s}\left( \varepsilon \otimes v\right) \left( u_{1}\right)
\cdot \left( \gamma ^{-1}\otimes \varepsilon \right) \left( u_{2}\right) -0+
\\
\delta _{\left\vert x\right\vert +\left\vert y\right\vert +\left\vert
z\right\vert ,s}\omega _{s}\left( u_{1}\right) \cdot \left( \gamma
^{-1}\otimes \varepsilon \right) \left( u_{2}\right) -0%
\end{array}%
\right] \\
&\overset{(\ref{form:gamma-1})}{=}&\left[
\begin{array}{c}
\delta _{\left\vert x\right\vert +\left\vert y\right\vert +\left\vert
z\right\vert ,s}\left( \left( \left( \varepsilon \otimes \varepsilon \right)
-v\right) \otimes \varepsilon \right) \left( u\right) + \\
-\delta _{\left\vert x\right\vert +\left\vert y\right\vert +\left\vert
z\right\vert ,s}v\left( m\otimes Q\right) \left( u_{1}\right) \cdot \left(
\left( \left( \varepsilon \otimes \varepsilon \right) -v\right) \otimes
\varepsilon \right) \left( u_{2}\right) + \\
\delta _{\left\vert x\right\vert +\left\vert y\right\vert +\left\vert
z\right\vert ,s}v\left( Q\otimes m\right) \left( u_{1}\right) \cdot \left(
\left( \left( \varepsilon \otimes \varepsilon \right) -v\right) \otimes
\varepsilon \right) \left( u_{2}\right) + \\
\delta _{\left\vert x\right\vert +\left\vert y\right\vert +\left\vert
z\right\vert ,s}\left( \varepsilon \otimes v\right) \left( u_{1}\right)
\cdot \left( \left( \left( \varepsilon \otimes \varepsilon \right) -v\right)
\otimes \varepsilon \right) \left( u_{2}\right) + \\
\delta _{\left\vert x\right\vert +\left\vert y\right\vert +\left\vert
z\right\vert ,s}\omega _{s}\left( u_{1}\right) \cdot \left( \left( \left(
\varepsilon \otimes \varepsilon \right) -v\right) \otimes \varepsilon
\right) \left( u_{2}\right)%
\end{array}%
\right] \\
&=&\left[
\begin{array}{c}
\delta _{\left\vert x\right\vert +\left\vert y\right\vert +\left\vert
z\right\vert ,s}\left( \varepsilon \otimes \varepsilon \otimes \varepsilon
\right) \left( u\right) -\delta _{\left\vert x\right\vert +\left\vert
y\right\vert +\left\vert z\right\vert ,s}\left( v\otimes \varepsilon \right)
\left( u\right) \\
-\delta _{\left\vert x\right\vert +\left\vert y\right\vert +\left\vert
z\right\vert ,s}v\left( m\otimes Q\right) \left( u_{1}\right) \cdot \left(
\varepsilon \otimes \varepsilon \otimes \varepsilon \right) \left(
u_{2}\right) + \\
\delta _{\left\vert x\right\vert +\left\vert y\right\vert +\left\vert
z\right\vert ,s}v\left( m\otimes Q\right) \left( u_{1}\right) \cdot \left(
v\otimes \varepsilon \right) \left( u_{2}\right) + \\
\delta _{\left\vert x\right\vert +\left\vert y\right\vert +\left\vert
z\right\vert ,s}v\left( Q\otimes m\right) \left( u_{1}\right) \cdot \left(
\varepsilon \otimes \varepsilon \otimes \varepsilon \right) \left(
u_{2}\right) + \\
-\delta _{\left\vert x\right\vert +\left\vert y\right\vert +\left\vert
z\right\vert ,s}v\left( Q\otimes m\right) \left( u_{1}\right) \cdot \left(
v\otimes \varepsilon \right) \left( u_{2}\right) + \\
\delta _{\left\vert x\right\vert +\left\vert y\right\vert +\left\vert
z\right\vert ,s}\left( \varepsilon \otimes v\right) \left( u_{1}\right)
\cdot \left( \varepsilon \otimes \varepsilon \otimes \varepsilon \right)
\left( u_{2}\right) + \\
-\delta _{\left\vert x\right\vert +\left\vert y\right\vert +\left\vert
z\right\vert ,s}\left( \varepsilon \otimes v\right) \left( u_{1}\right)
\cdot \left( v\otimes \varepsilon \right) \left( u_{2}\right) + \\
\delta _{\left\vert x\right\vert +\left\vert y\right\vert +\left\vert
z\right\vert ,s}\omega _{s}\left( u_{1}\right) \cdot \left( \varepsilon
\otimes \varepsilon \otimes \varepsilon \right) \left( u_{2}\right) + \\
-\delta _{\left\vert x\right\vert +\left\vert y\right\vert +\left\vert
z\right\vert ,s}\omega _{s}\left( u_{1}\right) \cdot \left( v\otimes
\varepsilon \right) \left( u_{2}\right)%
\end{array}%
\right] \\
&\overset{\left\vert u_{1}\right\vert +\left\vert u_{2}\right\vert \leq s}{=}%
&\left[
\begin{array}{c}
0-\delta _{\left\vert x\right\vert +\left\vert y\right\vert +\left\vert
z\right\vert ,s}\left( v\otimes \varepsilon \right) \left( u\right) \\
-\delta _{\left\vert x\right\vert +\left\vert y\right\vert +\left\vert
z\right\vert ,s}v\left( m\otimes Q\right) \left( u\right) +0+ \\
\delta _{\left\vert x\right\vert +\left\vert y\right\vert +\left\vert
z\right\vert ,s}v\left( Q\otimes m\right) \left( u\right) -0+ \\
\delta _{\left\vert x\right\vert +\left\vert y\right\vert +\left\vert
z\right\vert ,s}\left( \varepsilon \otimes v\right) \left( u\right) -0+ \\
\delta _{\left\vert x\right\vert +\left\vert y\right\vert +\left\vert
z\right\vert ,s}\omega _{s}\left( u\right) -0%
\end{array}%
\right] \\
&=&\left[
\begin{array}{c}
\delta _{\left\vert x\right\vert +\left\vert y\right\vert +\left\vert
z\right\vert ,s}\left[ \left( \varepsilon \otimes v\right) \left( u\right)
+v\left( Q\otimes m\right) \left( u\right) -v\left( m\otimes Q\right) \left(
u\right) -\left( v\otimes \varepsilon \right) \left( u\right) \right] + \\
\delta _{\left\vert x\right\vert +\left\vert y\right\vert +\left\vert
z\right\vert ,s}\omega _{s}\left( u\right)%
\end{array}%
\right] .
\end{eqnarray*}
\end{invisible}

Now $\omega _{s}\left( u\right) =\overline{\omega }_{s}\left( \overline{u}%
\right) $ while one proves that $\left( \varepsilon \otimes v\right) \left(
u\right) =\left( \varepsilon _{\overline{Q}}\otimes \overline{v}\right)
\left( \overline{u}\right) ,$ $\delta _{\left\vert x\right\vert +\left\vert
y\right\vert +\left\vert z\right\vert ,s}v\left( m\otimes Q\right) \left(
u\right) =\delta _{\left\vert x\right\vert +\left\vert y\right\vert
+\left\vert z\right\vert ,s}\overline{v}\left( m_{\overline{Q}}\otimes
\overline{Q}\right) \left( \overline{u}\right) $ and similarly for the other
pieces of the equality.

\begin{invisible}
Here is the full computation%
\begin{eqnarray*}
\left( \varepsilon \otimes v\right) \left( u\right) &=&\left( \varepsilon
\otimes v\right) \left( x\otimes y\otimes z\right) =\varepsilon \left(
x\right) v\left( y\otimes z\right) =\varepsilon \left( x\right) \overline{v}%
\left( \overline{y}\otimes \overline{z}\right) \\
&=&\delta _{\left\vert x\right\vert ,0}\varepsilon \left( x\right) \overline{%
v}\left( \overline{y}\otimes \overline{z}\right) =\varepsilon _{\overline{Q}%
}\left( \overline{x}\right) \overline{v}\left( \overline{y}\otimes \overline{%
z}\right) =\left( \varepsilon _{\overline{Q}}\otimes \overline{v}\right)
\left( \overline{u}\right) ,
\end{eqnarray*}%
and taking $x=x^{a,l}$ and $y=x^{a^{\prime },l^{\prime }}$%
\begin{eqnarray*}
\delta _{\left\vert x\right\vert +\left\vert y\right\vert +\left\vert
z\right\vert ,s}v\left( m\otimes Q\right) \left( u\right) &=&\delta
_{\left\vert x\right\vert +\left\vert y\right\vert +\left\vert z\right\vert
,s}v\left( x\cdot y\otimes z\right) \overset{(\ref{form:GelYD1})}{=}\delta
_{a+a^{\prime }+\left\vert z\right\vert ,s}v\left( \sum_{w\leq a+a^{\prime
}}\sum_{t}\mu _{w,t}^{a,l,a^{\prime },l^{\prime }}x^{w,t}\otimes z\right) \\
&=&\delta _{a+a^{\prime }+\left\vert z\right\vert ,s}\sum_{w\leq a+a^{\prime
}}\sum_{t}\mu _{w,t}^{a,l,a^{\prime },l^{\prime }}v\left( x^{w,t}\otimes
z\right) \\
&=&\delta _{a+a^{\prime }+\left\vert z\right\vert ,s}\sum_{w\leq a+a^{\prime
}}\sum_{t}\mu _{w,t}^{a,l,a^{\prime },l^{\prime }}\overline{v}\left(
\overline{x^{w,t}}\otimes \overline{z}\right) \\
&&\overset{(\ref{form:vbarYD})}{=}\delta _{a+a^{\prime }+\left\vert
z\right\vert ,s}\sum_{w\leq a+a^{\prime }}\sum_{t}\mu _{w,t}^{a,l,a^{\prime
},l^{\prime }}\delta _{w+\left\vert z\right\vert ,s}\overline{v}\left(
\overline{x^{w,t}}\otimes \overline{z}\right) \\
&=&\delta _{a+a^{\prime }+\left\vert z\right\vert ,s}\sum_{t}\mu
_{a+a^{\prime },t}^{a,l,a^{\prime },l^{\prime }}\overline{v}\left( \overline{%
x^{a+a^{\prime },t}}\otimes \overline{z}\right) \\
&&\overset{(\ref{form:GelYD2})}{=}\delta _{a+a^{\prime }+\left\vert
z\right\vert ,s}\overline{v}\left( \overline{x^{a,l}}\cdot \overline{%
x^{a^{\prime },l^{\prime }}}\otimes \overline{z}\right) \\
&=&\delta _{\left\vert x\right\vert +\left\vert y\right\vert +\left\vert
z\right\vert ,s}\overline{v}\left( \overline{x}\cdot \overline{y}\otimes
\overline{z}\right) \\
&=&\delta _{\left\vert x\right\vert +\left\vert y\right\vert +\left\vert
z\right\vert ,s}\overline{v}\left( m_{\overline{Q}}\otimes \overline{Q}%
\right) \left( \overline{u}\right)
\end{eqnarray*}%
Similarly one gets $\left( v\otimes \varepsilon \right) \left( u\right)
=\left( \overline{v}\otimes \varepsilon _{\overline{Q}}\right) \left(
\overline{u}\right) $ and $\delta _{\left\vert x\right\vert +\left\vert
y\right\vert +\left\vert z\right\vert ,s}v\left( Q\otimes m\right) \left(
u\right) =\delta _{\left\vert x\right\vert +\left\vert y\right\vert
+\left\vert z\right\vert ,s}\overline{v}\left( \overline{Q}\otimes m_{%
\overline{Q}}\right) \left( \overline{u}\right) .$
\end{invisible}

Thus one gets%
\begin{eqnarray*}
\left( \omega ^{\gamma }\right) _{s}\left( u\right) &=&\left[
\begin{array}{c}
\delta _{\left\vert x\right\vert +\left\vert y\right\vert +\left\vert
z\right\vert ,s}\left[ \left( \varepsilon _{\overline{Q}}\otimes \overline{v}%
\right) \left( \overline{u}\right) +\overline{v}\left( \overline{Q}\otimes
m_{\overline{Q}}\right) \left( \overline{u}\right) -\overline{v}\left( m_{%
\overline{Q}}\otimes \overline{Q}\right) \left( \overline{u}\right) -\left(
\overline{v}\otimes \varepsilon _{\overline{Q}}\right) \left( \overline{u}%
\right) \right] + \\
+\delta _{\left\vert x\right\vert +\left\vert y\right\vert +\left\vert
z\right\vert ,s}\overline{\omega }_{s}\left( \overline{u}\right)%
\end{array}%
\right] \\
&=&-\delta _{\left\vert x\right\vert +\left\vert y\right\vert +\left\vert
z\right\vert ,s}\partial ^{2}\overline{v}+\delta _{\left\vert x\right\vert
+\left\vert y\right\vert +\left\vert z\right\vert ,s}\overline{\omega }%
_{s}\left( \overline{u}\right) =0.
\end{eqnarray*}%
For $0\leq t\leq s-1$, analogously to the above, we compute
\begin{eqnarray*}
\left( \omega ^{\gamma }\right) _{t}\left( u\right) &=&\delta _{\left\vert
x\right\vert +\left\vert y\right\vert +\left\vert z\right\vert ,t}\omega
^{\gamma }\left( u\right) \\
&=&\delta _{\left\vert x\right\vert +\left\vert y\right\vert +\left\vert
z\right\vert ,t}\left[ \left( \varepsilon \otimes \gamma \right) \ast \gamma
\left( Q\otimes m\right) \ast \omega \ast \gamma ^{-1}\left( m\otimes
Q\right) \ast \left( \gamma ^{-1}\otimes \varepsilon \right) \right] \left(
u\right) \\
&=&\delta _{\left\vert x\right\vert +\left\vert y\right\vert +\left\vert
z\right\vert ,t}\left[ \left( \varepsilon \otimes \gamma \right) \ast \gamma
\left( Q\otimes m\right) \ast \omega _{0}\ast \gamma ^{-1}\left( m\otimes
Q\right) \ast \left( \gamma ^{-1}\otimes \varepsilon \right) \right] \left(
u\right) \\
&\overset{(\ref{form:Omega0YD})}{=}&\delta _{\left\vert x\right\vert
+\left\vert y\right\vert +\left\vert z\right\vert ,t}\left[ \left(
\varepsilon \otimes \gamma \right) \ast \gamma \left( Q\otimes m\right) \ast
\gamma ^{-1}\left( m\otimes Q\right) \ast \left( \gamma ^{-1}\otimes
\varepsilon \right) \right] \left( u\right) \\
&=&\delta _{\left\vert x\right\vert +\left\vert y\right\vert +\left\vert
z\right\vert ,t}\left( \varepsilon \otimes \varepsilon \otimes \varepsilon
\right) \left( u\right) =\delta _{0,t}\left( \varepsilon \otimes \varepsilon
\otimes \varepsilon \right) \left( u\right) .
\end{eqnarray*}

\begin{invisible}
Here are the datails%
\begin{eqnarray*}
&&\left( \omega ^{\gamma }\right) _{t}\left( u\right) =\delta _{\left\vert
x\right\vert +\left\vert y\right\vert +\left\vert z\right\vert ,t}\omega
^{\gamma }\left( u\right) \\
&=&\delta _{\left\vert x\right\vert +\left\vert y\right\vert +\left\vert
z\right\vert ,t}\left[ \left( \varepsilon \otimes \gamma \right) \ast \gamma
\left( Q\otimes m\right) \ast \omega \ast \gamma ^{-1}\left( m\otimes
Q\right) \ast \left( \gamma ^{-1}\otimes \varepsilon \right) \right] \left(
u\right) \\
&=&\delta _{\left\vert x\right\vert +\left\vert y\right\vert +\left\vert
z\right\vert ,t}\left[ \left( \varepsilon \otimes \gamma \right) \ast \gamma
\left( Q\otimes m\right) \ast \left( \sum\limits_{t^{\prime }\in \N_0%
}\omega _{t^{\prime }}\right) \ast \gamma ^{-1}\left( m\otimes Q\right) \ast
\left( \gamma ^{-1}\otimes \varepsilon \right) \right] \left( u\right) \\
&=&\delta _{\left\vert x\right\vert +\left\vert y\right\vert +\left\vert
z\right\vert ,t}\left[ \left( \varepsilon \otimes \gamma \right) \ast \gamma
\left( Q\otimes m\right) \ast \omega _{0}\ast \gamma ^{-1}\left( m\otimes
Q\right) \ast \left( \gamma ^{-1}\otimes \varepsilon \right) \right] \left(
u\right) \\
&\overset{(\ref{form:Omega0YD})}{=}&\delta _{\left\vert x\right\vert
+\left\vert y\right\vert +\left\vert z\right\vert ,t}\left[ \left(
\varepsilon \otimes \gamma \right) \ast \gamma \left( Q\otimes m\right) \ast
\gamma ^{-1}\left( m\otimes Q\right) \ast \left( \gamma ^{-1}\otimes
\varepsilon \right) \right] \left( u\right) \\
&=&\delta _{\left\vert x\right\vert +\left\vert y\right\vert +\left\vert
z\right\vert ,t}\left( \varepsilon \otimes \gamma \right) \left(
u_{1}\right) \cdot \gamma \left( Q\otimes m\right) \left( u_{2}\right) \cdot
\gamma ^{-1}\left( m\otimes Q\right) \left( u_{3}\right) \cdot \left( \gamma
^{-1}\otimes \varepsilon \right) \left( u_{4}\right) \\
&=&\delta _{\left\vert x\right\vert +\left\vert y\right\vert +\left\vert
z\right\vert ,t}\left( \left( \varepsilon \otimes \varepsilon \right)
+v\right) \left( u_{1}\right) \cdot \gamma \left( Q\otimes m\right) \left(
u_{2}\right) \cdot \gamma ^{-1}\left( m\otimes Q\right) \left( u_{3}\right)
\cdot \left( \gamma ^{-1}\otimes \varepsilon \right) \left( u_{4}\right) \\
&\overset{\left\vert u_{1}\right\vert <s}{=}&\delta _{\left\vert
x\right\vert +\left\vert y\right\vert +\left\vert z\right\vert ,t}\left(
\varepsilon \otimes \varepsilon \right) \left( u_{1}\right) \cdot \gamma
\left( Q\otimes m\right) \left( u_{2}\right) \cdot \gamma ^{-1}\left(
m\otimes Q\right) \left( u_{3}\right) \cdot \left( \gamma ^{-1}\otimes
\varepsilon \right) \left( u_{4}\right) \\
&=&\delta _{\left\vert x\right\vert +\left\vert y\right\vert +\left\vert
z\right\vert ,t}\gamma \left( Q\otimes m\right) \left( u_{1}\right) \cdot
\gamma ^{-1}\left( m\otimes Q\right) \left( u_{2}\right) \cdot \left( \gamma
^{-1}\otimes \varepsilon \right) \left( u_{3}\right) \\
&=&\delta _{\left\vert x\right\vert +\left\vert y\right\vert +\left\vert
z\right\vert ,t}\left( \left( \varepsilon \otimes \varepsilon \right)
+v\right) \left( Q\otimes m\right) \left( u_{1}\right) \cdot \gamma
^{-1}\left( m\otimes Q\right) \left( u_{2}\right) \cdot \left( \gamma
^{-1}\otimes \varepsilon \right) \left( u_{3}\right) \\
&\overset{\left\vert u_{1}\right\vert <s}{=}&\delta _{\left\vert
x\right\vert +\left\vert y\right\vert +\left\vert z\right\vert ,t}\left(
\varepsilon \otimes \varepsilon \right) \left( Q\otimes m\right) \left(
u_{1}\right) \cdot \gamma ^{-1}\left( m\otimes Q\right) \left( u_{2}\right)
\cdot \left( \gamma ^{-1}\otimes \varepsilon \right) \left( u_{3}\right) \\
&=&\delta _{\left\vert x\right\vert +\left\vert y\right\vert +\left\vert
z\right\vert ,t}\gamma ^{-1}\left( m\otimes Q\right) \left( u_{1}\right)
\cdot \left( \gamma ^{-1}\otimes \varepsilon \right) \left( u_{2}\right) \\
&\overset{(\ref{form:gamma-1})}{=}&\delta _{\left\vert x\right\vert
+\left\vert y\right\vert +\left\vert z\right\vert ,t}\left( \left(
\varepsilon \otimes \varepsilon \right) -v\right) \left( m\otimes Q\right)
\left( u_{1}\right) \cdot \left( \gamma ^{-1}\otimes \varepsilon \right)
\left( u_{2}\right) \\
&\overset{\left\vert u_{1}\right\vert <s}{=}&\delta _{\left\vert
x\right\vert +\left\vert y\right\vert +\left\vert z\right\vert ,t}\left(
\varepsilon \otimes \varepsilon \right) \left( m\otimes Q\right) \left(
u_{1}\right) \cdot \left( \gamma ^{-1}\otimes \varepsilon \right) \left(
u_{2}\right) \\
&=&\delta _{\left\vert x\right\vert +\left\vert y\right\vert +\left\vert
z\right\vert ,t}\left( \gamma ^{-1}\otimes \varepsilon \right) \left(
u\right) \\
&\overset{(\ref{form:gamma-1})}{=}&\delta _{\left\vert x\right\vert
+\left\vert y\right\vert +\left\vert z\right\vert ,t}\left( \left( \left(
\varepsilon \otimes \varepsilon \right) -v\right) \otimes \varepsilon
\right) \left( u\right) \\
&\overset{\left\vert u_{1}\right\vert <s}{=}&\delta _{\left\vert
x\right\vert +\left\vert y\right\vert +\left\vert z\right\vert ,t}\left(
\varepsilon \otimes \varepsilon \otimes \varepsilon \right) \left( u\right)
\\
&=&\delta _{0,t}\left( \varepsilon \otimes \varepsilon \otimes \varepsilon
\right) \left( u\right) =\delta _{0,t}\left( \varepsilon _{\overline{Q}%
}\otimes \varepsilon _{\overline{Q}}\otimes \varepsilon _{\overline{Q}%
}\right) \left( \overline{u}\right) .
\end{eqnarray*}
\end{invisible}

Therefore we can now repeat the argument on $\omega ^{\gamma }$ instead of $%
\omega .$ Deforming several times we will get a reassociator, say $\omega
^{\prime },$ whose first non trivial component $\omega _{t}^{\prime }$, with
$t\neq 0$, exceeds the dimension of $Q.$ In other words $\omega ^{\prime
}=\omega _{0}^{\prime }$ which is trivial. Hence $Q$ is gauge equivalent to
a connected bialgebra in ${_{H}^{H}\mathcal{YD}}$.
\end{proof}

\section{Invariants}\label{sec:4}

Given a $\Bbbk $-algebra $A,$ we denote by $\mathrm{H}^{n}\left( A,-\right) $
the $n$-th right derived functor of $\mathrm{Hom}_{A,A}\left( A,-\right) $
in the category of $A$-bimodules. In other words, for every $A$-bimodule $M$%
, $\mathrm{H}^{n}\left( A,M\right) $ is the Hochschild cohomology group of $%
A $ with coefficients in $M$. Denote by $\mathrm{Z}^{n}\left( A,M\right) $
and $\mathrm{B}^{n}\left( A,M\right) $ the abelian groups of $n$-cocycles
and of $n$-coboundaries respectively.

Let $H$ be a Hopf algebra, let $B$ be a left $H$-module algebra and let $M$
be a $B\#H$-bimodule, where $B\#H$ denotes the smash product algebra, see
e.g. \cite[Definition 4.1.3]{Mo}. Then $\mathrm{H}^{n}\left( B,M\right) $
becomes an $H$-bimodule as follows. Its structure of left $H$-module is
given via $\varepsilon _{H}$ and its structure of right $H$-module is
defined, for every $f\in \mathrm{Z}^{n}\left( B,M\right) $ and $h\in H,$ by
setting
\begin{equation*}
\left[ f\right] h:=\left[ \chi _{n}^{h}\left( M\right) \left( f\right) %
\right]
\end{equation*}%
where, for every $k\in \Bbbk ,b_{1},\ldots ,b_{n}\in B,$ we set%
\begin{eqnarray*}
\chi _{0}^{h}\left( M\right) \left( f\right) \left( k\right) := &&\left(
1_{B}\#S\left( h_{1}\right) \right) f\left( k\right) \left(
1_{B}\#h_{2}\right) \text{ for }n=0\text{ while and for }n\geq 1 \\
\chi _{n}^{h}\left( M\right) \left( f\right) \left( b_{1}\otimes
b_{2}\otimes \cdots \otimes b_{n}\right) := &&\left( 1_{B}\#S\left(
h_{1}\right) \right) f\left( h_{2}b_{1}\otimes h_{3}b_{2}\otimes \cdots
\otimes h_{n+1}b_{n}\right) \left( 1_{B}\#h_{n+2}\right) .
\end{eqnarray*}

Moreover
\begin{equation}
\partial ^{n}\circ \chi _{n}^{h}\left( M\right) =\chi _{n+1}^{h}\left(
M\right) \circ \partial ^{n},\text{ for every }n\geq -1,
\label{form:PartialChi}
\end{equation}%
where $\partial ^{n}:\mathrm{Hom}_{\Bbbk }\left( B^{\otimes n},M\right)
\rightarrow \mathrm{Hom}_{\Bbbk }\left( B^{\otimes \left( n+1\right)
},M\right) $ denotes the differential of the usual Hochschild cohomology.

Denote by $\mathrm{H}^{n}\left( B,M\right) ^{H}$ the space of $H$-invariant
elements of $\mathrm{H}^{n}\left( B,M\right) $.

\begin{proposition}
\label{pro:Dragos}Let $H$ be a semisimple Hopf algebra and let $B$ be a left
$H$-module algebra. Denote by $A:=B\#H$. Then, for each $n\in \N_0$
and for every $A$-bimodule $M$%
\begin{equation*}
\mathrm{H}^{n}\left( B\#H,M\right) \cong \mathrm{H}^{n}\left( B,M\right)
^{H}.
\end{equation*}
\end{proposition}

\begin{proof}
We will apply \cite[Equation (3.6.1)]{St}. To this aim we have to prove
first that $A/B$ is an $H$-Galois extension such that $A$ is flat as left
and right $B$-module. Now, $A=B\#_{\xi }H$ for $\xi :H\otimes H\rightarrow B$
defined by $\xi \left( x,y\right) =\varepsilon _{H}\left( x\right)
\varepsilon _{H}\left( y\right) 1_{A},$ cf. \cite[Definition 7.1.1]{Mo}.
Moreover a direct computation shows that $\iota :B\rightarrow A:b\mapsto
b\#1_{H}$ is a right $H$-extension where $A$ is regarded as a right $H$%
-comodule via $\rho :A\rightarrow A\otimes H:b\#h\mapsto \left(
b\#h_{1}\right) \otimes h_{2}.$ Thus, by \cite[Proposition 7.2.7]{Mo}, we
know that $\iota :B\rightarrow A$ is $H$-cleft and hence, by \cite[Theorem
8.2.4]{Mo}, it is $H$-Galois. The $B$-bimodule structure of $A$ is induced
by $\iota $ so that, explicitly, we have
\begin{eqnarray*}
b^{\prime }\left( b\#h\right) &=&\left( b^{\prime }\#1_{H}\right) \left(
b\#h\right) =b^{\prime }b\#h, \\
\left( b\#h\right) b^{\prime } &=&\left( b\#h\right) \left( b^{\prime
}\#1_{H}\right) =b\left( h_{1}b^{\prime }\right) \#h_{2}.
\end{eqnarray*}%
Note that $A=B\#H$ is flat as a left $B$-module as $H$ is a free $\Bbbk $%
-module ($\Bbbk $ is a field). Now consider the map $\alpha :H\otimes
B\rightarrow A$ defined by setting $\alpha \left( h\otimes b\right)
:=h_{1}b\otimes h_{2}$ (note that it is defined as the braiding in ${_{H}^{H}%
\mathcal{YD}}$). We have%
\begin{equation*}
\alpha \left( h\otimes bb^{\prime }\right) =h_{1}\left( bb^{\prime }\right)
\otimes h_{2}=\left( h_{1}b\right) \left( h_{2}b^{\prime }\right) \otimes
h_{3}=\left( h_{1}b\#h_{2}\right) b^{\prime }=\alpha \left( h\otimes
b\right) b^{\prime }
\end{equation*}%
so that $\alpha $ is right $B$-linear where $H\otimes B$ is regarded as a
right module via $\left( h\#b\right) b^{\prime }:=h\#bb^{\prime }.$ Now $H$
is semisimple and hence separable (see \cite[Corollary 3.7]{St}). Thus $H$
is finite-dimensional and hence it has bijective antipode $S_{H}$. Thus $%
\alpha $ is invertible with inverse given by $\alpha ^{-1}\left( b\#h\right)
:=h_{2}\otimes S_{H}^{-1}\left( h_{1}\right) b.$ Therefore $\alpha $ is an
isomorphism of right $B$-modules and hence $A$ is flat as a right $B$-module
as $H\otimes B$ is.

We have now the hypotheses necessary to apply \cite[Equation (3.6.1)]{St}
and obtain%
\begin{equation*}
\mathrm{H}^{n}\left( A,M\right) \cong \mathrm{Hom}_{-,H}\left( \Bbbk ,%
\mathrm{H}^{n}\left( B,M\right) \right) =\mathrm{Hom}_{\Bbbk }\left( \Bbbk ,%
\mathrm{H}^{n}\left( B,M\right) \right) ^{H}\cong \mathrm{H}^{n}\left(
B,M\right) ^{H}.
\end{equation*}

\begin{invisible}
Let us check for ourself that the structures of $\mathrm{H}^{n}\left(
B,M\right) $ are the ones claimed. First note that $\mathrm{H}^{n}\left(
B,M\right) $, through the isomorphism, is regarded as a left $H$-module via $%
\varepsilon _{H}$ i.e.
\begin{equation*}
\mathrm{H}^{n}\left( B,M\right) ^{H}=\left\{ z\in \mathrm{H}^{n}\left(
B,M\right) \mid \varepsilon _{H}\left( h\right) z=zh,\text{ for every }h\in
H\right\} .
\end{equation*}%
We would like to express explicitly the right $H$-module structure of $%
\mathrm{H}^{n}\left( B,M\right) $. It is the one of \cite[Proposition 2.4]%
{St}. First, following the results recalled above, the map making $A$ cleft
is $\gamma :H\rightarrow A$ given by $\gamma \left( h\right) =1_{B}\#h$
whose convolution inverse is defined by $\gamma ^{-1}\left( h\right) =\gamma
S\left( h\right) =1_{B}\#S\left( h\right) .$

Note that $\gamma $ is an algebra map as%
\begin{eqnarray*}
\gamma \left( h\right) \gamma \left( l\right) &=&\left( 1_{B}\#h\right)
\left( 1_{B}\#l\right) =\left( 1_{B}\left( h_{1}1_{B}\right) \right)
\#\left( h_{2}l\right) =\left( 1_{B}\left( \varepsilon _{H}\left(
h_{1}\right) 1_{B}\right) \right) \#\left( h_{2}l\right) =\gamma \left(
hl\right) , \\
\gamma \left( 1_{H}\right) &=&1_{B}\#1_{H}.
\end{eqnarray*}%
Now we have $\gamma ^{-1}\left( h\right) =\gamma S\left( h\right) $ so that
\begin{eqnarray*}
\gamma \left( h_{1}\right) \gamma ^{-1}\left( h_{2}\right) &=&\gamma \left(
h_{1}\right) \gamma \left( S\left( h_{2}\right) \right) =\gamma \left(
h_{1}S\left( h_{2}\right) \right) =\varepsilon _{H}\left( h\right) \gamma
\left( 1_{H}\right) =\varepsilon _{H}\left( h\right) 1_{B}\#1_{H}, \\
\gamma ^{-1}\left( h_{1}\right) \gamma \left( h_{2}\right) &=&\gamma \left(
S\left( h_{1}\right) \right) \gamma \left( h_{2}\right) =\gamma \left(
S\left( h_{1}\right) h_{2}\right) =\varepsilon _{H}\left( h\right)
1_{B}\#1_{H}.
\end{eqnarray*}%
Using $\gamma $ one has that the composition inverse of the canonical map
\begin{equation*}
\beta :A\otimes _{B}A\rightarrow A\#H:\left( b\#h\right) \otimes \left(
b^{\prime }\#h^{\prime }\right) \mapsto \left( b\#h\right) \rho \left(
b^{\prime }\#h^{\prime }\right) =\left( b\left( h_{1}b^{\prime }\right)
\#h_{2}h_{1}^{\prime }\right) \otimes h_{2}^{\prime }
\end{equation*}%
is given by%
\begin{eqnarray*}
\beta ^{-1}\left( \left( b\#h\right) \otimes l\right) &=&\left( b\#h\right)
\gamma ^{-1}\left( l_{1}\right) \otimes _{B}\gamma \left( l_{2}\right) \\
&=&\left( b\#h\right) \left( 1_{B}\#S\left( l_{1}\right) \right) \otimes
_{B}\left( 1_{B}\#l_{2}\right) \\
&=&\left( b\left( h_{1}1_{B}\right) \right) \#\left( h_{2}S\left(
l_{1}\right) \right) \otimes _{B}\left( 1_{B}\#l_{2}\right) \\
&=&\left( b\#hS\left( l_{1}\right) \right) \otimes _{B}\left(
1_{B}\#l_{2}\right) .
\end{eqnarray*}%
Hence
\begin{equation*}
r_{i}\left( h\right) \otimes _{B}l_{i}\left( h\right) =\beta ^{-1}\left(
\left( 1_{B}\#1_{H}\right) \otimes h\right) =\left( 1_{B}\#S\left(
h_{1}\right) \right) \otimes _{B}\left( 1_{B}\#h_{2}\right) .
\end{equation*}

Now $\mathrm{H}^{0}\left( B,M\right) =M^{B}$ and the right $H$-module
structure of the latter is given, for every $h\in H,m\in M^{B},$ by%
\begin{equation*}
mh=r_{i}\left( h\right) ml_{i}\left( h\right) =\left( 1_{B}\#S\left(
h_{1}\right) \right) m\left( 1_{B}\#h_{2}\right) .
\end{equation*}%
Consider a short exact sequence of $A$-bimodules
\begin{equation*}
0\rightarrow M\overset{s}{\rightarrow }I\overset{p}{\rightarrow }%
C\rightarrow 0
\end{equation*}%
with $I$ injective as an $A$-bimodule (then $\mathrm{H}^{n}\left( B,I\right)
=0\,$for every $n\geq 1$ by the proof of \cite[Lemma 2.2]{St}). This short
sequence induces a short exact sequence of chain complexes%
\begin{equation*}
\begin{array}{ccccccccc}
0 & \rightarrow & M^{B} & \overset{s^{B}}{\rightarrow } & I^{B} & \overset{%
p^{B}}{\rightarrow } & C^{B} &  &  \\
&  & \partial ^{-1}\downarrow &  & \partial ^{-1}\downarrow &  & \partial
^{-1}\downarrow &  &  \\
0 & \rightarrow & \mathrm{Hom}_{\Bbbk }\left( \Bbbk ,M\right) & \overset{%
\varphi ^{0}=\mathrm{Hom}_{\Bbbk }\left( \Bbbk ,s\right) }{\rightarrow } &
\mathrm{Hom}_{\Bbbk }\left( \Bbbk ,I\right) & \overset{\psi ^{0}=\mathrm{Hom}%
_{\Bbbk }\left( \Bbbk ,p\right) }{\rightarrow } & \mathrm{Hom}_{\Bbbk
}\left( \Bbbk ,C\right) & \rightarrow & 0 \\
&  & \partial ^{0}\downarrow &  & \partial ^{0}\downarrow &  & \partial
^{0}\downarrow &  &  \\
0 & \rightarrow & \mathrm{Hom}_{\Bbbk }\left( B,M\right) & \overset{\varphi
^{1}=\mathrm{Hom}_{\Bbbk }\left( B,s\right) }{\rightarrow } & \mathrm{Hom}%
_{\Bbbk }\left( B,I\right) & \overset{\psi ^{1}=\mathrm{Hom}_{\Bbbk }\left(
B,p\right) }{\rightarrow } & \mathrm{Hom}_{\Bbbk }\left( B,C\right) &
\rightarrow & 0 \\
&  & \partial ^{1}\downarrow &  & \partial ^{1}\downarrow &  & \partial
^{1}\downarrow &  &  \\
0 & \rightarrow & \mathrm{Hom}_{\Bbbk }\left( B^{\otimes 2},M\right) &
\overset{\varphi ^{2}=\mathrm{Hom}_{\Bbbk }\left( B^{\otimes 2},s\right) }{%
\rightarrow } & \mathrm{Hom}_{\Bbbk }\left( B^{\otimes 2},I\right) & \overset%
{\psi ^{2}=\mathrm{Hom}_{\Bbbk }\left( B^{\otimes 2},p\right) }{\rightarrow }
& \mathrm{Hom}_{\Bbbk }\left( B^{\otimes 2},C\right) & \rightarrow & 0 \\
&  & \partial ^{2}\downarrow &  & \partial ^{2}\downarrow &  & \partial
^{2}\downarrow &  &  \\
0 & \rightarrow & \mathrm{Hom}_{\Bbbk }\left( B^{\otimes 3},M\right) &
\overset{\varphi ^{3}=\mathrm{Hom}_{\Bbbk }\left( B^{\otimes 3},s\right) }{%
\rightarrow } & \mathrm{Hom}_{\Bbbk }\left( B^{\otimes 3},I\right) & \overset%
{\psi ^{3}=\mathrm{Hom}_{\Bbbk }\left( B^{\otimes 3},p\right) }{\rightarrow }
& \mathrm{Hom}_{\Bbbk }\left( B^{\otimes 3},C\right) & \rightarrow & 0 \\
&  & \partial ^{3}\downarrow &  & \partial ^{3}\downarrow &  & \partial
^{3}\downarrow &  &  \\
&  & \vdots &  & \vdots &  & \vdots &  &
\end{array}%
\end{equation*}%
It is straightforward to prove that the above squares commute.
\end{invisible}
\begin{invisible}
In fact, for $b_{1},\ldots ,b_{n+1}\in B,$%
\begin{eqnarray*}
&&\varphi ^{n+1}\partial ^{n}\left( f\right) \left( b_{1}\otimes
b_{2}\otimes \cdots \otimes b_{n+1}\right) \\
&=&s\left( \sum\limits_{i=0}^{n+1}\left( -1\right) ^{i}\partial
_{i}^{n}\left( f\right) \left( b_{1}\otimes b_{2}\otimes \cdots \otimes
b_{n+1}\right) \right) \\
&=&s\left(
\begin{array}{c}
f\left( b_{1}\otimes b_{2}\otimes \cdots \otimes b_{n}\right)
b_{n+1}-f\left( b_{1}\otimes b_{2}\cdots \otimes b_{n-1}\otimes
b_{n}b_{n+1}\right) \\
+f\left( b_{1}\otimes b_{2}\otimes \cdots \otimes b_{n-1}b_{n}\otimes
b_{n+1}\right) +\cdots +\left( -1\right) ^{n+1}b_{1}f\left( b_{2}\otimes
b_{3}\otimes \cdots \otimes b_{n}\otimes b_{n+1}\right)%
\end{array}%
\right) \\
&=&\left(
\begin{array}{c}
sf\left( b_{1}\otimes b_{2}\otimes \cdots \otimes b_{n}\right)
b_{n+1}-sf\left( b_{1}\otimes b_{2}\cdots \otimes b_{n-1}\otimes
b_{n}b_{n+1}\right) \\
+sf\left( b_{1}\otimes b_{2}\otimes \cdots \otimes b_{n-1}b_{n}\otimes
b_{n+1}\right) +\cdots +\left( -1\right) ^{n+1}b_{1}sf\left( b_{2}\otimes
b_{3}\otimes \cdots \otimes b_{n}\otimes b_{n+1}\right)%
\end{array}%
\right) \\
&=&\sum\limits_{i=0}^{n+1}\left( -1\right) ^{i}\partial _{i}^{n}\left(
sf\right) \left( b_{1}\otimes b_{2}\otimes \cdots \otimes b_{n+1}\right) \\
&=&\partial ^{n}\left( sf\right) \left( b_{1}\otimes b_{2}\otimes \cdots
\otimes b_{n+1}\right) =\partial ^{n}\varphi ^{n}\left( f\right) \left(
b_{1}\otimes b_{2}\otimes \cdots \otimes b_{n+1}\right) .
\end{eqnarray*}

Now consider the following commutative diagram%
\begin{equation*}
\begin{array}{ccccccccc}
0 & \rightarrow & \mathrm{Hom}_{\Bbbk }\left( B^{\otimes n},M\right) &
\overset{\varphi ^{n}=\mathrm{Hom}_{\Bbbk }\left( B^{\otimes n},s\right) }{%
\rightarrow } & \mathrm{Hom}_{\Bbbk }\left( B^{\otimes n},I\right) & \overset%
{\psi ^{n}=\mathrm{Hom}_{\Bbbk }\left( B^{\otimes n},p\right) }{\rightarrow }
& \mathrm{Hom}_{\Bbbk }\left( B^{\otimes n},C\right) & \rightarrow & 0 \\
&  & \chi _{n}^{h}\left( M\right) \downarrow &  & \chi _{n}^{h}\left(
I\right) \downarrow &  & \chi _{n}^{h}\left( C\right) \downarrow &  &  \\
0 & \rightarrow & \mathrm{Hom}_{\Bbbk }\left( B^{\otimes n},M\right) &
\overset{\varphi ^{n}=\mathrm{Hom}_{\Bbbk }\left( B^{\otimes n},s\right) }{%
\rightarrow } & \mathrm{Hom}_{\Bbbk }\left( B^{\otimes n},I\right) & \overset%
{\psi ^{n}=\mathrm{Hom}_{\Bbbk }\left( B^{\otimes n},p\right) }{\rightarrow }
& \mathrm{Hom}_{\Bbbk }\left( B^{\otimes n},C\right) & \rightarrow & 0%
\end{array}%
\end{equation*}

where, for every $h\in H,$ we set%
\begin{eqnarray*}
\chi _{0}^{h}\left( M\right) \left( f\right) \left( k\right) &:&=\left(
1_{B}\#S\left( h_{1}\right) \right) f\left( k\right) \left(
1_{B}\#h_{2}\right) , \\
\chi _{1}^{h}\left( M\right) \left( f\right) \left( b_{1}\right) &:&=\left(
1_{B}\#S\left( h_{1}\right) \right) f\left( h_{2}b_{1}\right) \left(
1_{B}\#h_{3}\right) \text{ and for }n>1 \\
\chi _{n}^{h}\left( M\right) \left( f\right) \left( b_{1}\otimes
b_{2}\otimes \cdots \otimes b_{n}\right) &:&=\left( 1_{B}\#S\left(
h_{1}\right) \right) f\left( h_{2}b_{1}\otimes h_{3}b_{2}\otimes \cdots
\otimes h_{n+1}b_{n}\right) \left( 1_{B}\#h_{n+2}\right) .
\end{eqnarray*}

Let us check the diagram above commutes for each $n\in \N_0$. We have%
\begin{eqnarray*}
&&\varphi ^{n}\chi _{n}^{h}\left( M\right) \left( f\right) \left(
b_{1}\otimes b_{2}\otimes \cdots \otimes b_{n}\right) \\
&=&s\left( \left( 1_{B}\#S\left( h_{1}\right) \right) f\left(
h_{2}b_{1}\otimes h_{3}b_{2}\otimes \cdots \otimes h_{n+1}b_{n}\right)
\left( 1_{B}\#h_{n+2}\right) \right) \\
&=&\left( 1_{B}\#S\left( h_{1}\right) \right) \left( sf\right) \left(
h_{2}b_{1}\otimes h_{3}b_{2}\otimes \cdots \otimes h_{n+1}b_{n}\right)
\left( 1_{B}\#h_{n+2}\right) \\
&=&\chi _{n}^{h}\left( I\right) \left( sf\right) \left( b_{1}\otimes
b_{2}\otimes \cdots \otimes b_{n}\right) =\chi _{n}^{h}\left( I\right)
\varphi ^{n}\left( f\right) \left( b_{1}\otimes b_{2}\otimes \cdots \otimes
b_{n}\right)
\end{eqnarray*}%
so that the left-hand square commutes and similarly the right-hand one does.

One also checks that $\chi _{n}^{h}\left( M\right) $ commutes with
differentials i.e. that the following diagram commutes%
\begin{equation*}
\begin{array}{ccccccccc}
0 & \rightarrow & M^{B} & \overset{\partial ^{-1}}{\longrightarrow } &
\mathrm{Hom}_{\Bbbk }\left( \Bbbk ,M\right) & \overset{\partial ^{0}}{%
\longrightarrow } & \mathrm{Hom}_{\Bbbk }\left( B,M\right) & \overset{%
\partial ^{1}}{\longrightarrow } & \cdots \\
&  & \chi _{-1}^{h}\left( M\right) =\varrho _{0}^{h}\left( M\right)
\downarrow &  & \chi _{0}^{h}\left( M\right) \downarrow &  & \chi
_{1}^{h}\left( M\right) \downarrow &  &  \\
0 & \rightarrow & M^{B} & \overset{\partial ^{-1}}{\longrightarrow } &
\mathrm{Hom}_{\Bbbk }\left( \Bbbk ,M\right) & \overset{\partial ^{0}}{%
\longrightarrow } & \mathrm{Hom}_{\Bbbk }\left( B,M\right) & \overset{%
\partial ^{1}}{\longrightarrow } & \cdots%
\end{array}%
\end{equation*}

We compute
\begin{eqnarray*}
\left( \partial ^{-1}\chi _{-1}^{h}\left( M\right) \left( m\right) \right)
\left( k\right) &=&\left( \partial ^{-1}\left( mh\right) \right) \left(
k\right) =k\left( mh\right) =\left( 1_{B}\#S\left( h_{1}\right) \right)
km\left( 1_{B}\#h_{2}\right) \\
&=&\left( 1_{B}\#S\left( h_{1}\right) \right) \partial ^{-1}\left( m\right)
\left( k\right) \left( 1_{B}\#h_{2}\right) =\left( \chi _{0}^{h}\left(
M\right) \partial ^{-1}\left( m\right) \right) \left( k\right) ,
\end{eqnarray*}%
\begin{eqnarray*}
\left( \partial ^{0}\chi _{0}^{h}\left( M\right) \left( f\right) \right)
\left( b_{1}\right) &=&\chi _{0}^{h}\left( M\right) \left( f\right) \left(
1\right) b_{1}-b_{1}\chi _{0}^{h}\left( M\right) \left( f\right) \left(
1\right) \\
&=&\left( 1_{B}\#S\left( h_{1}\right) \right) f\left( 1\right) \left(
1_{B}\#h_{2}\right) b_{1}-b_{1}\left( 1_{B}\#S\left( h_{1}\right) \right)
f\left( 1\right) \left( 1_{B}\#h_{2}\right) \\
&=&\left( 1_{B}\#S\left( h_{1}\right) \right) f\left( 1\right) \left(
1_{B}\#h_{2}\right) \left( b_{1}\#1_{H}\right) -\left( b_{1}\#1_{H}\right)
\left( 1_{B}\#S\left( h_{1}\right) \right) f\left( 1\right) \left(
1_{B}\#h_{2}\right) \\
&=&\left( 1_{B}\#S\left( h_{1}\right) \right) f\left( 1\right) \left(
h_{2}b_{1}\#h_{3}\right) -\left( b_{1}\#S\left( h_{1}\right) \right) f\left(
1\right) \left( 1_{B}\#h_{2}\right) \\
&=&\left( 1_{B}\#S\left( h_{1}\right) \right) \left( f\left( 1\right) \left(
h_{2}b_{1}\#1_{H}\right) \right) \left( 1_{B}\#h_{3}\right) -\left(
1_{B}\#S\left( h_{1}\right) \right) \left( \left( h_{2}b_{1}\#1_{H}\right)
f\left( 1\right) \right) \left( 1_{B}\#h_{3}\right) \\
&=&\left( 1_{B}\#S\left( h_{1}\right) \right) \left( f\left( 1\right) \left(
h_{2}b_{1}\right) \right) \left( 1_{B}\#h_{3}\right) -\left( 1_{B}\#S\left(
h_{1}\right) \right) \left( \left( h_{2}b_{1}\right) f\left( 1\right)
\right) \left( 1_{B}\#h_{3}\right) \\
&=&\left( 1_{B}\#S\left( h_{1}\right) \right) \partial ^{0}\left( f\right)
\left( h_{2}b_{1}\right) \left( 1_{B}\#h_{3}\right) =\left( \chi
_{1}^{h}\left( M\right) \partial ^{0}\left( f\right) \right) \left(
b_{1}\right)
\end{eqnarray*}%
\begin{eqnarray*}
&&\left( \partial ^{n}\chi _{n}^{h}\left( M\right) \left( f\right) \right)
\left( b_{1}\otimes b_{2}\otimes \cdots \otimes b_{n+1}\right) \\
&=&\sum\limits_{i=0}^{n+1}\left( -1\right) ^{i}\partial _{i}^{n}\chi
_{n}^{h}\left( M\right) \left( f\right) \left( b_{1}\otimes b_{2}\otimes
\cdots \otimes b_{n+1}\right) \\
&=&\left(
\begin{array}{c}
\chi _{n}^{h}\left( M\right) \left( f\right) \left( b_{1}\otimes
b_{2}\otimes \cdots \otimes b_{n}\right) b_{n+1}-\chi _{n}^{h}\left(
M\right) \left( f\right) \left( b_{1}\otimes b_{2}\cdots \otimes
b_{n-1}\otimes b_{n}b_{n+1}\right) \\
+\chi _{n}^{h}\left( M\right) \left( f\right) \left( b_{1}\otimes
b_{2}\otimes \cdots \otimes b_{n-1}b_{n}\otimes b_{n+1}\right) +\cdots
+\left( -1\right) ^{n+1}b_{1}\chi _{n}^{h}\left( M\right) \left( f\right)
\left( b_{2}\otimes b_{3}\otimes \cdots \otimes b_{n}\otimes b_{n+1}\right)%
\end{array}%
\right) \\
&=&\left(
\begin{array}{c}
\left( 1_{B}\#S\left( h_{1}\right) \right) f\left( h_{2}b_{1}\otimes
h_{3}b_{2}\otimes \cdots \otimes h_{n+1}b_{n}\right) \left(
1_{B}\#h_{n+2}\right) b_{n+1}+ \\
-\left( 1_{B}\#S\left( h_{1}\right) \right) f\left( h_{2}b_{1}\otimes
h_{3}b_{2}\otimes \cdots \otimes h_{n+1}\left( b_{n}b_{n+1}\right) \right)
\left( 1_{B}\#h_{n+2}\right) \\
+\left( 1_{B}\#S\left( h_{1}\right) \right) f\left( h_{2}b_{1}\otimes
h_{3}b_{2}\otimes \cdots \otimes h_{n}\left( b_{n-1}b_{n}\right) \otimes
h_{n+1}b_{n}\right) \left( 1_{B}\#h_{n+2}\right) +\cdots \\
+\left( -1\right) ^{n+1}b_{1}\left( 1_{B}\#S\left( h_{1}\right) \right)
f\left( h_{2}b_{2}\otimes h_{3}b_{3}\otimes \cdots \otimes
h_{n+1}b_{n+1}\right) \left( 1_{B}\#h_{n+2}\right)%
\end{array}%
\right) \\
&=&\left(
\begin{array}{c}
\left( 1_{B}\#S\left( h_{1}\right) \right) f\left( h_{2}b_{1}\otimes
h_{3}b_{2}\otimes \cdots \otimes h_{n+1}b_{n}\right) \left(
1_{B}\#h_{n+2}\right) \left( b_{n+1}\#1_{H}\right) + \\
-\left( 1_{B}\#S\left( h_{1}\right) \right) f\left( h_{2}b_{1}\otimes
h_{3}b_{2}\otimes \cdots \otimes \left( h_{n+1}b_{n}\right) \left(
h_{n+2}b_{n+1}\right) \right) \left( 1_{B}\#h_{n+3}\right) \\
+\left( 1_{B}\#S\left( h_{1}\right) \right) f\left( h_{2}b_{1}\otimes
h_{3}b_{2}\otimes \cdots \otimes \left( h_{n}b_{n-1}\right) \left(
h_{n+1}b_{n}\right) \otimes h_{n+2}b_{n}\right) \left( 1_{B}\#h_{n+3}\right)
+\cdots \\
+\left( -1\right) ^{n+1}\left( b_{1}\#1_{H}\right) \left( 1_{B}\#S\left(
h_{1}\right) \right) f\left( h_{2}b_{2}\otimes h_{3}b_{3}\otimes \cdots
\otimes h_{n+1}b_{n+1}\right) \left( 1_{B}\#h_{n+2}\right)%
\end{array}%
\right) \\
&=&\left(
\begin{array}{c}
\left( 1_{B}\#S\left( h_{1}\right) \right) f\left( h_{2}b_{1}\otimes
h_{3}b_{2}\otimes \cdots \otimes h_{n+1}b_{n}\right) \left(
h_{n+2}b_{n+1}\#h_{n+3}\right) + \\
-\left( 1_{B}\#S\left( h_{1}\right) \right) f\left( h_{2}b_{1}\otimes
h_{3}b_{2}\otimes \cdots \otimes \left( h_{n+1}b_{n}\right) \left(
h_{n+2}b_{n+1}\right) \right) \left( 1_{B}\#h_{n+3}\right) \\
+\left( 1_{B}\#S\left( h_{1}\right) \right) f\left( h_{2}b_{1}\otimes
h_{3}b_{2}\otimes \cdots \otimes \left( h_{n}b_{n-1}\right) \left(
h_{n+1}b_{n}\right) \otimes h_{n+2}b_{n}\right) \left( 1_{B}\#h_{n+3}\right)
+\cdots \\
+\left( -1\right) ^{n+1}\left( b_{1}\#S\left( h_{1}\right) \right) f\left(
h_{2}b_{2}\otimes h_{3}b_{3}\otimes \cdots \otimes h_{n+1}b_{n+1}\right)
\left( 1_{B}\#h_{n+2}\right)%
\end{array}%
\right) \\
&=&\left( 1_{B}\#S\left( h_{1}\right) \right) \left(
\begin{array}{c}
f\left( h_{2}b_{1}\otimes h_{3}b_{2}\otimes \cdots \otimes
h_{n+1}b_{n}\right) \left( h_{n+2}b_{n+1}\#1_{H}\right) \\
-f\left( h_{2}b_{1}\otimes h_{3}b_{2}\otimes \cdots \otimes \left(
h_{n+1}b_{n}\right) \left( h_{n+2}b_{n+1}\right) \right) + \\
+f\left( h_{2}b_{1}\otimes h_{3}b_{2}\otimes \cdots \otimes \left(
h_{n}b_{n-1}\right) \left( h_{n+1}b_{n}\right) \otimes h_{n+2}b_{n}\right)
+\cdots \\
+\left( -1\right) ^{n+1}\left( h_{2}b_{1}\#1_{H}\right) f\left(
h_{3}b_{2}\otimes h_{3}b_{3}\otimes \cdots \otimes h_{n+2}b_{n+1}\right)%
\end{array}%
\right) \left( 1_{B}\#h_{n+3}\right) \\
&=&\left( 1_{B}\#S\left( h_{1}\right) \right) \left(
\begin{array}{c}
f\left( h_{2}b_{1}\otimes h_{3}b_{2}\otimes \cdots \otimes
h_{n+1}b_{n}\right) h_{n+2}b_{n+1} \\
-f\left( h_{2}b_{1}\otimes h_{3}b_{2}\otimes \cdots \otimes \left(
h_{n+1}b_{n}\right) \left( h_{n+2}b_{n+1}\right) \right) + \\
+f\left( h_{2}b_{1}\otimes h_{3}b_{2}\otimes \cdots \otimes \left(
h_{n}b_{n-1}\right) \left( h_{n+1}b_{n}\right) \otimes h_{n+2}b_{n}\right)
+\cdots \\
+\left( -1\right) ^{n+1}h_{2}b_{1}f\left( h_{3}b_{2}\otimes
h_{3}b_{3}\otimes \cdots \otimes h_{n+2}b_{n+1}\right)%
\end{array}%
\right) \left( 1_{B}\#h_{n+3}\right) \\
&=&\left( 1_{B}\#S\left( h_{1}\right) \right) \partial ^{n}\left( f\right)
\left( h_{2}b_{1}\otimes h_{3}b_{2}\otimes \cdots \otimes
h_{n+2}b_{n+1}\right) \left( 1_{B}\#h_{n+3}\right) \\
&=&\left( \chi _{n+1}^{h}\left( M\right) \partial ^{n}\left( f\right)
\right) \left( b_{1}\otimes b_{2}\otimes \cdots \otimes b_{n+1}\right) .
\end{eqnarray*}

Therefore we got a commutative diagram with exact rows
\begin{equation*}
\begin{array}{ccccccccc}
0 & \rightarrow & \mathrm{Hom}_{\Bbbk }\left( B^{\otimes \ast },M\right) &
\overset{\varphi ^{\ast }}{\rightarrow } & \mathrm{Hom}_{\Bbbk }\left(
B^{\otimes \ast },I\right) & \overset{\psi ^{\ast }}{\rightarrow } & \mathrm{%
Hom}_{\Bbbk }\left( B^{\otimes \ast },C\right) & \rightarrow & 0 \\
&  & \chi _{\ast }^{h}\left( M\right) \downarrow &  & \chi _{\ast
}^{h}\left( I\right) \downarrow &  & \chi _{\ast }^{h}\left( C\right)
\downarrow &  &  \\
0 & \rightarrow & \mathrm{Hom}_{\Bbbk }\left( B^{\otimes \ast },M\right) &
\overset{\varphi ^{\ast }}{\rightarrow } & \mathrm{Hom}_{\Bbbk }\left(
B^{\otimes \ast },I\right) & \overset{\psi ^{\ast }}{\rightarrow } & \mathrm{%
Hom}_{\Bbbk }\left( B^{\otimes \ast },C\right) & \rightarrow & 0%
\end{array}%
\end{equation*}%
This yields a commutative square for each $n\in \N_0$%
\begin{equation*}
\begin{array}{ccc}
\mathrm{H}^{n-1}\left( B,C\right) & \overset{\delta ^{n-1}}{\rightarrow } &
\mathrm{H}^{n}\left( B,M\right) \\
\varrho _{n-1}^{h}\left( C\right) \downarrow &  & \downarrow \varrho
_{n}^{h}\left( M\right) \\
\mathrm{H}^{n-1}\left( B,C\right) & \overset{\delta ^{n-1}}{\rightarrow } &
\mathrm{H}^{n}\left( B,M\right)%
\end{array}%
\end{equation*}%
where%
\begin{equation*}
\varrho _{n}^{h}\left( M\right) \left( \left[ f\right] \right) :=\left[ \chi
_{n}^{h}\left( M\right) \left( f\right) \right]
\end{equation*}%
see e.g. \cite[Proposition 0.4, page 6]{Br} in its dual form. Coming back to
our short exact sequence
\begin{equation*}
0\rightarrow M\overset{s}{\rightarrow }I\overset{p}{\rightarrow }%
C\rightarrow 0,
\end{equation*}%
from $\mathrm{H}^{n}\left( B,I\right) =0$ for every $n\geq 1$ (as observed
above) we get a commutative diagram%
\begin{equation*}
\begin{array}{ccccccc}
\mathrm{H}^{n-1}\left( B,I\right) & \overset{\mathrm{H}^{n-1}\left(
B,p\right) }{\rightarrow } & \mathrm{H}^{n-1}\left( B,C\right) & \overset{%
\delta ^{n-1}}{\rightarrow } & \mathrm{H}^{n}\left( B,M\right) & \overset{%
\mathrm{H}^{n-1}\left( B,s\right) }{\rightarrow } & \mathrm{H}^{n}\left(
B,I\right) =0 \\
\varrho _{n-1}^{h}\left( I\right) \downarrow &  & \varrho _{n-1}^{h}\left(
C\right) \downarrow &  & \downarrow \varrho _{n}^{h}\left( M\right) &  &  \\
\mathrm{H}^{n-1}\left( B,I\right) & \overset{\mathrm{H}^{n-1}\left(
B,p\right) }{\rightarrow } & \mathrm{H}^{n-1}\left( B,C\right) & \overset{%
\delta ^{n-1}}{\rightarrow } & \mathrm{H}^{n}\left( B,M\right) &  &
\end{array}%
\end{equation*}%
so that the map $\varrho _{n}^{h}\left( M\right) $ is uniquely determined by
the universal property of the cokernel in the upper sequence. This map is
the one used in \cite[Proposition 2.4]{St} which is constructed using \cite[%
Theorem 7.5, page 78]{Br}. Summing up, the right $H$-module structure of $%
\mathrm{H}^{n}\left( B,M\right) $ is given by $\left[ f\right] h:=\varrho
_{n}^{h}\left( M\right) \left( \left[ f\right] \right) =\left[ \chi
_{n}^{h}\left( M\right) \left( f\right) \right] .$
\end{invisible}
\end{proof}

\begin{remark}\label{rem:DV}
 Proposition \ref{pro:Dragos} in the particular case when $M=\Bbbk$ and $B$ is finite-dimensional is \cite[Theorem 2.17]{SVay}. Note that in the notations therein, one has $E(B)=\oplus_{n \in \N_0}E_n(B,\Bbbk)$ where $E_n(B,\Bbbk)=\mathrm{Ext}^n_B(\Bbbk,\Bbbk)\cong\mathrm{H}^n(B,\Bbbk)$. The latter isomorphism is \cite[Corollary 4.4, page 170]{CE}.
\end{remark}

Let $H$ be a Hopf algebra and let $B$ be a bialgebra in the braided category
${_{H}^{H}\mathcal{YD}}$. Denote by $A:=B\#H$ the Radford-Majid bosonization
of $B$ by $H,$ see e.g. \cite[Theorem 1]{Ra-TheStruct}. Note that $A$ is
endowed with an algebra map $\varepsilon _{A}:A\rightarrow \Bbbk $ defined
by $\varepsilon _{A}\left( b\#h\right) =\varepsilon _{B}\left( b\right)
\varepsilon _{H}\left( h\right) $ so that we can regard $\Bbbk $ as an $A$%
-bimodule via $\varepsilon _{A}.$ Then we can consider $\mathrm{H}^{n}\left(
B,\Bbbk \right) $ as an $H$-bimodule as follows. Its structure of left $H$%
-module is given via $\varepsilon _{H}$ and its structure of right $H$%
-module is defined, for every $f\in \mathrm{Z}^{n}\left( B,\Bbbk \right) $
and $h\in H,$ by setting
\begin{equation*}
\left[ f\right] h:=\left[ fh\right] ,
\end{equation*}%
where $\left( fh\right) \left( z\right) =f\left( hz\right) ,$ for every $%
z\in B^{\otimes n}$. The latter is the usual right $H$-module structure of $%
\mathrm{Hom}_{\Bbbk }\left( B^{\otimes n},\Bbbk \right) .$ Indeed, for every
$n\geq -1$, the vector space $\mathrm{Hom}_{\Bbbk }\left( B^{\otimes
n},\Bbbk \right) $ is an $H$-bimodule with respect to this right $H$-module
structure and the left one induced by $\varepsilon _{H}.$

\begin{corollary}
\label{coro:K}Let $H$ be a semisimple Hopf algebra and let $B$ be a
bialgebra in the braided category ${_{H}^{H}\mathcal{YD}}$. Set $A:=B\#H.$
Then, for each $n\in \N_0$%
\begin{equation*}
\mathrm{H}^{n}\left( B\#H,\Bbbk \right) \cong \mathrm{H}^{n}\left( B,\Bbbk
\right) ^{H}
\end{equation*}%
and the differential $\partial ^{n}:\mathrm{Hom}_{\Bbbk }\left( B^{\otimes
n},\Bbbk \right) \rightarrow \mathrm{Hom}_{\Bbbk }\left( B^{\otimes \left(
n+1\right) },\Bbbk \right) $ of the usual Hochschild cohomology is $H$%
-bilinear.
\end{corollary}

\begin{proof}
In the particular case $M=\Bbbk ,$ the right module $H$-structure used in
Proposition \ref{pro:Dragos} simplifies as follows. It is defined, for every
$f\in \mathrm{Z}^{n}\left( B,\Bbbk \right) $ and $h\in H,$ by setting
\begin{equation*}
\left[ f\right] h:=\left[ \chi _{n}^{h}\left( \Bbbk \right) \left( f\right) %
\right]
\end{equation*}%
where, for every $k\in \Bbbk ,b_{1},\ldots ,b_{n}\in B,$ we set%
\begin{eqnarray*}
\chi _{0}^{h}\left( \Bbbk \right) \left( f\right) \left( k\right)
&:&=\varepsilon _{H}\left( h\right) f\left( k\right) \text{ for }n=0\text{
while and for }n\geq 1 \\
\chi _{n}^{h}\left( \Bbbk \right) \left( f\right) \left( b_{1}\otimes
b_{2}\otimes \cdots \otimes b_{n}\right) &:&=f\left( h_{1}b_{1}\otimes
h_{2}b_{2}\otimes \cdots \otimes h_{n}b_{n}\right) .
\end{eqnarray*}%
More concisely $\chi _{n}^{h}\left( \Bbbk \right) \left( f\right) \left(
z\right) =f\left( hz\right) $ for every $n\in \N_0$ and $z\in
B^{\otimes n}$ i.e. $\left[ f\right] h:=\left[ fh\right] $ where $fh:=\chi
_{n}^{h}\left( \Bbbk \right) \left( f\right) .$

Now consider the differential $\partial ^{n}:\mathrm{Hom}_{\Bbbk }\left(
B^{\otimes n},\Bbbk \right) \rightarrow \mathrm{Hom}_{\Bbbk }\left(
B^{\otimes \left( n+1\right) },\Bbbk \right) $ of the usual Hochschild
cohomology. Note that for each $n\in \N_0$, $\mathrm{Hom}_{\Bbbk
}\left( B^{\otimes n},\Bbbk \right) $ is regarded as a bimodule over $H$
using the left $H$-module structures of its arguments. By (\ref%
{form:PartialChi}), we have
\begin{equation*}
\partial ^{n}\chi _{n}^{h}\left( \Bbbk \right) \left( f\right) =\chi
_{n+1}^{h}\left( \Bbbk \right) \partial ^{n}\left( f\right)
\end{equation*}%
Since $\chi _{n}^{h}\left( \Bbbk \right) \left( f\right) =fh$, the last
displayed equality becomes $\partial ^{n}\left( fh\right) =\partial
^{n}\left( f\right) h\,$for every $n\in \N_0$. Thus $\partial ^{n}$ is
right $H$-linear. Since $hf=\varepsilon _{H}\left( h\right) f$ for every $%
f\in \mathrm{Hom}_{\Bbbk }\left( B^{\otimes n},\Bbbk \right) ,h\in H,$ we
get that $\partial ^{n}$ is also left $H$-linear whence $H$-bilinear.
\end{proof}

\begin{remark}
Note that, in the context of the proof of \cite[Proposition 5.1]{EG}, one
has
\begin{equation*}
\mathrm{H}^{3}\left( \mathcal{B}\left( V\right) \#\mathbb{C}\left[ \mathbb{%
\mathbb{Z}}_{p}\right] ,\mathbb{C}\right) \cong \mathrm{H}^{3}\left(
\mathcal{B}\left( V\right) ,\mathbb{C}\right) ^{\mathbb{\mathbb{Z}}_{p}}.
\end{equation*}
This is a particular case of Corollary \ref{coro:K} where $H=\mathbb{C}\left[
\mathbb{\mathbb{Z}}_{p}\right] ,$\thinspace $V\in {_{H}^{H}\mathcal{YD}}$
and $B=\mathcal{B}\left( V\right) $.
\end{remark}

\begin{proposition}
\label{pro:Prerad}Let $\mathcal{C}$ and $\mathcal{D}$ be abelian categories.
Let $r,\omega :\mathcal{C}\rightarrow \mathcal{D}$ be exact functors such
that $r$ is a subfunctor of $\omega $ i.e. there is a natural transformation
$\eta :r\rightarrow \omega $ which is a monomorphism when evaluated on
objects. If $X$ is a subobject of $Y$ then $r\left( X\right) =\omega \left(
X\right) \cap r\left( Y\right) .$ Moreover, for every morphism $%
f:X\rightarrow Y$ in $\mathcal{C}$ one has%
\begin{eqnarray*}
\mathrm{ker}\left( r\left( f\right) \right) &=&r\left( \mathrm{ker}\left(
f\right) \right) =\omega \left( \mathrm{ker}\left( f\right) \right) \cap
r\left( X\right) =\mathrm{ker}\left( \omega \left( f\right) \right) \cap
r\left( X\right) , \\
\mathrm{Im}\left( r\left( f\right) \right) &=&\mathrm{Im}\left( \omega
\left( f\right) \right) \cap r\left( Y\right) =r\left( \mathrm{Im}\left(
f\right) \right) .
\end{eqnarray*}
\end{proposition}

\begin{proof}
The proof is similar to \cite[Proposition 1.7, page 138]{Sten}.

\begin{invisible}
Consider an exact sequence $0\rightarrow X\overset{s}{\rightarrow }Y\overset{%
p}{\rightarrow }Y/X\rightarrow 0.$ By exactness we get a commutative diagram
as follows
\begin{equation*}
\begin{array}{ccccccccc}
&  &  &  & 0 &  &  &  &  \\
&  &  &  & \downarrow &  &  &  &  \\
0 & \rightarrow & r\left( X\right) & \overset{r\left( s\right) }{\rightarrow
} & r\left( Y\right) & \overset{r\left( p\right) }{\rightarrow } & r\left(
Y/X\right) & \rightarrow & 0 \\
&  & \downarrow \eta X &  & \downarrow \eta Y &  & \downarrow \eta \left(
Y/X\right) &  &  \\
0 & \rightarrow & \omega \left( X\right) & \overset{\omega \left( s\right) }{%
\rightarrow } & \omega \left( Y\right) & \overset{\omega \left( p\right) }{%
\rightarrow } & \omega \left( Y\right) /\omega \left( X\right) \cong \omega
\left( Y/X\right) & \rightarrow & 0 \\
&  &  &  & \downarrow \tau &  &  &  &  \\
&  &  &  & \omega \left( Y\right) /r\left( Y\right) &  &  &  &  \\
&  &  &  & \downarrow &  &  &  &  \\
&  &  &  & 0 &  &  &  &
\end{array}%
\end{equation*}%
where all rows and columns are exact sequences. Let us prove that $\left(
r\left( X\right) ,\eta Y\circ r\left( s\right) \right) =\omega \left(
X\right) \cap r\left( Y\right) .$ We have%
\begin{equation*}
\omega \left( p\right) \circ \left( \eta Y\circ r\left( s\right) \right)
=\eta \left( Y/X\right) \circ r\left( p\right) \circ r\left( s\right)
=0=\tau \circ \left( \eta Y\circ r\left( s\right) \right) .
\end{equation*}%
Let $\alpha :Z\rightarrow \omega \left( Y\right) $ be a morphism such that $%
\tau \circ \alpha =0=\omega \left( p\right) \circ \alpha .$ The first
equality entails that there exists a morphism $\alpha ^{\prime
}:Z\rightarrow r\left( Y\right) $ such that $\eta Y\circ \alpha ^{\prime
}=\alpha .$ We have%
\begin{equation*}
\eta \left( Y/X\right) \circ r\left( p\right) \circ \alpha ^{\prime }=\omega
\left( p\right) \circ \eta Y\circ \alpha ^{\prime }=\omega \left( p\right)
\circ \alpha =0.
\end{equation*}%
Since $\eta \left( Y/X\right) $ is a monomorphism, we deduce that $r\left(
p\right) \circ \alpha ^{\prime }=0$ so that there is a morphism $\alpha
^{\prime \prime }:Z\rightarrow r\left( X\right) $ such that

$r\left( s\right) \circ \alpha ^{\prime \prime }=\alpha ^{\prime }.$ Thus%
\begin{equation*}
\left( \eta Y\circ r\left( s\right) \right) \circ \alpha ^{\prime \prime
}=\eta Y\circ \alpha ^{\prime }=\alpha .
\end{equation*}%
Since $\eta Y\circ r\left( s\right) $ is a monomorphism, we deduce that $%
\alpha ^{\prime \prime }$ is uniquely determined by the equality above. This
proves that $\left( r\left( X\right) ,\eta Y\circ r\left( s\right) \right) =%
\mathrm{ker}\left( \Delta \left( \tau ,\omega \left( p\right) \right)
\right) =\omega \left( X\right) \cap r\left( Y\right) $ where $\Delta \left(
\tau ,\omega \left( p\right) \right) $ denotes the diagonal morphism of $%
\left( \tau ,\omega \left( p\right) \right) $. Let $f:X\rightarrow Y$ be a
morphism in $\mathcal{C}$. We have that $r\left( \mathrm{ker}\left( f\right)
\right) =\omega \left( \mathrm{ker}\left( f\right) \right) \cap r\left(
X\right) .$ By exactness we get $r\left( \mathrm{ker}\left( f\right) \right)
=\mathrm{ker}\left( r\left( f\right) \right) $ and $\omega \left( \mathrm{ker%
}\left( f\right) \right) =\mathrm{ker}\left( \omega \left( f\right) \right)
. $ Thus
\begin{eqnarray*}
\mathrm{Im}\left( r\left( f\right) \right) &=&\mathrm{kerCo\mathrm{ker}}%
\left( r\left( f\right) \right) =\mathrm{ker}\left( r\left( \mathrm{Co%
\mathrm{ker}}f\right) \right) \\
&=&r\left( \mathrm{ker}\left( \mathrm{Co\mathrm{ker}}f\right) \right)
=\omega \left( \mathrm{ker}\left( \mathrm{Co\mathrm{ker}}f\right) \right)
\cap r\left( Y\right) \\
&=&\mathrm{ker}\left( \mathrm{Co\mathrm{ker}}\left( \omega \left( f\right)
\right) \right) \cap r\left( Y\right) =\mathrm{Im}\left( \omega \left(
f\right) \right) \cap r\left( Y\right) .
\end{eqnarray*}%
Note that from the above computation one gets $\mathrm{Im}\left( r\left(
f\right) \right) =r\left( \mathrm{ker}\left( \mathrm{Co\mathrm{ker}}f\right)
\right) =r\left( \mathrm{Im}\left( f\right) \right) $.
\end{invisible}
\end{proof}

\begin{remark}
\label{rem:ExactInv}From Corollary \ref{coro:K}, we have%
\begin{eqnarray*}
\mathrm{H}^{n}\left( B,\Bbbk \right) ^{H} &=&\left\{ \left[ f\right] \mid
f\in \mathrm{Z}^{n}\left( B,\Bbbk \right) ,\varepsilon _{H}\left( h\right) %
\left[ f\right] =\left[ f\right] h,\text{ for every }h\in H\right\} \\
&=&\left\{ \left[ f\right] \mid f\in \mathrm{Z}^{n}\left( B,\Bbbk \right) ,%
\left[ \varepsilon _{H}\left( h\right) f\right] =\left[ fh\right] ,\text{
for every }h\in H\right\}
\end{eqnarray*}%
where, for every $z\in B^{\otimes n}$, we have%
\begin{equation*}
\left( fh\right) \left( z\right) =f\left( hz\right) .
\end{equation*}%
Note that, for any $H$-bimodule $M$ one has
\begin{equation*}
\mathrm{Hom}_{H,H}\left( H,M\right) \cong M^{H}=\left\{ m\in M\mid hm=mh,%
\text{ for every }h\in H\right\} .
\end{equation*}%
Note also that $H$ is a separable $\Bbbk $-algebra whence it is projective
in the category of $H$-bimodules. As a consequence $\mathrm{Hom}_{H,H}\left(
H,{-}\right) \cong \left( {-}\right) ^{H}:{_{H}}\mathfrak{M}_{H}\rightarrow
\mathfrak{M}$ is an exact functor (here ${_{H}}\mathfrak{M}_{H}$ is the
category of $H$-bimodules and $\mathfrak{M}$ the category of $\Bbbk $-vector
spaces). By Proposition \ref{pro:Prerad} applied to the case when $r:=\left(
{-}\right) ^{H}:{_{H}}\mathfrak{M}_{H}\rightarrow \mathfrak{M}$ and $\omega $
is the forgetful functor, for every morphism $f:X\rightarrow Y$ of $H$%
-bimodules one has%
\begin{equation*}
\mathrm{ker}\left( f^{H}\right) =\mathrm{ker}\left( f\right) \cap
X^{H}=\left( \mathrm{ker}\left( f\right) \right) ^{H}\qquad \text{and}\qquad
\mathrm{Im}\left( f^{H}\right) =\mathrm{Im}\left( f\right) \cap Y^{H}=\left(
\mathrm{Im}\left( f\right) \right) ^{H}.
\end{equation*}%
Still by Corollary \ref{coro:K}, we know that the differential $\partial
^{n}:\mathrm{Hom}_{\Bbbk }\left( B^{\otimes n},\Bbbk \right) \longrightarrow
\mathrm{Hom}_{\Bbbk }\left( B^{\otimes \left( n+1\right) },\Bbbk \right) $
of the usual Hochschild cohomology is $H$-bilinear. Thus we can apply the
argument above to get%
\begin{eqnarray*}
\mathrm{ker}\left( \left( \partial ^{n}\right) ^{H}\right) &=&\mathrm{ker}%
\left( \partial ^{n}\right) \cap \mathrm{Hom}_{\Bbbk }\left( B^{\otimes
n},\Bbbk \right) ^{H}=\left( \mathrm{ker}\left( \partial ^{n}\right) \right)
^{H}\qquad \text{and}\qquad \\
\mathrm{Im}\left( \left( \partial ^{n-1}\right) ^{H}\right) &=&\mathrm{Im}%
\left( \partial ^{n-1}\right) \cap \mathrm{Hom}_{\Bbbk }\left( B^{\otimes
n},\Bbbk \right) ^{H}=\left( \mathrm{Im}\left( \partial ^{n-1}\right)
\right) ^{H}.
\end{eqnarray*}%
Now $\mathrm{Hom}_{\Bbbk }\left( B^{\otimes n},\Bbbk \right) ^{H}=\mathrm{Hom%
}_{H,-}\left( B^{\otimes n},\Bbbk \right) $ so that we get%
\begin{eqnarray*}
\mathrm{Z}_{H\text{-}\mathrm{Mod}}^{n}\left( B,\Bbbk \right) &=&\mathrm{Z}%
^{n}\left( B,\Bbbk \right) \cap \mathrm{Hom}_{H,-}\left( B^{\otimes n},\Bbbk
\right) =\mathrm{Z}^{n}\left( B,\Bbbk \right) ^{H}\qquad \text{and}\qquad \\
\mathrm{B}_{H\text{-}\mathrm{Mod}}^{n}\left( B,\Bbbk \right) &=&\mathrm{B}%
^{n}\left( B,\Bbbk \right) \cap \mathrm{Hom}_{H,-}\left( B^{\otimes n},\Bbbk
\right) =\mathrm{B}^{n}\left( B,\Bbbk \right) ^{H}.
\end{eqnarray*}%
where $\mathrm{Z}_{H\text{-}\mathrm{Mod}}^{n}\left( B,\Bbbk \right) $ and $%
\mathrm{B}_{H\text{-}\mathrm{Mod}}^{n}\left( B,\Bbbk \right) $ denotes the
the abelian groups of $n$-cocycles, of $n$-coboundaries for the cohomology
of the algebra $B$ with coefficients in $\Bbbk $ computed in the monoidal
category $H$-$\mathrm{Mod}$ of left $H$-modules. The corresponding $n$-th
Hochschild cohomology group is
\begin{equation*}
\mathrm{H}_{H\text{-}\mathrm{Mod}}^{n}\left( B,\Bbbk \right) :=\frac{\mathrm{%
Z}_{H\text{-}\mathrm{Mod}}^{n}\left( B,\Bbbk \right) }{\mathrm{B}_{H\text{-}%
\mathrm{Mod}}^{n}\left( B,\Bbbk \right) }=\frac{\mathrm{Z}^{n}\left( B,\Bbbk
\right) ^{H}}{\mathrm{B}^{n}\left( B,\Bbbk \right) ^{H}}\cong \left( \frac{%
\mathrm{Z}^{n}\left( B,\Bbbk \right) }{\mathrm{B}^{n}\left( B,\Bbbk \right) }%
\right) ^{H}=\mathrm{H}^{n}\left( B,\Bbbk \right) ^{H}.
\end{equation*}
\end{remark}

Denote by $D\left( H\right) $ the Drinfeld double, see e.g. the first
structure of \cite[Theorem 7.1.1]{Maj}.

\begin{proposition}
\label{pro:D(H)}In the setting of Corollary \ref{coro:K} assume that $H$ is
also cosemisimple. Then, for $n\in \N_0$
\begin{equation*}
\mathrm{Z}_{{\mathcal{YD}}}^{n}\left( B,\Bbbk \right) =\mathrm{Z}^{n}\left(
B,\Bbbk \right) ^{D(H)},\quad \mathrm{B}_{{\mathcal{YD}}}^{n}\left( B,\Bbbk
\right) =\mathrm{B}^{n}\left( B,\Bbbk \right) ^{D(H)}\quad \text{and}\quad
\mathrm{H}_{{\mathcal{YD}}}^{n}\left( B,\Bbbk \right) \cong \mathrm{H}%
^{n}\left( B,\Bbbk \right) ^{D(H)}.
\end{equation*}%
where $\mathrm{Z}^{n}\left( B,\Bbbk \right) $ and $\mathrm{B}^{n}\left(
B,\Bbbk \right) $ are regarded as $D\left( H\right) $-subbimodules of $%
\mathrm{Hom}_{\Bbbk }\left( B^{\otimes n},\Bbbk \right) $ whose structure is
induced by the left $D\left( H\right) $-module structures of its arguments.

Moreover $\mathrm{H}^{n}\left( B,\Bbbk \right) ^{D(H)}$ is a subspace of $%
\mathrm{H}^{n}\left( B,\Bbbk \right) ^{H}.$
\end{proposition}

\begin{proof}
For shortness, in this proof, we denote $D(H)$ by $D$. Consider the analogue
of the standard complex as in Remark \ref{rem:HochYD}
\begin{equation*}
\xymatrixrowsep{25pt}\xymatrixcolsep{1cm}\xymatrix{\yd( \Bbbk ,\Bbbk)
\ar[r]^{\partial ^0}& \yd(B ,\Bbbk) \ar[r]^{\partial ^1}&\yd( B^{\otimes2}
,\Bbbk) \ar[r]^-{\partial ^2}&\cdots}
\end{equation*}%
where $\partial ^{n}$ is induced by the differential $\partial ^{n}:\mathrm{%
Hom}_{\Bbbk }\left( B^{\otimes n},\Bbbk \right) \longrightarrow \mathrm{Hom}%
_{\Bbbk }\left( B^{\otimes \left( n+1\right) },\Bbbk \right) $ of the
ordinary Hochschild cohomology. Now, since $H$ is semisimple, it is
finite-dimensional (whence it has bijective antipode) so that, by a result
essentially due to Majid (see \cite[Proposition 10.6.16]{Mo}) and by \cite[%
Proposition 6]{RT}, we get a category isomorphism ${_{H}^{H}\mathcal{YD}}%
\cong {{_{D}}}\mathfrak{M}$. Thus the complex above can be rewritten as
follows%
\begin{equation*}
\xymatrixrowsep{25pt}\xymatrixcolsep{0,5cm}\xymatrix{\mathrm{Hom}_{D,-}(
\Bbbk ,\Bbbk) \ar[r]^{\partial ^0}& \mathrm{Hom}_{D,-}(B ,\Bbbk)
\ar[r]^{\partial ^1}& \mathrm{Hom}_{D,-}( B^{\otimes2} ,\Bbbk)
\ar[r]^-{\partial ^2}&\cdots}
\end{equation*}%
Now, since, for each $n\in \N_0$, we have $\mathrm{Hom}_{D,-}\left(
B^{\otimes n},\Bbbk \right) =\mathrm{Hom}_{\Bbbk }\left( B^{\otimes n},\Bbbk
\right) ^{D},$ we obtain the complex%
\begin{equation*}
\xymatrixrowsep{25pt}\xymatrixcolsep{0,5cm}\xymatrix{\mathrm{Hom}_\Bbbk(
\Bbbk ,\Bbbk)^{D} \ar[r]^{\partial ^0}& \mathrm{Hom}_\Bbbk(B ,\Bbbk)^{D}
\ar[r]^{\partial ^1}& \mathrm{Hom}_\Bbbk( B^{\otimes2} ,\Bbbk)^{D}
\ar[r]^-{\partial ^2}&\cdots}
\end{equation*}%
We will write $\left( \partial ^{n}\right) ^{D}$ instead of $\partial ^{n}$
when we would like to stress that the map considered is the one induced on
invariants. Thus we will write equivalently%
\begin{equation*}
\xymatrixrowsep{25pt}\xymatrixcolsep{45pt}\xymatrix{\mathrm{Hom}_\Bbbk(
\Bbbk ,\Bbbk)^{D} \ar[r]^{(\partial ^0)^{D}}& \mathrm{Hom}_\Bbbk(B
,\Bbbk)^{D} \ar[r]^{(\partial ^1)^{D}}& \mathrm{Hom}_\Bbbk( B^{\otimes2}
,\Bbbk)^{D} \ar[r]^-{(\partial ^2)^{D}}&\cdots}
\end{equation*}

Now, assume $H$ is also cosemisimple. Since $H$ is both semisimple and
cosemisimple, by \cite[Proposition 7]{Ra} the Hopf algebra $D$ is semisimple
as an algebra. Thus, as in Remark \ref{rem:ExactInv} in case of $H$, the
functor $\left( -\right) ^{D}:{_{D}}\mathfrak{M}_{D}\rightarrow \mathfrak{M}$
is exact (here ${_{D}}\mathfrak{M}_{D}$ is the category of $D$-bimodules and
$\mathfrak{M}$ the category of $\Bbbk $-vector spaces). By Proposition \ref%
{pro:Prerad} applied to the case when $r:=\left( {-}\right) ^{D}:{_{D}}%
\mathfrak{M}_{D}\rightarrow \mathfrak{M}$ and $\omega $ is the forgetful
functor, for every morphism $f:X\rightarrow Y$ of $D$-bimodules one has%
\begin{equation*}
\mathrm{ker}\left( f^{D}\right) =\mathrm{ker}\left( f\right) \cap
X^{D}=\left( \mathrm{ker}\left( f\right) \right) ^{D}\qquad \text{and}\qquad
\mathrm{Im}\left( f^{D}\right) =\mathrm{Im}\left( f\right) \cap Y^{D}=\left(
\mathrm{Im}\left( f\right) \right) ^{D}.
\end{equation*}%
In particular we get%
\begin{eqnarray*}
\mathrm{ker}\left( \left( \partial ^{n}\right) ^{D}\right) &=&\mathrm{ker}%
\left( \partial ^{n}\right) \cap \mathrm{Hom}_{\Bbbk }\left( B^{\otimes
n},\Bbbk \right) ^{D}=\mathrm{ker}\left( \partial ^{n}\right) ^{D}\qquad
\text{and}\qquad \\
\mathrm{Im}\left( \left( \partial ^{n-1}\right) ^{D}\right) &=&\mathrm{Im}%
\left( \partial ^{n-1}\right) \cap \mathrm{Hom}_{\Bbbk }\left( B^{\otimes
n},\Bbbk \right) ^{D}=\mathrm{Im}\left( \partial ^{n-1}\right) ^{D}
\end{eqnarray*}%
and hence%
\begin{eqnarray*}
\mathrm{Z}_{{\mathcal{YD}}}^{n}\left( B,\Bbbk \right) &=&\mathrm{Z}%
^{n}\left( B,\Bbbk \right) \cap \mathrm{Hom}_{D,-}\left( B^{\otimes n},\Bbbk
\right) =\mathrm{Z}^{n}\left( B,\Bbbk \right) ^{D}\qquad \text{and}\qquad \\
\mathrm{B}_{{\mathcal{YD}}}^{n}\left( B,\Bbbk \right) &=&\mathrm{B}%
^{n}\left( B,\Bbbk \right) \cap \mathrm{Hom}_{D,-}\left( B^{\otimes n},\Bbbk
\right) =\mathrm{B}^{n}\left( B,\Bbbk \right) ^{D}
\end{eqnarray*}%
Then we obtain%
\begin{equation*}
\mathrm{H}_{{\mathcal{YD}}}^{n}\left( B,\Bbbk \right) =\frac{\mathrm{Z}_{{%
\mathcal{YD}}}^{n}\left( B,\Bbbk \right) }{\mathrm{B}_{{\mathcal{YD}}%
}^{n}\left( B,\Bbbk \right) }=\frac{\mathrm{Z}^{n}\left( B,\Bbbk \right) ^{D}%
}{\mathrm{B}^{n}\left( B,\Bbbk \right) ^{D}}\cong \mathrm{H}^{n}\left(
B,\Bbbk \right) ^{D}.
\end{equation*}%
Let us prove the last part of the statement. The correspondence between the
left $D$-module structure and the structure of Yetter-Drinfeld module over $%
H $ is written explicitly in \cite[Proposition 7.1.6]{Maj}. In particular $%
D=H^{\ast }\otimes H$ and given $V\in {_{H}^{H}\mathcal{YD}}$, the two
structures are related by the following equality $\left( f\otimes h\right)
\rhd v=f\left( \left( h\rhd v\right) _{-1}\right) \left( h\rhd v\right) _{0}$
for every $f\in H^{\ast },h\in H,v\in V.$ Thus $\left( \varepsilon
_{H}\otimes h\right) \rhd v=h\rhd v.$ Moreover $H$ is a Hopf subalgebra of $%
D $ via $h\mapsto \varepsilon _{H}\otimes h,$ where $D$ is considered with
the first structure of \cite[Theorem 7.1.1]{Maj}. Since the $D$-bimodule
structure of $\mathrm{H}^{n}\left( B,\Bbbk \right) $ is induced by the one
of $\mathrm{Hom}_{\Bbbk }\left( B^{\otimes n},\Bbbk \right) $ which comes
from the left $D$-module structures of its arguments and similarly for the $%
H $-bimodule structure of $\mathrm{H}^{n}\left( B,\Bbbk \right) ,$ we deduce
that $\mathrm{H}^{n}\left( B,\Bbbk \right) ^{D}$ is a subspace of $\mathrm{H}%
^{n}\left( B,\Bbbk \right) ^{H}.$
\end{proof}

\begin{example}
In the setting of the proof of \cite[Theorem 4.1.3]{An-Basic}, a Nichols
algebra $\mathcal{B}\left( V\right) $ such that $\mathrm{H}^{3}\left(
\mathcal{B}\left( V\right) ,\Bbbk \right) ^{\mathbb{Z}_{m}}=0$ is considered
where $\Bbbk $ is a field of characteristic zero. By Proposition \ref%
{pro:D(H)} applied in the case $H=\Bbbk \mathbb{Z}_{m}$ and $B=\mathcal{B}%
\left( V\right) ,$ we have that $\mathrm{H}_{{\mathcal{YD}}}^{3}\left(
\mathcal{B}\left( V\right) ,\Bbbk \right) \cong \mathrm{H}^{3}\left(
\mathcal{B}\left( V\right) ,\Bbbk \right) ^{D(H)}$ is a subspace of $\mathrm{%
H}^{3}\left( \mathcal{B}\left( V\right) ,\Bbbk \right) ^{H}=\mathrm{H}%
^{3}\left( \mathcal{B}\left( V\right) ,\Bbbk \right) ^{\mathbb{Z}_{m}}=0.$
Thus we get $\mathrm{H}_{{\mathcal{YD}}}^{3}\left( \mathcal{B}\left(
V\right) ,\Bbbk \right) =0.$ Therefore, in view of Theorem \ref{teo:GelakiYD}%
, if $\left( Q,m,u,\Delta ,\varepsilon ,\omega \right) $ is a f.d. connected
coquasi-bialgebra in ${_{H}^{H}\mathcal{YD}}$ such that $\mathrm{gr}Q\cong
\mathcal{B}\left( V\right) $ (as above) as augmented algebras in ${_{H}^{H}%
\mathcal{YD}}$ (the counit must be the same in order to have the same
Yetter-Drinfeld module structure on $\Bbbk $), then we can conclude that $Q$
is gauge equivalent to a connected bialgebra in ${_{H}^{H}\mathcal{YD}}$.
\end{example}

\begin{remark}
Let $A$ be a finite-dimensional coquasi-bialgebra with the dual Chevalley
property i.e. the coradical $H$ of $A$ is a coquasi-subbialgebra of $A$ (in
particular $H$ is cosemisimple). Assume the coquasi-bialgebra structure of $%
H $ has trivial reassociator (i.e. it is an ordinary bialgebra) and also
assume it has an antipode (i.e. it is a Hopf algebra). Then, by \cite[%
Corollary 6.4]{AP}, $\mathrm{gr}A$ is isomorphic to $R\#H$ as a
coquasi-bialgebra, where $R$ is a suitable connected bialgebra in ${_{H}^{H}%
\mathcal{YD}}$. Note that $R\#H$ is the usual Radford-Majid bosonization as $%
H$ has trivial reassociator, see \cite[Definition 5.4]{AP}. Hence we can
compute
\begin{equation*}
\mathrm{H}^{3}\left( \mathrm{gr}A,\Bbbk \right) =\mathrm{H}^{3}\left(
R\#H,\Bbbk \right) .
\end{equation*}%
Assume further that $H$ is semisimple. Then, by Corollary \ref{coro:K}, we
have
\begin{equation*}
\mathrm{H}^{n}\left( R\#H,\Bbbk \right) \cong \mathrm{H}^{n}\left( R,\Bbbk
\right) ^{H}
\end{equation*}%
so that $\mathrm{H}^{3}\left( \mathrm{gr}A,\Bbbk \right) \cong \mathrm{H}%
^{3}\left( R,\Bbbk \right) ^{H}.$ Thus, if $\mathrm{H}^{3}\left( R,\Bbbk
\right) ^{H}=0,$ one gets $\mathrm{H}^{3}\left( \mathrm{gr}A,\Bbbk \right)
=0 $ which is the analogue of the condition \cite[Proposition 2.3]{EG} (note
that our $A$ is the dual of the one considered therein) which guarantees
that $A$ is gauge equivalent to an ordinary Hopf algebra, if $A$ has an a
quasi-antipode and $\Bbbk =\mathbb{C}$. Next we will give another approach
to arrive at the same conclusion but just requiring $\mathrm{H}_{{\mathcal{YD%
}}}^{3}\left( R,\Bbbk \right) =0$. Note that a priori $\mathrm{H}_{{\mathcal{%
YD}}}^{3}\left( R,\Bbbk \right) \cong \mathrm{H}^{3}\left( R,\Bbbk \right)
^{D\left( H\right) }$ is smaller than $\mathrm{H}^{3}\left( R,\Bbbk \right)
^{H}$.
\end{remark}

\section{Dual Chevalley}\label{sec:5}

The main aim of this section is to prove Theorem \ref{teo:main}. Let $A$ be
a Hopf algebra over a field $\Bbbk $ of characteristic zero such that the
coradical $H$ of $A$ is a sub-Hopf algebra (i.e. $A$ has the dual Chevalley
Property). Assume $H$ is finite-dimensional so that $H$ is semisimple. By
\cite[Theorem I]{ABM}, there is a gauge transformation $\zeta :A\otimes
A\rightarrow \Bbbk $ such that $A^{\zeta }$ is isomorphic, as a
coquasi-bialgebra, to the bosonization $Q\#H$ of a connected
coquasi-bialgebra $Q$ in ${_{H}^{H}\mathcal{YD}}$ by $H.$ By construction $%
\zeta $ is $H$-bilinear and $H$-balanced: this follows from \cite[%
Proposition 5.7]{ABM} (note that gauge transformation $v_{B}:B\otimes
B\rightarrow \Bbbk $, used therein for $B:=R\#_{\xi }H$, is $H$-bilinear and
$H$-balanced, as observed in the proof) and the fact that there is an $H$%
-bilinear Hopf algebra isomorphism $\psi :B\rightarrow A$ (see \cite[Proof
of Theorem I, page 36 and Theorem 6.1]{ABM} which is a consequence of \cite[%
Theorem 3.64]{AMS}) where $\left( R,\xi \right) $ is a suitable connected
pre-bialgebra with cocycle in ${_{H}^{H}\mathcal{YD}}$ (note that $\zeta
=v_{B}\circ \left( \psi ^{-1}\otimes \psi ^{-1}\right) $): here by connected
pre-bialgebra we mean that the coradical $R_{0}$ of $R$ is $\Bbbk 1_{R}$ (by
the properties of $1_{R}$ this implies that $R_{0}$ is a subcoalgebra in ${%
_{H}^{H}\mathcal{YD}}$ of $R$). Assume that $A$ is finite-dimensional. Then $%
Q\#H$ and hence $Q$ is finite dimensional.

Thus, by Theorem \ref{teo:GelakiYD}, if $\mathrm{H}_{{\mathcal{YD}}%
}^{3}\left( \mathrm{gr}Q,\Bbbk \right) =0$, then $Q$ is gauge equivalent to
a connected bialgebra in ${_{H}^{H}\mathcal{YD}}$.

First let us check which condition on $A$ guarantee that $\mathrm{H}_{{%
\mathcal{YD}}}^{3}\left( \mathrm{gr}Q,\Bbbk \right) =0.$ Note that by
construction $Q=R^{v}$ (see \cite[Proposition 5.7]{ABM}) where $v:=\left(
\lambda \xi \right) ^{-1}$, the convolution inverse of $\lambda \xi $ and $%
\lambda :H\rightarrow \Bbbk $ denotes the total integral on $H$.$\ $Thus we
can rewrite $\mathrm{gr}\left( Q\right) $ as $\mathrm{gr}\left( R^{v}\right)
.$

Moreover $v_{B}$ is given by $v_{B}\left( \left( r\#h\right) \otimes \left(
r^{\prime }\#h^{\prime }\right) \right) =v\left( r\otimes hr^{\prime
}\right) \varepsilon _{H}\left( h^{\prime }\right) $ for every $r,r^{\prime
}\in R,h,h^{\prime }\in H.$ By \cite[Proposition 2.5]{AMStu-Small}, $\mathrm{%
gr}\left( R\right) $ inherits the pre-bialgebra structure in ${_{H}^{H}%
\mathcal{YD}}$ of $R$. This is proved by checking that $R_{i}\cdot
R_{j}\subseteq R_{i+j}$ for every $i,j\in \N_0$, where $R_{i}$ denotes
the $i$-th term of the coradical filtration of $R$. Moreover $R_{i}$ is a
subcoalgebra of $R$ in ${_{H}^{H}\mathcal{YD}}$.

\begin{lemma}
\label{lem:ciccio}Keep the above hypotheses and notations. Then $\mathrm{gr}%
\left( R^{v}\right) $ and $\mathrm{gr}\left( R\right) $ coincide as
bialgebras in ${_{H}^{H}\mathcal{YD}}$ where the structures of $\mathrm{gr}%
\left( R\right) $ are induced by the ones of $\left( R,\xi \right) .$
\end{lemma}

\begin{proof}
By Theorem \ref{teo:grHopf}, $\mathrm{gr}\left( R^{v}\right) =\mathrm{gr}%
\left( Q\right) $ is a connected bialgebras in ${_{H}^{H}\mathcal{YD}}$.

Note that $R^{v}$ and $R$ coincide as coalgebras in ${_{H}^{H}\mathcal{YD}}$
so that $\mathrm{gr}\left( R^{v}\right) $ and $\mathrm{gr}\left( R\right) $
coincide as coalgebras in ${_{H}^{H}\mathcal{YD}}$. They also have the same
unit. It remains to check that their two multiplications coincide too.

Since $\xi $ is unital, by \cite[Proposition 4.8]{AMS}, we have that $v$ is
unital and this is equivalent to $v^{-1}$ unital (see the proof therein).

Let $C:=R\otimes R$. Let $n>0$ and let $w\in C_{\left( n\right)
}=\sum_{i+j\leq n}R_{i}\otimes R_{j}.$ By \cite[Lemma 3.69]{AMS}, we have
that%
\begin{equation*}
\Delta _{C}\left( w\right) -w\otimes \left( 1_{R}\right) ^{\otimes 2}-\left(
1_{R}\right) ^{\otimes 2}\otimes w\in C_{\left( n-1\right) }\otimes
C_{\left( n-1\right) }.
\end{equation*}%
Thus we get%
\begin{equation*}
w_{1}\otimes w_{2}\otimes w_{3}-\Delta _{C}\left( w\right) \otimes \left(
1_{R}\right) ^{\otimes 2}-\Delta _{C}\left( \left( 1_{R}\right) ^{\otimes
2}\right) \otimes w\in \Delta _{C}\left( C_{\left( n-1\right) }\right)
\otimes C_{\left( n-1\right) }
\end{equation*}%
and hence%
\begin{equation*}
w_{1}\otimes w_{2}\otimes w_{3}-w\otimes \left( 1_{R}\right) ^{\otimes
2}\otimes \left( 1_{R}\right) ^{\otimes 2}-\left( 1_{R}\right) ^{\otimes
2}\otimes w\otimes \left( 1_{R}\right) ^{\otimes 2}-\left( 1_{R}\right)
^{\otimes 4}\otimes w\in C_{\left( n-1\right) }\otimes C_{\left( n-1\right)
}\otimes C_{\left( n-1\right) }.
\end{equation*}%
Since $m\left( C_{\left( n-1\right) }\right) \subseteq \sum_{i+j\leq
n}m\left( R_{i}\otimes R_{j}\right) \subseteq R_{n-1}$ we get%
\begin{equation*}
w_{1}\otimes m\left( w_{2}\right) \otimes w_{3}-w\otimes 1_{R}\otimes \left(
1_{R}\right) ^{\otimes 2}-\left( 1_{R}\right) ^{\otimes 2}\otimes m\left(
w\right) \otimes \left( 1_{R}\right) ^{\otimes 2}-\left( 1_{R}\right)
^{\otimes 3}\otimes w\in C_{\left( n-1\right) }\otimes R_{n-1}\otimes
C_{\left( n-1\right) }
\end{equation*}

and hence%
\begin{equation}
w_{1}\otimes \left( m\left( w_{2}\right) +R_{n-1}\right) \otimes
w_{3}=\left( 1_{R}\right) ^{\otimes 2}\otimes \left( m\left( w\right)
+R_{n-1}\right) \otimes \left( 1_{R}\right) ^{\otimes 2}.  \label{form:gr2}
\end{equation}

Let $x,y\in R$. We compute%
\begin{eqnarray*}
\overline{x}\cdot _{v}\overline{y} &=&\left( x+R_{\left\vert x\right\vert
-1}\right) \cdot _{v}\left( y+R_{\left\vert y\right\vert -1}\right) \\
&=&\left( x\cdot _{v}y\right) +R_{\left\vert x\right\vert +\left\vert
y\right\vert -1}=m^{v}\left( x\otimes y\right) +R_{\left\vert x\right\vert
+\left\vert y\right\vert -1} \\
&=&v\left( \left( x\otimes y\right) _{1}\right) m\left( \left( x\otimes
y\right) _{2}\right) v^{-1}\left( \left( x\otimes y\right) _{3}\right)
+R_{\left\vert x\right\vert +\left\vert y\right\vert -1} \\
&=&v\left( \left( x\otimes y\right) _{1}\right) \left( m\left( \left(
x\otimes y\right) _{2}\right) +R_{\left\vert x\right\vert +\left\vert
y\right\vert -1}\right) v^{-1}\left( \left( x\otimes y\right) _{3}\right) \\
&\overset{(\ref{form:gr2})}{=}&v\left( \left( 1_{R}\right) ^{\otimes
2}\right) \left( m\left( x\otimes y\right) +R_{\left\vert x\right\vert
+\left\vert y\right\vert -1}\right) v^{-1}\left( \left( 1_{R}\right)
^{\otimes 2}\right) \\
&=&m\left( x\otimes y\right) +R_{\left\vert x\right\vert +\left\vert
y\right\vert -1}=\left( x\cdot y\right) +R_{\left\vert x\right\vert
+\left\vert y\right\vert -1}=\overline{x}\cdot \overline{y}.
\end{eqnarray*}
\end{proof}

The following result is inspired by \cite[Theorem
3.71]{AMS}.

\begin{lemma}
\label{lem:CoradSmash}Let $H$ be a cosemisimple Hopf algebra. Let $C$ be a
left $H$-comodule coalgebra such that $C_{0}$ is a one-dimensional left $H$%
-comodule subcoalgebra of $C$. Let $B=C\#H$ be the smash coproduct of $C$ by
$H$ i.e. the coalgebra defined by
\begin{eqnarray}
\Delta _{B}\left( c\#h\right) &=&\sum \left( c_{1}\#\left( c_{2}\right)
_{\left\langle -1\right\rangle }h_{1}\right) \otimes \left( \left(
c_{2}\right) _{\left\langle 0\right\rangle }\#h_{2}\right) ,
\label{eq:DeltaCosmash} \\
\varepsilon _{B}\left( c\#h\right) &=&\varepsilon _{C}\left( c\right)
\varepsilon _{H}\left( h\right) .  \notag
\end{eqnarray}%
Then, for every $n\in \N_0$ we have $B_{n}=C_{n}\#H.$
\end{lemma}

\begin{proof}
Since $C_{0}$ is a subcoalgebra of $C$ in ${^{H}}\mathfrak{M}\ $and, for $%
n\geq 1$, one has $C_{n}=C_{n-1}\wedge _{C}C_{0},$ then inductively one
proves that $C_{n}$ is a subcoalgebra of $C$ in ${^{H}}\mathfrak{M}$. Set $%
B_{\left( n\right) }:=C_{n}\#H$ for every $n\in \N_0$. Let us check
that $B_{\left( n\right) }=B_{n}$ by induction on $n\in \N_0.$

Let $n=0.$ First note $B=\cup _{m\in \N_0}B_{\left( m\right) }$ and,
since $\Delta _{C}\left( C_{m}\right) \subseteq \sum_{0\leq i\leq
m}C_{i}\otimes C_{m-i}$, we also have
\begin{eqnarray*}
\Delta _{B}\left( B_{\left( m\right) }\right) &=&\Delta _{B}\left(
C_{m}\#H\right) \subseteq \sum_{0\leq i\leq m}\sum \left( C_{i}\#\left(
C_{m-i}\right) _{\left\langle -1\right\rangle }\left( H\right) _{1}\right)
\otimes \left( \left( C_{m-i}\right) _{\left\langle 0\right\rangle }\#\left(
H\right) _{2}\right) \\
&\subseteq &\sum_{0\leq i\leq m}\left( C_{i}\#H\right) \otimes \left(
C_{m-i}\#\left( H\right) \right) =\sum_{0\leq i\leq m}B_{\left( i\right)
}\otimes B_{\left( m-i\right) }.
\end{eqnarray*}%
Therefore $\left( B_{\left( m\right) }\right) _{m\in \N_0}$ is a
coalgebra filtration for $B$ and hence, by \cite[Proposition 11.1.1]{Sw}, we
get that $B_{\left( 0\right) }\supseteq B_{0}.$ Since $C_{0}$ is
one-dimensional, there is a grouplike element $1_{C}\in C_{0}$ such that $%
C_{0}=\Bbbk 1_{C}.$ Moreover one checks that $C_{0}$ is a subcoalgebra of $C$
in ${^{H}}\mathfrak{M}$ implies $\sum \left( 1_{C}\right) _{\left\langle
-1\right\rangle }\otimes \left( 1_{C}\right) _{\left\langle 0\right\rangle
}=1_{H}\otimes 1_{C}.$

\begin{invisible}
Since $\rho \left( C_{0}\right) \subseteq H\otimes C_{0},$ we get that $\rho
\left( 1_{C}\right) =x\otimes 1_{C}$ for some $x\in H.$ Since $C_{0}$ is a
subcoalgebra of $C$ in ${^{H}}\mathfrak{M,}$ the counit $\varepsilon
:C_{0}\rightarrow \Bbbk $ is left $H$-colinear i.e.
\begin{equation*}
1_{H}\otimes 1_{\Bbbk }=\rho _{\Bbbk }\left( 1_{\Bbbk }\right) =\rho _{\Bbbk
}\left( \varepsilon \left( 1_{C}\right) \right) =\left( H\otimes \varepsilon
\right) \rho \left( 1_{C}\right) =x\otimes \varepsilon \left( 1_{C}\right)
=x\otimes 1_{\Bbbk }
\end{equation*}%
and hence $x=1_{H}.$
\end{invisible}

Let $\sigma :H\rightarrow C\otimes H:h\mapsto 1_{C}\otimes h$ be the
canonical injection. We have%
\begin{eqnarray*}
\Delta _{B}\sigma \left( h\right) &=&\Delta _{B}\left( 1_{C}\otimes h\right)
=\sum \left( 1_{C}\#\left( 1_{C}\right) _{\left\langle -1\right\rangle
}h_{1}\right) \otimes \left( \left( 1_{C}\right) _{\left\langle
0\right\rangle }\#h_{2}\right) \\
&=&\sum \left( 1_{C}\#1_{H}h_{1}\right) \otimes \left( 1_{C}\#h_{2}\right)
=\sum \sigma \left( h_{1}\right) \otimes \sigma \left( h_{2}\right) =\left(
\sigma \otimes \sigma \right) \Delta _{H}\left( h\right) , \\
\varepsilon _{B}\sigma \left( h\right) &=&\varepsilon _{B}\left(
1_{C}\otimes h\right) =\varepsilon _{C}\left( 1_{C}\right) \varepsilon
_{H}\left( h\right) =\varepsilon _{H}\left( h\right)
\end{eqnarray*}%
so that $\sigma $ is a coalgebra map. Since $H$ is cosemisimple and $\sigma $
an injective coalgebra map we deduce that also $\sigma \left( H\right)
=C_{0}\otimes H=B_{\left( 0\right) }$ is a cosemisimple subcoalgebra of $B$
whence $B_{\left( 0\right) }\subseteq B_{0}.$

Let $n>0$ and assume that $B_{i}=B_{\left( i\right) }$ for $0\leq i\leq n-1.$
Let $\sum\limits_{i\in I}c_{i}\#h_{i}\in B_{n}.$ Then
\begin{equation*}
\Delta _{B}\left( \sum\limits_{i\in I}c_{i}\#h_{i}\right) \in B_{n-1}\otimes
B+B\otimes B_{0}=C_{n-1}\otimes H\otimes C\otimes H+C\otimes H\otimes
C_{0}\otimes H.
\end{equation*}%
Let $p_{n}:C\rightarrow \frac{C}{C_{n}}$ be the canonical projection. If we
apply $\left( p_{n-1}\otimes \varepsilon _{H}\otimes p_{0}\otimes H\right) $
we get%
\begin{eqnarray*}
0 &=&\left( p_{n-1}\otimes \varepsilon _{H}\otimes p_{0}\otimes H\right)
\Delta _{B}\left( \sum\limits_{i\in I}c_{i}\#h_{i}\right) \\
&=&\left( p_{n-1}\otimes \varepsilon _{H}\otimes p_{0}\otimes H\right)
\left( \sum\limits_{i\in I}\left( \left( c_{i}\right) _{1}\#\left( \left(
c_{i}\right) _{2}\right) _{\left\langle -1\right\rangle }\left( h_{i}\right)
_{1}\right) \otimes \left( \left( \left( c_{i}\right) _{2}\right)
_{\left\langle 0\right\rangle }\#\left( h_{i}\right) _{2}\right) \right) \\
&=&\left( p_{n-1}\otimes p_{0}\otimes H\right) \left( \sum\limits_{i\in
I}\left( c_{i}\right) _{1}\otimes \left( c_{i}\right) _{2}\otimes
h_{i}\right) =\left( \left( p_{n-1}\otimes p_{0}\right) \Delta _{C}\otimes
H\right) \left( \sum\limits_{i\in I}c_{i}\#h_{i}\right) .
\end{eqnarray*}%
Thus $\sum\limits_{i\in I}c_{i}\#h_{i}\in \mathrm{ke}$\textrm{$r$}$\left(
\left( p_{n-1}\otimes p_{0}\right) \Delta _{C}\otimes H\right) =\left[
\mathrm{ker}\left( \left( p_{n-1}\otimes p_{0}\right) \Delta _{C}\right) %
\right] \otimes H=C_{n}\otimes H=B_{\left( n\right) }.$ Thus $B_{n}\subseteq
B_{\left( n\right) }.$ On the other hand, form $\Delta _{C}\left(
C_{n}\right) \subseteq C_{n-1}\otimes C+C\otimes C_{0}$ we deduce
\begin{eqnarray*}
\Delta _{B}\left( B_{\left( n\right) }\right) &=&\Delta _{B}\left(
C_{n}\otimes H\right) \\
&\subseteq &\sum \left( \left( C_{n}\right) _{1}\#\left( \left( C_{n}\right)
_{2}\right) _{\left\langle -1\right\rangle }\left( H\right) _{1}\right)
\otimes \left( \left( \left( C_{n}\right) _{2}\right) _{\left\langle
0\right\rangle }\#\left( H\right) _{2}\right) \\
&\subseteq &\sum \left( C_{n-1}\#\left( C\right) _{\left\langle
-1\right\rangle }H\right) \otimes \left( \left( C\right) _{\left\langle
0\right\rangle }\#H\right) +\sum \left( C\#\left( C_{0}\right)
_{\left\langle -1\right\rangle }H\right) \otimes \left( \left( C_{0}\right)
_{\left\langle 0\right\rangle }\#H\right) \\
&\subseteq &\left( C_{n-1}\#H\right) \otimes \left( C\#H\right) +\left(
C\#H\right) \otimes \left( C_{0}\#H\right) \\
&=&B_{\left( n-1\right) }\otimes B+B\otimes B_{\left( 0\right)
}=B_{n-1}\otimes B+B\otimes B_{0}
\end{eqnarray*}

and hence $B_{\left( n\right) }\subseteq B_{n}.$
\end{proof}

\begin{definition}
Let $A$ be a Hopf algebra over a field $\Bbbk $ such that the coradical $H$
of $A$ is a sub-Hopf algebra (i.e. $A$ has the dual Chevalley Property). Set
$G:=\mathrm{gr}\left( A\right) .$ There are two canonical Hopf algebra maps%
\begin{eqnarray*}
\sigma _{G} &:&H\rightarrow \mathrm{gr}\left( A\right) :h\mapsto h+A_{-1}, \\
\pi _{G} &:&\mathrm{gr}\left( A\right) \rightarrow H:a+A_{n-1}\mapsto
a\delta _{n,0},\qquad n\in \N_0\text{.}
\end{eqnarray*}%
The diagram of $A$ (see \cite[page 659]{AS-Lifting}) is the vector space%
\begin{equation*}
\mathcal{D}\left( A\right) :=\left\{ d\in \mathrm{gr}\left( A\right) \mid
\sum d_{1}\otimes \pi _{G}\left( d_{2}\right) =d\otimes 1_{H}\right\} .
\end{equation*}%
It is a bialgebra in ${_{H}^{H}\mathcal{YD}}$ as follows. $\mathcal{D}\left(
A\right) $ is a subalgebra of $G.$ The left $H$-action, the left $H$%
-coaction of $\mathcal{D}\left( A\right) ,$ the comultiplication and counit
are given respectively by%
\begin{gather*}
h\vartriangleright d:=\sum \sigma _{G}\left( h_{1}\right) d\sigma
_{G}S\left( h_{2}\right) ,\qquad \rho \left( d\right) =\sum \pi _{G}\left(
d_{1}\right) \otimes d_{2}, \\
\Delta _{D\left( A\right) }\left( d\right) :=\sum d_{1}\sigma _{G}S_{H}\pi
_{G}\left( d_{2}\right) \otimes d_{3},\qquad \varepsilon _{D\left( A\right)
}\left( d\right) =\varepsilon _{G}\left( d\right) .
\end{gather*}
\end{definition}

Although the following result seems to be folklore, we include here its statement for future
references.

\begin{proposition}
\label{pro:D(f)}Let $A$ be a Hopf algebra over a field $\Bbbk $ such that
the coradical $H$ of $A$ is a sub-Hopf algebra. Let $A^{\prime }$ be a Hopf
algebra over a field $\Bbbk $. Let $f:A^{\prime }\rightarrow A$ be an
isomorphism of Hopf algebras. Then $H^{\prime }:=f^{-1}\left( H\right) \cong
H$ is the coradical of $A^{\prime }$ and it is a sub-Hopf algebra of $%
A^{\prime }$. Thus we can identify $H^{\prime }$ with $H.$ Moreover $f$
induces an isomorphism $\mathcal{D}\left( f\right) :\mathcal{D}\left(
A^{\prime }\right) \rightarrow \mathcal{D}\left( A\right) $ of bialgebras in
${_{H}^{H}\mathcal{YD}}$.
\end{proposition}

\begin{invisible}
\begin{proof}
Set $G:=\mathrm{gr}\left( A\right) $ and set $G^{\prime }:=\mathrm{gr}\left(
A^{\prime }\right) .$ Using the identification $H^{\prime }\rightarrow
H:h^{\prime }\mapsto f\left( h^{\prime }\right) $ we can rewrite the
canonical bialgebra maps%
\begin{eqnarray*}
\sigma _{G^{\prime }} &:&H^{\prime }\rightarrow \mathrm{gr}\left( A^{\prime
}\right) :h^{\prime }\mapsto h^{\prime }+A_{-1}^{\prime }, \\
\pi _{G^{\prime }} &:&\mathrm{gr}\left( A^{\prime }\right) \rightarrow
H^{\prime }:a^{\prime }+A_{n-1}^{\prime }\mapsto a^{\prime }\delta
_{n,0},\qquad n\in \N_0
\end{eqnarray*}%
as follows%
\begin{eqnarray*}
\sigma _{G^{\prime }} &:&H\rightarrow \mathrm{gr}\left( A^{\prime }\right)
:h\mapsto f^{-1}\left( h\right) +A_{-1}^{\prime }, \\
\pi _{G^{\prime }} &:&\mathrm{gr}\left( A^{\prime }\right) \rightarrow
H:a^{\prime }+A_{n-1}^{\prime }\mapsto f\left( a^{\prime }\delta
_{n,0}\right) ,\qquad n\in \N_0\text{.}
\end{eqnarray*}%
Since $f$ is an isomorphism it induces an isomorphism $A_{n}^{\prime }\cong
A_{n}$ and hence an isomorphism of graded Hopf algebras $\mathrm{gr}\left(
f\right) :G^{\prime }=\mathrm{gr}\left( A^{\prime }\right) \rightarrow
\mathrm{gr}\left( A\right) =G:a^{\prime }+A_{n-1}^{\prime }\mapsto f\left(
a^{\prime }\right) +A_{n-1}$ for every $n\in \N_0,a^{\prime }\in
A_{n}^{\prime }\backslash A_{n-1}^{\prime }.$ We have%
\begin{eqnarray*}
\pi _{G}\mathrm{gr}\left( f\right) \left( a^{\prime }+A_{n-1}^{\prime
}\right) &=&\pi _{G}\left( f\left( a^{\prime }\right) +A_{n-1}\right)
=f\left( a^{\prime }\right) \delta _{n,0}=f\left( a^{\prime }\delta
_{n,0}\right) =\pi _{G^{\prime }}\left( a^{\prime }+A_{n-1}^{\prime }\right)
, \\
\mathrm{gr}\left( f\right) \sigma _{G^{\prime }}\left( h\right) &=&\mathrm{gr%
}\left( f\right) \left( f^{-1}\left( h\right) +A_{-1}^{\prime }\right)
=h+A_{-1}=\sigma _{G}\left( h\right)
\end{eqnarray*}%
so that%
\begin{equation*}
\pi _{G}\circ \mathrm{gr}\left( f\right) =\pi _{G^{\prime }}\qquad \text{and}%
\qquad \mathrm{gr}\left( f\right) \circ \sigma _{G^{\prime }}=\sigma _{G}.
\end{equation*}%
Let $a^{\prime }+A_{n-1}^{\prime }\in \mathcal{D}\left( A^{\prime }\right) $
and consider $d:=\mathrm{gr}\left( f\right) \left( a^{\prime
}+A_{n-1}^{\prime }\right) =f\left( a^{\prime }\right) +A_{n-1}.$ Then
\begin{eqnarray*}
\sum d_{1}\otimes \pi _{G}\left( d_{2}\right) &=&\sum \left( \mathrm{gr}%
\left( f\right) \left( a^{\prime }+A_{n-1}^{\prime }\right) \right)
_{1}\otimes \pi _{G}\left( \left( \mathrm{gr}\left( f\right) \left(
a^{\prime }+A_{n-1}^{\prime }\right) \right) _{2}\right) \\
&=&\sum \mathrm{gr}\left( f\right) \left( \left( a^{\prime }+A_{n-1}^{\prime
}\right) _{1}\right) \otimes \pi _{G}\mathrm{gr}\left( f\right) \left(
\left( a^{\prime }+A_{n-1}^{\prime }\right) _{2}\right) \\
&=&\sum \mathrm{gr}\left( f\right) \left( \left( a^{\prime }+A_{n-1}^{\prime
}\right) _{1}\right) \otimes \pi _{G^{\prime }}\left( \left( a^{\prime
}+A_{n-1}^{\prime }\right) _{2}\right) \\
&=&\mathrm{gr}\left( f\right) \left( a^{\prime }+A_{n-1}^{\prime }\right)
\otimes 1_{H^{\prime }}=d\otimes 1_{H^{\prime }}.
\end{eqnarray*}%
Therefore $\mathrm{gr}\left( f\right) $ induces an isomorphism of vector
spaces $\mathcal{D}\left( f\right) :D\left( A^{\prime }\right) \rightarrow
D\left( A\right) :d^{\prime }\mapsto \mathrm{gr}\left( f\right) \left(
d^{\prime }\right) .$ Let us check $\mathcal{D}\left( f\right) $ is a
morphism in ${_{H}^{H}\mathcal{YD}}$. For every $d^{\prime }\in \mathcal{D}%
\left( A^{\prime }\right) ,$ we have%
\begin{eqnarray*}
\mathcal{D}\left( f\right) \left( h\vartriangleright d^{\prime }\right) &=&%
\mathcal{D}\left( f\right) \left( \sum \sigma _{G^{\prime }}\left(
h_{1}\right) d^{\prime }\sigma _{G^{\prime }}S\left( h_{2}\right) \right) \\
&=&\mathrm{gr}\left( f\right) \left( \sum \sigma _{G^{\prime }}\left(
h_{1}\right) d^{\prime }\sigma _{G^{\prime }}S\left( h_{2}\right) \right) \\
&=&\sum \mathrm{gr}\left( f\right) \sigma _{G^{\prime }}\left( h_{1}\right)
\cdot \mathrm{gr}\left( f\right) \left( d^{\prime }\right) \cdot \mathrm{gr}%
\left( f\right) \sigma _{G^{\prime }}S\left( h_{2}\right) \\
&=&\sum \sigma _{G}\left( h_{1}\right) \cdot \mathcal{D}\left( f\right)
\left( d^{\prime }\right) \cdot \sigma _{G}S\left( h_{2}\right)
=h\vartriangleright \mathcal{D}\left( f\right) \left( d^{\prime }\right)
\end{eqnarray*}%
and also
\begin{eqnarray*}
\rho \mathcal{D}\left( f\right) \left( d^{\prime }\right) &=&\rho \mathrm{gr}%
\left( f\right) \left( d^{\prime }\right) =\sum \pi _{G}\left( \left(
\mathrm{gr}\left( f\right) \left( d^{\prime }\right) \right) _{1}\right)
\otimes \left( \left( \mathrm{gr}\left( f\right) \left( d^{\prime }\right)
\right) \right) _{2} \\
&=&\sum \pi _{G}\left( \mathrm{gr}\left( f\right) \left( d_{1}^{\prime
}\right) \right) \otimes \mathrm{gr}\left( f\right) \left( d_{2}^{\prime
}\right) \\
&=&\sum \pi _{G^{\prime }}\left( d_{1}^{\prime }\right) \otimes \mathrm{gr}%
\left( f\right) \left( d_{2}^{\prime }\right) =\left( H\otimes \mathrm{gr}%
\left( f\right) \right) \rho \left( d^{\prime }\right) =\left( H\otimes
\mathcal{D}\left( f\right) \right) \rho \left( d^{\prime }\right) .
\end{eqnarray*}

Now $\mathcal{D}\left( A^{\prime }\right) $ is a subalgebra of $G^{\prime }$
and $\mathcal{D}\left( A\right) $ is a subalgebra of $G.$ Since $\mathrm{gr}%
\left( f\right) :G^{\prime }\rightarrow G$ is multiplicative and unitary so
that also $\mathcal{D}\left( f\right) :\mathcal{D}\left( A^{\prime }\right)
\rightarrow \mathcal{D}\left( A\right) $ is multiplicative and unitary.
Moreover we have%
\begin{eqnarray*}
\Delta _{\mathcal{D}\left( A\right) }\mathcal{D}\left( f\right) \left(
d^{\prime }\right) &=&\sum \left( \mathcal{D}\left( f\right) \left(
d^{\prime }\right) \right) _{1}\sigma _{G}S_{H}\pi _{G}\left( \left(
\mathcal{D}\left( f\right) \left( d^{\prime }\right) \right) _{2}\right)
\otimes \left( \mathcal{D}\left( f\right) \left( d^{\prime }\right) \right)
_{3} \\
&=&\sum \left( \mathrm{gr}\left( f\right) \left( d^{\prime }\right) \right)
_{1}\sigma _{G}S_{H}\pi _{G}\left( \left( \mathrm{gr}\left( f\right) \left(
d^{\prime }\right) \right) _{2}\right) \otimes \left( \mathrm{gr}\left(
f\right) \left( d^{\prime }\right) \right) _{3} \\
&=&\sum \mathrm{gr}\left( f\right) \left( d_{1}^{\prime }\right) \cdot
\sigma _{G}S_{H}\pi _{G}\mathrm{gr}\left( f\right) \left( d_{2}^{\prime
}\right) \otimes \mathrm{gr}\left( f\right) \left( d_{3}^{\prime }\right) \\
&=&\sum \mathrm{gr}\left( f\right) \left( d_{1}^{\prime }\right) \cdot
\mathrm{gr}\left( f\right) \sigma _{G^{\prime }}S_{H}\pi _{G^{\prime
}}\left( d_{2}^{\prime }\right) \otimes \mathrm{gr}\left( f\right) \left(
d_{3}^{\prime }\right) \\
&=&\left( \mathrm{gr}\left( f\right) \otimes \mathrm{gr}\left( f\right)
\right) \left( \sum d_{1}^{\prime }\sigma _{G^{\prime }}S_{H}\pi _{G^{\prime
}}\left( d_{2}^{\prime }\right) \otimes d_{3}^{\prime }\right) \\
&=&\left( \mathrm{gr}\left( f\right) \otimes \mathrm{gr}\left( f\right)
\right) \Delta _{\mathcal{D}\left( A^{\prime }\right) }\left( d^{\prime
}\right) =\left( \mathcal{D}\left( f\right) \otimes \mathcal{D}\left(
f\right) \right) \Delta _{\mathcal{D}\left( A^{\prime }\right) }\left(
d^{\prime }\right)
\end{eqnarray*}%
and%
\begin{equation*}
\varepsilon _{\mathcal{D}\left( A\right) }\mathcal{D}\left( f\right) \left(
d^{\prime }\right) =\varepsilon _{G}\mathrm{gr}\left( f\right) \left(
d^{\prime }\right) =\varepsilon _{G^{\prime }}\left( d^{\prime }\right)
=\varepsilon _{\mathcal{D}\left( A^{\prime }\right) }\left( d^{\prime
}\right) .
\end{equation*}
\end{proof}
\end{invisible}

\begin{proposition}
\label{pro:ciccio}Keep the hypotheses and notations of the beginning of the
section. Then $\mathcal{D}\left( A\right) \cong \mathcal{D}\left( R\#_{\xi
}H\right) \cong \mathrm{gr}\left( R\right) $ as bialgebras in ${_{H}^{H}%
\mathcal{YD}}$.
\end{proposition}

\begin{proof}
Apply Proposition \ref{pro:D(f)} to the canonical isomorphism $\psi
:B:=R\#_{\xi }H\rightarrow A$ that we recalled at the beginning of the
section to get that $\mathcal{D}\left( R\#_{\xi }H\right) \cong \mathcal{D}%
\left( A\right) .$ Note that, by $H$-linearity we have
\begin{equation*}
\psi \left( 1_{R}\#h\right) =\psi \left( \left( 1_{R}\#1_{H}\right) \left(
1_{R}\#h\right) \right) =\psi \left( \left( 1_{R}\#1_{H}\right) h\right)
=\psi \left( 1_{R}\#1_{H}\right) h=h
\end{equation*}%
so that $\psi \left( \Bbbk 1_{R}\otimes H\right) =H$ and hence $H^{\prime
}=\psi ^{-1}\left( H\right) =\Bbbk 1_{R}\otimes H$ with the notation of
Proposition \ref{pro:D(f)}. Thus $B_{0}=\Bbbk 1_{R}\otimes H=R_{0}\otimes H$
so that we can identify $B_{0}$ with $H$ via the canonical isomorphism $%
H\rightarrow R_{0}\otimes H:h\mapsto 1_{R}\otimes h$. Its inverse is $%
R_{0}\otimes H\rightarrow H:r\otimes h\mapsto \varepsilon _{R}\left(
r\right) h.$ With this identification and by setting $G:=\mathrm{gr}\left(
B\right) ,$ we can consider the canonical bialgebra maps%
\begin{eqnarray*}
\sigma _{G} &:&H\rightarrow \mathrm{gr}\left( B\right) :h\mapsto
1_{R}\#h+\left( R\#_{\xi }H\right) _{-1}, \\
\pi _{G} &:&\mathrm{gr}\left( B\right) \rightarrow H:r\#h+\left( R\#_{\xi
}H\right) _{n-1}\mapsto \varepsilon _{R}\left( r\right) h\delta _{n,0},\text{
where }r\#h\in \left( R\#_{\xi }H\right) _{n},n\in \N_0\text{.}
\end{eqnarray*}

Since the underlying coalgebra of $B$ is exactly the smash coproduct of $R$
by $H$ and $\left( R,\xi \right) $ is a connected pre-bialgebra with cocycle
in ${_{H}^{H}\mathcal{YD}}$, by Lemma \ref{lem:CoradSmash}, we have that $%
B_{n}=R_{n}\otimes H.$ Let us compute $\mathcal{D}:=\mathcal{D}\left(
B\right) .$ As a vector space it is%
\begin{equation*}
\mathcal{D}:=\left\{ d\in G\mid \sum d_{1}\otimes \pi _{G}\left(
d_{2}\right) =d\otimes 1_{H}\right\} .
\end{equation*}%
By \cite[Lemma 2.1]{AS-Lifting}, we have that $\mathcal{D}=\oplus _{n\in
\N_0}\mathcal{D}^{n}$ where $\mathcal{D}^{n}=\mathcal{D}\cap G^{n}=%
\mathcal{D}\cap \frac{B_{n}}{B_{n-1}}.$ Let $d:=\overline{\sum\limits_{i\in
I}r_{i}\#h_{i}}\in \mathcal{D}^{n}$ where we can assume $\sum\limits_{i\in
I}r_{i}\#h_{i}\in B_{n}\backslash B_{n-1}$ and, for every $i\in I$, $%
r_{i}\#h_{i}\in B_{n}\backslash B_{n-1}$. Then $\overline{\sum\limits_{i\in
I}r_{i}\#h_{i}}=\sum\limits_{i\in I}\overline{r_{i}\#h_{i}}$ and hence the
fact that $d\mathcal{\ }$is coinvariant rewrites as%
\begin{equation}
\sum\limits_{i\in I}\left( \overline{r_{i}\#h_{i}}\right) _{1}\otimes \pi
_{G}\left( \left( \overline{r_{i}\#h_{i}}\right) _{2}\right)
=\sum\limits_{i\in I}\overline{r_{i}\#h_{i}}\otimes 1_{H}.  \label{eq: pig1}
\end{equation}%
By definition of $\pi _{G}$ and (\ref{form:DeltaGr}), the left-hand side
becomes%
\begin{equation*}
\sum\limits_{i\in I}\left( \overline{r_{i}\#h_{i}}\right) _{1}\otimes \pi
_{G}\left( \left( \overline{r_{i}\#h_{i}}\right) _{2}\right)
=\sum\limits_{i\in I}\left( \left( r_{i}\#\left( h_{i}\right) _{1}\right)
+B_{n-1}\right) \otimes \left( h_{i}\right) _{2}
\end{equation*}

\begin{invisible}
Here is the computation
\begin{eqnarray*}
&&\sum\limits_{i\in I}\left( \overline{r_{i}\#h_{i}}\right) _{1}\otimes \pi
_{G}\left( \left( \overline{r_{i}\#h_{i}}\right) _{2}\right) \overset{(\ref%
{form:DeltaGr})}{=}\sum\limits_{i\in I}\sum_{0\leq t\leq n}\left( \left(
r_{i}\#h_{i}\right) _{1}+B_{t-1}\right) \otimes \pi _{G}\left( \left(
r_{i}\#h_{i}\right) _{2}+B_{n-t-1}\right) \\
&=&\sum\limits_{i\in I}\sum_{0\leq t\leq n}\left( \left( r_{i}^{\left(
1\right) }\#\left( r_{i}^{\left( 2\right) }\right) _{\left\langle
-1\right\rangle }\left( h_{i}\right) _{1}\right) +B_{t-1}\right) \otimes \pi
_{G}\left( \left( \left( r_{i}^{\left( 2\right) }\right) _{\left\langle
0\right\rangle }\#\left( h_{i}\right) _{2}\right) +B_{n-t-1}\right) \\
&=&\sum\limits_{i\in I}\sum_{0\leq t\leq n}\left( \left( r_{i}^{\left(
1\right) }\#\left( r_{i}^{\left( 2\right) }\right) _{\left\langle
-1\right\rangle }\left( h_{i}\right) _{1}\right) +B_{t-1}\right) \otimes
\varepsilon _{R}\left( \left( r_{i}^{\left( 2\right) }\right) _{\left\langle
0\right\rangle }\right) \left( h_{i}\right) _{2}\delta _{n-t,0} \\
&=&\sum\limits_{i\in I}\left( \left( r_{i}^{\left( 1\right) }\#\left(
h_{i}\right) _{1}\right) +B_{n-1}\right) \otimes \varepsilon _{R}\left(
r_{i}^{\left( 2\right) }\right) \left( h_{i}\right) _{2} \\
&=&\sum\limits_{i\in I}\left( \left( r_{i}\#\left( h_{i}\right) _{1}\right)
+B_{n-1}\right) \otimes \left( h_{i}\right) _{2}
\end{eqnarray*}
\end{invisible}

so that (\ref{eq: pig1}) becomes%
\begin{equation*}
\sum\limits_{i\in I}\left( \left( r_{i}\#\left( h_{i}\right) _{1}\right)
+B_{n-1}\right) \otimes \left( h_{i}\right) _{2}=\sum\limits_{i\in I}%
\overline{r_{i}\#h_{i}}\otimes 1_{H}=\sum\limits_{i\in I}\left(
r_{i}\#h_{i}+B_{n-1}\right) \otimes 1_{H}
\end{equation*}%
i.e.%
\begin{equation*}
\sum\limits_{i\in I}\left( r_{i}\#\left( h_{i}\right) _{1}\right) \otimes
\left( h_{i}\right) _{2}-\sum\limits_{i\in I}r_{i}\#h_{i}\otimes 1_{H}\in
B_{n-1}\otimes H=R_{n-1}\otimes H\otimes H.
\end{equation*}%
If we apply $R\otimes \varepsilon _{H}\otimes H$, we get%
\begin{equation*}
\sum\limits_{i\in I}r_{i}\otimes h_{i}-\sum\limits_{i\in I}r_{i}\varepsilon
_{H}\left( h_{i}\right) \otimes 1_{H}\in R_{n-1}\otimes H=B_{n-1}.
\end{equation*}%
Thus $\overline{\sum\limits_{i\in I}r_{i}\#h_{i}}=\sum\limits_{i\in I}%
\overline{r_{i}\#h_{i}}=\sum\limits_{i\in I}\left(
r_{i}\#h_{i}+B_{n-1}\right) =\sum\limits_{i\in I}\left( r_{i}\varepsilon
_{H}\left( h_{i}\right) \otimes 1_{H}\right) +B_{n-1}.$

Since $\sum\limits_{i\in I}r_{i}\#h_{i}\in B_{n}\backslash B_{n-1}$ we get
that $\left( \sum\limits_{i\in I}r_{i}\varepsilon _{H}\left( h_{i}\right)
\right) \otimes 1_{H}\notin B_{n-1}$ and hence $\sum\limits_{i\in
I}r_{i}\varepsilon _{H}\left( h_{i}\right) \notin R_{n-1}$ and we can write%
\begin{equation*}
\overline{\sum\limits_{i\in I}r_{i}\#h_{i}}=\overline{\left(
\sum\limits_{i\in I}r_{i}\varepsilon _{H}\left( h_{i}\right) \right) \otimes
1_{H}}.
\end{equation*}%
Therefore we have proved that the map
\begin{equation*}
\varphi _{n}:\frac{R_{n}}{R_{n-1}}\rightarrow \mathcal{D}^{n}:\overline{r}%
\mapsto \overline{r\otimes 1_{H}},
\end{equation*}%
which is well-defined as $\mathcal{D}^{n}=\mathcal{D}\cap G^{n}=\mathcal{D}%
\cap \frac{B_{n}}{B_{n-1}}=\mathcal{D}\cap \frac{R_{n}\otimes H}{%
R_{n-1}\otimes H},$ is also surjective.

It is also injective as $\varphi _{n}\left( \overline{r}\right) =\varphi
_{n}\left( \overline{s}\right) $ implies $r\otimes 1_{H}-s\otimes 1_{H}\in
B_{n-1}=R_{n-1}\otimes H$ and hence, by applying $R\otimes \varepsilon _{H},$
we get $r-s\in R_{n-1}$ i.e. $\overline{r}=\overline{s}.$ Therefore $\varphi
_{n}$ is an isomorphism such that $\overline{\sum\limits_{i\in I}r_{i}\#h_{i}%
}=\varphi _{n}\left( \overline{\sum\limits_{i\in I}r_{i}\varepsilon
_{H}\left( h_{i}\right) }\right) $ and hence
\begin{equation*}
\varphi _{n}^{-1}\left( \overline{\sum\limits_{i\in I}r_{i}\#h_{i}}\right) =%
\overline{\sum\limits_{i\in I}r_{i}\varepsilon _{H}\left( h_{i}\right) }.
\end{equation*}%
Clearly this extends to a graded $\Bbbk $-linear isomorphism
\begin{equation*}
\varphi :\mathrm{gr}\left( R\right) \rightarrow \mathcal{D}.
\end{equation*}%
Let us check that $\varphi $ is a morphism in ${_{H}^{H}\mathcal{YD}}$.
First note that, for every $r\in R_{n}$, we have%
\begin{eqnarray*}
\varphi \left( r+R_{n-1}\right) &=&\delta _{\left\vert r\right\vert
,n}\varphi \left( r+R_{n-1}\right) =\delta _{\left\vert r\right\vert
,n}\varphi _{n}\left( r+R_{n-1}\right) =\delta _{\left\vert r\right\vert
,n}\varphi _{n}\left( \overline{r}\right) \\
&=&\delta _{\left\vert r\right\vert ,n}\overline{r\otimes 1_{H}}=\delta
_{\left\vert r\right\vert ,n}\left( r\otimes 1_{H}+\left( R\#_{\xi }H\right)
_{n-1}\right) =r\otimes 1_{H}+\left( R\#_{\xi }H\right) _{n-1}.
\end{eqnarray*}%
Thus%
\begin{equation}
\varphi \left( r+R_{n-1}\right) =r\otimes 1_{H}+\left( R\#_{\xi }H\right)
_{n-1},\text{ for every }r\in R_{n}\text{.}  \label{form:phiTrick}
\end{equation}%
For every $r\in R_{n}\backslash R_{n-1},$ by using (\ref{form:phiTrick}), it
is straighforward to prove that $h\vartriangleright \varphi \left( \overline{%
r}\right) =\varphi \left( h\overline{r}\right) .$

\begin{invisible}
In fact we have%
\begin{eqnarray*}
h\vartriangleright \varphi \left( \overline{r}\right) &=&\sum \sigma
_{G}\left( h_{1}\right) \varphi \left( \overline{r}\right) \sigma
_{G}S\left( h_{2}\right) \\
&&\overset{(\ref{form:phiTrick})}{=}\sum \left( 1_{R}\#h_{1}+\left( R\#_{\xi
}H\right) _{-1}\right) \left( r\#1_{H}+\left( R\#_{\xi }H\right)
_{n-1}\right) \left( 1_{R}\#S\left( h_{2}\right) +\left( R\#_{\xi }H\right)
_{-1}\right) \\
&=&\sum \left( 1_{R}\#h_{1}\right) \left( r\#1_{H}\right) \left(
1_{R}\#S\left( h_{2}\right) \right) +\left( R\#_{\xi }H\right) _{n-1} \\
&=&\sum \left( h_{1}r\#h_{2}S\left( h_{3}\right) \right) +\left( R\#_{\xi
}H\right) _{n-1} \\
&=&\left( hr\#1_{H}\right) +\left( R\#_{\xi }H\right) _{n-1}\overset{(\ref%
{form:phiTrick})}{=}\varphi \left( hr+R_{n-1}\right) \\
&=&\varphi \left( h\left( r+R_{n-1}\right) \right) =\varphi \left( h%
\overline{r}\right) .
\end{eqnarray*}
\end{invisible}

Moreover, by applying (\ref{form:DeltaGr}), (\ref{eq:DeltaCosmash}), the
definition of $\pi _{G}$ and (\ref{form:phiTrick}), we get that $\rho
\varphi \left( \overline{r}\right) =\left( H\otimes \varphi \right) \rho
\left( \overline{r}\right) .$

\begin{invisible}
In fact we have%
\begin{eqnarray*}
\rho \varphi \left( \overline{r}\right) &=&\sum \pi _{G}\left( \varphi
\left( \overline{r}\right) _{1}\right) \otimes \varphi \left( \overline{r}%
\right) _{2} \\
&=&\sum \pi _{G}\left( \left( \overline{r\otimes 1_{H}}\right) _{1}\right)
\otimes \left( \overline{r\otimes 1_{H}}\right) _{2} \\
&&\overset{(\ref{form:DeltaGr})}{=}\sum_{0\leq i\leq n}\pi _{G}\left( \left(
r\otimes 1_{H}\right) _{1}+\left( R\#_{\xi }H\right) _{i-1}\right) \otimes
\left( \left( r\otimes 1_{H}\right) _{2}+\left( R\#_{\xi }H\right)
_{n-i-1}\right) .
\end{eqnarray*}%
We compute%
\begin{equation}
\Delta _{B}\left( r\otimes 1_{H}\right) =\sum r^{\left( 1\right) }\otimes
\left( r^{\left( 2\right) }\right) _{\left\langle -1\right\rangle }\otimes
\left( r^{\left( 2\right) }\right) _{\left\langle 0\right\rangle }\otimes
1_{H}  \label{form:DeltaSmash1}
\end{equation}%
so that%
\begin{eqnarray*}
&&\rho \varphi \left( \overline{r}\right) \overset{(\ref{form:DeltaSmash1})}{%
=}\sum_{0\leq i\leq n}\sum \pi _{G}\left( r^{\left( 1\right) }\otimes \left(
r^{\left( 2\right) }\right) _{\left\langle -1\right\rangle }+\left( R\#_{\xi
}H\right) _{i-1}\right) \otimes \left( \left( r^{\left( 2\right) }\right)
_{\left\langle 0\right\rangle }\otimes 1_{H}+\left( R\#_{\xi }H\right)
_{n-i-1}\right) \\
&=&\sum_{0\leq i\leq n}\sum \left( \varepsilon _{R}\left( r^{\left( 1\right)
}\right) \left( r^{\left( 2\right) }\right) _{\left\langle -1\right\rangle
}\delta _{i,0}\right) \otimes \left( \left( r^{\left( 2\right) }\right)
_{\left\langle 0\right\rangle }\otimes 1_{H}+\left( R\#_{\xi }H\right)
_{n-i-1}\right) \\
&=&\sum r_{\left\langle -1\right\rangle }\otimes \left( r_{\left\langle
0\right\rangle }\otimes 1_{H}+\left( R\#_{\xi }H\right) _{n-1}\right) \\
&&\overset{(\ref{form:phiTrick})}{=}r_{\left\langle -1\right\rangle }\otimes
\varphi \left( r_{\left\langle 0\right\rangle }+R_{n}\right) \\
&=&\left( H\otimes \varphi \right) \left( r_{\left\langle -1\right\rangle
}\otimes \left( r_{\left\langle 0\right\rangle }+R_{n}\right) \right)
=\left( H\otimes \varphi \right) \rho \left( r+R_{n}\right) =\left( H\otimes
\varphi \right) \rho \left( \overline{r}\right) .
\end{eqnarray*}
\end{invisible}

Let us check that $\varphi $ is a morphism of bialgebras in ${_{H}^{H}%
\mathcal{YD}}$. Fix $r\in R_{n}\backslash R_{n-1}.$

Using the definition of $\Delta _{\mathcal{D}}$, (\ref{form:DeltaGr}), (\ref%
{eq:DeltaCosmash}), the definition of $\pi _{G}$, the definition of $\sigma
_{G}$, (\ref{form:phiTrick}) and (\ref{form:DeltaGr}) again, we obtain $%
\Delta _{\mathcal{D}}\varphi \left( \overline{r}\right) =\left( \varphi
\otimes \varphi \right) \Delta _{\mathrm{gr}\left( R\right) }\left(
\overline{r}\right) .$

\begin{invisible}
Let us check $\varphi $ is comultiplicative.
\begin{eqnarray*}
\Delta _{\mathcal{D}}\varphi \left( \overline{r}\right) &=&\Delta _{\mathcal{%
D}}\left( \overline{r\otimes 1_{H}}\right) \\
&=&\sum \left( \overline{r\otimes 1_{H}}\right) _{1}\sigma _{G}S_{H}\pi
_{G}\left( \left( \overline{r\otimes 1_{H}}\right) _{2}\right) \otimes
\left( \overline{r\otimes 1_{H}}\right) _{3} \\
&&\overset{(\ref{form:DeltaGr})}{=}\sum \sum_{a+b+c=n}\left( \left( r\otimes
1_{H}\right) _{1}+B_{a-1}\right) \sigma _{G}S_{H}\pi _{G}\left( \left(
r\otimes 1_{H}\right) _{2}+B_{b-1}\right) \otimes \left( \left( r\otimes
1_{H}\right) _{3}+B_{c-1}\right) .
\end{eqnarray*}%
We compute%
\begin{eqnarray*}
&&\left( B\otimes \Delta _{B}\right) \Delta _{B}\left( r\otimes 1_{H}\right)
\\
&&\overset{(\ref{form:DeltaSmash1})}{=}\sum r^{\left( 1\right) }\otimes
\left( r^{\left( 2\right) }\right) _{\left\langle -1\right\rangle }\otimes
\left( \left( r^{\left( 2\right) }\right) _{\left\langle 0\right\rangle
}\right) ^{\left( 1\right) }\otimes \left( \left( \left( r^{\left( 2\right)
}\right) _{\left\langle 0\right\rangle }\right) ^{\left( 2\right) }\right)
_{\left\langle -1\right\rangle }\otimes \left( \left( \left( r^{\left(
2\right) }\right) _{\left\langle 0\right\rangle }\right) ^{\left( 2\right)
}\right) _{\left\langle 0\right\rangle }\otimes 1_{H} \\
&=&\sum r^{\left( 1\right) }\otimes \left( r^{\left( 2\right) }\right)
_{\left\langle -1\right\rangle }\left( r^{\left( 3\right) }\right)
_{\left\langle -1\right\rangle }\otimes \left( r^{\left( 2\right) }\right)
_{\left\langle 0\right\rangle }\otimes \left( \left( \left( r^{\left(
3\right) }\right) _{\left\langle 0\right\rangle }\right) \right)
_{\left\langle -1\right\rangle }\otimes \left( \left( \left( r^{\left(
3\right) }\right) _{\left\langle 0\right\rangle }\right) \right)
_{\left\langle 0\right\rangle }\otimes 1_{H} \\
&=&\sum \left( r^{\left( 1\right) }\otimes \left( r^{\left( 2\right)
}\right) _{\left\langle -1\right\rangle }\left( r^{\left( 3\right) }\right)
_{\left\langle -2\right\rangle }\right) \otimes \left( \left( r^{\left(
2\right) }\right) _{\left\langle 0\right\rangle }\otimes \left( r^{\left(
3\right) }\right) _{\left\langle -1\right\rangle }\right) \otimes \left(
\left( r^{\left( 3\right) }\right) _{\left\langle 0\right\rangle }\otimes
1_{H}\right)
\end{eqnarray*}%
so that%
\begin{eqnarray*}
\Delta _{\mathcal{D}}\varphi \left( \overline{r}\right) &=&\sum
\sum_{a+b+c=n}\left( \left( r\otimes 1_{H}\right) _{1}+B_{a-1}\right) \sigma
_{G}S_{H}\pi _{G}\left( \left( r\otimes 1_{H}\right) _{2}+B_{b-1}\right)
\otimes \left( \left( r\otimes 1_{H}\right) _{3}+B_{c-1}\right) \\
&=&\left[
\begin{array}{c}
\sum \sum_{a+b+c=n}\left( \left( r^{\left( 1\right) }\otimes \left(
r^{\left( 2\right) }\right) _{\left\langle -1\right\rangle }\left( r^{\left(
3\right) }\right) _{\left\langle -2\right\rangle }\right) +B_{a-1}\right)
\sigma _{G}S_{H}\pi _{G}\left( \left( \left( r^{\left( 2\right) }\right)
_{\left\langle 0\right\rangle }\otimes \left( r^{\left( 3\right) }\right)
_{\left\langle -1\right\rangle }\right) +B_{b-1}\right) \\
\otimes \left( \left( \left( r^{\left( 3\right) }\right) _{\left\langle
0\right\rangle }\otimes 1_{H}\right) +B_{c-1}\right)%
\end{array}%
\right] \\
&=&\left[
\begin{array}{c}
\sum \sum_{a+b+c=n}\left( \left( r^{\left( 1\right) }\otimes \left(
r^{\left( 2\right) }\right) _{\left\langle -1\right\rangle }\left( r^{\left(
3\right) }\right) _{\left\langle -2\right\rangle }\right) +B_{a-1}\right)
\sigma _{G}S_{H}\left( \left( \varepsilon _{R}\left( \left( r^{\left(
2\right) }\right) _{\left\langle 0\right\rangle }\right) \left( r^{\left(
3\right) }\right) _{\left\langle -1\right\rangle }\right) \right) \delta
_{b,0} \\
\otimes \left( \left( \left( r^{\left( 3\right) }\right) _{\left\langle
0\right\rangle }\otimes 1_{H}\right) +B_{c-1}\right)%
\end{array}%
\right] \\
&=&\left[
\begin{array}{c}
\sum \sum_{a+c=n}\left( \left( r^{\left( 1\right) }\otimes \left( r^{\left(
3\right) }\right) _{\left\langle -2\right\rangle }\right) +B_{a-1}\right)
S_{H}\left( \varepsilon _{R}\left( r^{\left( 2\right) }\right) \left(
r^{\left( 3\right) }\right) _{\left\langle -1\right\rangle }\right) \\
\otimes \left( \left( \left( r^{\left( 3\right) }\right) _{\left\langle
0\right\rangle }\otimes 1_{H}\right) +B_{c-1}\right)%
\end{array}%
\right] \\
&=&\sum \sum_{a+c=n}\left( \left( r^{\left( 1\right) }\otimes \left(
r^{\left( 2\right) }\right) _{\left\langle -2\right\rangle }\right)
+B_{a-1}\right) \sigma _{G}S_{H}\left( \left( r^{\left( 2\right) }\right)
_{\left\langle -1\right\rangle }\right) \otimes \left( \left( \left(
r^{\left( 2\right) }\right) _{\left\langle 0\right\rangle }\otimes
1_{H}\right) +B_{c-1}\right) \\
&=&\sum \sum_{a+c=n}\left( \left( r^{\left( 1\right) }\otimes \left(
r^{\left( 2\right) }\right) _{\left\langle -2\right\rangle }\right)
+B_{a-1}\right) \left( 1_{R}\otimes S_{H}\left( \left( r^{\left( 2\right)
}\right) _{\left\langle -1\right\rangle }\right) +B_{-1}\right) \otimes
\left( \left( \left( r^{\left( 2\right) }\right) _{\left\langle
0\right\rangle }\otimes 1_{H}\right) +B_{c-1}\right) \\
&=&\sum \sum_{a+c=n}\left( \left( r^{\left( 1\right) }\otimes \left(
r^{\left( 2\right) }\right) _{\left\langle -2\right\rangle }\right) \left(
1_{R}\otimes S_{H}\left( \left( r^{\left( 2\right) }\right) _{\left\langle
-1\right\rangle }\right) \right) +B_{a-1}\right) \otimes \left( \left(
\left( r^{\left( 2\right) }\right) _{\left\langle 0\right\rangle }\otimes
1_{H}\right) +B_{c-1}\right) \\
&=&\sum \sum_{a+c=n}\left( \left( r^{\left( 1\right) }\otimes \left(
r^{\left( 2\right) }\right) _{\left\langle -2\right\rangle }S_{H}\left(
\left( r^{\left( 2\right) }\right) _{\left\langle -1\right\rangle }\right)
\right) +B_{a-1}\right) \otimes \left( \left( \left( r^{\left( 2\right)
}\right) _{\left\langle 0\right\rangle }\otimes 1_{H}\right) +B_{c-1}\right)
\\
&=&\sum \sum_{a+c=n}\left( \left( r^{\left( 1\right) }\otimes 1_{H}\right)
+B_{a-1}\right) \otimes \left( \left( r^{\left( 2\right) }\otimes
1_{H}\right) +B_{c-1}\right) \\
&&\overset{(\ref{form:phiTrick})}{=}\left( \varphi \otimes \varphi \right)
\left( \sum_{a+c=n}\left( r^{\left( 1\right) }+R_{a-1}\right) \otimes \left(
r^{\left( 2\right) }+R_{c-1}\right) \right) \\
&&\overset{(\ref{form:DeltaGr})}{=}\left( \varphi \otimes \varphi \right)
\Delta _{\mathrm{gr}\left( R\right) }\left( \overline{r}\right) .
\end{eqnarray*}
\end{invisible}

Let us check $\varphi $ is counitary:%
\begin{eqnarray*}
\varepsilon _{\mathcal{D}}\varphi \left( \overline{r}\right) &=&\varepsilon
_{G}\varphi \left( \overline{r}\right) =\varepsilon _{G}\left( \overline{%
r\otimes 1_{H}}\right) \overset{(\ref{form:EpsGr})}{=}\delta
_{n,0}\varepsilon _{B}\left( r\otimes 1_{H}\right) \\
&=&\delta _{n,0}\varepsilon _{R}\left( r\right) \overset{(\ref{form:EpsGr})}{%
=}\varepsilon _{\mathrm{gr}\left( R\right) }\left( \overline{r}\right) .
\end{eqnarray*}%
Let us check $\varphi $ is multiplicative. Let $s\in R_{m}\backslash
R_{m-1}. $ Then, by definition of $\varphi ,$ of $m_{\mathcal{D}}$ and of
the multiplication of $R\#_{\xi }H,$ we have that%
\begin{equation*}
m_{\mathcal{D}}\left( \varphi \otimes \varphi \right) \left( \overline{s}%
\otimes \overline{r}\right) =\sum \left( s^{\left( 1\right) }\left( \left(
s^{\left( 2\right) }\right) _{\left\langle -1\right\rangle }r^{\left(
1\right) }\right) \#\xi \left( \left( s^{\left( 2\right) }\right)
_{\left\langle 0\right\rangle }\otimes r^{\left( 2\right) }\right) \right)
+\left( R\#_{\xi }H\right) _{m+n-1}.
\end{equation*}

\begin{invisible}
Explicitly we have%
\begin{eqnarray*}
m_{\mathcal{D}}\left( \varphi \otimes \varphi \right) \left( \overline{s}%
\otimes \overline{r}\right) &=&m_{\mathcal{D}}\left( \overline{s\otimes 1_{H}%
}\otimes \overline{r\otimes 1_{H}}\right) \\
&=&m_{G}\left( \left( s\otimes 1_{H}+\left( R\#_{\xi }H\right) _{m-1}\right)
\otimes \left( r\otimes 1_{H}+\left( R\#_{\xi }H\right) _{n-1}\right) \right)
\\
&=&\left( s\#1_{H}\right) \left( r\#1_{H}\right) +\left( R\#_{\xi }H\right)
_{m+n-1} \\
&=&\sum \left( s^{\left( 1\right) }\left( \left( s^{\left( 2\right) }\right)
_{\left\langle -1\right\rangle }r^{\left( 1\right) }\right) \#\xi \left(
\left( s^{\left( 2\right) }\right) _{\left\langle 0\right\rangle }\otimes
r^{\left( 2\right) }\right) \right) +\left( R\#_{\xi }H\right) _{m+n-1}.
\end{eqnarray*}
\end{invisible}

Now write $\sum s^{\left( 1\right) }\otimes s^{\left( 2\right) }=\sum_{0\leq
i\leq m}s_{i}\otimes s_{m-i}^{\prime }$ for some $s_{i},s_{i}^{\prime }\in
R_{i}$ and similarly $\sum r^{\left( 1\right) }\otimes r^{\left( 2\right)
}=\sum_{0\leq j\leq n}r_{j}\otimes r_{n-j}^{\prime }$ for some $%
r_{j},r_{j}^{\prime }\in R_{j}.$ Then%
\begin{eqnarray*}
m_{\mathcal{D}}\left( \varphi \otimes \varphi \right) \left( \overline{s}%
\otimes \overline{r}\right) &=&\sum_{\substack{ 0\leq i\leq m  \\ 0\leq
j\leq n}}\left( s_{i}\left( \left( s_{m-i}^{\prime }\right) _{\left\langle
-1\right\rangle }r_{j}\right) \#\xi \left( \left( s_{m-i}^{\prime }\right)
_{\left\langle 0\right\rangle }\otimes r_{n-j}^{\prime }\right) \right)
+\left( R\#_{\xi }H\right) _{m+n-1} \\
&=&\sum_{\substack{ 0\leq i\leq m  \\ 0\leq j\leq n}}\delta _{i,m}\delta
_{j,n}\left( s_{i}\left( \left( s_{m-i}^{\prime }\right) _{\left\langle
-1\right\rangle }r_{j}\right) \#\xi \left( \left( s_{m-i}^{\prime }\right)
_{\left\langle 0\right\rangle }\otimes r_{n-j}^{\prime }\right) \right)
+\left( R\#_{\xi }H\right) _{m+n-1} \\
&=&\sum \left( s_{m}\left( \left( s_{0}^{\prime }\right) _{\left\langle
-1\right\rangle }r_{n}\right) \#\xi \left( \left( s_{0}^{\prime }\right)
_{\left\langle 0\right\rangle }\otimes r_{0}^{\prime }\right) \right)
+\left( R\#_{\xi }H\right) _{m+n-1}\text{ } \\
&&\overset{R_{0}=\Bbbk 1_{R}}{=}\sum s_{m}\left( \left( s_{0}^{\prime
}\right) _{\left\langle -1\right\rangle }r_{n}\right) \#\varepsilon
_{R}\left( \left( s_{0}^{\prime }\right) _{\left\langle 0\right\rangle
}\right) \varepsilon _{R}\left( r_{0}^{\prime }\right) 1_{H}+\left( R\#_{\xi
}H\right) _{m+n-1} \\
&=&\sum s_{m}\varepsilon _{R}\left( s_{0}^{\prime }\right) r_{n}\varepsilon
_{R}\left( r_{0}^{\prime }\right) \#1_{H}+\left( R\#_{\xi }H\right) _{m+n-1}
\\
&=&\sum_{\substack{ 0\leq i\leq m  \\ 0\leq j\leq n}}\delta _{i,m}\delta
_{j,n}\left( s_{i}\varepsilon _{R}\left( s_{m-i}^{\prime }\right)
r_{j}\varepsilon _{R}\left( r_{m-j}^{\prime }\right) \#1_{H}\right) +\left(
R\#_{\xi }H\right) _{m+n-1} \\
&=&\sum_{\substack{ 0\leq i\leq m  \\ 0\leq j\leq n}}\left( s_{i}\varepsilon
_{R}\left( s_{m-i}^{\prime }\right) r_{j}\varepsilon _{R}\left(
r_{m-j}^{\prime }\right) \#1_{H}\right) +\left( R\#_{\xi }H\right) _{m+n-1}
\\
&=&\sum \left( s^{\left( 1\right) }\varepsilon _{R}\left( s^{\left( 2\right)
}\right) r^{\left( 1\right) }\varepsilon _{R}\left( r^{\left( 2\right)
}\right) \#1_{H}\right) +\left( R\#_{\xi }H\right) _{m+n-1} \\
&=&\left( sr\#1_{H}\right) +\left( R\#_{\xi }H\right) _{m+n-1}\overset{(\ref%
{form:phiTrick})}{=}\varphi \left( sr+R_{m+n-1}\right) \\
&=&\varphi \left( \left( s+R_{m-1}\right) \left( r+R_{n-1}\right) \right)
=\varphi m_{\mathrm{gr}\left( R\right) }\left( \overline{s}\otimes \overline{%
r}\right) .
\end{eqnarray*}%
Let us check $\varphi $ is unitary. We have%
\begin{equation*}
\varphi \left( 1_{\mathrm{gr}\left( R\right) }\right) =\varphi \left(
1_{R}+R_{-1}\right) =\varphi \left( \overline{1_{R}}\right) =\overline{%
1_{R}\otimes 1_{H}}=\left( 1_{R}\otimes 1_{H}\right) +\left( R\#_{\xi
}H\right) _{-1}=1_{B}+B_{-1}=1_{G}.
\end{equation*}
\end{proof}

Summing up we have proved that

\begin{equation*}
\mathrm{gr}\left( Q\right) \overset{Q=R^{v}}{=}\mathrm{gr}\left(
R^{v}\right) \overset{\text{Lem. \ref{lem:ciccio}}}{\cong }\mathrm{gr}\left(
R\right) \overset{\text{Pro. \ref{pro:ciccio}}}{\cong }\mathcal{D}\left(
R\#_{\xi }H\right) \overset{\text{Pro. \ref{pro:D(f)}}}{\cong }\mathcal{D}%
\left( A\right)
\end{equation*}%
as bialgebras in ${_{H}^{H}\mathcal{YD}}$. Therefore $\mathrm{H}_{{\mathcal{%
YD}}}^{3}\left( \mathcal{D}\left( A\right) ,\Bbbk \right) =0$ (the
Hochschild cohomology in ${_{H}^{H}\mathcal{YD}}$ of the algebra $\mathcal{D}%
\left( A\right) $ with values in $\Bbbk $) if, and only if, $\mathrm{H}_{{%
\mathcal{YD}}}^{3}\left( \mathrm{gr}Q,\Bbbk \right) =0$. In this case, by
the foregoing, we get that $Q$ is gauge equivalent to a connected bialgebra
in ${_{H}^{H}\mathcal{YD}}$.

Now let $E$ be a connected bialgebra in ${_{H}^{H}\mathcal{YD}}$ and let $%
\gamma :E\otimes E\rightarrow \Bbbk $ be a gauge transformation in ${_{H}^{H}%
\mathcal{YD}}$ such that $Q=E^{\gamma }.$ We proved that $A^{\zeta }\cong
Q\#H\cong E^{\gamma }\#H$ as coquasi-bialgebras. By Proposition \ref%
{pro:deformSmash}, we have that $\left( E\#H\right) ^{\Gamma }=E^{\gamma
}\#H $ as an ordinary coquasi-bialgebras. Recall that two coquasi-bialgebras
$A$ and $A^{\prime }$ are called \textbf{gauge equivalent} or \textbf{%
quasi-isomorphic }whenever there is some gauge transformation $\gamma
:Q\otimes Q\rightarrow \Bbbk $ in $\mathbf{Vec}_{\Bbbk }$ such that $%
A^{\gamma }\cong A^{\prime }$ as coquasi-bialgebras. We point out that, if $A$ and $A'$ are ordinary bialgebras and $%
A^{\gamma }\cong A^{\prime }$, then $\gamma$ comes out to be a unitary cocycle. This is encoded in the triviality of the reassociators of $A$ and $A'$.

\begin{theorem}
\label{teo:main}Let $A$ be a finite-dimensional Hopf algebra over a field $%
\Bbbk $ of characteristic zero such that the coradical $H$ of $A$ is a
sub-Hopf algebra (i.e. $A$ has the dual Chevalley Property). If $\mathrm{H}_{%
{\mathcal{YD}}}^{3}\left( \mathcal{D}\left( A\right) ,\Bbbk \right) =0$,
then $A$ is quasi-isomorphic to the Radford-Majid bosonization $E\#H$ of
some connected bialgebra $E$ in ${_{H}^{H}\mathcal{YD}}$ by $H$. Moreover $%
\mathrm{gr}\left( E\right) \cong \mathcal{D}\left( A\right) $ as bialgebras
in ${_{H}^{H}\mathcal{YD}}$.
\end{theorem}

\begin{proof}
By the foregoing $A^{\zeta }\cong Q\#H\cong E^{\gamma }\#H=\left(
E\#H\right) ^{\Gamma }$ as coquasi-bialgebras. Now $A$ is quasi-isomorphic
to $A^{\zeta }$ which is quasi-isomorphic to $E\#H$ so that $A$ is
quasi-isomorphic to $E\#H.$ Moreover%
\begin{equation*}
\mathrm{gr}\left( E\right) =\mathrm{gr}\left( E^{\gamma }\right) =\mathrm{gr}%
\left( Q\right) \cong \mathcal{D}\left( A\right) .
\end{equation*}%
where the first equality holds by Proposition \ref{pro:grgaugeYD}.

\begin{invisible}
Let us check for our sake that the quasi-isomorphism is an equivalence
relation.

First recall that two coquasi-bialgebras $A$ and $B$ are called
quasi-isomorphism whenever $A^{\alpha }\cong B$ as coquasi-bialgebras for
some gauge transformation $\alpha :A\otimes A\rightarrow \Bbbk $ for $A$.
Write $A\sim B$ in this case.

Clearly $A\sim A$ taking $\alpha :=\varepsilon _{A}\otimes \varepsilon _{A}.$

If $A\sim B$, then there is $\alpha $ and an isomorphism of
coquasi-bialgebras $\sigma :B\rightarrow A^{\alpha }.$ Apply \cite[%
Proposition 2.5]{ABM} to this morphism and $v:=\alpha ^{-1}$ which is a
gauge transformation for $A^{\alpha },$ see \cite[Remark 2.2(ii)]{ABM}. Then
$\beta :=v\circ \left( \sigma \otimes \sigma \right) $ is a gauge
transformation for $B$ and $\sigma :B^{\beta }\rightarrow \left(
A^{v}\right) ^{v^{-1}}$ is also an isomorphism of coquasi-bialgebras. By
\cite[Remark 2.2(ii)]{ABM} we have $\left( A^{v}\right) ^{v^{-1}}\cong A$
and hence $B^{\beta }\cong A$ which means $B\sim A.$

Is $A\sim B$ and $B\sim C$ then there are $\alpha $ and $\beta $ such that $%
A^{\alpha }\cong B$ and $B^{\beta }\cong C.$

Call $\sigma :A^{\alpha }\rightarrow B$ the first isomorphism. Apply again
\cite[Proposition 2.5]{ABM} with $v:=\beta .$ Then $\alpha ^{\prime
}:=v\circ \left( \sigma \otimes \sigma \right) $ is a gauge transformation
for $A^{\alpha }$ and $\sigma :\left( A^{\alpha }\right) ^{\alpha ^{\prime
}}\rightarrow B^{\beta }$ is an isomorphism of coquasi-bialgebras. Thus $%
\left( A^{\alpha }\right) ^{\alpha ^{\prime }}\cong B^{\beta }\cong C.$
Since $A$ and $A^{\alpha }$ have the same coalgebra structure, we have that
the monoids $\left( \mathrm{Hom}_{\Bbbk }\left( A^{\alpha }\otimes A^{\alpha
},\Bbbk \right) ,\ast ,\varepsilon _{A^{\alpha }\otimes A^{\alpha }}\right) $
and $\left( \mathrm{Hom}_{\Bbbk }\left( A\otimes A,\Bbbk \right) \ast
,\varepsilon _{A\otimes A}\right) $ are the same. Thus we can regard $\alpha
^{\prime }$ and its inverse as elements in $\mathrm{Hom}_{\Bbbk }\left(
A\otimes A,\Bbbk \right) $ and we get that $\alpha ^{\prime }\ast \alpha
:A\otimes A\rightarrow \Bbbk $ is convolution invertible with convolution
inverse $\alpha ^{-1}\ast \left( \alpha ^{\prime }\right) ^{-1}.$ Since $A$
and $A^{\alpha }$ have also the same unit, we get that $\alpha ^{\prime
}\ast \alpha $ is unitary too and hence a gauge transformation. Let us check
that
\begin{equation*}
\left( A^{\alpha }\right) ^{\alpha ^{\prime }}=A^{\alpha ^{\prime }\ast
\alpha }.
\end{equation*}%
Note that%
\begin{eqnarray*}
&&\omega _{A^{\alpha ^{\prime }\ast \alpha }} \\
&=&\left( \varepsilon _{A}\otimes \left( \alpha ^{\prime }\ast \alpha
\right) \right) \ast \left( \alpha ^{\prime }\ast \alpha \right) \left(
A\otimes m_{A}\right) \ast \omega _{A}\ast \left( \alpha ^{\prime }\ast
\alpha \right) ^{-1}\left( m_{A}\otimes A\right) \ast \left( \left( \alpha
^{\prime }\ast \alpha \right) ^{-1}\otimes \varepsilon _{A}\right) \\
&=&\left(
\begin{array}{c}
\left( \varepsilon _{A}\otimes \alpha ^{\prime }\right) \ast \left(
\varepsilon _{A}\otimes \alpha \right) \ast \alpha ^{\prime }\left( A\otimes
m_{A}\right) \ast \alpha \left( A\otimes m_{A}\right) \\
\ast \omega _{A}\ast \left( \alpha ^{-1}\ast \left( \alpha ^{\prime }\right)
^{-1}\right) \left( m_{A}\otimes A\right) \ast \left( \left( \alpha
^{-1}\ast \left( \alpha ^{\prime }\right) ^{-1}\right) \otimes \varepsilon
_{A}\right)%
\end{array}%
\right) \\
&=&\left(
\begin{array}{c}
\left( \varepsilon _{A}\otimes \alpha ^{\prime }\right) \ast \left(
\varepsilon _{A}\otimes \alpha \right) \ast \alpha ^{\prime }\left( A\otimes
m_{A}\right) \ast \alpha \left( A\otimes m_{A}\right) \\
\ast \omega _{A}\ast \alpha ^{-1}\left( m_{A}\otimes A\right) \ast \left(
\alpha ^{\prime }\right) ^{-1}\left( m_{A}\otimes A\right) \ast \left(
\alpha ^{-1}\otimes \varepsilon _{A}\right) \ast \left( \left( \alpha
^{\prime }\right) ^{-1}\otimes \varepsilon _{A}\right)%
\end{array}%
\right) \\
&=&\left(
\begin{array}{c}
\left( \varepsilon _{A}\otimes \alpha ^{\prime }\right) \ast \alpha ^{\prime
}\left( A\otimes \left( \alpha \ast m_{A}\right) \right) \ast \alpha \left(
A\otimes m_{A}\right) \\
\ast \omega _{A}\ast \alpha ^{-1}\left( m_{A}\otimes A\right) \ast \left(
\alpha ^{\prime }\right) ^{-1}\left( \left( m_{A}\ast \alpha ^{-1}\right)
\otimes A\right) \ast \left( \left( \alpha ^{\prime }\right) ^{-1}\otimes
\varepsilon _{A}\right)%
\end{array}%
\right) \\
&=&\left(
\begin{array}{c}
\left( \varepsilon _{A}\otimes \alpha ^{\prime }\right) \ast \alpha ^{\prime
}\left( A\otimes \left( m_{A^{\alpha }}\ast \alpha \right) \right) \ast
\alpha \left( A\otimes m_{A}\right) \\
\ast \omega _{A}\ast \alpha ^{-1}\left( m_{A}\otimes A\right) \ast \left(
\alpha ^{\prime }\right) ^{-1}\left( \left( \alpha ^{-1}\ast m_{A^{\alpha
}}\right) \otimes A\right) \ast \left( \left( \alpha ^{\prime }\right)
^{-1}\otimes \varepsilon _{A}\right)%
\end{array}%
\right) \\
&=&\left(
\begin{array}{c}
\left( \varepsilon _{A}\otimes \alpha ^{\prime }\right) \ast \alpha ^{\prime
}\left( A\otimes m_{A^{\alpha }}\right) \ast \left( \varepsilon _{A}\otimes
\alpha \right) \ast \alpha \left( A\otimes m_{A}\right) \\
\ast \omega _{A}\ast \alpha ^{-1}\left( m_{A}\otimes A\right) \ast \left(
\alpha ^{-1}\otimes \varepsilon _{A}\right) \ast \left( \alpha ^{\prime
}\right) ^{-1}\left( m_{A^{\alpha }}\otimes A\right) \ast \left( \left(
\alpha ^{\prime }\right) ^{-1}\otimes \varepsilon _{A}\right)%
\end{array}%
\right) \\
&=&\left( \varepsilon _{A^{\alpha }}\otimes \alpha ^{\prime }\right) \ast
\alpha ^{\prime }\left( A^{\alpha }\otimes m_{A^{\alpha }}\right) \ast
\omega _{A^{\alpha }}\ast \left( \alpha ^{\prime }\right) ^{-1}\left(
m_{A^{\alpha }}\otimes A^{\alpha }\right) \ast \left( \left( \alpha ^{\prime
}\right) ^{-1}\otimes \varepsilon _{A^{\alpha }}\right) \\
&=&\omega _{\left( A^{\alpha }\right) ^{\alpha ^{\prime }}}.
\end{eqnarray*}%
Moreover%
\begin{eqnarray*}
m_{A^{\alpha ^{\prime }\ast \alpha }} &=&\left( \alpha ^{\prime }\ast \alpha
\right) \ast m_{A}\ast \left( \alpha ^{\prime }\ast \alpha \right) ^{-1} \\
&=&\alpha ^{\prime }\ast \left( \alpha \ast m_{A}\ast \alpha ^{-1}\right)
\ast \left( \alpha ^{\prime }\right) ^{-1} \\
&=&\alpha ^{\prime }\ast m_{A^{\alpha ^{\prime }}}\ast \left( \alpha
^{\prime }\right) ^{-1}=m_{\left( A^{\alpha }\right) ^{\alpha ^{\prime }}}.
\end{eqnarray*}%
Thus $\left( A^{\alpha }\right) ^{\alpha ^{\prime }}=A^{\alpha ^{\prime
}\ast \alpha }$ as coquasi-bialgebras. We conclude that $A^{\alpha ^{\prime
}\ast \alpha }\cong C$ and hence $A\sim C.$
\end{invisible}
\end{proof}

More generally, given $A$ a (finite-dimensional) Hopf algebra whose coradical $H$ is a sub-Hopf algebra, then if $H$ is also semisimple, we expect that $A$ is quasi-isomorphic to the Radford-Majid bosonization $E\#H$ of
some connected bialgebra $E$ in ${_{H}^{H}\mathcal{YD}}$ by $H$. See e.g. \cite[Corollary 3.4 and the proof therein]{Grunenfelder-Mastnak} and \cite{AAGMV,AAG} for a further clue in this direction.

\section{Examples} \label{sec:6}

We notice that the Hochschild cohomology of a finite-dimensional Nichols algebras has been computed in few examples.
We consider here those Nichols algebras to compute $\mathrm{H}_{{\mathcal{YD}}}^{3}\left( \mathcal{B}\left( V\right)
,\Bbbk \right)$.

\subsection{Braidings of Cartan type}

Let $A=(a_{ij})_{1\le i,j\le\theta}$ be a finite Cartan matrix, $\Delta$ the corresponding root system, $(\alpha_i)_{1\le i\le\theta}$
a set of simple roots and $W$ its Weyl group.
Let $w_0=s_{i_1}\cdots s_{i_M}$ be a reduced expression of the element $w_0\in W$ of maximal length as a product of simple
reflections, $\beta_j=s_{i_1}\cdots s_{i_{j-1}}(\alpha_{i_j})$, $1\le j\le M$. Then $\beta_j\neq\beta_k$ if $j\neq k$ and
$\Delta^+=\{\beta_j|1\le j\le M\}$,  see \cite[page 25 and Proposition 3.6]{Hiller}.

\begin{invisible}
This can be deduced as follows. By \cite[page 25]{Hiller}, there is $w_0\in W$ such that $w_0(\Delta^+)=\Delta^-$ and $M=l(w_0)=|\Delta^+|$. Since $\Delta^+=-\Delta^-$, we also have $w_0(\Delta^-)=\Delta^+$. By \cite[Proposition 3.6]{Hiller} we have $\Delta^+=\Delta^+\cap w_0(\Delta^-)=\{\beta_1,\ldots, \beta_M\}$. Since $M=l(w_0)=|\Delta^+|$, then $\beta_1,\ldots, \beta_M$ are distinct.
\end{invisible}

Let $\Gamma$ be a finite abelian group, $\widehat{\Gamma}$ its group of characters.
$\cD=(\Gamma,(g_i)_{1\le i\le\theta},(\chi_i)_{1\le i\le\theta},A)$ is a \emph{datum of finite Cartan type} \cite{AS-Classif} associated to
$\Gamma$ and $A$ if $g_i\in\Gamma$, $\chi_j\in\widehat{\Gamma}$, $1\le i,j\le\theta$, satisfy $\chi_i(g_i)\neq 1$,
$\chi_i(g_j)\chi_j(g_i)=\chi_i(g_i)^{a_{ij}}$ for all $i,j$. Set $\bq=(q_{ij})_{1\le i,j\le\theta}$, where $q_{ij}=\chi_j(g_i)$.

In what follows $V$ denotes the Yetter-Drinfeld module over $\Bbbk\Gamma$, $\dim V=\theta$, with a fixed basis $x_1,\ldots, x_\theta$, where the action and the coaction over each $x_i$ is given by $\chi_i$ and $g_i$, respectively. Then the associated braiding is $c(x_i\otimes x_j)=q_{ij}x_j\otimes x_i$ for all $i,j$.
Let $\cB_\bq=\cB(V)$. The tensor algebra $T(V)$ is $\N_0^\theta$-graded with grading $\alpha_i$ for each $x_i$.  For $\beta=\sum_{i=1}^\theta a_i\alpha_i\in\Delta^+$, set
\begin{align*}
g_\beta &= g_1^{a_1}\cdots g_\theta^{a_\theta} , & \chi_\beta &= \chi_1^{a_1}\cdots \chi_\theta^{a_\theta}, & q_\beta&=\chi_\beta(g_\beta).
\end{align*}
Given $\alpha,\beta\in\Delta^+$, we denote $q_{\alpha\beta}=\chi_\beta(g_\alpha)$.

We assume as in \cite{AS-Classif,MPSW} that \emph{the order of $q_{ii}$ is odd for all $i$, and not divisible by 3 for each
connected component of the Dynkin diagram of $A$ of type $G_2$}. Therefore the order of $q_{ii}$ is the same for all the $i$
in the same connected component $J$. Given $\beta\in J$, we denote by $N_\beta$ the order of the corresponding $q_{ii}$ in $J$,
which is also the order of $q_\beta$.

By \cite{L} there exist homogeneous elements $x_{\beta}$ of degree $\beta$, $\beta\in\Delta^+$, such that the Nichols algebra $\cB_\bq$ of $V$
is presented by generators $x_1,\dots,x_\theta$ and relations
\begin{align*}
(\ad_c x_i)^{1-a_{ij}}x_j&=0, & &1\le i\neq j\le\theta; \\
x_\beta^{N_\beta}&=0, &  &\beta\in \Delta_+.
\end{align*}
Moreover $\{x_{\beta_1}^{n_1}\dots x_{\beta_M}^{n_M}| 0\le n_i<N_{\beta_i}\}$ is a basis of $\cB_\bq$.

We shall prove that $\mathrm{H}_{{\mathcal{YD}}}^{3}\left( \mathcal{B}_\bq,\Bbbk \right)=0$. We need first some technical results.

\begin{lemma}\label{lemma:non trivial pair 1}
Let $\alpha,\beta\in\Delta_+$. Then either $g_\alpha g_\beta^{N_\beta}\neq e$, or else $\chi_\alpha \chi_\beta^{N_\beta}\neq \epsilon$.
\end{lemma}
\begin{proof}
Suppose on the contrary that $g_\alpha g_\beta^{N_\beta}=e$, $\chi_\alpha \chi_\beta^{N_\beta}=\epsilon$. Then
$$ q_\alpha=\chi_\alpha^{-1}(g_\alpha^{-1})= \chi_\beta^{N_\beta}(g_\beta^{N_\beta})=q_\beta^{N_\beta^2}=1, $$
since $q_\beta$ is a root of unity of order $N_\beta$. But this is a contradiction, since $q_\alpha\neq 1$.
\end{proof}

\begin{lemma}\label{lemma:non trivial pair 2}
Let $\alpha,\beta,\gamma\in\Delta^+$ be pairwise different. Then either $g_\alpha g_\beta g_\gamma \neq e$, or else
$\chi_\alpha \chi_\beta \chi_\gamma \neq \epsilon$.
\end{lemma}
\begin{proof}
Suppose on the contrary that $g_\alpha g_\beta g_\gamma=e$ and $\chi_\alpha \chi_\beta \chi_\gamma=\epsilon$. Then
\begin{align}\label{eq:conditions alpha,beta,gamma}
q_{\alpha}&= \chi_\alpha^{-1}(g_\alpha^{-1})= \chi_\beta\chi_\gamma(g_\beta g_\gamma)=q_\beta q_\gamma q_{\beta\gamma}q_{\gamma\beta}, &
q_\beta&=q_\alpha q_\gamma q_{\alpha\gamma}q_{\gamma\alpha}, & q_\gamma&=q_\alpha q_\beta q_{\alpha\beta}q_{\beta\alpha}.
\end{align}
Notice that $\alpha,\beta,\gamma$ belong to the same connected component. Indeed, if $\gamma$ belongs
to a different connected component, then $q_{\beta\gamma}q_{\gamma\beta}=q_{\alpha\gamma}q_{\gamma\alpha}=1$. Thus
$q_\beta=q_\alpha q_\gamma = q_\beta q_\gamma^2$, so $q_\gamma^2=1$, which is a contradiction.
Therefore we may assume that the Dynkin diagram is connected.

One can prove that $q_{s_i(\alpha)}=q_{\alpha}$ for every $\alpha\in\Delta$. As we observed that $\Delta^+=\{\beta_j|1\le j\le M\}$, we deduce that   for every $\beta\in\Delta^+$ there is some $j$ such that $q_{\beta}=q_{j}$.
One can prove that there is some $q\in\Bbbk$ such that $q_{\alpha}=q^{(\alpha,\alpha)/2}$ and $q_{\alpha\gamma}q_{\gamma\alpha}=q^{(\alpha,\gamma)}$, where $(\cdot,\cdot)$ is the invariant bilinear form on
the simple Lie algebra $\mathfrak{g}$ associated with the finite Cartan matrix  \cite[Ch. VI, $\S1$, Proposition 3 and Definition 3]{B} and the basis of the root systems given in \cite[Ch. VI, $\S 4$]{B} should be normalized in such a way that $q=q_\delta$, $(\delta,\delta)={ 2}$ for each short root $\delta\in\Delta$. Note that $q_{\alpha}=q^{(\alpha,\alpha){/2}}\neq 1$ for all $\alpha$ as $(\alpha,\alpha)\neq 0$. Thus
\begin{itemize}
  \item $q_\alpha=q_\beta=q_\gamma=q$ if the Dynkin diagram is simply laced,
  \item $q_\alpha,q_\beta,q_\gamma\in\{q,q^2\}$ if the Dynkin diagram has a double arrow,
  \item $q_\alpha,q_\beta,q_\gamma\in\{q,q^3\}$ if the Dynkin diagram is of type $G_2$.
\end{itemize}
 If the Dynkin diagram is simply laced, then, by \eqref{eq:conditions alpha,beta,gamma}, we have $q_{\beta\gamma}q_{\gamma\beta}=q_{\alpha\gamma}q_{\gamma\alpha}=q_{\alpha\beta}q_{\beta\alpha}=q^{-1}$. Then $q^{(\alpha,\gamma)}=q^{-1}$.  Now set $n(\alpha,\beta):=2(\alpha,\beta)/(\beta,\beta)=(\alpha,\beta)$. Then $n(\alpha,\beta)$ is symmetric whence, by \cite[Ch. VI, $\S1$, page 148]{B} we have $(\alpha,\gamma)=-1$ as the order of $q$ is odd,
 so $\alpha+\gamma\in\Delta^+$,  by \cite[VI, $\S1$, Corollary, page 149]{B}. Now  the same argument we used above shows that also $(\alpha,\beta)=-1=(\gamma,\beta)$ and hence $(\alpha+\gamma,\beta)=-2$, so $\alpha+\beta+\gamma\in\Delta^+$, since $\alpha+\gamma\neq -\beta$ (as $\alpha+\gamma$ and $\beta$ are both in $\Delta^+$). But
$ q_{\alpha+\beta+\gamma}= q_{\alpha}q_\beta q_\gamma q_{\beta\gamma}q_{\gamma\beta}q_{\alpha\gamma}q_{\gamma\alpha}
q_{\alpha\beta}q_{\beta\alpha}=1$, which is a contradiction.

If the Dynkin diagram has a double arrow, then $q_{\alpha}$, $q_\beta$, $q_\gamma\in\{q,q^2\}$.
%[It is not clear to us that $q_\gamma\in\{q,q^2\}$. For example, in $B_3$, take $\gamma=\epsilon_1+\epsilon_3=\alpha_1+\alpha_2+2\alpha_3$. Then $q_\gamma=q^5$, where $q:=q_1=q_2$ and $q_3=q^2$.]
If $q_{\alpha}=q_\beta=q_\gamma$, then the proof follows as for the simply-laced case because $n(u,v)=n(v,u)$ for $u,v\in\{\alpha,\beta,\gamma\}$. If $q_\alpha=q_\beta=q$ and $q_\gamma=q^2$,
then $q_{\beta\gamma}q_{\gamma\beta}=q_{\alpha\gamma}q_{\gamma\alpha}=q^{-2}$, and $q_{\alpha\beta}q_{\beta\alpha}=1$, by \eqref{eq:conditions alpha,beta,gamma}.
Then a simple calculation yields $(\beta,\gamma)=-2$ so that $\beta+\gamma\in\Delta^+$. One also gets $(\alpha,\beta)=0$ and $(\alpha,\gamma)=-2$ so that $(\alpha,\beta+\gamma)=(\alpha,\beta)+(\alpha,\gamma)=-2<0$ by the conditions on the order of $q$, so again $\alpha+\beta+\gamma\in\Delta^+$; but again we obtain
$q_{\alpha+\beta+\gamma}=1$, which is a contradiction. The proof for $q_\alpha=q_\beta=q^2$ and $q_\gamma=q$ follows analogously.

Finally, if the Dynkin diagram is of type $G_2$, then a similar analysis gives a contradiction.
\end{proof}

For each $1\le k\le M$, set $\cB_\bq(k)$ as the subspace of $\cB_\bq$ spanned by $\{x_{\beta_1}^{n_1}\dots x_{\beta_k}^{n_k}| 0\le n_i<N_{\beta_i}\}$.
By \cite{DP} this gives an algebra filtration, and the graded algebra $\Gr\cB_\bq$ associated to this filtration is presented by generators
$\bx_\beta$, $\beta\in\Delta^+$, and relations
\begin{align*}
\bx_\beta\bx_\gamma &= q_{\beta\gamma} \bx_\gamma\bx_\beta, & \bx_\beta^{N_\beta}&=0, & \beta & <\gamma\in\Delta_+.
\end{align*}
In \cite{MPSW} $\Gr\cB_\bq$ is viewed as an algebra in $\ydg$, which (as an algebra) is the Nichols algebra of Cartan type
$A_1\times\dots\times A_1$, $M$ copies,
with action and coaction on $\bx_\beta$ given by $\chi_\beta$, $g_\beta$, respectively. By \cite[Theorem 4.1]{MPSW}, $\mathrm{H}^\bullet(\Gr\cB_\bq,\Bbbk)$ is the algebra generated by $\xi_\beta$, $\eta_\beta$, $\beta\in\Delta^+$, where $\deg\xi_\beta=2$, $\deg\eta_\beta=1$, and relations
\begin{align*}
\xi_\beta\xi_\gamma &= q_{\beta\gamma}^{N_\beta N_\gamma} \xi_\gamma\xi_\beta, &
\eta_\beta\xi_\gamma &= q_{\beta\gamma}^{N_\gamma} \xi_\gamma\eta_\beta, &
\eta_\beta\eta_\gamma &= -q_{\beta\gamma} \eta_\gamma\eta_\beta, &
\beta,\gamma & \in\Delta^+.
\end{align*}
As we assume that all the $q_{ii}$ have odd order, we deduce in particular from the last equality that $\eta_\beta^2=0$ for all $\beta\in\Delta^+$.
%From $\eta_\beta\eta_\gamma = -q_{\beta\gamma} \eta_\gamma\eta_\beta$ we deduce that $\eta_\beta^2= -q_{\gamma} \eta_\beta^2$. If $\eta_\beta^2\neq 0$ then $q_{\gamma}=-1$ and hence $q_{\gamma}^2=1.$ Now $q_{\gamma}=q^t$ for some $t$ and hence $q^{2t}=1.$ Since $q$ has odd order, we deduce that $o(q)\mid t$ so that $q_{\gamma}=q^t=-1$, contradiction.
As an algebra in $\ydg$,
the action and coaction on $\xi_\beta$ is given by $\chi_\beta^{-N_\beta}$, $g_\beta^{-N_\beta}$, while
the action and coaction on $\eta_\beta$ is given by $\chi_\beta^{-1}$, $g_\beta^{-1}$.

\begin{theorem}
$\mathrm{H}_{{\mathcal{YD}}}^{3}\left( \cB_\bq,\Bbbk \right)=0$.
\end{theorem}

\begin{proof}
First we will prove that $\mathrm{H}^{3}\left( \Gr\cB_\bq,\Bbbk \right)^{D} =0$ for $D:=D(\Bbbk\Gamma).$ Now, the invariants are with respect to the ${D}$-bimodule structure that $\mathrm{H}^{3}\left( \Gr\cB_\bq,\Bbbk \right)$ inherits from $\mathrm{Hom}\left( (\Gr\cB_\bq)^{\otimes 3},\Bbbk \right)$ (this is a ${D}$-bimodule as its arguments are left ${D}$-modules). Since the left $D$-module structure is induced by the one of $\Bbbk$, it is trivial. Thus the invariants of $\mathrm{H}^{3}\left( \Gr\cB_\bq,\Bbbk \right)$ as a $D$-bimodule reduce to the its invariants as a right $D$-module. Since right  $D$-modules are equivalent to left $D$-modules, via the antipode of $D$ which is invertible as $D$ is finite-dimensional, the right $D$-module structure of $\mathrm{H}^{3}\left( \Gr\cB_\bq,\Bbbk \right)$ becomes the structure of object in $\ydg$ described above. Thus, in order to prove that $\mathrm{H}^{3}\left( \Gr\cB_\bq,\Bbbk \right)^{D} =0$ we just have to check that the invariants of $\mathrm{H}^{3}\left( \Gr\cB_\bq,\Bbbk \right)$ as a left-left Yetter-Drinfeld modules are zero.

Now, by the defining relations of $\mathrm{H}^\bullet(\Gr\cB_\bq,\Bbbk)$, a basis $B$ of $\mathrm{H}^3(\Gr\cB_\bq,\Bbbk)$ is given by
$\{\xi_\alpha\eta_\beta\}\cup\{\eta_\alpha\eta_\beta\eta_\gamma|\alpha<\beta<\gamma\}$.
If $v\in\mathrm{H}^3(\Gr\cB_\bq,\Bbbk)$ is invariant, then $v$ is written as a linear combination of elements in the trivial component. Indeed, write $v=\sum_{b\in B} c_b \, b$ for some $c_b\in\Bbbk$, and let $g_b$, $\chi_b$ be the elements describing the component of $b\in B$. Then
\begin{align*}
v &= g\cdot v = \sum_{b\in B} c_b \, g\cdot b = \sum_{b\in B} c_b \chi_b(g) \, b, & \mbox{for all }& g\in \Gamma, \\
1\otimes v &= \rho(v) = \sum_{b\in B} c_b \, \rho\cdot b = \sum_{b\in B} c_b g_b \otimes b.
\end{align*}
If $c_b\neq 0$, then $\chi_b(g)=1$ for all $g\in \Gamma$ so $\chi_b=\epsilon$, and $g_b=1$. Thus $b$ is invariant. We have so proved that the existence of $v\neq 0$ invariant implies the existence of $b\in B$ invariant. Hence, if $B$ has no invariant element then there is no invariant element at all.
Note that, for all $h\in H$, we have $h\cdot(\xi_\alpha\eta_\beta)=(\chi_\alpha^{-N_\alpha}\chi_\beta^{-1})(h)\xi_\alpha\eta_\beta$ and $\rho(\xi_\alpha\eta_\beta)=g_\alpha^{-N_\alpha}g_\beta^{-1}\otimes\xi_\alpha\eta_\beta$ so that, by Lemma \ref{lemma:non trivial pair 1}, the element $\xi_\alpha\eta_\beta$ is not $D$-invariant. A similar argument, using Lemma \ref{lemma:non trivial pair 2}, shows that also $\eta_\alpha\eta_\beta\eta_\gamma$ is not $D$-invariant. Thus the elements in $B$ are not $D$-invariant, so  $\mathrm{H}^{3}\left( \Gr\cB_\bq,\Bbbk \right)^{D} =0$. Since the elements in $\{x_{\beta_1}^{n_1}\dots x_{\beta_k}^{n_k}| 0\le n_i<N_{\beta_i}\}$ are eigenvectors for $D$, we can mimic the argument in \cite[Section 5]{MPSW} by taking into account the spectral sequence associated to the filtration of algebras therein; see for example \cite[Corollary 5.5]{MPSW} for a similar argument.  Thus $\mathrm{H}_{{\mathcal{YD}}}^{3}\left( \cB_\bq,\Bbbk \right)\cong\mathrm{H}^{3}\left( \cB_\bq,\Bbbk \right)^{D} =0$.
\end{proof}

\begin{remark}
Notice that $\mathrm{H}_{{\mathcal{YD}}}^{3}\left( \cB_\bq,\Bbbk \right)\cong \mathrm{H}^{3}\left( \cB_\bq,\Bbbk \right)^{D(\Bbbk\Gamma)} =0$
although $\mathrm{H}^3\left( \cB_\bq\#\Bbbk\Gamma,\Bbbk \right)\cong \mathrm{H}^{3}\left( \cB_\bq,\Bbbk \right)^{\Gamma}$ can be non-trivial, see for example \cite[Example 5.8]{MPSW}.
\end{remark}

\subsection{Braidings of non-diagonal type}

For $n\geq3$, $\mathcal{FK}_{n}$ denotes the quadratic algebra \cite{FK} with a presentation by generators $x_{(ij)}$, $1\leq i<j\leq n$,
and relations
\begin{align*}
x_{(ij)}^{2}&=0,&  &1\leq i<j\leq n,\\
x_{(ij)}x_{(jk)}&=x_{(jk)}x_{(ik)}+x_{(ik)}x_{(ij)},&  &1\leq i<j<k\leq n,\\
x_{(jk)}x_{(ij)}&=x_{(ik)}x_{(jk)}+x_{(ij)}x_{(ik)},&  &1\leq i<j<k\leq n,\\
x_{(ij)}x_{(kl)}&=x_{(kl)}x_{(ij)},&  & \# \{i,j,k,l\}=4.
\end{align*}
According to \cite{MiS} each $\mathcal{FK}_n$ is a graded bialgebra in the category of Yetter-Drinfeld modules over the symmetric group $S_n$,
generated as an algebra by the vector space $V_n$ with basis $\{x_{(ij)}\mid 1\leq i<j\leq n\}$. The action is described by identifying $(ij)$
with the corresponding transposition in $S_n$ and then consider the conjugation twisted by the sign, while the coaction is given by
declaring $x_\sigma$ a homogeneous element of degree $\sigma$. Then the braiding on $V_n$ becomes
\[ c(x_\sigma\otimes x_\tau)=\chi(\sigma,\tau)x_{\sigma\tau\sigma^{-1}}\otimes x_\sigma,\quad\quad
 \chi(\sigma,\tau)=
 \begin{cases}
 	1 & \sigma(i)<\sigma(j), \tau=(ij), \, i<j,\\
  -1 & \text{otherwise,}
\end{cases}
\]
where $\sigma$ and $\tau$ are transpositions. Moreover $\mathcal{FK}_n$ projects onto the Nichols algebra $\cB(V_n)$. For $n=3,4,5$, it is known that $\mathcal{FK}_n=\cB(V_n)$ and
has dimension, respectively, $12$, $576$ and $8294400$.

The Hochschild cohomology of $\mathcal{FK}_3$ is a consequence of the results in \cite{SVay} as follows.

\begin{theorem}
$\mathrm{H}_{\Bbbk S_3\text{-}\mathrm{Mod}}^{\bullet}\left( \mathcal{FK}_3 ,\Bbbk \right)$ is isomorphic to the graded algebra
\begin{align*}
&\Bbbk[X,U,V]/(U^2V-VU^2), & \mbox{where }\deg U=\deg V=2, & \, \deg X=4.
\end{align*}
\end{theorem}

\begin{proof}
By \cite[Theorem 4.19]{SVay}, we have that $E(B\#\Bbbk S_3)$ is isomorphic to the algebra in the claim, where $B=\mathcal{FK}_3$. By \cite[Theorem 2.17]{SVay}, we know that $E(B\#\Bbbk S_3)\cong E(B)^{\Bbbk S_3}$ as graded algebras. As observed in Remark \ref{rem:DV}, we have that $E(B)\cong \mathrm{H}^{\bullet}\left( B ,\Bbbk \right)$. By Remark \ref{rem:ExactInv}, we have $\mathrm{H}^{\bullet}\left( B ,\Bbbk \right)^{\Bbbk S_3}\cong\mathrm{H}_{\Bbbk S_3\text{-}\mathrm{Mod}}^{\bullet}\left( \mathcal{FK}_3 ,\Bbbk \right)$.
\end{proof}

From this result we get $\mathrm{H}_{\Bbbk S_3\text{-}\mathrm{Mod}}^{3}\left( \mathcal{FK}_3 ,\Bbbk \right)=0$ so that, by Proposition \ref{pro:D(H)} we conclude that

\begin{corollary}
$\mathrm{H}_{{\mathcal{YD}}}^{3}\left( \mathcal{FK}_3,\Bbbk \right)=0$.
\end{corollary}

\end{document}